\documentclass[12pt,letterpaper]{article}
\usepackage{graphicx}
\usepackage{subcaption}
\usepackage{multirow}
\usepackage{amssymb,amsmath,amsthm,latexsym,tikz,pgf}
\usepackage{mathrsfs} % this is for ``\mathscr''
\usepackage{amssymb}
\usepackage{epsfig}
\usepackage{epstopdf}
\usepackage{hyperref}
\usepackage{soul} %% for "\ul" instead of "\underline"
\usepackage{mathtools}% for "\xhookrightarrow"

\newtheorem{thm}{Theorem}[section]
\newtheorem{cor}[thm]{Corollary}
\newtheorem{lem}[thm]{Lemma}

\newtheorem{defn}[thm]{Definition}
\newtheorem{rem}[thm]{Remark}
\renewcommand{\S}{\mathcal{S}}
\renewcommand{\SS}{\mathscr{S}}

\newcommand{\C}{\mathcal{C}}
\newcommand{\R}{\mathcal{R}}
\newcommand{\T}{\mathcal{T}}
\renewcommand{\L}{\mathcal{L}}
\renewcommand{\H}{\mathcal{H}}
\renewcommand{\P}{\mathcal{P}}
\renewcommand{\O}{\mathcal{O}}

\newcommand{\B}{\mathcal{B}}
\newcommand{\D}{\mathcal{D}}
\newcommand{\F}{\mathcal{F}}
\newcommand{\Z}{\mathcal{Z}}
\newcommand{\x}{\mathbf{x}}
\renewcommand{\phi}{\varphi}
\newcommand{\RR}{\mathbb{R}}
\newcommand{\RRn}{\mathbb{R}^n}
\newcommand{\RRp}{\mathbb{R}_{>0}}
\newcommand{\RRpn}{\mathbb{R}_{>0}^n}

\newcommand{\eps}{\varepsilon}
\newcommand{\dist}{\mathrm{dist}}
\newcommand{\et}{\tilde\eps}

\newcommand*{\medcap}{\mathbin{\scalebox{1.25}{\ensuremath{\cap}}}}%
\renewcommand{\bar}{\overline}
\newcommand{\fiber}{\mathit{fiber}}
\newcommand{\nbhd}{\mathit{nbhd}}
\renewcommand{\a}{\alpha}
\renewcommand{\d}{\Delta}
\begin{document}

\title{Toric Differential Inclusions and \\ a Proof of the Global Attractor Conjecture}

\author{
Gheorghe Craciun\\ 
Department of Mathematics and \\ 
Department of Biomolecular Chemistry\\
University of Wisconsin-Madison\\
%Madison, WI 53706\\
e-mail: \texttt{craciun@math.wisc.edu}\\
\\
\\
{\tt version 3.0}
}

\date{\today}

\maketitle

\begin{abstract}
\noindent The  {\it global attractor conjecture} says that {\em toric dynamical systems} 
%(i.e., a class of polynomial dynamical systems on the positive orthant) 
{\em have a globally attracting point} (up to linear conservation relations)
%within each affine invariant subspace
 --- or, equivalently, {\em complex balanced systems have a globally attracting point within each stoichiometric compatibility class}.  
A proof of this conjecture implies that a large class of nonlinear dynamical systems on the positive orthant have very simple and stable dynamics. The conjecture originates from the 1972 breakthrough work by Fritz Horn and Roy Jackson, and was  formulated in its current  form by Horn in 1974.

Toric dynamical systems can be embedded into  {\it toric differential inclusions}. We show that each bounded positive solution of a toric differential inclusion is contained in an invariant region that prevents it from approaching the boundary of the positive orthant. %In particular, we show that similar invariant regions prevent positive solutions of weakly reversible variable-$k$  polynomial dynamical systems from approaching the origin. 
We use this result to prove the global attractor conjecture. 
 In particular, it  follows that all detailed balanced mass action systems and all deficiency zero weakly reversible systems have the global attractor property.
\end{abstract}

\tableofcontents

\newpage

\section{Introduction}

Any autonomous polynomial dynamical system\footnote{or more generally any {\em power law system}~\cite{CNP}, since we do not assume that the coordinates of the vertex points $y$ are either integer or non-negative. In particular, if vertex points have non-negative integer coordinates  we obtain {\em mass-action systems}.
} 
on the strictly positive orthant $\RR^n_{>0}$ can be represented as 
\begin{equation}\label{polynomial}
\frac{dx}{dt} = \sum_{y \to y' \in G} k_{y \to y'} x^{y} (y' - y) 
\end{equation}
for some Euclidean embedded graph\footnote{
A {\em Euclidean embedded graph} $G$ is a finite directed graph whose set of vertices is $Y  \subset \RR^n$. See Fig.~\ref{fig:1} for examples.
} 
$G$ and some positive constants $k_{y \to y'}$, one for each edge\footnote{Inspired by notation from reaction networks, we denote an directed edge from $y$ to $y'$ by $y \to y'$. Moreover, if $y \to y'$ is an edge in $G$ we simply write $y \to y' \in G$.
}  
$y \to y'$ of $G$. 
(Here $x^y$ denotes the monomial $\prod_{i=1}^n {x_i} ^ {y_i}$.)
Similarly, any non-autonomous polynomial dynamical system on $\RR^n_{>0}$ can be represented as 
\begin{equation}\label{polynomial_nonaut}
\frac{dx}{dt} = \sum_{y \to y' \in G} k_{y \to y'}(t)x^{y} (y' - y) 
\end{equation}
for some nonnegative scalar functions $k_{y \to y'}(t)$.

\begin{figure}[h!]
  \begin{center}
    \begin{tabular}{c}
      \hskip-1.0cm
      \includegraphics[width=2.7in]{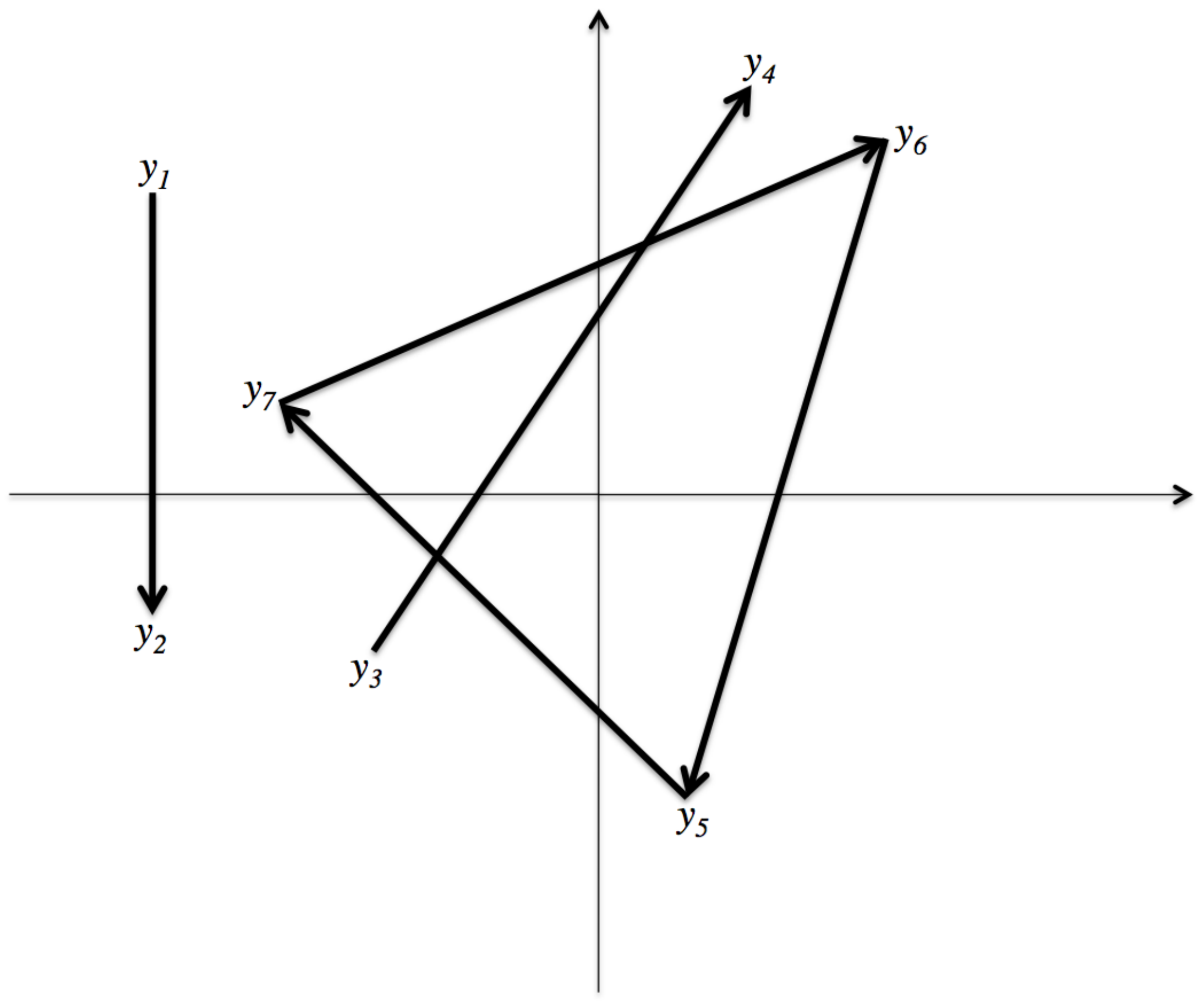}
      \includegraphics[width=2.7in]{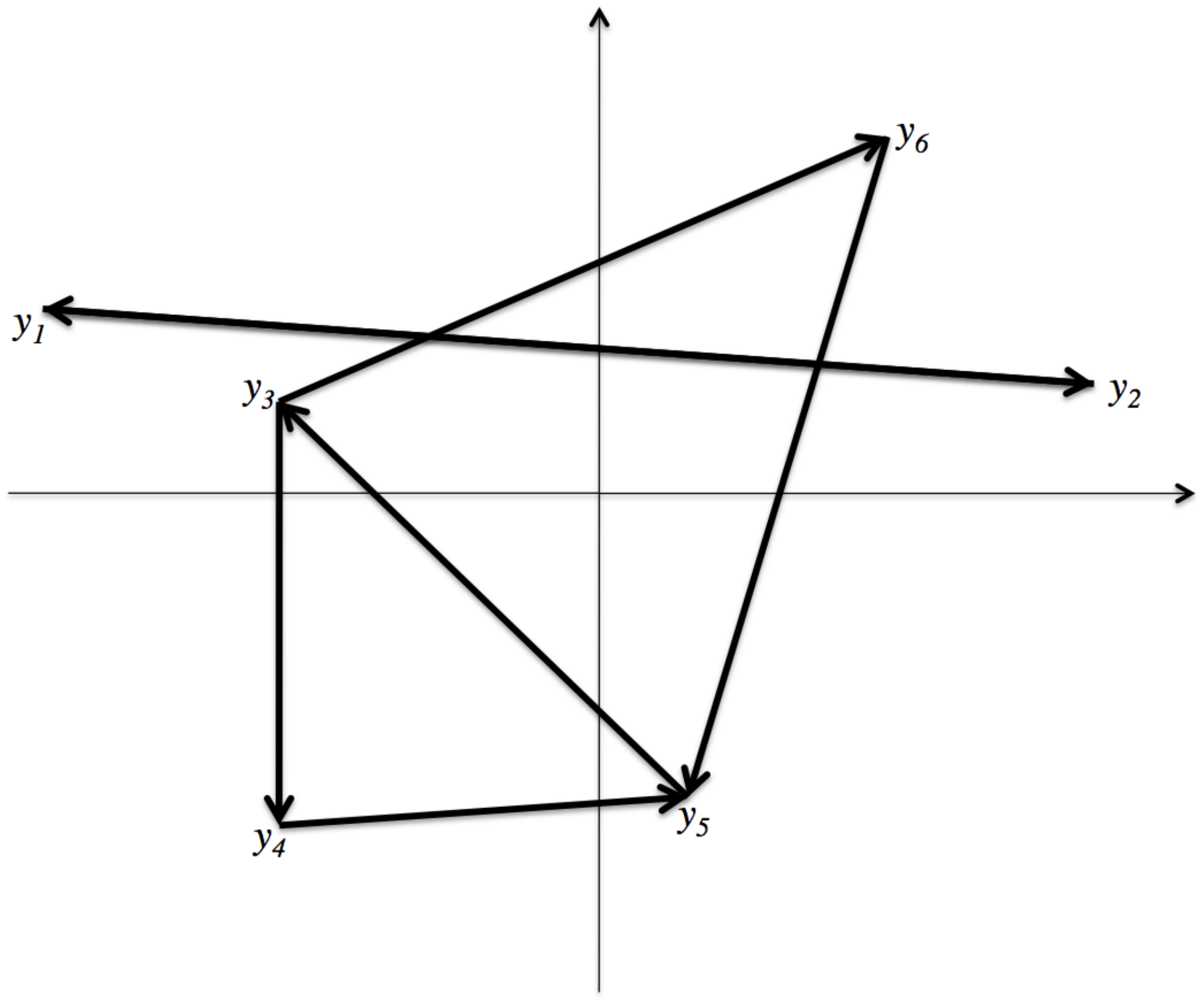}\\
      \hskip-1.0cm
      \includegraphics[width=2.7in]{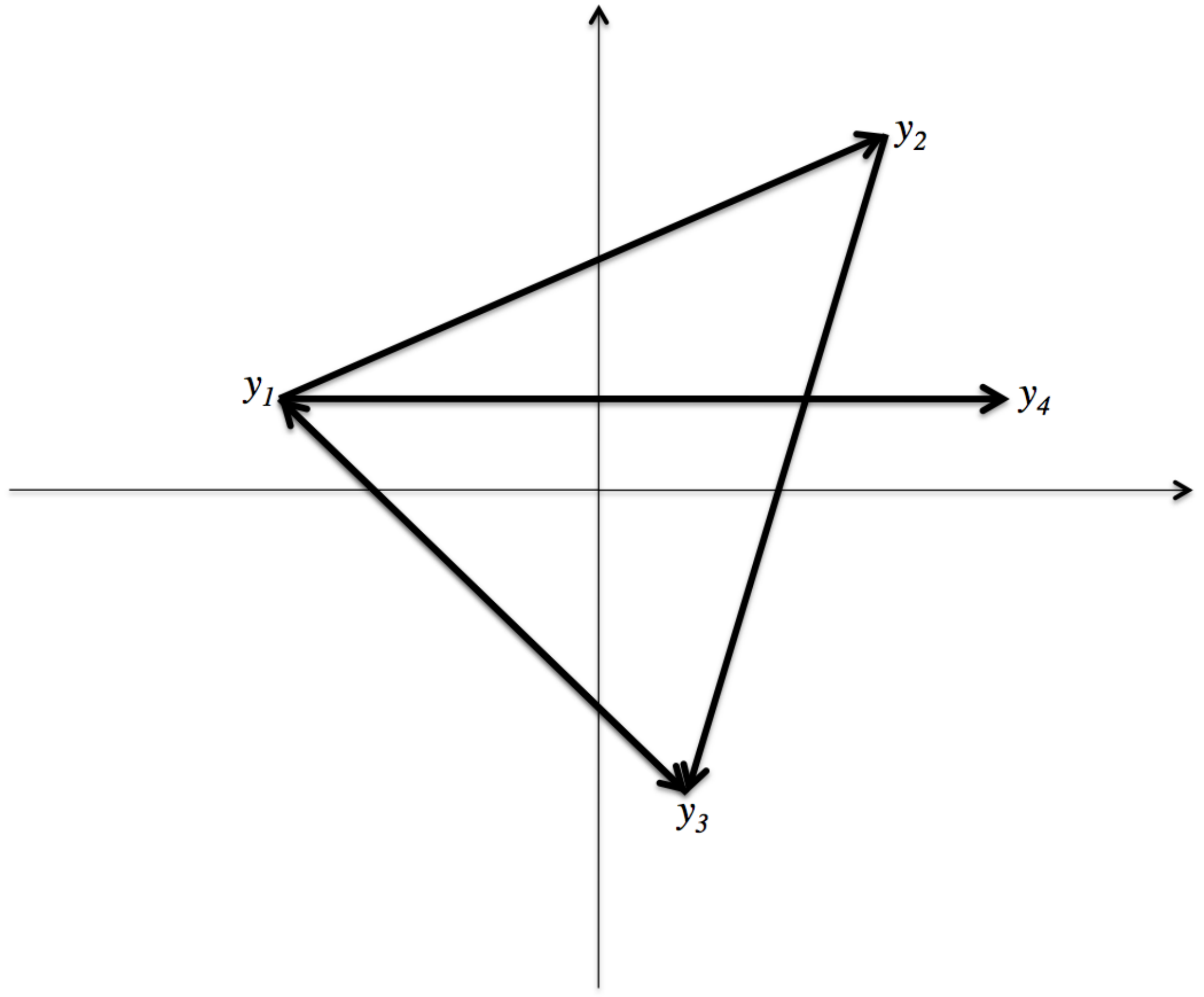}
      \includegraphics[width=2.7in]{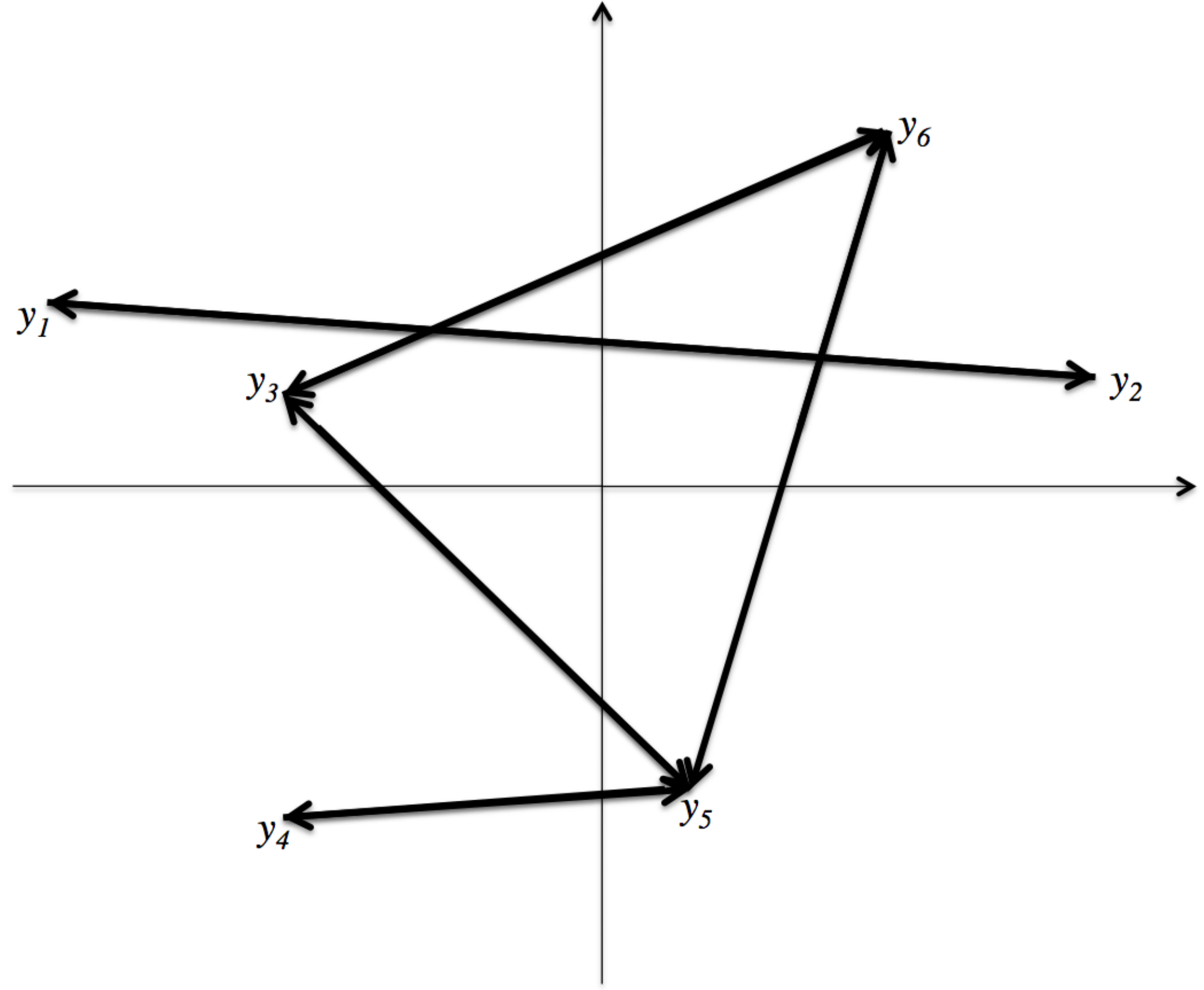}
    \end{tabular}
  \end{center}
  \caption{\scriptsize \label{fig:1}{\it Four examples of Euclidean embedded graphs in $\RR^2$. 
  %Note that, while we {\em do} assume that the points $y_1, ..., y_k\in\RR^2$ are distinct, we do {\em not} assume that the line segments (i.e., arrows) representing the vectors $y_j - y_i$ are disjoint. 
 The two graphs on the right are  {\em weakly reversible}, and the bottom-right graph is {\em reversible}.
Although the bottom-left graph is not weakly reversible, it generates dynamical systems~(\ref{polynomial}) which {\em can} also be represented by some weakly reversible graph. This is not true for the top-left graph.}}
\end{figure}

Positive trajectories of  (\ref{polynomial}) and (\ref{polynomial_nonaut}) are also trajectories of {\em mass-action systems}~\cite{Horn_Jackson}. 
There is great interest in understanding the {\em persistence} and {\em global stability} properties of such dynamical systems~\cite{Horn_Jackson, Horn_1974, mf72, Feinberg_1979, Feinberg_1987, Feinberg_1995, siegel_maclean, Sontag1, gunawardena, cf05, cf06, ctf06, Banaji_Craciun_2009, TDS, Anderson_Shiu_2010, ShiuSturmfels, Anderson_2011, angeli_deleenheer_sontag_2011, sf11, CNP, Banaji_2013, persistence2}. 
A natural question is the following: for which systems (\ref{polynomial}) or (\ref{polynomial_nonaut}) is it true that trajectories that start in $\RR^n_{>0}$ stay away\footnote{
i.e., for a trajectory $x(t)$, we have $\displaystyle \liminf_{t\to\infty} x_i(t) > 0$ for all $i=1,...,n$. Then we say that the trajectory $x(t)$ is {\em persistent}.
} from the boundary of $\RR^n_{>0}$? 

Given a polynomial dynamical system, the choice of graph $G$ above is not unique. We will see that if the graph $G$ can be chosen such that each edge of $G$ is contained in a cycle\footnote{
or, equivalently, each connected component of $G$ is strongly connected. If this is the case, we will say that the graph $G$ is {\em weakly reversible}.
}, 
%and if the functions $k_{y \to y'}(t)$ can be chosen to be bounded away from zero and infinity, 
then any bounded trajectory of (\ref{polynomial}) must be \emph{persistent} in $\RR^n_{>0}$. 
More generally, this is also true for solutions of (\ref{polynomial_nonaut}), provided that the nonnegative scalar functions $k_{y \to y'}(t)$ are {\em bounded away from zero and infinity}\footnote{
i.e., there exists $\eps>0$ such that $k_{y \to y'}(t) \in [\eps, \frac{1}{\eps}]$ for all $t$.
}. 

\medskip

In this paper we prove the results mentioned above, and use them to prove the {\em global attractor conjecture}, which says that toric dynamical systems have a global attractor within any linear invariant subspace\footnote{
A {\em linear invariant subspace} of~(\ref{polynomial}) is an invariant set $\S_0 = (x_0 + S_0) \cap \RRp^n$, where $S_0$ is a linear subset of $\RR^n$. For toric dynamical systems we will see that $S_0 = span\{y'-y:y\to y'\in G\}$ gives rise to a linear invariant subspace for any $x_0 \in \RRp^n$.
}.
A \emph{toric dynamical system} is a polynomial dynamical system\footnote{or more generally a power law system~\cite{CNP}.
} 
for which there exists a representation of the form~(\ref{polynomial}) that admits a {\em vertex-balanced equilibrium} in $\RR^n_{>0}$, i.e., there exists $x_0\in \RR^n_{>0}$ such that for each vertex $\bar y\in Y$ we have
\begin{equation}\label{vertex_balanced_equil}
\sum_{y \to \bar y \in G} k_{y \to \bar y}x_0^{y}  = 
\sum_{\bar y \to y \in G} k_{\bar y \to y}x_0^{\bar y} .
\end{equation}
In other words, if we think of the positive number $k_{y \to y'}x^{y}$ as a {\em flow} from the vertex $y$ to the vertex $y'$, then condition (\ref{vertex_balanced_equil}) says that at each vertex of the graph $G$, the sum of all the incoming flows equals the sum of all outgoing flows. The identity (\ref{vertex_balanced_equil}) implies that the graph $G$ is weakly reversible~\cite{Horn_Jackson}.

The name ``toric dynamical system" has been introduced~\cite{TDS} in order to emphasize the remarkable algebraic properties of these systems, but this class of polynomial dynamical systems has been first studied in depth in the 1972 breakthrough paper of Horn and Jackson~\cite{Horn_Jackson}, where they have been called \emph{complex balanced mass-action systems}. Also in~\cite{Horn_Jackson} Horn and Jackson have shown that these systems enjoy remarkable stability properties, and in particular {\em they have a unique positive equilibrium within each linear invariant subspace, and this equilibrium is locally asymptotically stable.} 
Actually,  in~\cite{Horn_Jackson} Horn and Jackson stated that the unique positive equilibrium within each linear invariant subspace is a {\em global attractor}\footnote{
i.e., if the positive equilibrium $x_0$ belongs to the linear invariant subspace $\S_0$, then all trajectories that start in $\S_0$ converge to $x_0$.
}, 
but soon afterwards Horn explained that they have not actually proved this claim, {\em and in 1974 he proposed this global convergence property as a conjecture~\cite{Horn_1974}}. 

The conjecture that toric dynamical systems have a {\em global attractor} within each linear invariant subspace (or, in the language of Horn, that complex balanced mass-action systems have a global attractor within each reaction simplex) was later called the {\em global attractor conjecture}~\cite{TDS}.
 
The global attractor conjecture has resisted efforts for a proof for over five decades, but proofs of many special cases have been obtained during this time, for example~\cite{siegel_maclean, Sontag1, TDS, Anderson_Shiu_2010, Anderson_2011, sf11, CNP, persistence2}.

In particular, Craciun, Nazarov and Pantea~\cite{CNP} have  proved the three-dimensional case, and Pantea has generalized this result for the case where the dimension of the linear invariant subspaces is at most three~\cite{persistence2}. Using a different approach, Anderson has proved the conjecture under the additional hypothesis that the graph $G$ has a single connected component~\cite{Anderson_2011}, and this result has been generalized by Gopalkrishnan, Miller, and Shiu for the case where the graph $G$ is strongly endotactic~\cite{Gopalkrishnan_Miller_Shiu_2013}.

The results described above do not provide a proof of the global attractor conjecture in full generality, and also do not provide a proof for any of the following three important special cases: $(i)$ the case where $G$ is {\em reversible}\footnote{
i.e., if $y \to y' \in G$ then also $y' \to y \in G$.
}, $(ii)$ the case where $G$ is reversible and the system (\ref{polynomial}) is {\em detailed balanced}\footnote{
i.e., we replace the {\em vertex-balance} equilibrium condition (\ref{vertex_balanced_equil}) with the more restrictive {\em edge-balance} equilibrium condition: $k_{y \to \bar{y}} x_0^y = k_{\bar{y} \to y} x_0^{\bar{y}}$ for all $y \to \bar{y} \in G$. This condition says that, at the equilibrium point $x_0$, the ``forward flux" balances the ``reverse flux" for any edge in $G$.
},  and $(iii)$ the case where $G$ is reversible and $k_{y \to y'} = k_{y' \to y}$ for all $y \to y' \in G$.

\bigskip

The goal of this paper is to use {\em toric differential inclusions} in order to construct a proof of the global attractor conjecture in full generality. 

\bigskip

{\em Toric differential inclusions} are defined as follows\footnote{See section~\ref{section:strategy.1} for more details and examples.}. Consider a finite set $\F$ of polyhedral cones\footnote{
A {\em polyhedral cone} in $\RR^n$ is the intersection of a finite set of half-spaces of $\RR^n$.
} that cover $\RR^n$, such that $\F$ is a polyhedral fan\footnote{
A {\em polyhedral fan} in $\RR^n$ is a finite set $\F$ of polyhedral cones such that $(i)$ any face of a cone in $\F$ is also in $\F$, and $(ii)$ the intersection of two cones in $\F$ is a face of both cones. 
%We say that a polyhedral fan $\F$ {\em covers} $\RR^n$ if \ $\bigcup_{C\in\F} C = \RR^n$.
Simple examples of polyhedral fans $\F_\H$ in $\RR^n$ are given by the polyhedral cones delimited by a finite set $\H$ of hyperplanes that contain the origin. This particular case of ``hyperplane-generated polyhedral fan" is especially relevant for motivating our definition of toric differential inclusions. 
} ~\cite{Cox_Little_Schenck_BOOK,Fulton}, and fix some real number $\delta>0$. For each cone $C\in \F$ denote by $C^o$ the {\em polar cone}\footnote{
The {\em polar cone} of a cone $C\in \RR^n$ is the cone $C^o = \{ y\in \RR^n | \ x \cdot y \le 0 \textrm{ for all } x\in C \}$~\cite{Rockafellar_Convex_Analysis}. 
}
of $C$. Then for each $x\in \RR^n$ define $F_{\F, \delta}(x)$ to be the convex cone generated by the union of polar cones $C^o$ for all $C\in\F$ such that $dist(x,C) < \delta$.
%$$
%\displaystyle F_{\F, \delta}(x) = \bigcup_{C\in\F:dist(x,C) < \delta} C^o
%$$
Then a {\em toric differential inclusion} is a differential inclusion on $\RR^n_{>0}$ given by 
\begin{equation}
\label{TDI_intro}
\frac{dx}{dt} \in F_{\F, \delta}(\log x). 
\end{equation}
%
%for some polyhedral fan $\F$ as above, and some  $\delta>0$.

%Here we will be interested in {\em variable-$k$  weakly reversible polynomial dynamical systems}, i.e., systems~(\ref{polynomial_nonaut}) where all connected components of $G$ are strongly connected, and all scalar functions $k_{y \to y'}(t)$ take values in an interval $[\eps, \frac{1}{\eps}]$ for some $\eps > 0$. These systems generalize {\em weakly reversible mass-action systems}, which are systems~(\ref{polynomial_nonaut}) where all connected components of $G$ are strongly connected, all scalar functions $k_{y \to y'}(t)$ are constant, and the set of vertices of $G$ is a subset of $\ZZ_{>0}^n$. 

\bigskip

It turns out that, although the global attractor conjecture is formulated in terms of some systems if the form~(\ref{polynomial}), it will be very useful to understand some properties of systems of the form~(\ref{TDI_intro}) in order to construct a proof of the conjecture. Indeed, we can ``embed" some systems of the form~(\ref{polynomial}) (and even some systems of the form~(\ref{polynomial_nonaut})) into toric differential inclusions:

\bigskip

\noindent {\bf Theorem A.} \emph{If a polynomial dynamical system has a representation (\ref{polynomial_nonaut}) such that $G$ is weakly reversible and the non-negative functions $k_{y \to y'}(t)$ are bounded away from zero and infinity, then it can be embedded\footnote{
We say that the dynamical system $\frac{dx}{dt} = f(x)$ is {\em embedded} \ into the differential inclusion $\frac{dx}{dt} \in F(x)$ in the domain $\Omega$, if $f(x)\in F(x)$ for all $x\in\Omega$. 
} in a toric differential inclusion.}

\bigskip

\noindent The theorem has already been proved in~\cite{Craciun_SIAGA_2019}. In particular, it follows that if $G$ is weakly reversible, then (\ref{polynomial}) can be embedded in a toric differential inclusion for any choice of parameters $k_{y \to y'}\!>\!0$. 
We then prove that, for any toric differential inclusion $\T$ and for any positive point $x_0 \in \RRp^n$, there exists a hypersurface $\H \subset \RRp^n$ that separates $\RR^n_{\ge 0}$ into two regions, such that $\partial \RRp^n$ and $x_0$ are in different regions, and the {\em ``upper" region} of $\H$ (i.e., the region that contains the point $x_0$) is at positive distance from $\partial \RRp^n$, and is a forward-invariant region for the solutions of~$\T$ that remain inside some cube $[0,M]^n$. We will refer to this kind of hypersurface $\H$ as a {\em zero-separating hypersurface}\footnote{Note that the (fundamental) notion of ``zero-separating hypersurface" here is \ul{\em different} from (but still very similar with) the identically-named notion of ``zero-separating hypersurface" in~\cite{GAC_Proof_version_2}. See section~\ref{section:strategy} for more details.}:

\bigskip

\noindent {\bf Theorem B.} \emph{Toric differential inclusions have an exhaustive\footnote{For our goals, it is {\em not} sufficient to construct a {\em single} zero-separating hypersurface; instead we need a {\em family} of zero-separating hypersurfaces $\Z_\eps$ such that for any $x_0 \in \RRp^n$ there exists some $\eps$ such that $x_0$ is in the {\em ``upper" region} of $\Z_\eps$. This is called {\em an exhaustive family of zero-separating hypersurfaces}.} family of zero-separating hypersurfaces.}

\bigskip

\noindent {\em Most of our effort in this paper \ul{will be spent on the proof of Theorem B.}}
Then we will use this result to conclude that  trajectories of toric dynamical systems must stay away from the boundary of $\RRp^n$:

\bigskip

\noindent {\bf Theorem C.} \emph{Toric dynamical systems are persistent. }

\bigskip

\noindent Finally, we will be able to conclude that the global attractor conjecture is true:

\bigskip

\noindent {\bf Theorem D.} \emph{Toric dynamical systems have a globally attracting point within each positive linear invariant subspace.}

\newpage
\section{Definitions and notation}

\noindent
As we mentioned in the previous section, some of the main objects of interest in this paper are non-autonomous polynomial dynamical systems of the form~(\ref{polynomial_nonaut}) such that the graph $G$ is weakly reversible, and the nonnegative scalar functions $k_{y \to y'}(t)$ are bounded away from zero and infinity. We will refer to this class of dynamical systems as {\em variable-$k$  weakly reversible dynamical systems}\footnote{{\em Variable-$k$  weakly reversible dynamical systems} have also been called 
{\em variable-$k$  toric dynamical systems}~\cite{GAC_Proof_version_2}, because  any such dynamical system can be obtained from a toric dynamical system by using the same graph $G$, and replacing the constants $k_{y \to y'}$ with time-dependent non-negative functions $k_{y \to y'}(t)$ that are bounded away from zero and infinity. 
}. 

%In this paper we are interested in a generalization of autonomous polynomial dynamical systems~(\ref{polynomial}) that can be written in the non-autonomous form~(\ref{polynomial_nonaut}) such that the nonnegative scalar functions $k_{y \to y'}(t)$ are bounded away from zero and infinity. 

More precisely, a {\em variable-$k$  weakly reversible dynamical system} is a dynamical system on $\RR^n_{>0}$ of the form 
\begin{equation}\label{kvTOR}
\frac{dx}{dt} = \sum_{y\to y'\in G}k_{y\to y'}(t) \, x^{y} (y' - y) 
\end{equation}
such that $G$ is weakly reversible, and there exists a fixed positive number $\eps$ such that  $\eps \le k_{y \to y'}(t) \le \frac{1}{\eps}$ for all~$t$ and for all $y\to y'\in G$.

Note that, under these conditions, solutions of the variable-$k$  weakly reversible dynamical system above are also solutions of the {\em differential inclusion} on $\RR^n_{>0}$ given by 
\begin{equation}\label{DIkvTOR}
\frac{dx}{dt} \in F(x), \mathrm{\ where \ } F(x) = \left\{  \sum_{y\to y'\in G}k_{y\to y'} \, x^{y} (y' - y)  \ \ | \ \ \eps \le k_{y\to y'} \le \frac{1}{\eps} \right\}. 
\end{equation}

We will show that solutions of the differential inclusion~(\ref{DIkvTOR}) are also solutions of a type of differential inclusions with rich geometric properties, called {\it toric differential inclusions}.
%In this section we define {\it toric differential inclusions} on $\RR^n_{>0}$, which are the main objects we study in this paper. We begin by defining the special class of {\it toric differential inclusions given by a set of lines}. 
Later in this section we will define the general notion of toric differential inclusions in $\RR^n_{>0}$. 
We first define the class of {\em polar differential inclusions} in $\RR^n$.

\subsection{Polar differential inclusions}

%\bigskip

%\noindent {\em \underline{Idea}: maybe first define "thin polar differential inclusions" where practically $\delta = 0$, and then define "polar differential inclusions"?}

%\bigskip

%
%\noindent {\em \underline{Related idea}: use the fact that (if $span\{h_1^\perp, ...,h_r^\perp\} = \RR^n$) for "thin polar differential inclusions" you can just say that the differential inclusion is given by $F_{thin}(x)=polar(Cone(x))$, even if $Cone(x)$ does not span the whole space $\RR^n$.\\
%Moreover, in general (for $\delta > 0$) we can define $F_\delta(x)=Cone(\cup_{\{y | dist(x,y)<\delta\}}F_{thin}(y))$.}

%\bigskip

If the graph $G$ is {\em reversible}, then the equations (\ref{kvTOR}) can be written as 
\begin{equation}\label{kvREV}
\frac{dx}{dt} = \sum_{y\rightleftharpoons y'\in G} \left( k_{y\to y'}(t)x^{y} - k_{y'\to y}(t)x^{y'} \right) (y' - y), 
\end{equation}
by grouping together terms given by an edge $y\to y'$ and its reverse $y'\to y$. Then, for each reversible edge $y\rightleftharpoons y'\in G$, there is now a single term in the sum (\ref{kvREV}), which can be thought of as a ``tug-of-war" between the forward and reverse terms. Indeed, both the forward and the reverse terms are trying to ``push" the state $x(t)$ of the system along the same line\footnote{
i.e., a line whose direction is given by the vector $y' - y$.
}, but in opposite directions. Then, the domain $\RR^n_{>0}$ can be partitioned into three regions: the region where the inequality $\eps x^{y} \ge \frac{1}{\eps} x^{y'}$ holds (which implies $k_{y\to y'}(t)x^{y} \ge k_{y'\to y}(t)x^{y'}$), the region where the inequality $\frac{1}{\eps} x^{y} \le \eps x^{y'}$ holds (which implies $k_{y\to y'}(t)x^{y} \le k_{y'\to y}(t)x^{y'}$), and an {\em uncertainty region} where neither one of these two inequalities are satisfied, and either one of the two terms $k_{y\to y'}(t)x^y$ and $k_{y'\to y}(t)x^{y'}$ may win the tug-of-war, due to the fact that $k_{y\to y'}(t)$ and $k_{y'\to y}(t)$ may take any values between $\eps$ and $\frac{1}{\eps}$. 
Some concrete examples can be found in Section 3 of~\cite{CNP}. In particular, note that if one term {\em does} win the tug-of-war, then the direction of the sum of the two terms is {\em towards} the uncertainty region, and also towards the hypersurface $x^{y} = x^{y'}$ in $\RR^n_{>0}$. 

The simplest way to understand how these regions adjoin each other is to look at them in $\RR^n$ instead of $\RR^n_{>0}$, after applying a logarithmic transformation. If we denote $X=\log x$, then the hypersurface $x^{y} = x^{y'}$ becomes the hyperplane $(y'-y)\cdot X = 0$. Moreover, the uncertainty region described above becomes just the set of points at distance $<\delta = \frac{2 |\!\log\eps|}{||y'-y||}$ from this hyperplane. Note also that  the  direction given by the vector $y' - y$ is exactly the direction orthogonal to the hyperplane $(y'-y)\cdot X = 0$. 

Therefore, if the graph $G$ consists of a {\em single} reversible edge $y\rightleftharpoons y'$, then the dynamics of the system~(\ref{kvREV}) at a point $x$ can be described as follows: we know that $x$ will move along a line of support vector $(y'-y)$, and we are able to specify one direction or another by mapping $x$ to $X=\log x$, and checking whether the distance between $X$ and the hyperplane $(y'-y)\cdot X = 0$ is $\ge \delta$.  Moreover, if we {\em can} specify this direction, then it is always the direction {\em towards} (and not away from) this hyperplane, and orthogonal to it.

This characterization has the advantage that it can be carried over to the more general case, where $G$ contains  {\em several} reversible edges. In that case we have {\em several tug-of-wars} going on at the same time, but for each one of them we can specify the winning direction (if any) at $x$ by calculating the distance between $X=\log x$ and some hyperplane in $\RR^n$. Depending on whether $X$ falls outside an uncertainty region or not, each reversible edge $y\rightleftharpoons y'$ of $G$ contributes one or two resultant vectors (if one, then it is either $y'-y$ or $y-y'$, and if two, then they are $\pm(y'-y)$). 

If follows that any system (\ref{kvREV}) can be embedded into a differential inclusion on $\RR^n_{>0}$ given by a set $\H$ of hyperplanes in $\RR^n$ and a number $\delta>0$, as follows. For each $x\in\RR^n_{>0}$ we define $F_{\H,\delta}(\log x)$ to be the  convex cone generated by vectors orthogonal to the hyperplanes of $\H$, in the direction that goes from the point $X=\log x$ {\em towards} each hyperplane, and also the opposite direction if $X$ is at distance $<\delta$ from some hyperplane. If $X$ does not belong to any uncertainty region, then $F_{\H,\delta}(\log x)$ is defined to be exactly the {\em polar cone} $C^o$ of a cone $C$ bounded by a subset of hyperplanes\footnote{
the cone $C$ is the largest cone that contains the point $X$ within the polyhedral fan $\F(\H)$ determined by the set of hyperplanes $\H$. The convex cone generated by the ``outer normal" vectors orthogonal to the hyperplane faces of $C$ is its polar cone $C^o$~\cite{Fulton}.
} from $\H$. Moreover, if $X$ does belong to some uncertainty regions, then we can still describe $F_{\H,\delta}(\log x)$ in terms of polar cones, by including not just the polar of the largest cone of $\F(\H)$ that contains $X$, but also the polar of each cone of $\F(\H)$ that is at distance $\le \delta$ from $X$.

Of course, not all polyhedral fans are determined by a set of hyperplanes as above. Nevertheless, we can generalize the construction described above to define {\em polar differential inclusions} given by a polyhedral fan $\F$ in $\RR^n$, as described in the previous section: we define $F_{\F, \delta}(X)$ to be the convex cone generated by the union of polar cones $C^o$ for all $C\in\F$ such that $dist(X,C) < \delta$.

\subsection{Toric differential inclusions}

Polar differential inclusions can produce sets of cones that contain the right-hand-side of the vector fields (\ref{kvTOR}) for reversible $G$, but they are defined on $\RR^n$, while the system (\ref{kvTOR}) is defined on $\RR^n_{>0}$. To obtain a proper embedding we need to introduce {\em toric differential inclusions}, which are obtained from polar differential inclusions by a logarithmic transformation: a toric differential inclusion given by a polyhedral fan $\F$ in $\RR^n$ is a differential inclusion on $\RR^n_{>0}$ given by 
\begin{equation}\label{T.D.I}
\frac{dx}{dt} \in F_{\F, \delta}(\log x) 
\end{equation}
where $F_{\F, \delta}$ is a polar differential inclusion as defined above.

Due to the way this definition is related to the system (\ref{kvREV}), it follows that variable-$k$  weakly reversible dynamical systems (\ref{kvTOR}) that are generated by a {\em reversible} graph $G$ can be embedded into polar differential inclusions.

As we discussed above, for a reversible graph $G$, the polyhedral fan $\F$ could be chosen simply as the fan generated by the set of hyperplanes that are orthogonal to the vectors $y'-y$ for $y\to y' \in G$.

We need to remove this strong assumption of reversibility, 
%at least close enough to the origin\footnote{in follow-up work we will show that proximity to the origin is not actually necessary, i.e., any variable-$k$  weakly reversible dynamical system (\ref{kvTOR}) can be embedded into a toric differential inclusion globally in $\RR^n_{>0}$~\cite{Craciun_future}. We do not discuss this extension here, since it is not needed for a proof of the global attractor conjecture.
in order to be able to address the global attractor conjecture in full generality. 
In the next section we explain how to find such a polyhedral fan for any {\em weakly reversible} graph~$G$.

\section{Embedding of variable-$k$  weakly reversible dynamical systems into toric differential inclusions} 
\label{section:embedding_WR_in_TDI}

As we discussed in the previous section, the simplest examples of polar differential inclusions are generated by polyhedral fans $\F_\H$ that are determined by a finite set $\H$ of hyperplanes\footnote{
i.e., $\F_\H$ consists of all the cones delimited by hyperplanes in $\H$.
} that pass through the origin. 
Then, the simplest examples of toric differential inclusions are also generated by polyhedral fans $\F_\H$ as above. We will refer to this class of toric differential inclusions as {\em hyperplane-generated toric differential inclusions}. 

We have seen in the previous section that any {\em reversible} variable-$k$  weakly reversible dynamical system in $\RR^n_{>0}$ can be embedded into a (hyperplane-generated) toric differential inclusion. 
Here we show that a similar fact is true for all variable-$k$  weakly reversible dynamical systems\footnote{
In order to prove the global attractor conjecture we plan to construct some forward-invariant sets for  any given {\em toric dynamical system} (not just for the reversible ones). Theorem~\ref{thm_tor_TOR} will help us argue that we can accomplish our goal by constructing something called {\em zero-separating surfaces (ZSH)} for toric differential inclusions. 
}.

\begin{thm}
\label{thm_tor_TOR}
Consider a variable-$k$  weakly reversible dynamical system~{\rm (\ref{kvTOR})}. Then this system can be embedded into a toric differential inclusion.
\end{thm}

\begin{proof}
This theorem has already been proved in~\cite{Craciun_SIAGA_2019}. We discuss a sketch  of the proof below, see~\cite{Craciun_SIAGA_2019} for more details.

Assume first that the weakly reversible graph $G$ is made up of a single directed cycle. 
If our single-cycle graph is given by $y_1 \to y_2 \to ... \to y_r \to y_1$, then the variable-$k$  weakly reversible dynamical system it generates is of the form 
\begin{equation}\label{MAcycle}
\frac{dx}{dt} = \sum_{i=1}^r k_i(t) \, x^{y_i} (y_{i+1} - y_i), 
\end{equation}
where $y_{r+1} = y_1$ and $\eps_0 \le k_i(t) \le \frac{1}{\eps_0}$ for some ${\eps_0}>0$.

Consider the set $\L$ of lines through the origin in the direction of vectors $y_i - y_j$ for all $i \neq j$, and denote by $\H$ the set of all hyperplanes that are orthogonal complements of lines in $\L$\footnote{i.e., a hyperplane belongs to $\H$ iff it is orthogonal to a line in $\L$.
}. 
Denote by $\F_\H$ the polyhedral fan generated by the set of hyperplanes $\H$, and, for $\delta>0$, denote by $\T_{\H,\delta}$ the corresponding hyperplane-generated toric differential inclusion. We will show that there exists $\delta_0>0$ such that the single-cycle variable-$k$  weakly reversible dynamical system~(\ref{MAcycle}) is embedded in the toric differential inclusion $\T_{\H,\delta_0}$. 

Choose $\delta_0>0$ large enough such that the uncertainty regions given by the {\em reversible} edges $y_i\rightleftharpoons y_j$ and $\eps_0$ are contained within the uncertainty regions of the toric differential inclusion $\T_{\H,\delta_0}$. 

Fix a point $x\in\RRp^n$ such that $X = \log x$ belongs to some cone $C$ in $\F_\H$, and does {\em not} belong to any uncertainty region of $\T_{\H,\delta_0}$. In particular, it follows that the cone $C$ has dimension $n$, otherwise $X$ would be contained in some hyperplane in $\H$, which in turn would be contained in an uncertainty region of $\T_{\H,\delta_0}$. We want to show that the right-hand-side of~(\ref{MAcycle}) is contained in the dual cone $C^o$. 

%Note that the generators of the cone $C^o$ are orthogonal to the {\em maximal faces}\footnote{maximal faces are also called {\em facets}.} of the cone $C$, which have dimension $n-1$. Therefore, the generators of the cone $C^o$ are orthogonal to the boundary hyperplanes of $C$, and these hyperplanes belong to the set $\H$. 

Consider a vector $w$ in the interior of $C$, and project $y_1, y_2 ,..., y_r$ on the line $l_w$ that passes through the origin in the direction given by $w$. Then no two projections are the same, because $w$ does not belong to any of the hyperplanes in $\H$. 
We now give a second set of names to the vectors $y_1, y_2 ,..., y_r$, say $v_1, v_2 ,..., v_r$, to record the {\em distance to the origin} of the projections along the line (i.e., if the projection of $y_{i_1}$ is the closest to origin, then $v_1 = y_{i_1}$, and if the projection of $y_{i_2}$ is the second closest to origin, then $v_2 = y_{i_2}$,  and so on). Note that, since the interior of $C$ is disjoint from all the hyperplanes in $\H$, it follows that the new names $v_1, v_2 ,..., v_r$ do {\em not} depend on the particular choice of vector $w$ in the interior of $C$. In other words, the dot products 
$
(v_{l+1}-v_l) \cdot w
$
are all {\em negative} numbers, for all $w$ in the interior of $C$.
Therefore, the vectors $v_2-v_1, v_3-v_2, ..., v_r-v_{r-1}$ belong to $C^o$. 

So, in order to show that the right-hand side of (\ref{MAcycle}) is included in $C^o$, it is enough to show that it can be written as a positive linear combination of the vectors $v_2-v_1, v_3-v_2, ..., v_r-v_{r-1}$. 

Note that $(i_1,i_2,...,i_r)$ is a permutation of $(1,2,...,r)$. If we denote the inverse permutation by $(j_1,j_2,...,j_r)$, it follows that $y_1 = v_{j_1}$, $y_2 = v_{j_2}$, and so on.

Then we have $y_2 - y_1 = v_{j_2} - v_{j_1}$. If $j_2 > j_1$ we write
$$
y_2 - y_1 = \sum_{l=j_1}^{j_2-1} (v_{l+1} - v_l),
$$
and if $j_2 < j_1$ we write 
$$
y_2 - y_1 = -\sum_{l=j_2}^{j_1-1} (v_{l+1} - v_l),
$$
We do the same for $y_3 - y_2, \ y_4 - y_3$, and so on. 
This way, we write each difference $y_{i+1} - y_i$ from the right-hand-side of~(\ref{MAcycle}) in terms of the vectors $\pm(y_{l+1} - y_l)$, with $l=1,2, ...,r-1$. 
Therefore we can re-group terms to obtain
\begin{equation}\label{MAcycle_new}
\frac{dx}{dt} = \sum_{i=1}^{r-1} \Phi_l (v_{l+1} - v_l), 
\end{equation}
where $\Phi_l$ is a sum of several terms of the form $k_i \, x^{v_i}$, with various signs. 

Note now that the positive terms inside $\Phi_l$ correspond to edges of the form $v_{m} \to v_{n}$ with $m \le l < n$, and negative terms inside $\Phi_l$ correspond to edges of the form $v_{m} \to v_{n}$ with $n \le l < m$. This means that the positive terms inside $\Phi_l$ contain $k_i \, x^{v_i}$ with $i \le l$, and the negative terms inside $\Phi_l$ contain $k_i \, x^{v_i}$ with $i > l$. Since $\log x \in C$  (and is not in an uncertainty region), and due to our choice of $\delta_0$, it follows that $k_1 x^{v_1} < k_2 x^{v_2} < ... < k_r x^{v_r}$. Therefore, the sum of the positive terms inside $\Phi_l$ dominates the sum of the negative terms inside $\Phi_l$, for each $l$\footnote{
note that the number of positive terms inside $\Phi_l$ is the same as the number of negative terms inside $\Phi_l$, because the graph $G$ is a cycle.
}. In conclusion, the right-hand-side of~(\ref{MAcycle_new}) (and therefore~(\ref{MAcycle})) is a positive linear combination of the vectors $v_{l+1} - v_l$, for $1 \le l \le r-1$, so it belongs to $C^o$.

\bigskip

If $x$ {\em does} belong to an uncertainty region of the toric differential inclusion $\T_{\H,\delta_0}$, then denote by $F(x)$ the cone of $\T_{\H,\delta_0}$ at $x$.  It follows that the maximal linear subspace contained in $F(x)$ has dimension $\ge 1$ in $\RR^n$. We project the problem on the orthogonal complement of that subspace, and then we reason the same way as above\footnote{
When projecting on the smaller dimensional subspace, the graph $G$ is replaced by its projection $\tilde G$. All vertices of $G$ that project to the same vertex of $\tilde G$ are interchangeable with each other when checking that the right-hand-side of~(\ref{MAcycle}) is contained in the (degenerate) cone $F(x)$, because the projections are done along linear subspaces contained in $F(x)$.
}. This gives us the desired conclusion for the case when $G$ is made up of a single directed cycle.

\bigskip

If the weakly reversible graph $G$ is {\it not} a single directed cycle, then we write it as a union of cyclic graphs, $\displaystyle G = \bigcup_{i=1}^g G_i$, and we can argue as above for each such $G_i$. We obtain that the variable-$k$  weakly reversible dynamical systems given by the cycle $G_i$ are embedded in toric differential inclusions generated by some set of hyperplanes $\H_i$.
Note now that the right-hand-side of a variable-$k$  weakly reversible dynamical system given by $G$ can be decomposed into a sum of terms, such that each term is of the form\footnote{
We may have to use smaller $\eps_i$ values for the terms in the decomposition, because the same edge of $G$ may belong to {\em several} (but of course {\em only finitely many}) graphs $G_i$.
} given by the right-hand-side of a variable-$k$  weakly reversible dynamical system determined by $G_i$. 
Then we conclude that any variable-$k$  mass-action system given by $G$ can be embedded into a toric differential inclusion generated by the set of hyperplanes~$\displaystyle \H = \bigcup_{i=1}^g \H_i$.
\end{proof}

\newpage

\section{Strategy for the proof of the global attractor conjecture: the big picture} 
\label{section:strategy}

%In this section we show that for any toric differential inclusion $\T_n$ on $\RRp^n$ and any  neighborhood $V_0$ of the origin in $\RR^n$ there exists a {\em zero-separating surface}, i.e., a hypersurface $Z_{\T_n,V_0}$ such that $\RRp^n \setminus Z_{\T_n,V_0}$ is the union of two disjoint and connected open sets $Z_{\T_n,V_0}^0$ and $Z_{\T_n,V_0}^1$, such that $Z_{\T_n,V_0}^0 \subset V_0 $, $Z_{\T_n,V_0}^1$ is an invariant region of $\T_n$, and the closure of $Z_{\T_n,V_0}^1$ does not contain the origin. Informally speaking, {\em the zero-separating surface $Z_{\T_n,V_0}$ ``cuts off the corner" of $\RRp^n$}, and what you obtain after this ``cut" is an invariant region of $\T_n$.

%For any given $x_0\in\RRp^n$ we can choose  $V_0$  that does not contain $x_0$, to ensure that $x_0$ belongs to the invariant region that is left over after the ``cut".  
%Then, it will follow that any solution of $\T_n$ with initial condition $x_0\in\RRp^n$ must be contained in some invariant region $Z_{\T_n,V_0}^1$ as described above, and in particular it cannot have an $\omega$-limit point at the origin.

\bigskip

In order to prove the global attractor conjecture we will prove that \ul{\em toric dynamical systems are persistent}. It has been known for a long time that persistence of toric dynamical systems, together with a simple application of the LaSalle invariance principle would immediately imply global stability~\cite{siegel_maclean, Sontag1, Extinction_in_reaction_networks}. 

More precisely, we will prove the following stronger fact:  {\em for any positive initial condition $x_0 \in \RRp^n$ of a toric dynamical system there exists a compact forward invariant region $K_{x_0} \subset \RRp^n$ such that $x_0 \in K_{x_0}$.} Of course, this  implies that toric dynamical systems are persistent. 

In order to accomplish this goal, our main focus will be the construction of a specific type of hypersurfaces in $\RRp^n$ called \underline{\em zero-separating hypersurfaces (ZSH)}, which, when considered together with well-chosen level sets of the Horn-Jackson Lyapunov function, give rise to compact forward invariant regions $K_{x_0}$ like we mentioned above; see Figure~\ref{FIG_1} for a two-dimensional example (in this example the ZSH is just a {\em curve}). 

More precisely, our strategy will be to construct {\em \underline{polyhedral} zero-separating hypersurfaces (polyhedral ZSH)}. 
%A polyhedral ZSH will be a polyhedral hypersurface in $\RRpn$ that, together with a well-chosen level set of the Horn-Jackson Lyapunov function, determines a compact forward invariant region for a toric dynamical system of interest. 
Our construction of a polyhedral ZSH will rely on a \underline{{\em ZSH blueprint}}. A {\em ZSH blueprint} is a structure associated to a given {\em polyhedral fan}, which contains inside complete instructions for the construction of an exhaustive family of ZSH. Polyhedral fans are relevant because of the theorem on embedding toric dynamical systems into toric differential inclusions; and a toric differential inclusion (denoted $\T_{\F, \delta}$) is given by a {\em polyhedral fan} $\mathcal{F}$ and a positive number $\delta$. 

In our construction of {\em polyhedral ZSH} we will see that the specific number $\delta$ above will practically become irrelevant, and it will turn out that we will not need to look in much depth at  detailed properties of toric differential inclusions; {\em it will be sufficient to focus almost exclusively on the much simpler notion of \ul{polyhedral fan}}. 

\medskip

Therefore, given a polyhedral fan $\F$, we will develop a step-by-step process that allows us to construct some polyhedral  hypersurfaces associated to $\F$. More precisely, we will show how to construct such families of polyhedral hypersurfaces by taking advantage of a {\em ZSH blueprint}, and we will see later that {\em these hypersurfaces are indeed zero-separating hypersurfaces}. 

\bigskip

\noindent
In summary, we will proceed according to the following plan\footnote{For convenience of notation we assume that we start in $(n+1)$-dimensional space; then the fundamentals steps (i.e., {\em Steps 2 and 3}) take place in $n$-dimensional space.}:

\medskip

{\em 

\underline{Step~0.} Given a toric dynamical system $\mathbf{T}_{n+1}$ in $\RR^{n+1}_{>0}$, we embed it into a toric differential inclusion $\T_{n+1}$, generated by a polyhedral fan $\F_{n+1}$ and a number $\delta>0$.

\smallskip

\underline{Step~1.} Given the polyhedral fan $\F_{n+1}$ in $\RR^{n+1}$ we construct the {$x_1$-$x_n$-subdivision} $\S_n$ of $\RRn$ generated by $\F_{n+1}$.

\smallskip

\underline{Step~2.} Given the $x_1$-$x_n$-subdivision $\S_n$ we construct a {ZSH blueprint} $\B_n$ for $\S_n$. 

\smallskip

\underline{Step~3.} Given the ZSH blueprint $\B_n$ we construct an exhaustive family of $n$-dimensional polyhedral ZSHs denoted $\Z_{\eps}$ in $\RRp^{n+1}$.

\smallskip

\underline{Step~4.} We show that for any $x_0 \in \RRp^{n+1}$ there exists a polyhedral hypersurface $\Z_{\eps}$ among the ones that we have constructed in step 3, such that $\Z_{\eps}$ (together with a well-chosen level set of the Horn-Jackson Lyapunov function) determines a compact set $K_{x_0} \subset \RRp^{n+1}$ such that $x_0 \in K_{x_0}$ and $K_{x_0}$  is a forward invariant region for the toric dynamical system $\mathbf{T}_{n+1}$ that we started with in Step 0.

}

\medskip

Step~0 has already been done in~\cite{Craciun_SIAGA_2019}, and we have discussed it in section~\ref{section:embedding_WR_in_TDI}. Step 1 will not be difficult, but we will need to we restrict our attention (without loss of generality) to $x_1$-$x_{n+1}$-fans.
{\em Steps~2 and 3 are very much interconnected, and most of our effort will focus on these two steps.} Step~4 will be a relatively quick corollary of the previous steps.

%%%%%%%%%%%%%%%%%%%%%%%%%%%%%%%%%%%%%%%%%%%%%%
\begin{figure}[htb!]
  \begin{center}
%    \begin{tabular}{c}
%      \hskip-2cm
      \includegraphics[angle=0,width=2.2in]{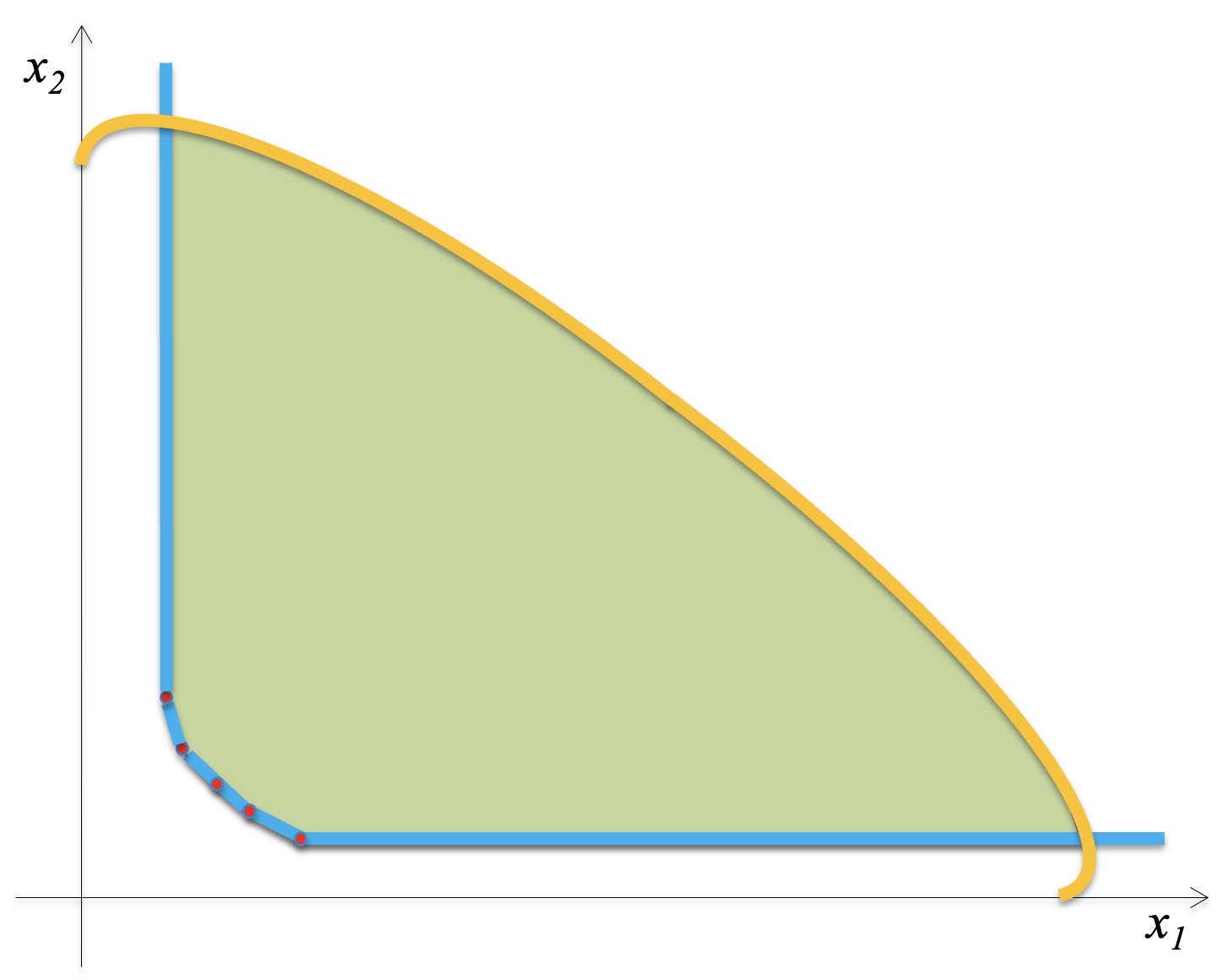}
%    \end{tabular}
  \end{center}
  \caption{\scriptsize \label{FIG_1} A two-dimensional illustration of the interplay between the ZSH (the blue curve) and a well chosen level set of the Horn-Jackson Lyapunov function (the orange curve), such that together they give rise to a compact forward invariant region (the green region).}
\end{figure}
%%%%%%%%%%%%%%%%%%%%%%%%%%%%%%%%%%%%%%%%%%%%%%

\subsection{Polyhedral fans and toric differential inclusions}
\label{section:strategy.1}

 {\em In this subsection we revisit some definitions and results that have  also been mentioned in a previous section, and we look at a few simple examples that illustrate these notions.} The reader who is familiar with the notions and results from the previous sections can skip this subsection. 

\begin{defn}
A \underline{\em polyhedral fan} in $\RR^n$ is a finite collection $\F$ of closed convex polyhedral cones\footnote{All the cones we are interested in here are closed convex polyhedral cones; therefore, from now on, whenever we say ``cone", we mean ``closed convex polyhedral cone".}  in $\RR^n$,
$$
\F = \{  C_1, ..., C_m  \},
$$ 
such that we have

$(i)$ for any cone $C_i \in \F$ any face of $C_i$ also belongs to $\F$,

\noindent
and 

$(ii)$ for any $C_i, C_j \in \F$ their intersection $C_i \cap C_j$ is a common face of both of them.
\end{defn}

We say that the polyhedral fan $\F$ is {\em pointed} if any cone $C \in \F$ is a pointed cone. Also, we say that a polyhedral fan $\F$ in $\RR^n$ is {\em complete} if the union of all the cones in $\F$ is the whole $\RR^n$. 

Here we are only interested in {\em pointed and complete polyhedral fans}, and for simplicity we will refer to them simply as \underline{{\em fans}} (in other words, from now on whenever we say ``fan", we mean ``pointed and complete polyhedral fan")\footnote{In some instances we will need to refer to the set of all the cones of a fan $\F$ that are contained in a specific subset of $\RR^n$; in this case we will say {\em ``restricted fan"}.}.

See Figure~\ref{FIG_2}$(a)$  for an illustration of a fan in $\RR^2$. This fan consists of 13 cones: six cones of dimension 2, six cones of dimension 1, and one cone of dimension 0 (i.e., the origin of $\RR^2$). 
%%%%%%%%%%%%%%%%%%%%%%%%%%%%%%%

\begin{defn}

Given a  fan $\mathcal{F}$ and a real number $\delta > 0$ we  define the \\ \underline{toric differential inclusion generated by $\mathcal{F}$ and $\delta$} as follows.
For any $x \in \mathbb{R}^n_{>0}$ we consider all the cones  $C_j \in \mathcal{F}$ such that $\text{dist}(\log x, C_j) < \delta$ and define ${F}_{\mathcal{F}, \delta}(x)$ to be the cone generated by the union of all polar cones $C_j^\circ$ .
Then the toric differential inclusion generated by $\mathcal{F}$ and $\delta$ is given by 
\begin{equation}
\label{TDI}
\frac{dx}{dt} \in {F}_{\mathcal{F}, \delta}(x). 
%\leftarrow \text{denoted } T_{\delta}
\end{equation}
\end{defn}

\bigskip

For example, in the case of the fan $\mathcal{F}_2$ 
shown in Figure~\ref{FIG_2}$(a)$, 
we can see that the right-hand side of the corresponding toric differential inclusion (i.e., ${F}_{\mathcal{F}_2, \delta}(x)$) consists of a piecewise constant choice of convex cones (one for each $C_j \in \mathcal{F}_2$), as shown in Figure~\ref{FIG_2}$(c)$.

\bigskip

Toric  differential inclusions are useful in the study of the global attractor conjecture because of the following embedding theorem~\cite{Craciun_SIAGA_2019}, which implies that one can reduce the problem of building forward-invariant regions for toric dynamical systems to the analogous problem for toric differential inclusions.

\begin{thm} Consider a variable-$k$  weakly reversible\footnote{Note that toric dynamical systems are special cases of weakly reversible systems; in particular, every toric dynamical system has a {\em fixed} vector of parameters (i.e., reaction rate constants) $k$, but {\em not every} weakly reversible system with a fixed vector of parameters is a toric dynamical system (because, in general, the ``vertex balance condition" is sufficient but not necessary for weak reversibility).} dynamical system. Then this dynamical system can be embedded into a toric differential inclusion. More precisely, if our variable-$k$  weakly reversible dynamical system is given by 
$$
\frac{dx}{dt} = f(t,x),
$$
then there exists a fan $\mathcal{F}$ 
%such that for any $\varepsilon > 0$ [in the inequalities $\varepsilon \le k_{\text{gap}}(x) \le \frac{1}{\varepsilon}$] there exists 
and some $\delta > 0$ such that we have  
$$
f(t,x) \in {F}_{\mathcal{F}, \delta}(x),
$$
for all $t \in \RR$ and all $x \in \RRpn$.
\end{thm}

\subsection{Forward invariant regions and ZSH for toric differential inclusions}

We are interested in constructing forward-invariant regions for toric differential inclusions. Let us first define the following notion:

\begin{defn}
Consider a piecewise smooth surface $\Z$ in $\RRp^n$, and a toric differential inclusion $\T$ in $\RRp^n$ generated by a fan $\F$ and a number $\delta>0$. If $P$ is a smooth point of $\Z$, we say that \underline{$\T$ cannot cross $\Z$ at $P$} if the tangent hyperplane of $\Z$ at $P$ is a  supporting hyperplane of the cone $F_{\F, \delta}(P)$. Also, for a set $\Omega \in \RRp^n$, we say that \underline{$\T$ cannot cross $\Z$ within $\Omega$} if $\T$ cannot cross $\Z$ at any smooth point $P \in \Omega$.
\end{defn}

Let us note the following simple characterization of the condition that $\T$ cannot cross $\Z$ at $P$:

%for the tangent hyperplane at a smooth point $P$ of a surface $\Z$ to be a support hyperplane\footnote{Of course, if the tangent hyperplane of $\Z$ at $P$ is a  support hyperplane of $F_{\F, \delta}(P)$, then solutions of $\T$ cannot cross $\Z$ at $P$.} of the cone $F_{\F, \delta}(P)$ of the toric differential inclusion generated by $\F$ and $\delta$:

\begin{rem}
\label{rem:T_cannot_cross_S}
Consider a piecewise smooth surface $\Z$ in $\RRp^n$, and a toric differential inclusion $\T$ in $\RRp^n$ generated by a fan $\F$ and a number $\delta>0$. Consider a smooth point of $\Z$ denoted  $P$. Then $\T$ cannot cross $\Z$ at $P$ if and only if there exists a normal vector $\vec{n}_P$ of $\Z$ at $P$ such that $\vec{n}_P \in \C$ for any cone $\C \in \F$ that satisfies 
the condition
$$
\dist(\C, \log P) < \delta.
$$
\end{rem}

Indeed, this follows immediately from the definition of the cone $F_{\F, \delta}(P)$ in~(\ref{TDI}), together with the general observation that the cone generated by a union of polar cones equals the polar cone of the intersection of those cones~\cite{Rockafellar_Convex_Analysis}.

\bigskip

The key notion in this section is the {\em zero-separating hypersurface}, or {\em ZSH} for short:

\begin{defn}
Consider a piecewise smooth surface $\Z$ in $\RRp^n$, and a toric differential inclusion $\T$ in $\RRp^n$ generated by a fan $\F$ and a number $\delta>0$. 
We say that \underline{$\Z$ is a zero-separating hypersurface for $\T$} (or \underline{$\Z$ is ZSH for $\T$}) if the following are true:

$(i)$ there exists some $\eta \in (0,1)$ such that 
$$
\Z \subset \RRp^n \setminus (\eta + \RRp^n),
$$

$(ii)$ $\Z$ intersects every toric ray\footnote{A {\em toric ray} in $\RRn$ is the image via the exponential function of a line that passes through the origin of $\RR^n$. Note that every toric ray is contained in $\RRp^n$ and passes through the point of coordinates $(1,1, \dots, 1) \in \RRpn$.} in $\RRp^n$ exactly once,

\noindent
and

$(iii)$ there exists some $M \ge 1$ such that the toric differential inclusion $\T$ cannot cross $\Z$ within the set $(0,M)^n$. 
\end{defn}

%there exists some $M \ge 1$ such that the toric differential inclusion $\T$ cannot cross $\Z$ within the set $(0,M)^n$.

%In some cases we will need to emphasize the values of the numbers $\eta$ and $M$ in the definition above; in that case we will say that {\em $\Z$ is a ZSH for $\T$}.

%Here is a short summary of our strategy in what follows: we start with a toric dynamical system $\bf T$ in $\RRp^n$, and we embed it into a 

Recall that, given a toric differential inclusion $\T$ in $\RRp^n$, we want to construct a family of ZSH for $\T$ denoted $\Z_\eps$ such that each such ZSH (together with a well-chosen level set of the Horn-Jackson Lyapunov function) gives rise to a compact forward invariant set $K_{\eps} \subset \RRp^n$ for the toric dynamical system that we have started with. Moreover, we want to make sure that for any initial state $x_0 \in \RRp^n$ there is some $\eps$ such that  $x_0 \in K_{\eps}$.  Then, as we mentioned before, we obtain a proof of the global attractor conjecture.

\bigskip

Indeed, see Figure~\ref{FIG_1} for an illustration in $\RRp^2$. More concretely, if we think of the blue polyhedral line in Figure~\ref{FIG_1} as a ZSH denoted $\Z_\eps$, we will want the polyhedral line $\Z_\eps$ to get closer and closer to the axes of the positive quadrant as $\eps$ gets closer and closer to zero; at the same time, we will want the orange curve (which is a level curve of the Horn-Jackson Lyapunov function in $\RRp^2$) to move farther and farther away, so the green forward-invariant region grows bigger and bigger, such that any $x_0 \in \RRp^2$ will be contained to such a green region. In other words, we want the polyhedral lines $\Z_\eps$ to form an exhaustive family of ZSHs.

%\begin{rem}
%Let us illustrate the usefulness of a 
%\end{rem}

\newpage

\section{Proof of the global attractor conjecture in the two-dimensional case.}

In this section we discuss the proof of the two-dimensional global attractor conjecture, {\em by relying on two-dimensional toric differential inclusions}. 

Recall that the two-dimensional global attractor conjecture has first been proved in~\cite{Anderson_Shiu_2010}, {\em without} the use of toric differential inclusions; but, due to the  relevance of toric differential inclusions in our proof of the $n$-dimensional global attractor conjecture 
%(in Sections \dots \dots \dots \dots \dots \dots \dots \dots \dots \dots \dots \dots \dots \dots ) 
we will focus here on how two-dimensional toric differential inclusions can be used in the proof of the two-dimensional global attractor conjecture. Our main interest  is {\em to illustrate the use of toric differential inclusions (and especially \underline{polyhedral fans}\,) in such a proof}, where several important aspects of a \ul{\em rich geometric structure} of the problem become important, and easy to see.

\subsection{Construction of ZSH for two-dimensional toric differential inclusions}

Two-dimensional toric differential inclusions admit ZSHs that are polygonal curves, such as the blue curve in Figure~\ref{FIG_2}$(d)$. 
Note that if this curve is constructed close enough to boundary of the positive quadrant 
% (which can be used as part of the boundary of a forward-invariant region) 
then it does {\em not} visit the more complicated regions in the middle of the figure.
%; instead, the blue curve stays close to the boundary of the positive quadrant. 
Therefore, in order to construct such a ZSH it is sufficient for us to focus on the right-hand side of the toric differential inclusion {\em near the boundary of the positive quadrant}.

We can see in Figure~\ref{FIG_2}$(d)$ that the (explicitly specified) cones of the toric differential inclusion can be used in order to construct a polygonal curve that cannot be crossed by solutions of toric differential inclusions, and therefore cannot be crossed by solutions of a toric dynamical system that is embedded in this toric differential inclusion (this is  discussed in detail in~\cite{Craciun_SIAGA_2019}).

More precisely, we can see in Figure~\ref{FIG_2}$(e)$ that the blue polygonal line can be used together with an level set of the Horn-Jackson Lyapunov function (the orange curve)  to create a {\em forward invariant region} that allows us to conclude immediately  (using the LaSalle invariance principle~\cite{siegel_maclean, Sontag1, Extinction_in_reaction_networks}) that the global convergence property holds for any solution with initial value inside this region. In other words, {\em the blue polygonal curve is a ZSH}. 

If we choose some initial point $x_0 \in \mathbb{R}^2_{>0}$ that is not in this region, then 
another such forward invariant region can be constructed that contains this $x_0$, as follows. First we construct another blue  polygonal curve, much closer to the boundary of the positive quadrant. This not only ensures that $x_0$ lies above this curve, but also ensures that the curve can be extended farther up, and farther to the right\footnote{Note that when we extend the blue curve up or to the right within some \ul{\em uncertainty region} for {\em large values} of $x_1$ or $x_2$,  then {\em it does \ul{not} satisfy the ZSH property with that uncertainty region}. But, if we construct it {\em closer to the boundary of the positive quadrant} then we can extended it while preserving the ZSH property farther up and farther to the right.}; then we can also intersect it with a level set of the Horn-Jackson Lyapunov function that lies {\em above} the point $x_0$, and obtain another  forward invariant region that is now guaranteed to contain $x_0$.

\noindent
Therefore we can conclude the desired global stability result in full generality.

\bigskip

%%%%%%%%%%%%%%%%%%%%%%%%%%%%%%%
\begin{figure}[htbp!] % option "p" places the figure on a dedicated page of floats (and helps avoid the figure being sent to the end of the document !!!!!!!:)
  \begin{center}
  \vspace*{-2.5cm}  
  \hspace*{-3.5cm} 
    \begin{tabular}{c}
$(a)$      \includegraphics[width=1.4in]{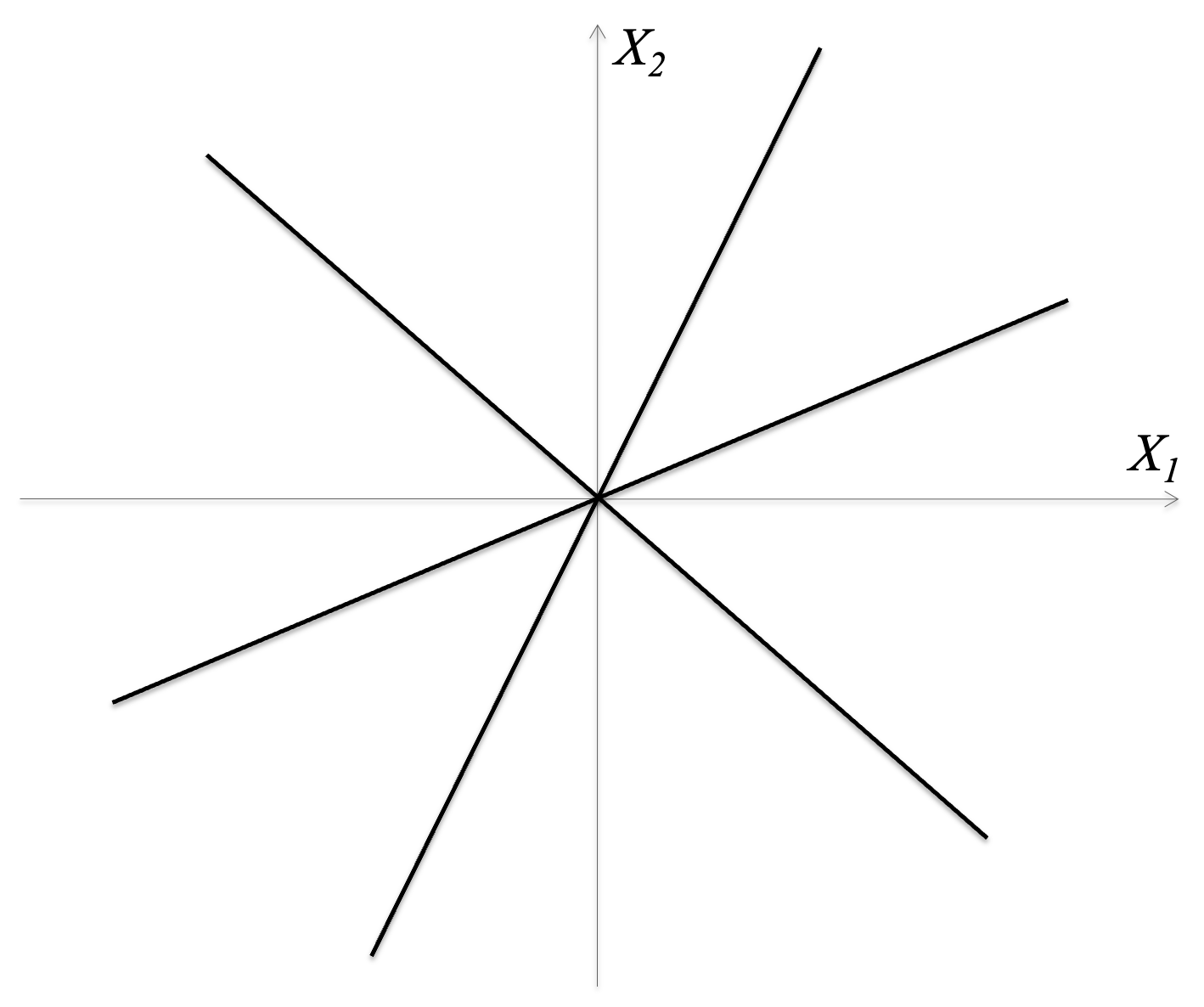} \ \ \ 
$(b)$      \includegraphics[width=1.4in]{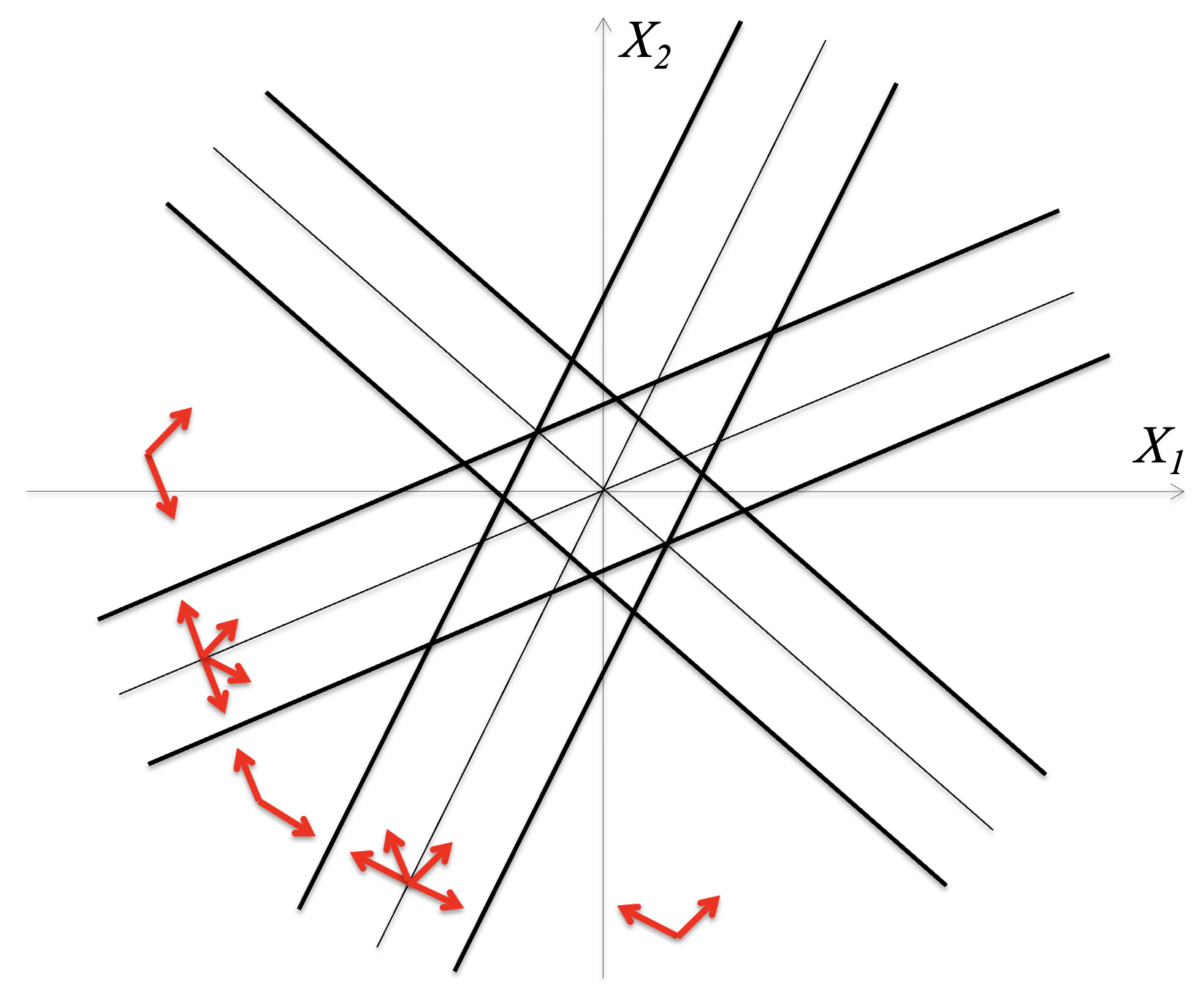}  \ \ \ 
$(c)$      \includegraphics[width=1.4in]{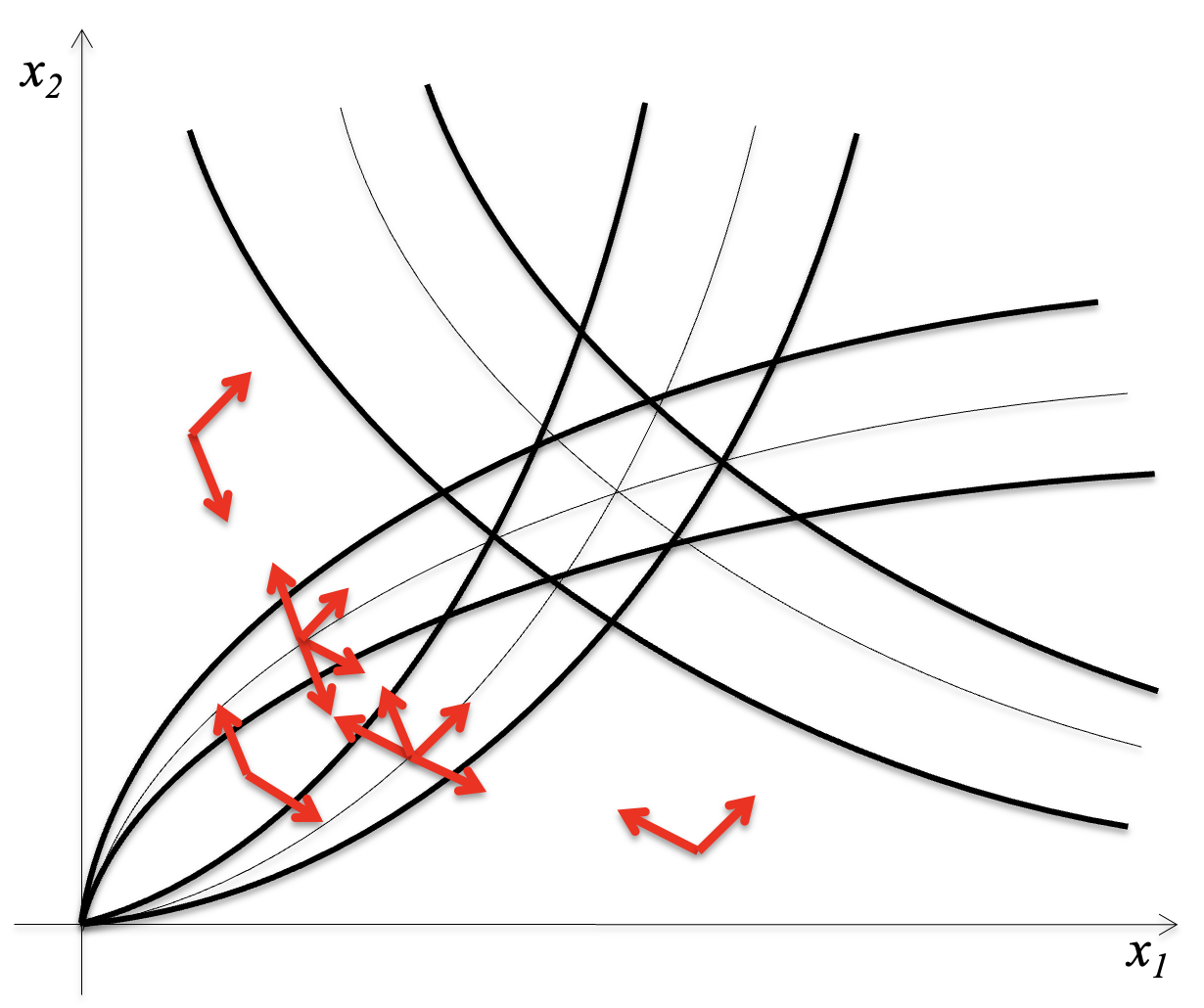}  \ \ \  
$(d)$      \includegraphics[width=1.4in]{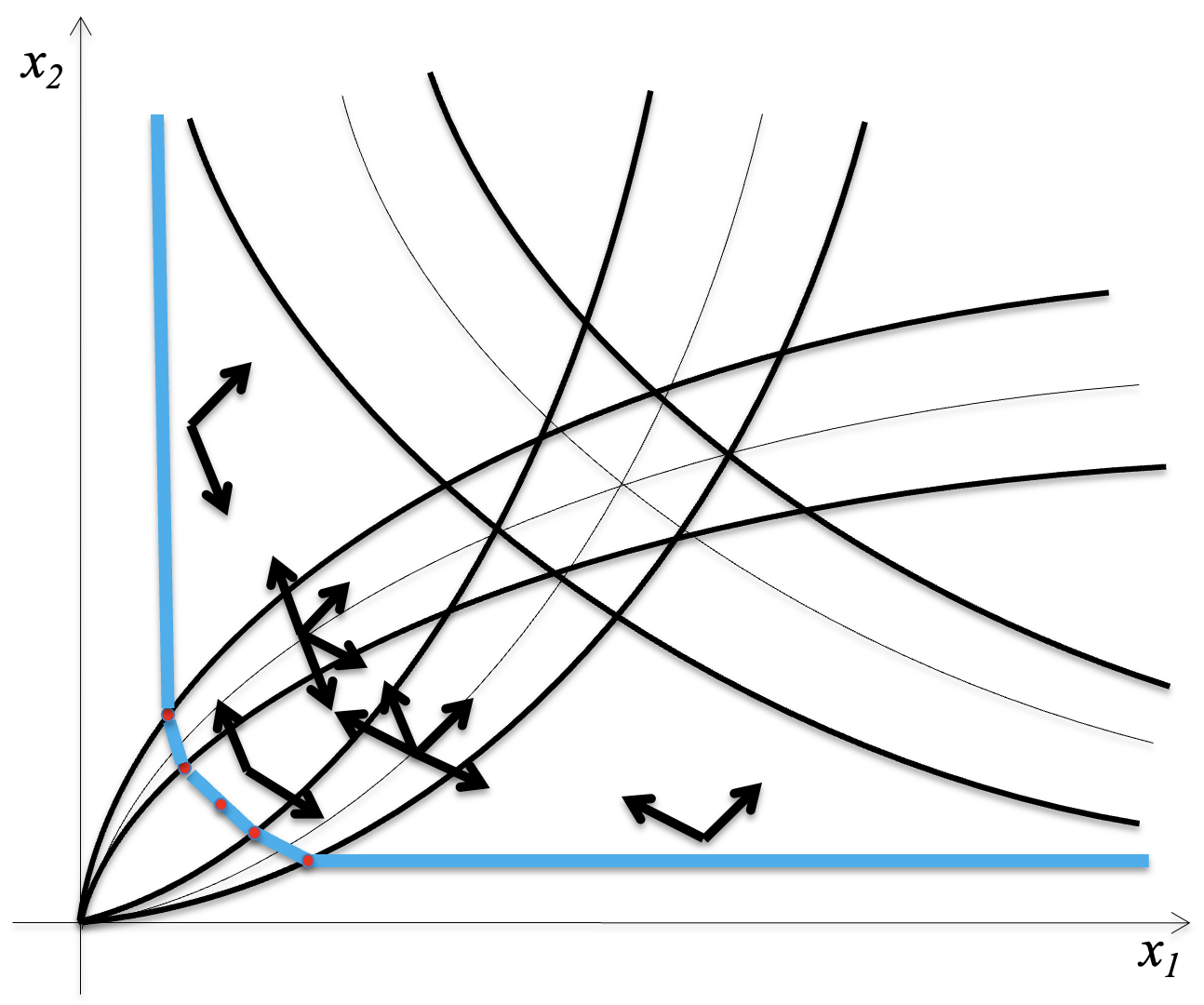} 
\\
\\
$(e)$      \includegraphics[height=1.6in]{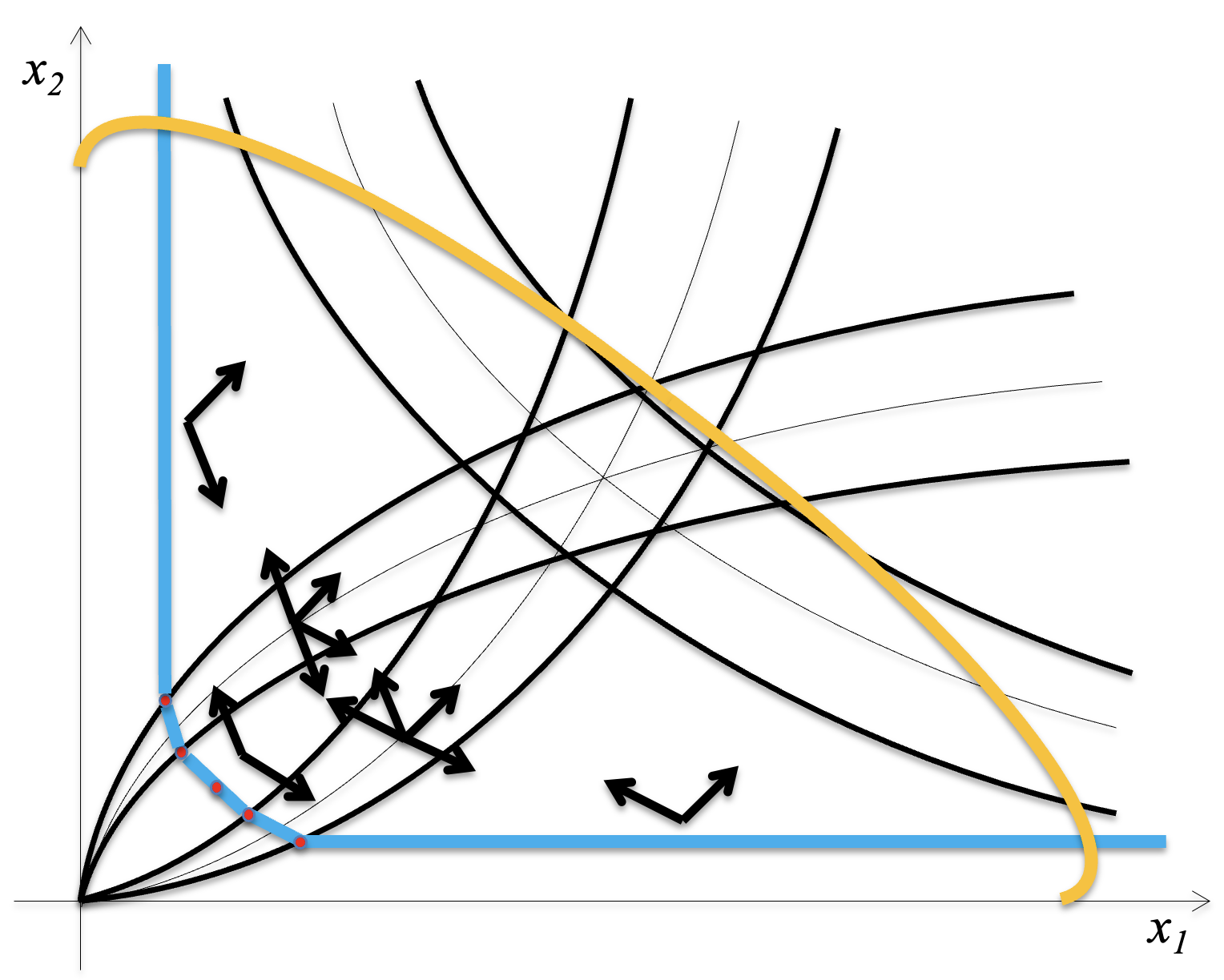} \ \ \ \ \ \ \ \ \ 
$(f)$      \includegraphics[height=1.6in]{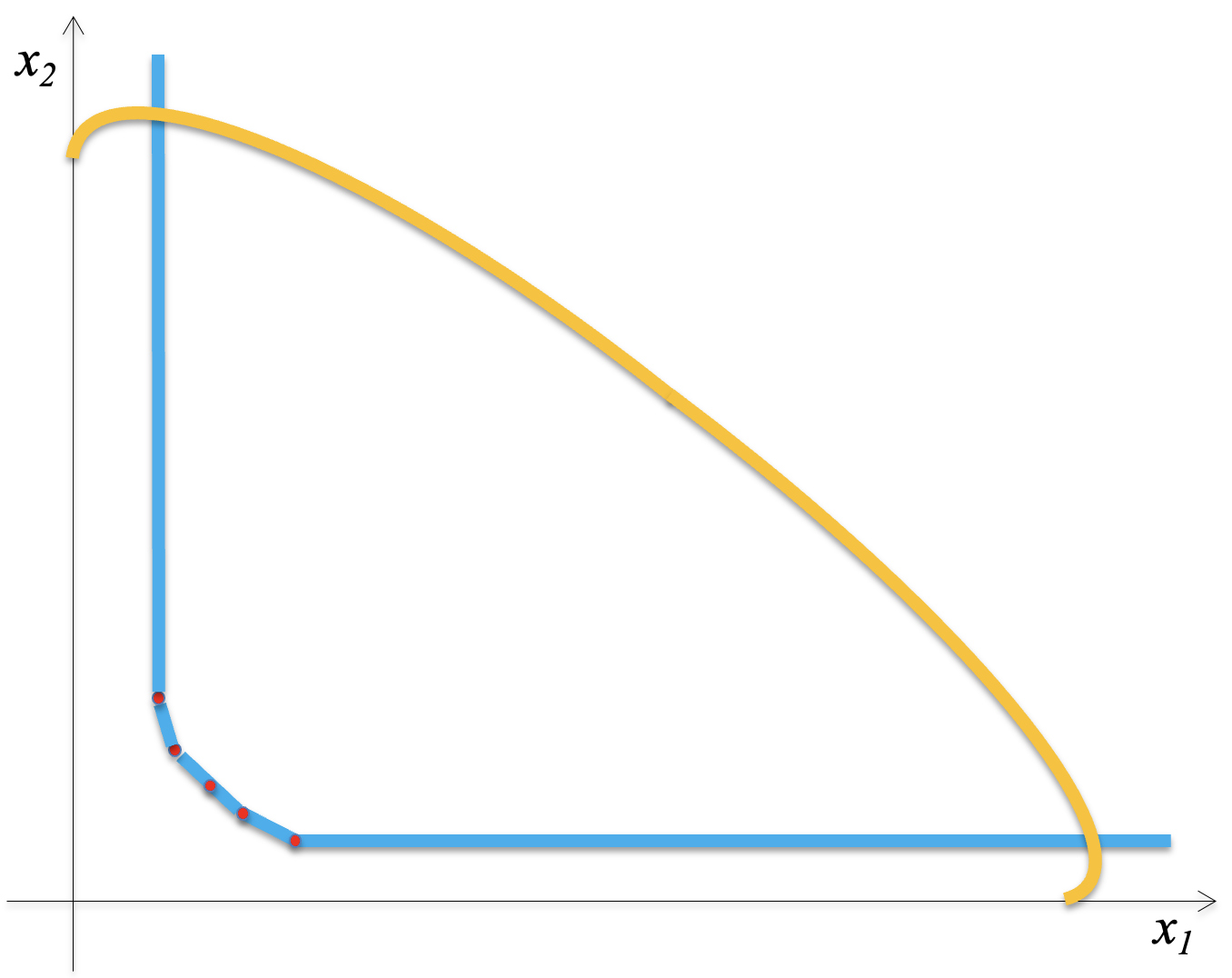}
\\
$(g)$      \includegraphics[width=1.4in]{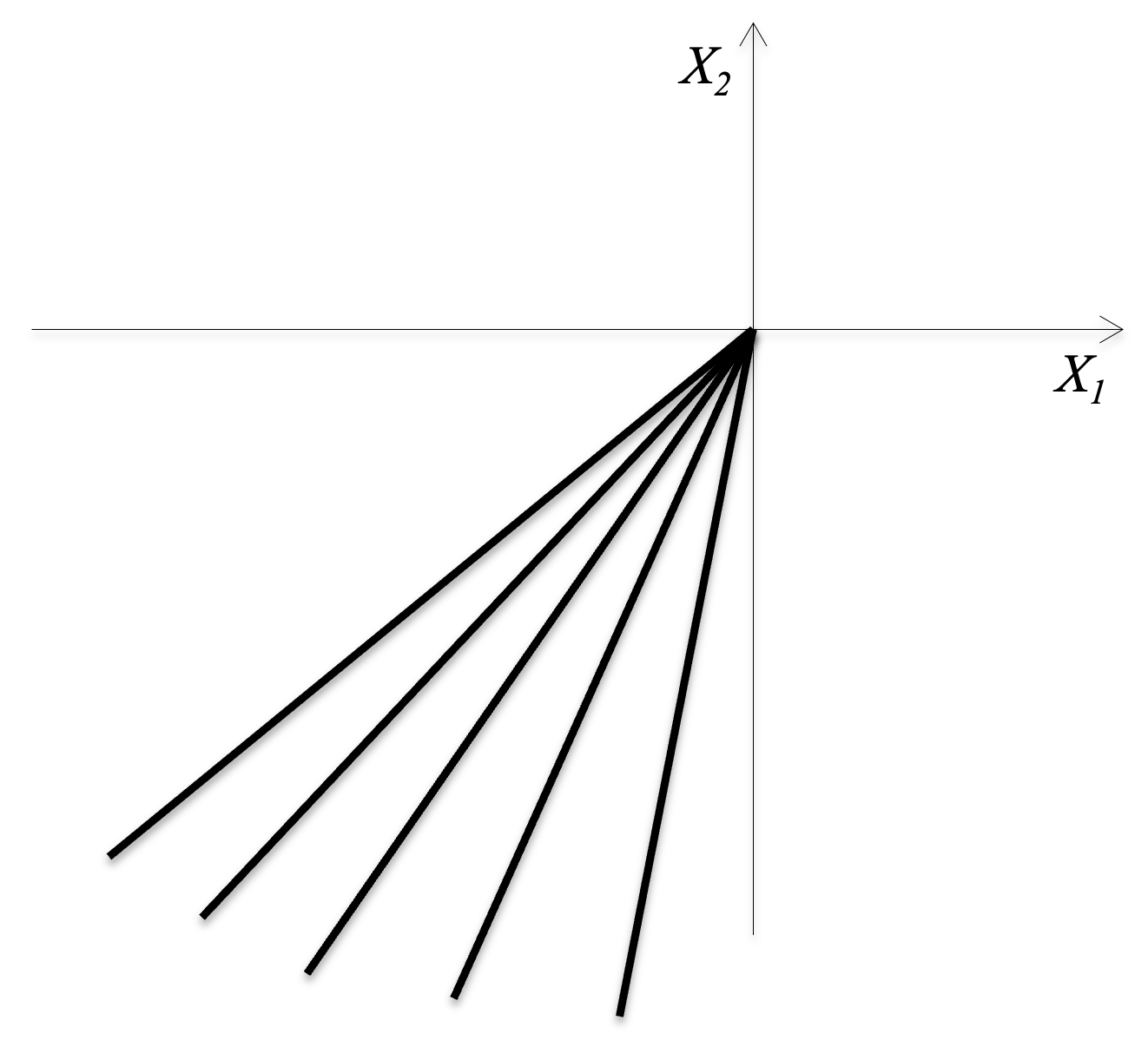}  \ \ \ 
$(h)$      \includegraphics[width=1.4in]{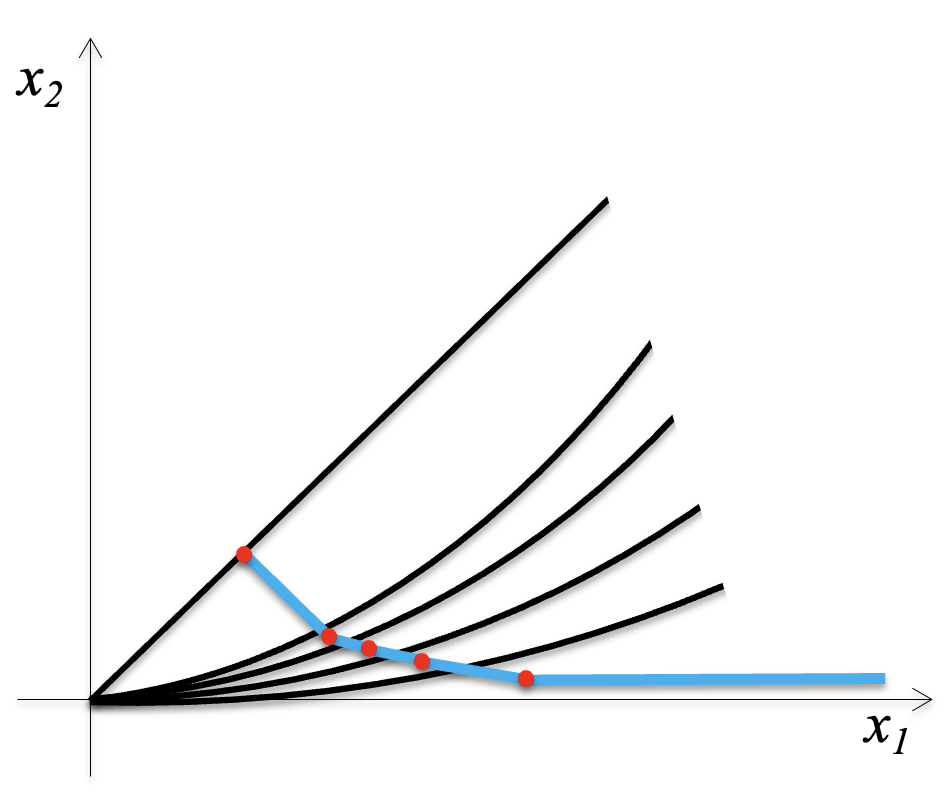} \ \ \ 
$(i)$      \includegraphics[width=1.7in]{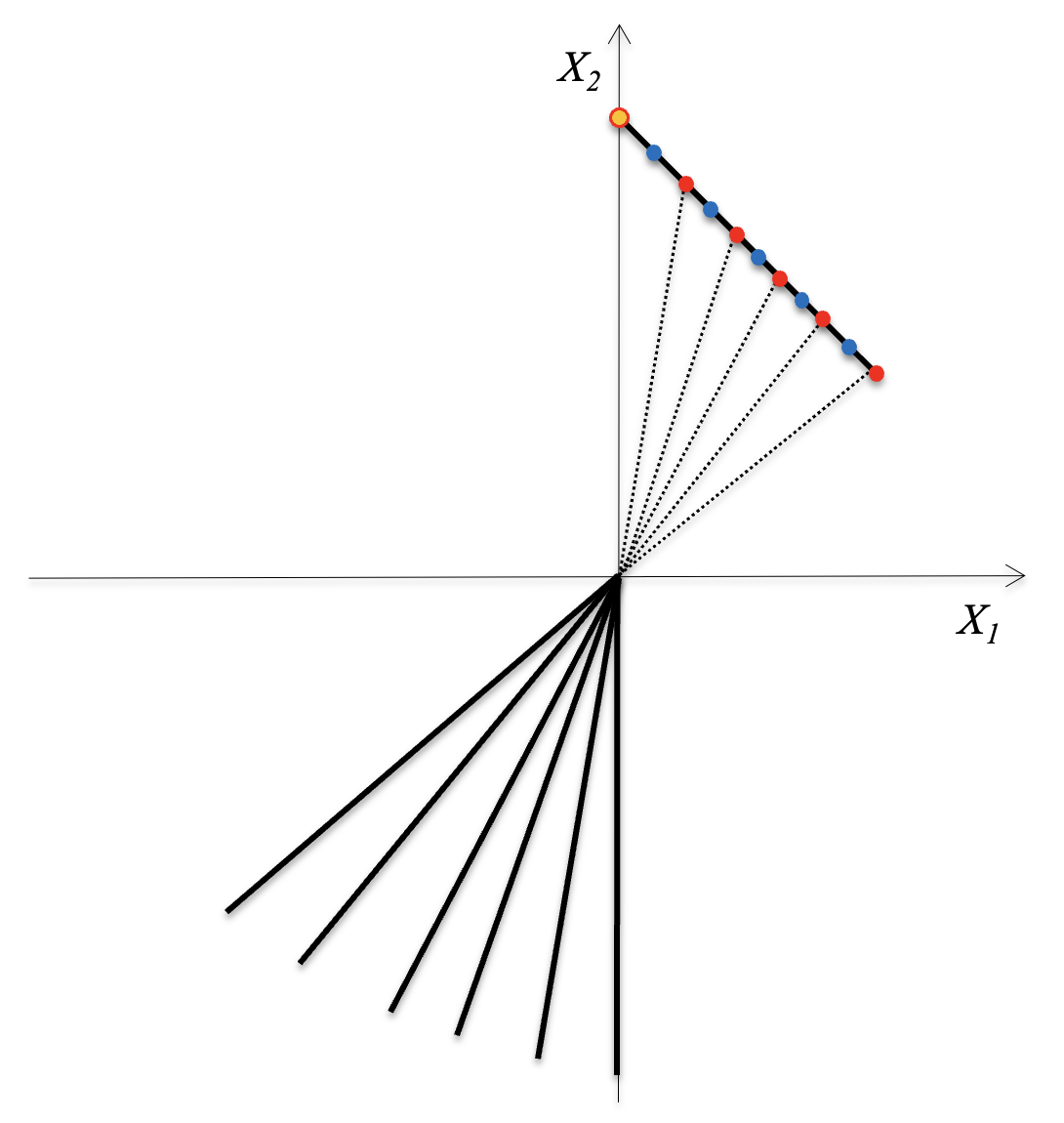} \ \ \
$(j)$      \includegraphics[width=1.2in]{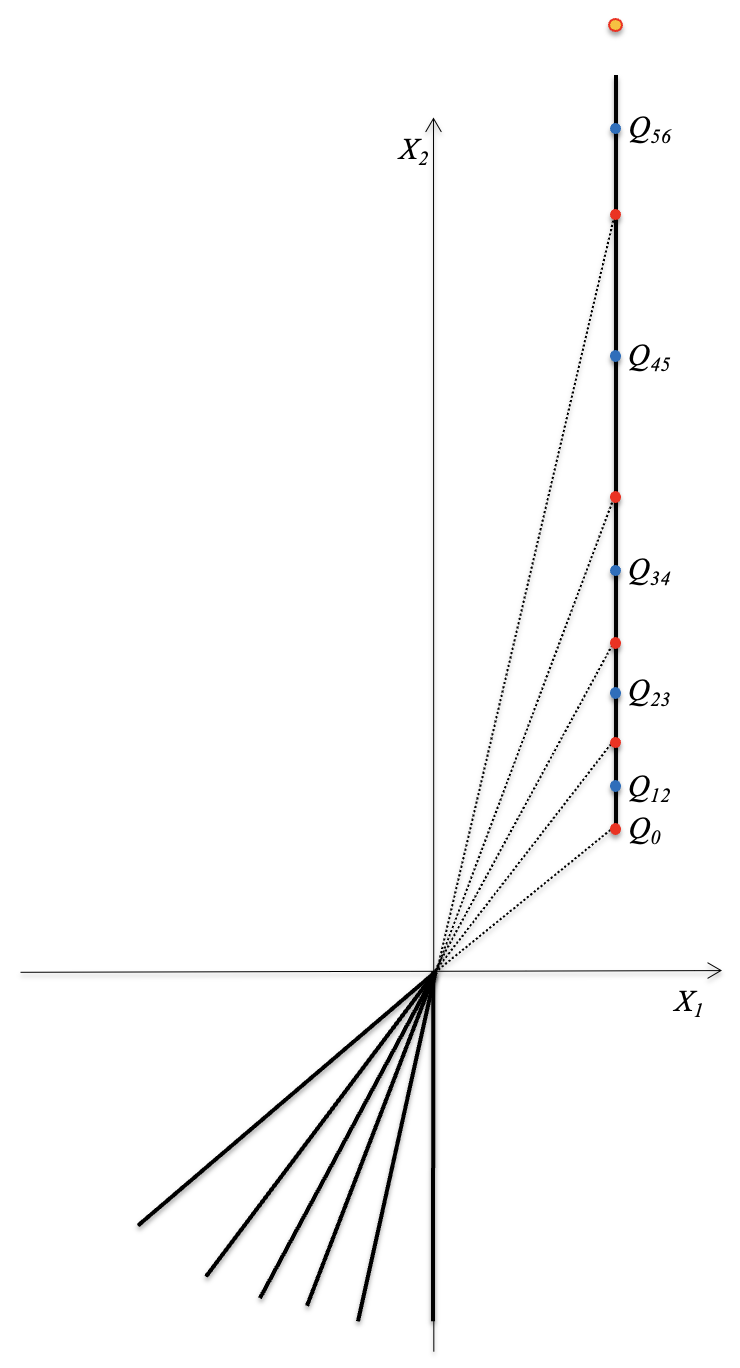}
\\
\\
\\
$(k)$      \includegraphics[width=3.4in]{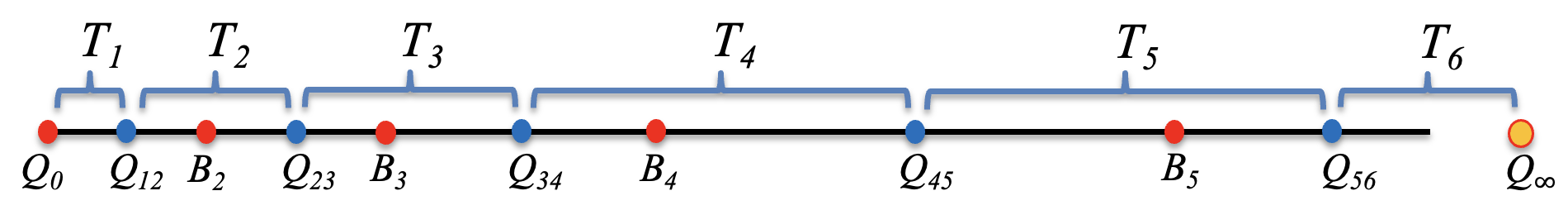}
    \end{tabular}
  \end{center}
  \caption
{\scriptsize \label{FIG_2} 
$(a)$ An example of a two-dimensional polyhedral fan.
$(b)$ An example of a {\em \ul{polar} differential inclusion}. Here we show regions at distance $<\delta$ from the cones of the fan, and we also draw some of the relevant polar cones, i.e., the polars of the cones of the fan that intersect the negative quadrant. 
$(c)$ An example of a {\em \ul{toric} differential inclusion}. The {\em black curves} that partition  the positive quadrant are images via the (coordinate-wise) exponential function of the {\em black lines} in $(b)$. Note that while the black lines \ul{are} mapped via the exponential function,  the (red) polar cones are just the same in $(c)$ as $(b)$ (i.e., they are \underline{not}  mapped via the exponential function), and therefore $(c)$ is  \underline{not} obtained from  $(b)$ simply via a change of variables. 
$(d)$ Construction of a ZSH for the toric differential inclusion shown in $(c)$.
$(e)$ Construction of a compact forward-invariant region within $\RRp^2$. This region consists of points that are above the ZSH and below a specific level set of the Horn-Jackson Lyapunov function.
$(f)$ The same compact forward-invariant region without the complicated details of the toric differential inclusion. Note that if we want to obtain  an invariant region of this type that contains a fixed (but arbitrary) $x_0 \in \RRp^2$, then we need to be able to construct the ZSH such that its horizontal and vertical edges {\em are very close to the coordinate axes, and reach very far outside the unit square $(0,1)^2$}.
$(g)$ For a different fan $\F_2$, after we \ul{\em symmetrize} it, it is sufficient to focus of the cones that lie within the lower half of the negative orthant; the set of these cones is denoted $\F_2^\le$.
$(h)$ Construction of (``half" of) a ZSH for a toric differential inclusion associated to the symmetrized fan.
$(i)$ Construction of a ZSH blueprint within $\D_1^P$.  
$(j)$ The same ZSH blueprint shown within $\D_1^H$ (instead of $\D_1^P$). 
$(k)$ The same ZSH blueprint for $\D_1^P$ as in $(i)$, but now shown within $\RR$ (the variable along the line is $X_2 \ge 1$, while the constant $X_1=1$ is {\em not} shown).
}
\end{figure}

%\newpage
%
%%%%%%%%%%%%%%%%%%%%%%%%%%%%%%%%%%%%%%%%%%%%%%%
%\begin{figure}[h!]
%  \begin{center}
%    \begin{tabular}{c}
%%      \hskip-2cm
%      \includegraphics[angle=180,width=4.2in]{IMG_0122.jpeg} \\
%      \includegraphics[angle=180,width=4.2in]{IMG_0123.jpeg}       
%    \end{tabular}
%  \end{center}
%  \caption{\label{fig:0122_0123}    Blackboards 0122 and 0123.}
%\end{figure}
%%%%%%%%%%%%%%%%%%%%%%%%%%%%%%%%%%%%%%%%%%%%%%%

\subsection{A different approach: the use of a ZSH blueprint and representations of a fan using subdivisions}

In preparation for the higher dimensional cases, we will now discuss a different point of view for the construction of a ZSH in $\RRp^2$, which is similar with the approach shown in Figure~\ref{FIG_2}$(d)$, but emphasizes less the two-dimensional  fan; instead, we introduce two (equivalent) notions of {\em  subdivisions} induced by the fan, which (in this case) are  one-dimensional objects that already contain all the relevant information about the two-dimensional fan, see Figures~\ref{FIG_2}$(i)$ and~\ref{FIG_2}$(j)$. Then, to any subdivision we associate a {\em ZSH blueprint} (see Figure~\ref{FIG_2}$(k)$), and we use it as a essential tool that guides our construction of the ZSH. (The ZSH constructed this way  will still be a polygonal curve, but we only need to construct it within the set $\{ 0 < x_2 \le x_1 \}$ because we can then infer the missing half by symmetry, see Figure~\ref{FIG_2}$(h)$.)

%In this approach, we first {\em symmetrize} the fan $\mathcal{F}_2$ (more precisely, we symmetrize the part of the fan that intersects the negative quadrant), and also make sure that  half-lines of slope 1, and the half-lines that lie along the coordinate axes of $\mathbb{R}^2$ belong to $\mathcal{F}_2$. In other words, we replace $\mathcal{F}_2$ with another fan (denoted also $\mathcal{F}_2$, for simplicity) which is a subdivision of the original $\mathcal{F}_2$ and also satisfies these additional properties [first add coordinate axes, then symmetrize within the negative quadrant]. 

\smallskip

{\em 
Note that, compared to the simplicity of the approach we used in the previous subsection, our approach in \underline{this} subsection will appear  \underline{unnecessarily complicated}. This is actually true, but, as we mentioned before, our goal in this subsection is not just to construct the  ZSH in the two-dimensional case, but to go through each and every step of the more complicated  construction that we have to perform in higher dimensions, while we are still in the \underline{simplest nontrivial setting:} the  two-dimensional case. 
}

\smallskip

Consider a fan $\mathcal{F}_2$ that is $X_1$-$X_2$-symmetric, and we also assume\footnote{These assumptions can be made without loss of generality, because, if needed, we are just replacing $\mathcal{F}_2$ with a finer version of $\mathcal{F}_2$; any ZSH for the finer fan will also be a ZSH for the coarser fan, see Lemma~\ref{lemma_1}.} that the three  lines given by $\{X_1 = X_2\}$, $\{X_1 = 0\}$, and $\{X_2 = 0\}$ can be realized as unions of cones that belong to $\mathcal{F}_2$. 
Recall that {\em we only need to specify  $\mathcal{F}_2$ within the (closed) negative orthant}; moreover, due to symmetry, we only need to specify  $\mathcal{F}_2$ within the subset of the negative orthant where we have $X_2 \le X_1 \le 0$. Let us denote the {\em restriction} of $\mathcal{F}_2$  to this subset by $\mathcal{F}_2^\le$ (here the superscript ``$\le$" is meant to remind us the inequality $X_2 \le X_1 \le 0$).
Then the part of $\mathcal{F}_2$ within the negative orthant is uniquely determined by $\mathcal{F}_2^\le$, and for now on we will focus on $\mathcal{F}_2^\le$.

Note now that  the information within $\mathcal{F}_2^\le$ is completely determined  if we specify the  subdivision\footnote{By {\em  subdivision} of $\mathscr{D}_1^H$ we mean a finite set of closed intervals and points that cover all of $\mathscr{D}_1^H$,  and satisfies properties that are similar to the properties of a fan: any endpoint of an interval of the subdivision is also an element of the subdivision, and the intersection of two intervals of the subdivision is an endpoint of both.} $\mathcal{F}_1^H$ of $\mathscr{D}_1^H$, obtained by considering the sets of the form
$$
(-C_i) \cap \mathscr{D}_1^H
$$
for all the cones $C_i \in \mathcal{F}_2^\le$, where
$$
\mathscr{D}_1^H = \{ (X_1^H, X_2^H) \in \RR^2_{\ge 0} \ | \  X_1^H \le X_2^H \mathrm{ and } \  X_1^H + X_2^H =1\}\footnote{The best way to think about the set $\mathscr{D}_1^H$ is geometrically: $\mathscr{D}_1^H$ is a concrete representation of the set of rays (i.e., one-dimensional cones) contained in the set $\{ X_2 \le X_1 \le 0 \}$. The element of $\mathscr{D}_1^H$ that represents a specific ray $\ell$ is the point $(-\ell) \cap \mathscr{D}_1^H$. This also allows us to understand why it is natural to ``see" the nonzero cones of $\mathcal{F}_2^\le$ inside $\mathscr{D}_1^H$: each such cone is a set of rays, so each such cone is determined by a subset of $\mathscr{D}_1^H$.},
$$ 
see Figure~\ref{FIG_2}$(i)$.

%Analogously, we obtain the subdivision $\mathcal{F}_2^P$ of $\mathscr{D}_2^P$, obtained by considering the sets of the form
%$$
%(-C_i) \cap \mathscr{D}_2^P
%$$
%for all the cones $C_i \in \mathcal{F}_2^\le$. 

Analogously, the information within $\mathcal{F}_2^\le$ is completely determined  if we specify the polygonal subdivision $\mathcal{F}_2^P$ of $\mathscr{D}_2^P$, obtained by considering the sets of the form
$$
(-C_i) \cap \mathscr{D}_2^P
$$
for all the cones $C_i \in \mathcal{F}_2^\le$, where\footnote{For simplicity, when we draw the set $\mathscr{D}_1^P$, we will show it as a 1D set within the $X_2^P$-space, automatically understood to lie within the subset of $\RR^2$ where the first coordinate $X_1^P$ is 1, see Figure~\ref{FIG_2}$(k)$ (and compare it with Figure~\ref{FIG_2}$(j)$).} 
$$
\mathscr{D}_1^P = {\rm compactification}(\{ (1, X_2^P) \in \RR^2_{\ge 0} |  1 \le X_2^P \})\footnote{We need a one-point compactification because we want $\mathscr{D}_1^P$ to be able to play the same role as $\mathscr{D}_1^H$, i.e., to be in one-to-one correspondence with the set of rays contained in the set $\{ X_2 \le X_1 \le 0 \}$. Then we can also ``see" all the nonzero cones of $\mathcal{F}_2^\le$ inside $\mathscr{D}_1^P$. Without this compactification we would ``fail to see" the cone of $\mathcal{F}_2^\le$ that goes along the negative $X_2$-axis.},
$$ 
see Figure~\ref{FIG_2}$(j)$. Note that the point at infinity is shown in Figure~\ref{FIG_2}$(j)$ as the red point with orange center.
%This compactification should be understood in the {\em classical} sense, i.e., each point at infinity corresponds to the set of half-lines within the interior of $\mathscr{D}_2^P$ that are parallel to one specific half-line; equivalently, each point at infinity corresponds to a specific slope that can be realized by half-lines contained within $\mathscr{D}_2^P$. Then, literally, parallel half-lines ``meet at infinity" within $\mathscr{D}_2^P$.

Equivalently, we can define the elements of  $\mathcal{F}_1^H$ as the images of the elements of $\mathcal{F}_2^\le$ (except for the zero vector) via the transformation $\psi_2^H$ given by
$$
\psi_2^H(X_1, X_2) = \left(\frac{X_1}{X_1+X_2}, \frac{X_2}{X_1+X_2} \right),
$$
and we can define the elements of  $\mathcal{F}_1^P$ as the images of the elements of $\mathcal{F}_2^{\ge}$ via the transformation  $\psi_2^P$ given by\footnote{In the case of $\psi_2^P$ we {\em allow division by zero} when $X_1=0$, and then the point  at infinity belongs to the image of $\psi_2^P$.}
$$
\psi_2^P(X_1, X_2) = \left( 1, \frac{X_2}{X_1} \right).
$$

\bigskip

\noindent
{\bf Remark.} {\em Since we have compactified $\mathscr{D}_1^P$ with a point at infinity, it follows that the two transformations $\psi_2^H$ and $\psi_2^P$ described above become \underline{homeomorphisms} from $\mathscr{D}_1^P$ to $\mathscr{D}_1^H$, and from $\mathscr{D}_1^H$ to $\mathscr{D}_1^P$, respectively. This fact gives another way to think about our compactification of  $\mathscr{D}_1^P$: it is the  compactification of the set $\{ (1, X_2^P) \in \RR^2_{\ge 0} |  1 \le X_2^P  \}$ for which 
$$
\psi_2^P : \mathscr{D}_1^H \to \mathscr{D}_1^P
$$ 
is a homeomorphism. 
\qed
}

\bigskip

%\bigskip

Let us now explain how one can take advantage of the subdivision $\F_1^P$ of  $\D_1^P$ to construct a ZSH for $\F_2$. 

The first step (see Figure~\ref{FIG_2}$(k)$) is to construct one more subdivision, with intervals $T_i$, as shown in Figure~\ref{FIG_2}$(k)$. Note that each interval $T_i$ contains exactly one red point (the red points are elements of $\F_1^P$). The red point inside $T_i$ is called the {\em basepoint} of $T_i$. For example, using the notation from Figure~\ref{FIG_2}$(k)$ the point $Q_0$ is the basepoint of $T_1$, $B_i$ is the basepoint of $T_i$ for $2 \le i \le 5$, and the point $Q_\infty$ (which is the point at infinity of $\D_1^P$) is the basepoint of $T_6$.

The subdivision with intervals $T_i$ is called a \underline{\em ZSH blueprint of $\F_2$}, and each interval $T_i$ is called a \underline{\em tile} of the ZSH blueprint.

Once we have the {\em ZSH blueprint} (which contains information about all the {\em tiles} and their {\em basepoints}) we will see that this is sufficient information in order to construct  a ZSH. More precisely, we will impose that each edge of the ZSH  maps to a tile of the ZSH blueprint via the map
$$
\Psi_2(X_1, X_2) = \left( 1, \ \frac{\log X_2}{\log X_1} \right),
$$
and moreover each edge is orthogonal to the basepoint of its corresponding  tile.

Recall that we plan to construct only the ``half" of the ZSH that lies below the line $\{ x_1=x_2 \}$, because we can then obtain the other half by symmetry. 

Let us now start the construction of a ZSH (like the one shown in Figure~\ref{FIG_2}$(h)$) by choosing first the point $P_0$ of the ZSH that lies on the line $\{ x_1 = x_2 \}$, say the point of coordinates $(\varepsilon, \varepsilon)$ for some very small number $\eps > 0$. 
Note that we have 
$$
\Psi_2(P_0) = Q_0,
$$ 
which is the left-side endpoint of the tile $T_1$. Then we proceed as follows. The first line segment of the ZSH we are constructing starts at the point $P_0$ and goes below the line $\{ x_1 = x_2 \}$ in a direction orthogonal to the basepoint of $T_1$, until it reaches a point $P_{12}$ such that 
$$
\Psi_2(P_{12}) = Q_{12},
$$
which is the right-side endpoint of the tile $T_1$. 

The next line segment of the ZSH will start at the point $P_{12}$ and go in a direction\footnote{Strictly speaking, there are {\em two} directions that start at the point $P_{12}$ and go orthogonal to the basepoint of $T_2$, and there is an angle of $180^o$ between them. We choose the one that {\em can} result in a point $P_{23}$ such that $\Psi_2(P_{23}) = Q_{23}$, i.e., the one where {\em $X_2^P$ increases along the edge}.} orthogonal to the basepoint of $T_2$, until it reaches  a point $P_{23}$ such that 
$$
\Psi_2(P_{23}) = Q_{23},
$$
which is the right-side endpoint of the tile $T_2$. 

We continue similarly until we reach the point $P_{56}$ such that 
$$
\Psi_2(P_{56}) = Q_{56},
$$
which is the right-side endpoint of the tile $T_5$. 

Finally, the last edge of our ZSH starts at $P_{56}$ and goes in a direction orthogonal to the basepoint of $T_6$; but the basepoint of $T_6$ is the point at infinity denoted $Q_\infty$, which means that this last edge must be perfectly horizontal\footnote{Because, informally speaking, the coordinates of $Q_\infty$ in $\D_1^P$ are $(1, \infty)$. If we want to avoid dealing with ``$\infty$" we can use the version of $Q_\infty$ from $\D_1^H$ (see Figure~\ref{FIG_2}$(i)$) instead of the $Q_\infty$ from $\D_1^P$; that one is simply the vector (0,1) so, of course, orthogonal to it means perfectly horizontal.}. 

To see that we have indeed succeeded in constructing a ZSH we claim now that we are in the setting of Lemma~\ref{lemma_outer_normal_sufficient}, i.e.,  for any nonzero cone $\C$ in $\F_2^\le$, the polygonal curve we have constructed above has the property that there exists a  cone $K_\C$ such that $\C$ is contained in the interior of $K_\C$ and such that for any point $P$ of the ZSH such that $\log(P)$ belongs to $K_\C$ we have that the outer normal of the ZSH at $P$ belongs to $\C$.
 
Indeed, if $\C$ is a one-dimensional cone in $\F_2^\le$ then the image of  $\C$ within $\D_2^P$ is the basepoint of a tile $T_i$. Then we can define $K_\C$ to be the two-dimensional cone that corresponds to the tile $T_i$. Also, consider the case where $\C$ is a two-dimensional cone in $\F_2^\le$. Then, regarded within $\D_2^P$, the cone   $\C$ ``spans" two adjacent tiles, say $T_i$ and $T_{i+1}$, and we can  define $K_\C$ to be the two-dimensional cone that corresponds to the union of the tiles $T_i$ and $T_{i+1}$. 

Note also that we can choose the initial point  $(\eps,\eps)$ to have an arbitrarily small $\eps$.
Therefore, we are indeed in the setting of Lemma~\ref{lemma_outer_normal_sufficient}\footnote{Note that Lemma~\ref{lemma_outer_normal_sufficient} only guarantees  that our toric differential inclusion cannot cross our polygonal curve  {\em within the unit square $(0,1)^2$}. Nevertheless, it is sufficient for our purposes, because this will remain true for a large enough distance outside the unit square (more precisely, until the curve reaches a new uncertainly region along one of the axes, which, if $\eps$ is small enough, will only happen for very large values of $x_1$), see Figure~\ref{FIG_2}$(e)$.}, which implies that our polygonal curve is a ZSH.

\newpage 

\section{Proof of the global attractor conjecture in the three-dimensional case}
\label{section:proof_in_3D}

We now discuss the construction of ZSH for three-dimensional toric differential inclusions. 

For our construction of ZSH in $\RRp^3$ we will rely very heavily on two-dimensional representations of three-dimensional fans using subdivisions of the sets $\mathscr{D}_2^P$ (and/or $\mathscr{D}_2^H$) described below, and on the corresponding notion of {\em two-dimensional ZSH blueprint}.

Recall that in the previous section we have been able to use {\em one-dimensional ZSH blueprints} for the construction of one-dimensional ZSH in $\RRp^2$ (i.e., for the construction of polygonal curves that satisfy the ZSH property). Moreover, our construction was ``canonical", in the following sense: once we specified the starting point $P_0$ of coordinates $(\eps, \eps)$, everything else was uniquely determined by the following  simple conditions: $(i)$ the {\em edges} of the ZSH are in one-to-one correspondence with the {\em tiles} of the ZSH blueprint,  $(ii)$ each edge is {\em orthogonal to the basepoint}  of its corresponding tile, and $(iii)$ the image of each edge via the map $\Psi_2$ is {\em exactly} its corresponding tile.

In the next subsection  we will choose an especially convenient fan $\F_3$,  define a ZSH blueprint for it, and   devise a construction process for the ZSH which is based on this ZSH blueprint. We will see that we have to adapt and modify some of the rules $(i)$-$(iii)$ described above, but in the end we {\em can} create such a construction process, and to some extent it can also be seen as ``canonical", under somewhat relaxed restrictions. More precisely, our construction give rise to a {\em polyhedral ZSH} such that: $(i')$ the {\em faces} of the ZSH are in one-to-one correspondence with the {\em tiles} of the ZSH blueprint, $(ii')$ each face is {\em orthogonal to the basepoint} of its corresponding tile, and $(iii')$ the image of each edge via the map $\Psi_3$ is {\em arbitrarily close} to its corresponding tile (in a sense that will be clarified below). 

In a later subsection we will describe the construction of a polyhedral ZSH for a fully general choice of fan $\F_3$. 

\subsection{The proof in a special case}

Similar to the previous section, we consider a fan $\mathcal{F}_3$ that is $X_1$-$X_2$-$X_3$-symmetric, and we  assume\footnote{Recall that these assumptions can be made without loss of generality, because, if needed, we are just replacing $\mathcal{F}_3$ with a finer version of $\mathcal{F}_3$; any ZSH for the finer fan will also be a ZSH for the coarser fan.} that all planes $\{X_i = X_j\}$ and $\{X_i = 0\}$ can be realized as unions of cones that belong to $\mathcal{F}_3$. 
Recall that we only need to specify  $\mathcal{F}_3$ within the negative orthant; moreover, due to symmetry, we only need to specify  $\mathcal{F}_3$ within the subset of the negative orthant where we have $X_3 \le X_2 \le X_1 \le 0$. Let us denote the {\em restriction} of $\mathcal{F}_3$  to this subset by $\mathcal{F}_3^\le$ (here the superscript ``$\le$" is meant to remind us the inequality $X_3 \le X_2 \le X_1 \le 0$).
Then the part of $\mathcal{F}_{3}$ within the negative orthant is uniquely determined by $\mathcal{F}_{3}^\le$, and for now on we will focus on $\mathcal{F}_{3}^\le$.

Also similar to the previous section, the information within $\mathcal{F}_{3}^\le$ is completely determined  if we specify the polygonal subdivision\footnote{By {\em polygonal subdivision} (or simply {\em  subdivision}) of $\mathscr{D}_2^H$ we mean a set of convex polygons, line segments, and points,  which satisfies properties that are similar to the properties of a fan: any face of an element of the subdivision is also an element of the subdivision, and the intersection of two elements of the subdivision is a face of both.} $\mathcal{F}_2^H$ of $\mathscr{D}_2^H$, obtained by considering the sets of the form
$$
(-C_i) \cap \mathscr{D}_2^H
$$
for all the cones $C_i \in \mathcal{F}_{3}^\le$, where
$$
\mathscr{D}_2^H = \{ (X_1^H, X_2^H, X_3^H) \in \RR^3_{\ge 0} \ | \  X_1^H \le X_2^H \le X_3^H \mathrm{ and } \  X_1^H + X_2^H + X_3^H =1\}\footnote{The best way to think about the role/meaning of $\mathscr{D}_2^H$ is geometrically: $\mathscr{D}_2^H$ is a concrete representation of the set of rays (i.e., one-dimensional cones) contained in the set $\{ X_3 \le X_2 \le X_1 \le 0 \}$. The element of $\mathscr{D}_2^H$ that represents a specific ray $\ell$ is the point $(-\ell) \cap \mathscr{D}_2^H$. This also allows us to understand why it is natural to ``see" the nonzero cones of $\mathcal{F}_{3}^\le$ inside $\mathscr{D}_2^H$: each such cone is a set of rays, so each such cone is determined by a subset of $\mathscr{D}_2^H$.}.  
$$

Analogously, the information within $\mathcal{F}_{3}^\le$ is completely determined  if we specify the polygonal subdivision $\mathcal{F}_2^P$ of $\mathscr{D}_2^P$, obtained by considering the sets of the form
$$
(-C_i) \cap \mathscr{D}_2^P
$$
for all the cones $C_i \in \mathcal{F}_{3}^\le$, where\footnote{For simplicity, when we draw the set $\mathscr{D}_2^P$, we will show it as a 2D set within the $X_2^P$-$X_3^P$-space, automatically understood to lie within the subset of $\RR^3$ where the first coordinate is 1, see Figure~\ref{fig:3a}$(a)$.} 
$$
\mathscr{D}_2^P = {\rm compactification}(\{ (1, X_2^P, X_3^P) \in \RR^3_{\ge 0} |  1 \le X_2^P \le X_3^P \})\footnote{We need this compactification because we want $\mathscr{D}_2^P$ to be able to play the same role as $\mathscr{D}_2^H$, i.e., to be in one-to-one correspondence with the set of rays contained in the set $\{ X_3 \le X_2 \le X_1 \le 0 \}$. Then we can also ``see" all the nonzero cones of $\mathcal{F}_{3}^\le$ inside $\mathscr{D}_2^P$, including the ones that lie in the plane $X_1=0$.}.
$$ 
This compactification should be understood in the {\em classical} sense, i.e., each point at infinity corresponds to the set of half-lines within the interior of $\mathscr{D}_2^P$ that are parallel to one specific half-line; equivalently, each point at infinity corresponds to a specific slope that can be realized by half-lines contained within $\mathscr{D}_2^P$. Then, literally, parallel half-lines ``meet at infinity" within $\mathscr{D}_2^P$.

Note that, equivalently, we can define the elements of  $\mathcal{F}_2^H$ as the images of the elements of $\mathcal{F}_{3}^\le$ (except for the zero vector) via the transformation $\psi_3^H$ given by
$$
\psi_3^H(X_1, X_2, X_3) = \left(\frac{X_1}{X_1+X_2+X_3}, \frac{X_2}{X_1+X_2+X_3}, \frac{X_3}{X_1+X_2+X_3}\right),
$$
and we can define the elements of  $\mathcal{F}_2^P$ as the images of the elements of $\mathcal{F}_{3}^{\ge}$ via the transformation  $\psi_3^P$ given by\footnote{In the case of $\psi_3^P$ we {\em allow division by zero} when $X_1=0$, in order for the points at infinity to be included in the image of $\psi_3^P$.}
$$
\psi_3^P(X_1, X_2, X_3) = \left(1, \frac{X_2}{X_1}, \frac{X_3}{X_1}\right).
$$

\bigskip

\noindent
{\bf Remark.} {\em Since we have compactified $\mathscr{D}_2^P$ with points at infinity, it follows that the two transformations $\psi_3^H$ and $\psi_3^P$ described above become \underline{homeomorphisms} from $\mathscr{D}_2^P$ to $\mathscr{D}_2^H$, and from $\mathscr{D}_2^H$ to $\mathscr{D}_2^P$, respectively. This fact gives the simplest way to understand clearly our compactification of  $\mathscr{D}_2^P$: it is \underline{the (unique) compactification} of the set $\{ (1, X_2^P, X_3^P) \in \RR^3_{\ge 0} |  1 \le X_2^P \le X_3^P \}$ for which 
$$
\psi_3^P : \mathscr{D}_2^H \to \mathscr{D}_2^P
$$ 
is a homeomorphism. 
\qed
}

\bigskip

\subsubsection{Construction of a ZSH blueprint $\B(\F_2^P)$ for a conveniently chosen $\mathcal{F}_3^\le$}

\noindent
For the sake of simplicity, let us now assume that  $\mathcal{F}_3^\le$ gives rise to the $\mathcal{F}_2^P$ which is \underline{exactly} the subdivision of $\mathscr{D}_2^P$ that is shown in Figure~\ref{fig:3a}$(a)$. We have chosen this particular fan because we can now quickly build our key tool for the construction of a ZSH, i.e., a \underline{ZSH blueprint} associated to the subdivision $\mathcal{F}_2^P$.
Indeed, a ZSH blueprint is shown (together with $\mathcal{F}_2^P$) in Figure~\ref{fig:3a}$(b)$ with blue lines; and the same ZSH blueprint is shown (without $\mathcal{F}_2^P$)  in Figure~\ref{fig:3a}$(c)$. 

%%%%%%%%%%%%%%%%%%%%%%%%%%%%%%%%%%%%%%%%%%%%%%
\begin{figure}[h!]
  \begin{center}
    \begin{tabular}{c}
      \hskip-1.4cm
$(a)$\includegraphics[angle=0,width=1.7in]{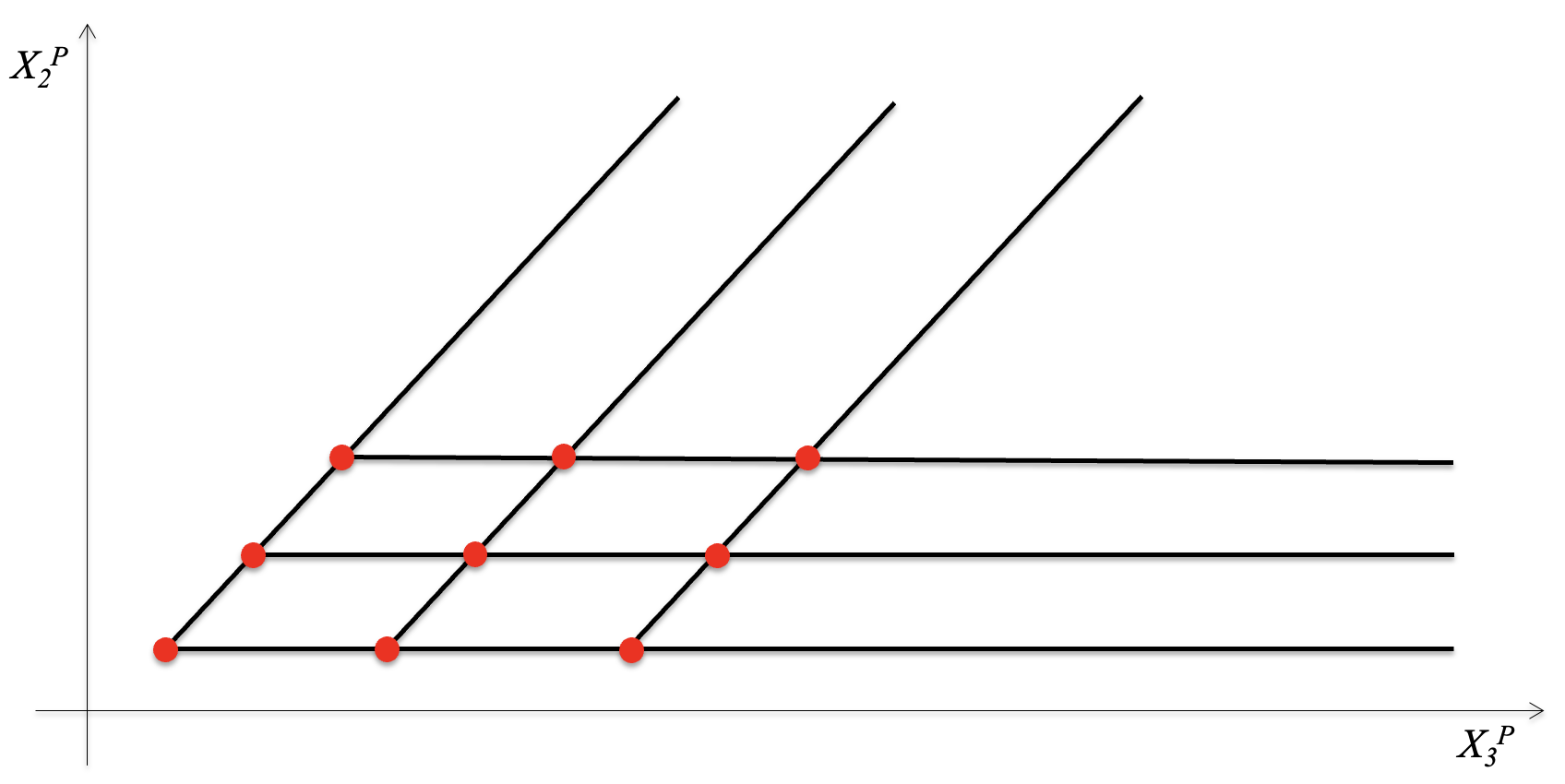} 
$(b)$\includegraphics[angle=0,width=1.7in]{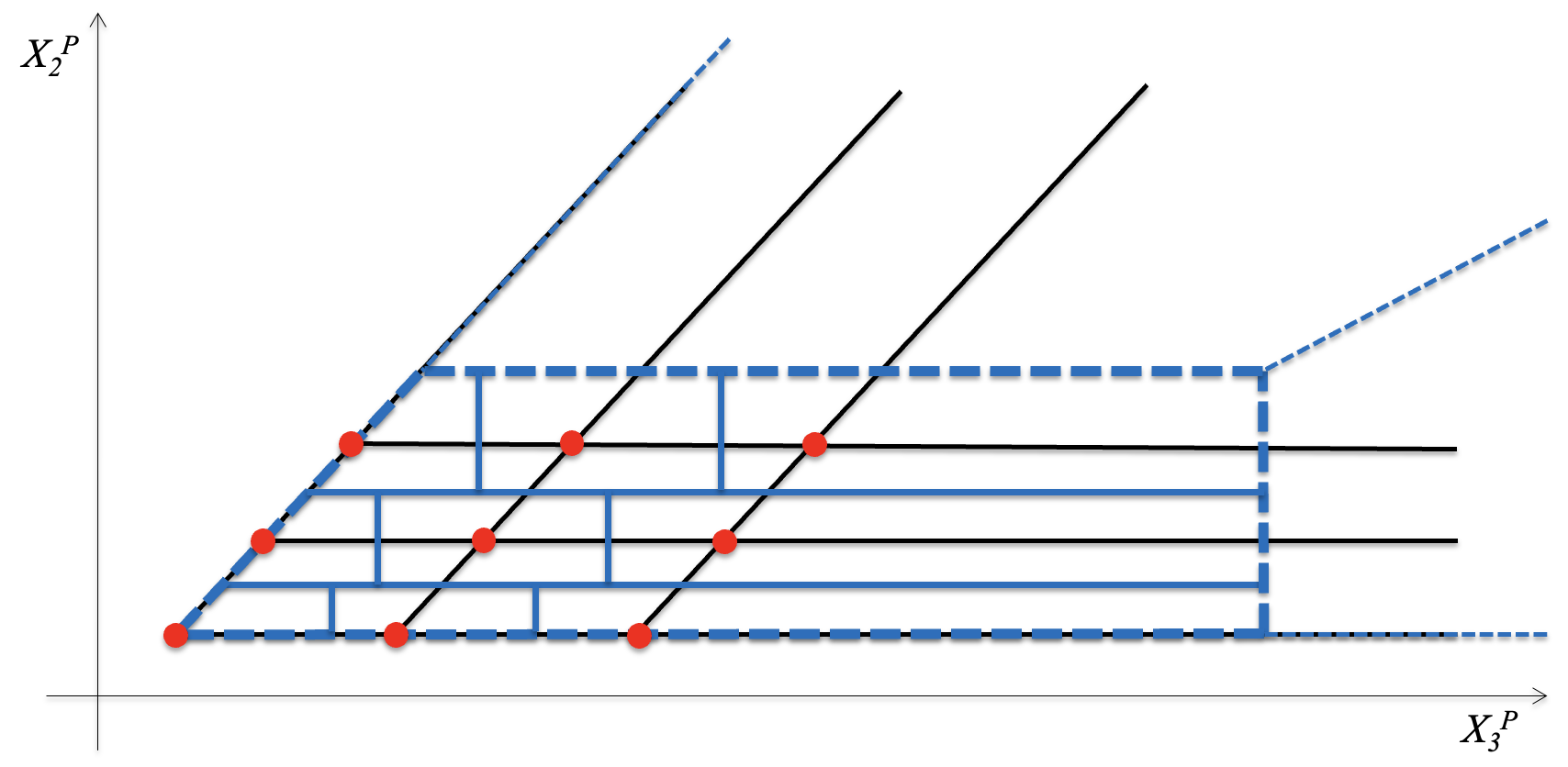} 
$(c)$\includegraphics[angle=0,width=1.7in]{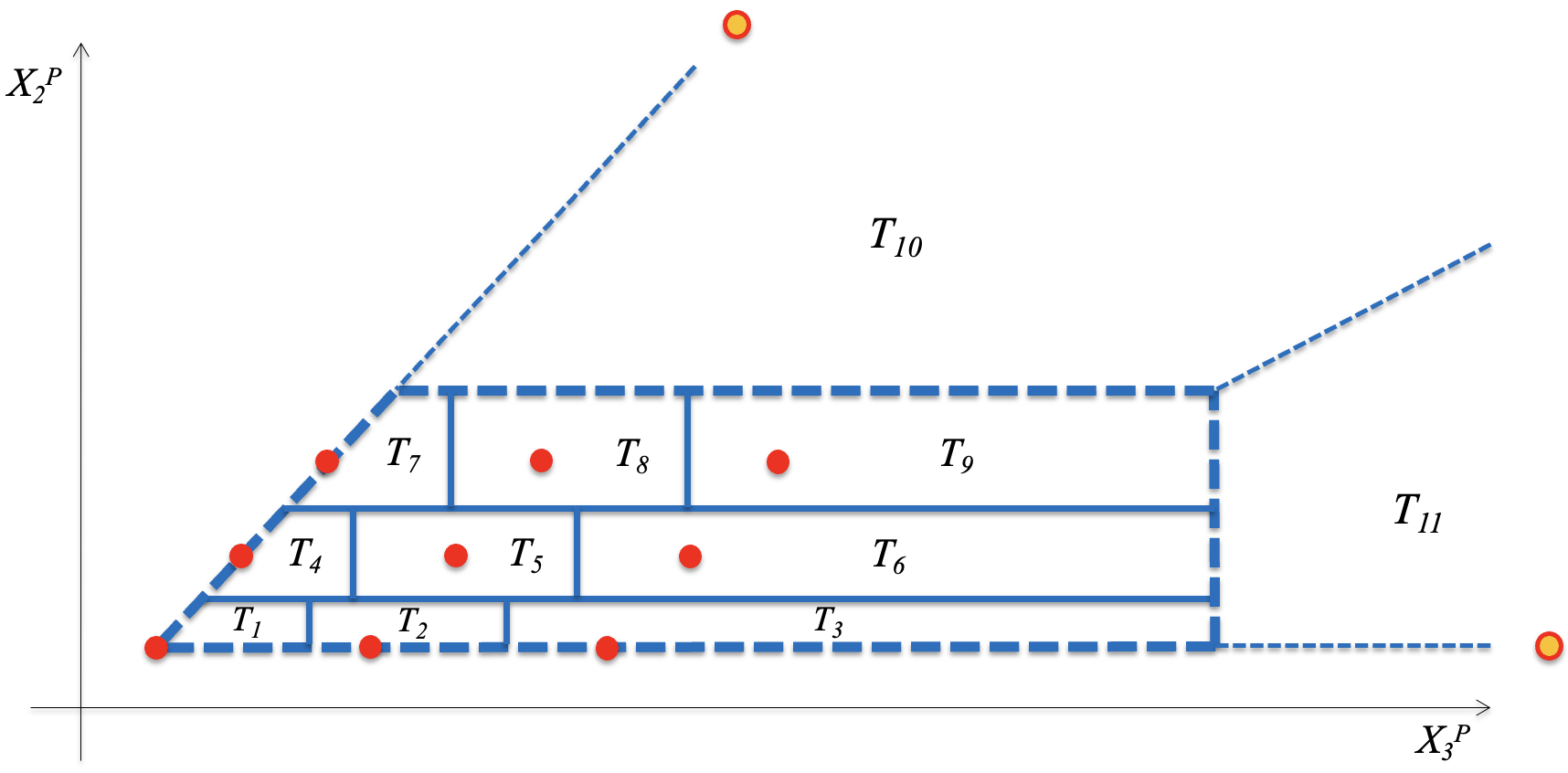}       \\
    \hskip-1.4cm
$(d)$\includegraphics[angle=0,width=1.7in]{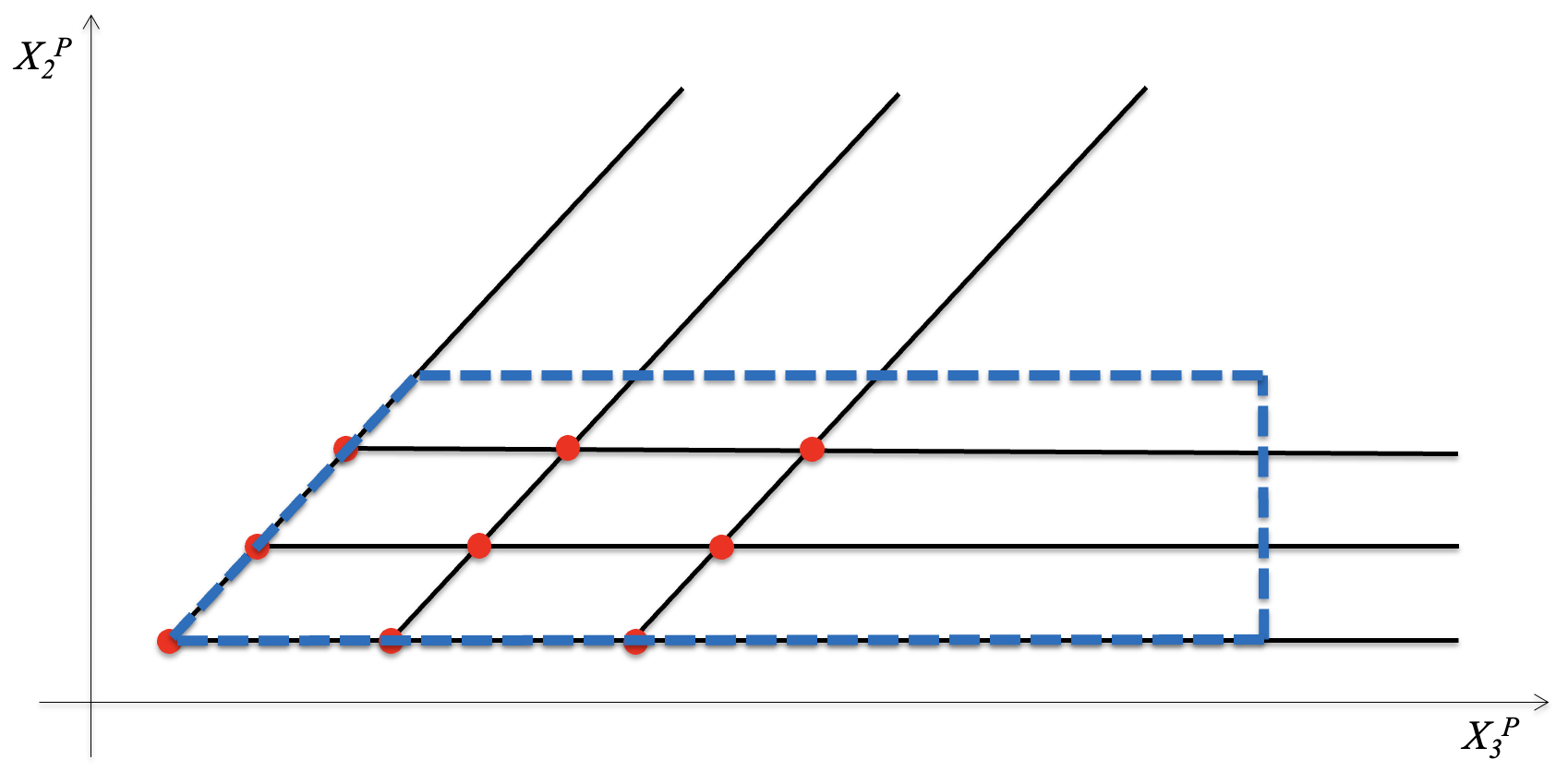} 
$(e)$\includegraphics[angle=0,width=1.7in]{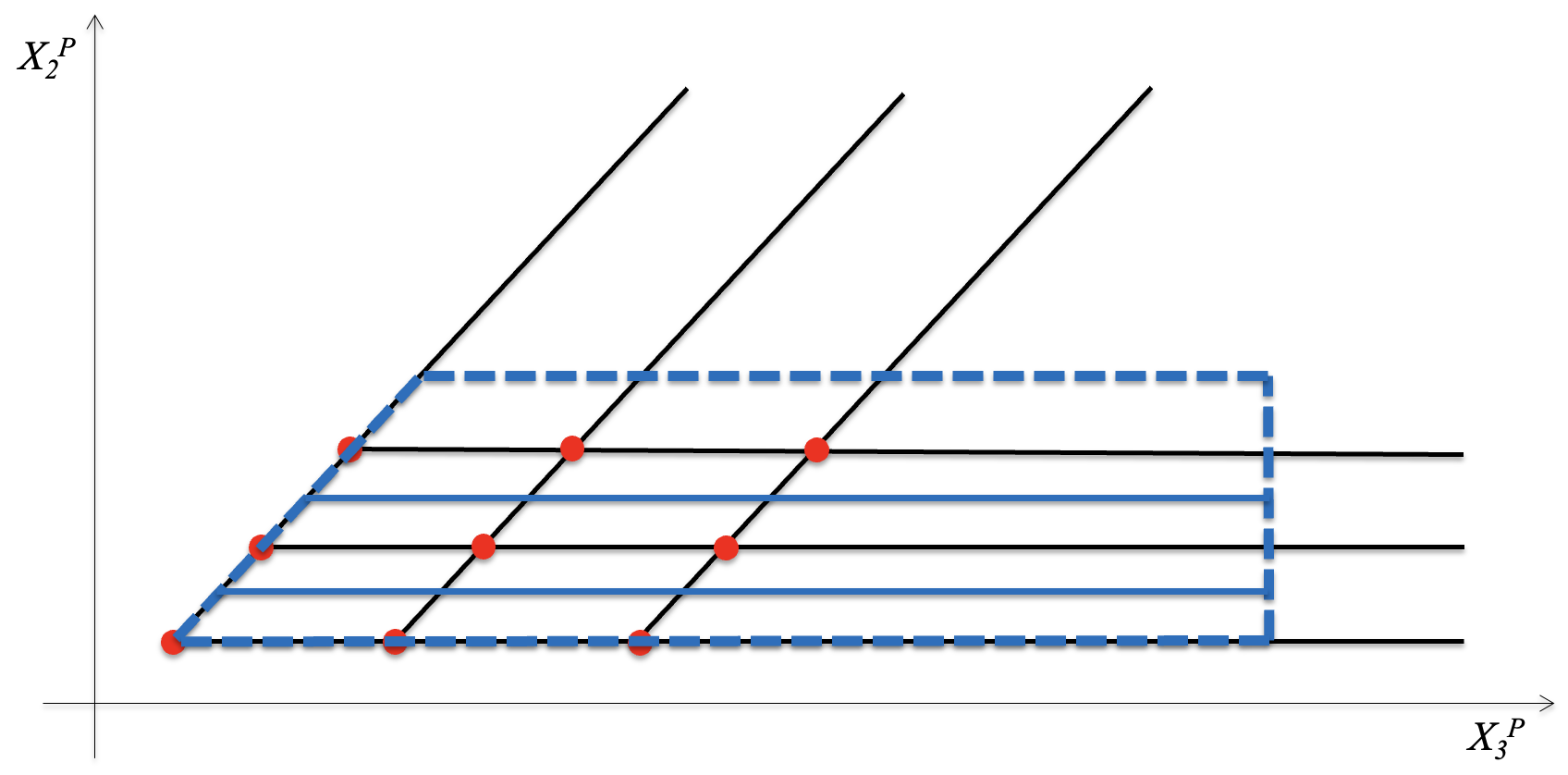} 
$(f)$\includegraphics[angle=0,width=1.7in]{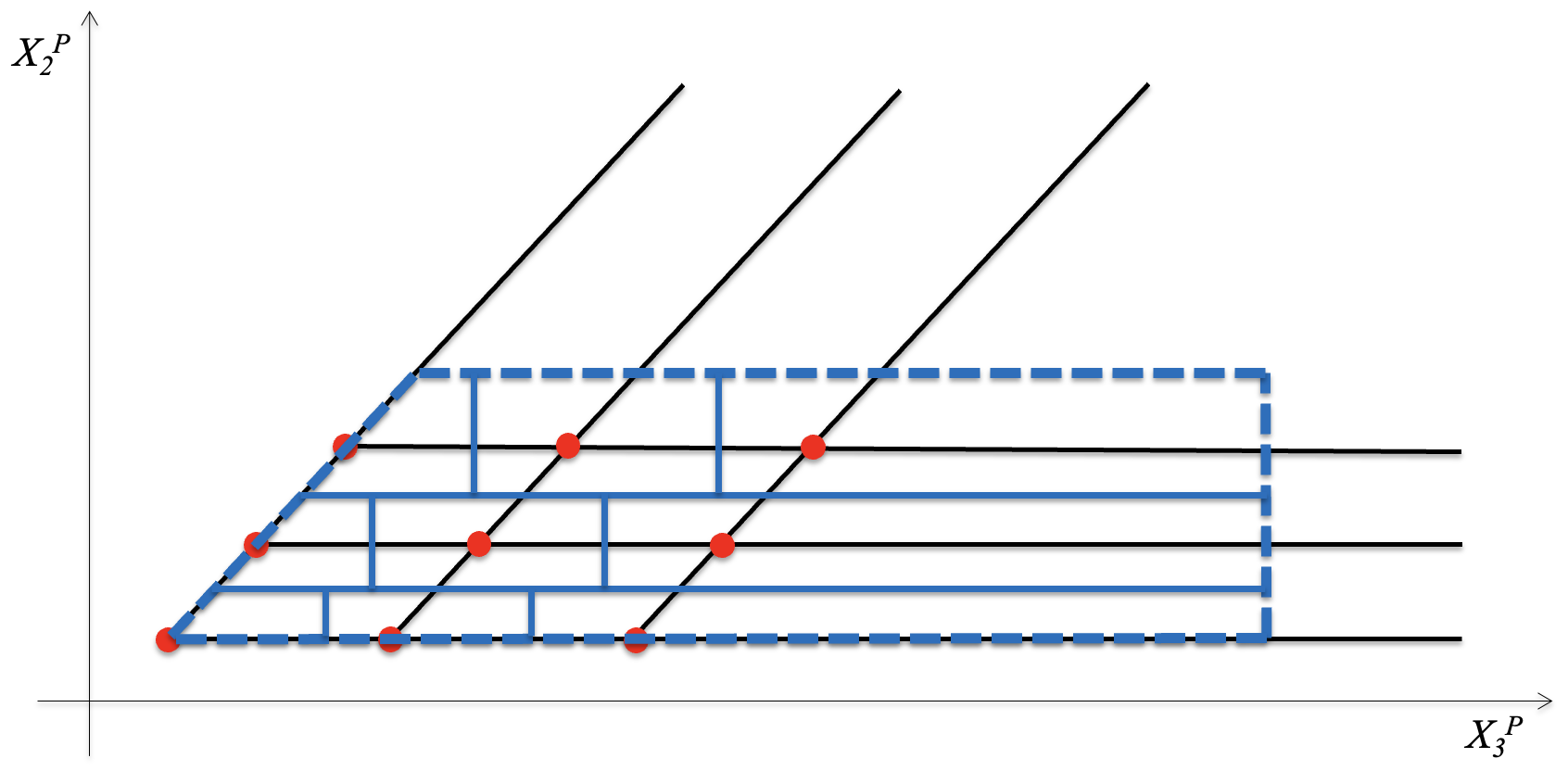}  
    \end{tabular}
  \end{center}
  \caption{\scriptsize \label{fig:3a} In order to illustrate our construction in a simple case we choose an especially convenient 3D fan $\mathcal{F}_{3}$, whose representation as a subdivision $\mathcal{F}_{2}^P$ of $\mathscr{D}_2^P$ is shown in $(a)$. In $(b)$ we show together $\mathcal{F}_{2}^P$ (with black lines) and a choice of ZSH blueprint (with blue lines). In $(c)$ we show {\em just} the ZSH blueprint, without the subdivision that generated it; note that the \underline{\underline{``basepoints"}} (i.e., the red points) are part of the definition of the ZSH blueprint. Note also that  \underline{two basepoints {\em lie at infinity} } (these are the two the points with orange center). 
In $(d)$, $(e)$, $(f)$ we show the intermediate steps in the construction of the ZSH blueprint. 
}
\end{figure}
%%%%%%%%%%%%%%%%%%%%%%%%%%%%%%%%%%%%%%%%%%%%%%

\bigskip

%\noindent
Let us now clarify the notion of ZSH blueprint for this example. 
In general, given a subdivision $\mathcal{F}_2^P$ of $\mathscr{D}_2^P$, a ZSH blueprint $\B(\F_2^P)$ {\em is another  subdivision of $\mathscr{D}_2^P$ into quadrilateral ``tiles", together with the choice of a preferred point within each tile, called the ``basepoint" of the tile}, which satisfies several additional properties. 
To obtain the ZSH blueprint shown in Figure~\ref{fig:3a}$(b)$ we have done the following:

\smallskip

\underline{Step 1.} We created the {\em blue box} shown with thick blue dashed lines in Figure~\ref{fig:3a}$(d)$, such that all the {\em bounded\footnote{$\F_2^P$ also has two {\em ``unbounded vertices"}, i.e., vertices that lie at infinity; these will be addressed in Step 4.} vertices}  of $\mathcal{F}_2^P$ (i.e., {\em the nine red dots}) are contained in the interior of this blue box. Note that the choice of blue box depends just on two numbers: a large enough value of $X_3^P$ where the {\em vertical} side of the box lies, and a large enough value of $X_2^P$ where the {\em horizontal} side of the blue box lies. 

\smallskip

\underline{Step 2.} We took advantage of the fact that the bounded vertices (i.e., the red dots) of $\mathcal{F}_2^P$ lie along horizontal lines, and we drew the {\em two blue horizontal line segments} that separate rows of vertices of $\mathcal{F}_2^P$ within the blue box, see Figure~\ref{fig:3a}$(e)$.

\smallskip

\underline{Step 3.} We drew the {\em six blue vertical line segments} such that they completely separate the nine vertices of $\mathcal{F}_2^P$ within nine different tiles inside the blue box, see Figure~\ref{fig:3a}$(f)$. Also, we chose each such vertex to be the {\em basepoint} of its corresponding tile.

\smallskip

\underline{Step 4.} For the region of $\mathscr{D}_2^P$ that lies {\em outside} the blue box, we drew the {\em three thin dashed half-lines} that give rise to two additional quadrilateral tiles ``near infinity". (These three thin dashed half-lines can  be seen most clearly in Figure~\ref{fig:3a}$(c)$.) The slope of the {\em middle dashed half-line} is chosen such that its support line passes through the origin. Note that these two tiles also {\em need to have basepoints}; we chose these two basepoints to lie at infinity, and to correspond to the slopes of the half-lines of $\mathcal{F}_2^P$ within those tiles. (We are showing these two points at infinity symbolically in Figure~\ref{fig:3a}$(c)$.)

\subsubsection{The use of the ZSH blueprint $\B(\F_2^P)$ in the construction of the polyhedral surface $\P(\B(\mathcal{F}_2^P))$}

After we have constructed the ZSH blueprint $\B(\F_2^P)$ shown in Figure~\ref{fig:3a}$(c)$, we will use it for the construction of a polyhedral surface (denoted $\P(\B(\mathcal{F}_2^P))$) within the set where $0 < x_3 \le x_2 \le x_1 \le 1$. 

Later we will argue that, when we also take into account its symmetric images (and also extend it outside the unit cube, in a way that will also be specified later), {\em  the surface $\P(\B(\mathcal{F}_2^P))$ will give rise to a ZSH for $\mathcal{F}_3$.}

%In general, given a subdivision $\mathcal{F}_2^P$ of $\mathscr{D}_2^P$, a ZSH blueprint $\B(\F_2^P)$ is another  subdivision of $\mathscr{D}_2^P$ {\em into quadrilateral ``tiles"} such that:
%
%\smallskip
%
%$(i)$ to each tile of $\B(\F_2^P)$ we have associated a ``basepoint", which is a distinguished point contained inside that quadrilateral tile (the red points in Figure~\ref{fig:3a}$(c)$ are the basepoints of the corresponding tiles).
%
%\smallskip
%
%$(ii)$ if a vertex $V$ of $\mathcal{F}_2^P$ belongs to tile $R$ of $\B(\F_2^P)$, then $V$ is the basepoint of $R$. 
%
%$(iii)$ if an edge $E$ of $\mathcal{F}_2^P$ intersects a tile $R$ of $\B(\F_2^P)$, then the basepoint of $R$ lies on $E$. 

Our plan is that the polyhedral surface $\P(\B(\mathcal{F}_2^P))$ will have {\em 11 faces, which will be in a one-to-one correspondence with the  tiles of $\B(\F_2^P)$}. Moreover, each face will be orthogonal to the basepoint of its corresponding tile. Even more, the image of each face via the map $\Psi_3^P$ will  approximate ``arbitrarily closely"\footnote{We will see that, while for each face of the polyhedral surface $\P(\B(\mathcal{F}_2^P))$ we \underline{\em can} impose  that it is  {\em exactly} orthogonal to the basepoint of its corresponding tile, we \underline{\em cannot} actually impose that the image of each face via  $\Psi_3^P$ is {\em exactly} the corresponding tile. Nevertheless, we \underline{\em can} impose that this image approximates the corresponding tile as accurately as we want, and this will be sufficient for our goals.}  its corresponding tile (see Figure~\ref{fig:3_bcd}$(e)$ for an illustration of this approximation).

Before we start to  describe the details of this construction, let us make the following general comment: {\em while the construction of the  polyhedral surface $\P(\B(\mathcal{F}_2^P))$ takes place in $\mathbb{R}^3_{>0}$, we plan to observe our progress and collect information about construction rules by looking at the image via 
$$
\Psi_3^P : \P(\B(\mathcal{F}_2^P)) \to \mathscr{D}_2^P
$$ of each piece of the surface \underline{while we are constructing it}, and we will compare this image to the ZSH blueprint that we  already have available to us within $\mathscr{D}_2^P$.}

\bigskip

% We will first discuss about how to 

Let us  ``zoom-in'' in the lower-left corner of Figure~\ref{fig:3a}$(c)$, which is shown in Figure~\ref{fig:3_bcd}$(a)$.

We will now focus just on the regions (tiles) $T_1, T_2, T_3, T_4, T_5$, that are shown in Figure~\ref{fig:3_bcd}$(a)$. Recall that each tile $T_i$ contains exactly one red point, which is the {\em basepoint} of the tile.

% recall that we need the ZSH to contain some piece (or face) that is p-orthogonal to $P_i$.] 

We now explain in  detail how we can use the information present in Figure~\ref{fig:3_bcd}$(a)$ in order to construct the corresponding faces $F_1, F_2, F_3, F_4, F_5$ of the polyhedral surface $\P(\B(\mathcal{F}_2^P))$, such that each face $F_i$ corresponds to a tile $T_i$ in Figure~\ref{fig:3_bcd}$(a)$. See Figure~\ref{fig:3_bcd}$(b)$ for an informal representation of the faces $F_1, F_2, F_3, F_4, F_5$ in $\RR^3$. (The adjacency information in Figure~\ref{fig:3_bcd}$(b)$ is correct, but the faces are {\em not shown to scale}).

%(To be clear, the \underline{full} ZSH will have many \underline{other} faces, but now we focus on constructing just these 4 faces.)

\bigskip

Before we begin the construction of the faces $F_i$, let us first explain a very convenient {\em  \underline{abuse of terminology} with respect to the tiles $T_i$}. If a vertex $Q$ of a tile $T_i$ lies on the boundary of another tile  $T_j$, we will say that $Q$ is also a vertex of $T_j$. In other words, even though all tiles are actually quadrilaterals, we convene that some tiles may have more than four vertices (i.e., we regard some of these quadrilaterals as {\em degenerate pentagons}, or {\em degenerate hexagons}, etc., in the sense that {\em some of their vertices are degenerate}, i.e., the interior angle at such a vertex is exactly $180^o$). For example, according to this convention, the tile $T_1$ still has just four vertices, but the tile $T_4$ has \underline{five} vertices (because we convene that the point denoted $Q_{124}$ is a vertex of $T_4$). We also extend this abuse of terminology to edges, and we say that the tile $T_4$ has five edges. {\em (This unusual terminology will become very convenient soon, because it will allow us to say that each face $F_i$ of the polyhedral surface $\P(\B(\mathcal{F}_2^P))$ has the same number of vertices and the same number of edges as its corresponding tile $T_i$.)}

\bigskip

\noindent
{\bf Remark.} Ideally, if it was possible, we would have liked to simply construct polygonal faces $F_i$ of the polyhedral surface $\P(\B(\mathcal{F}_2^P))$ such that 
$$
\Psi_3^P(F_i) = T_i,
$$ 
i.e., such that the ``logarithmic projective image" of each face $F_i$ is the tile $T_i$. {\em Unfortunately, it turns out that this is \underline{impossible} in 3D (although it \underline{was} possible in 2D).} Instead, we will make sure that the logarithmic projective image of each face $F_i$ is a very good approximation of the tile $T_i$, i.e., 
$$
\Psi_3^P(F_i) \approx T_i,
$$ 
in a sense that will be clarified below. In particular, the vertices and edges of $F_i$  will be in a ``canonical" one-to-one correspondence with vertices and edges\footnote{Recall our unusual convention of what we consider to be vertices and edges of each $T_i$.} of $T_i$, and the images via $\Psi_3^P$ of vertices and edges of $F_i$ can be made to lie arbitrarily closely to the corresponding vertices and edges of $T_i$. 
\qed

\bigskip

We now start with the tile $T_1$, and explain how we construct the face $F_1$. 
We see that $T_1$ has just four edges, so $F_1$ will also have four edges. First we build an edge of $F_1$ that corresponds to the left-side edge of $T_1$, denoted $[Q_0 Q_{14}]$ in Figure~\ref{fig:3_bcd}$(a)$. For this, we first choose a point $P_0$ of coordinates $(\varepsilon, \varepsilon, \varepsilon)$, for some positive number $\varepsilon \ll 1$. Note that for any positive $\varepsilon$ we have 
$$
\Psi_3^P(P_{0}) = Q_{0}.
$$
Then we draw a line segment that lies within the plane $\{x_2=x_3\}$, is orthogonal to the point $(1,1,1)$, and advances inside the set $\{x_1 \ge x_2 = x_3\}$ (see Figure~\ref{fig:3_bcd}$(b)$)\footnote{Note that the point $(1,1,1)$ is the basepoint of $T_1$. Actually, throughout our construction, each edge of the face $F_i$ will be orthogonal to the basepoint of the tile $T_i$. This is a necessary condition for our construction: {\em we want the basepoint of the tile $T_i$ to be orthogonal to the face $F_i$}, and therefore we want the basepoint of the tile $T_i$ to be orthogonal to each edge of the face $F_i$.}. We extend this line segment within the set $\{x_1 \ge x_2 = x_3\}$ until we ``reach'' the point $Q_{14}$, i.e., until we reach a point $P_{14}$ such that we have 
$$
\Psi_3^P(P_{14}) = Q_{14}.
$$
{\em (Note that the construction of this line segment within the plane $\{x_2=x_3\}$ should remind us of the construction of line segments for the ZSH in 2D, as we have done in the previous section.)}

%%%%%%%%%%%%%%%%%%%%%%%%%%%%%%%%%%%%%%%%%%%%%%
\begin{figure}[h!]
  \begin{center}
    \begin{tabular}{c}
      \hskip-0.21cm
	$(a)$\includegraphics[angle=0,width=2.1in]{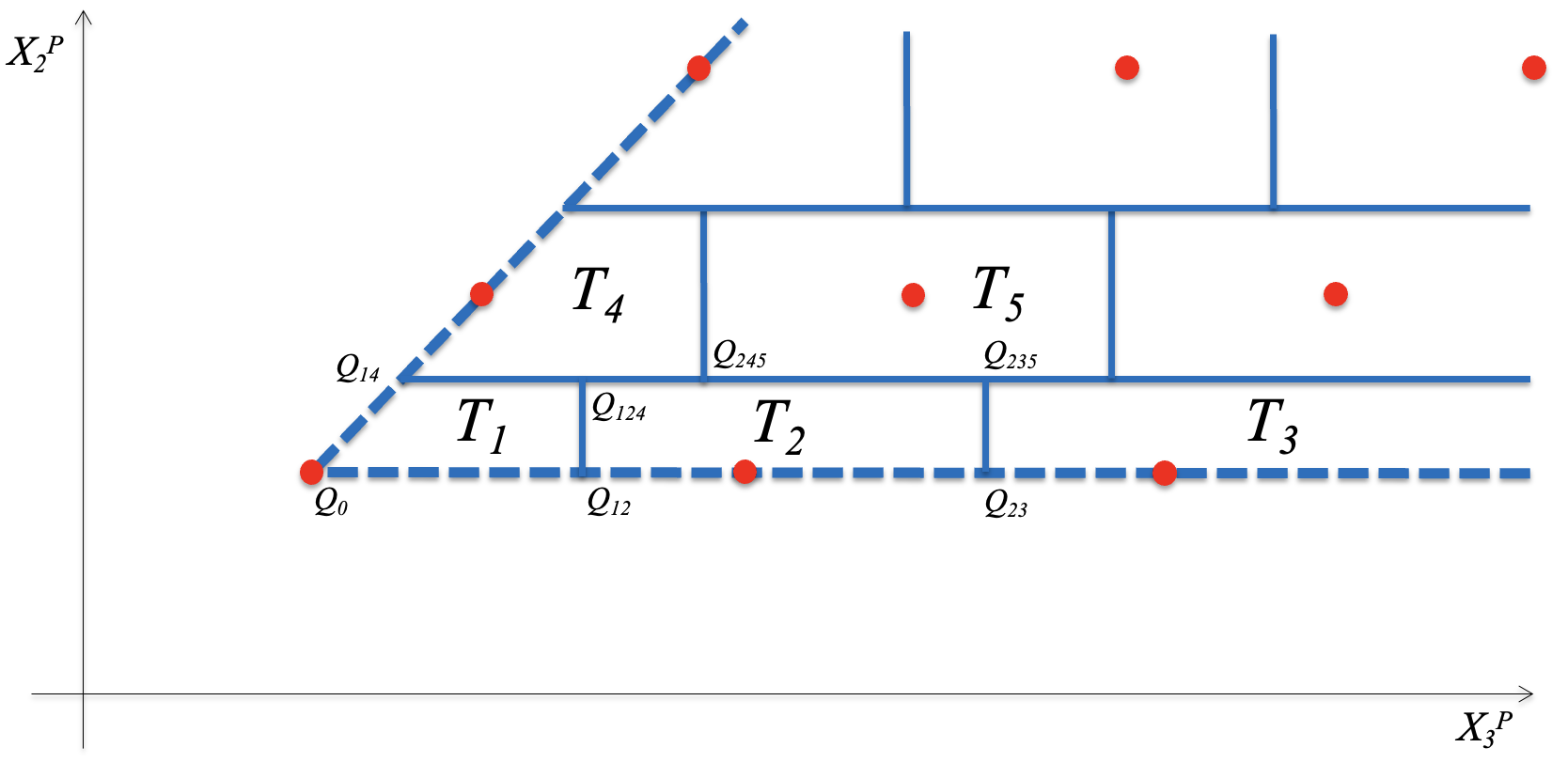} \ \ \ \ \ \ 
	$(b)$\includegraphics[angle=0,width=1.3in]{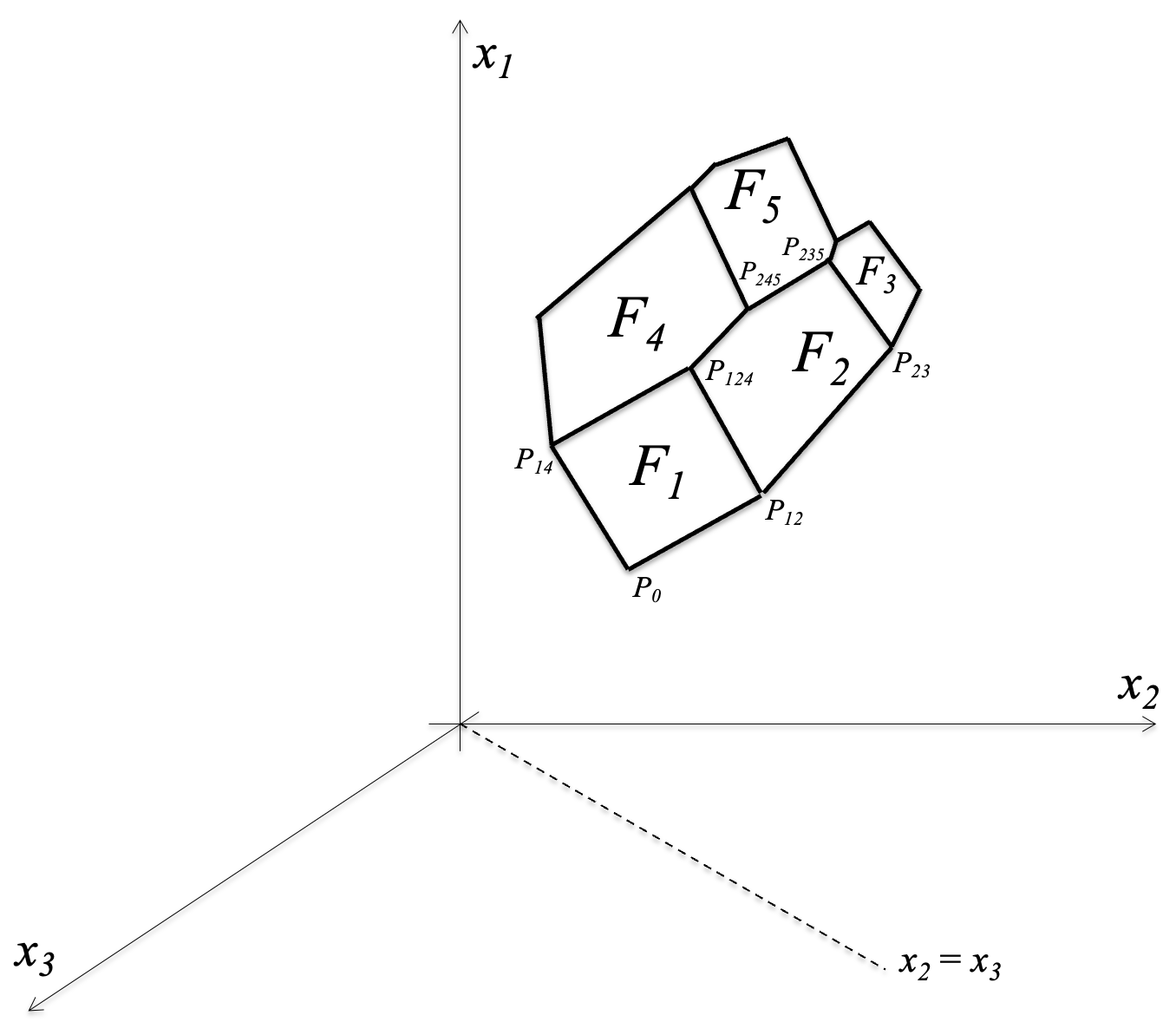}         \\ \\
      \hskip-0.21cm       
	$(c)$\includegraphics[angle=0,width=2.1in]{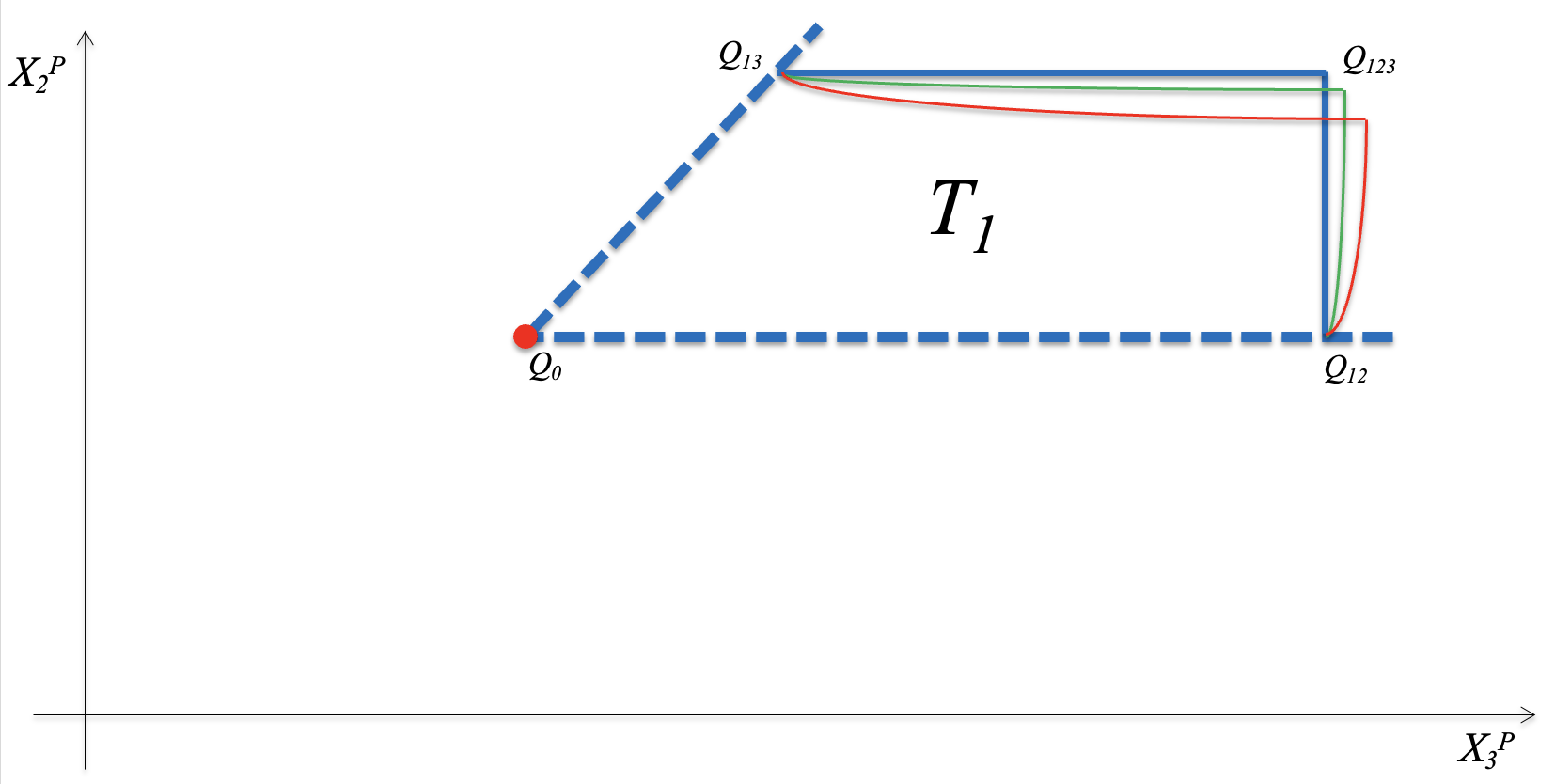} \ \ \ \ \ \       	
	$(d)$\includegraphics[angle=0,width=1.3in]{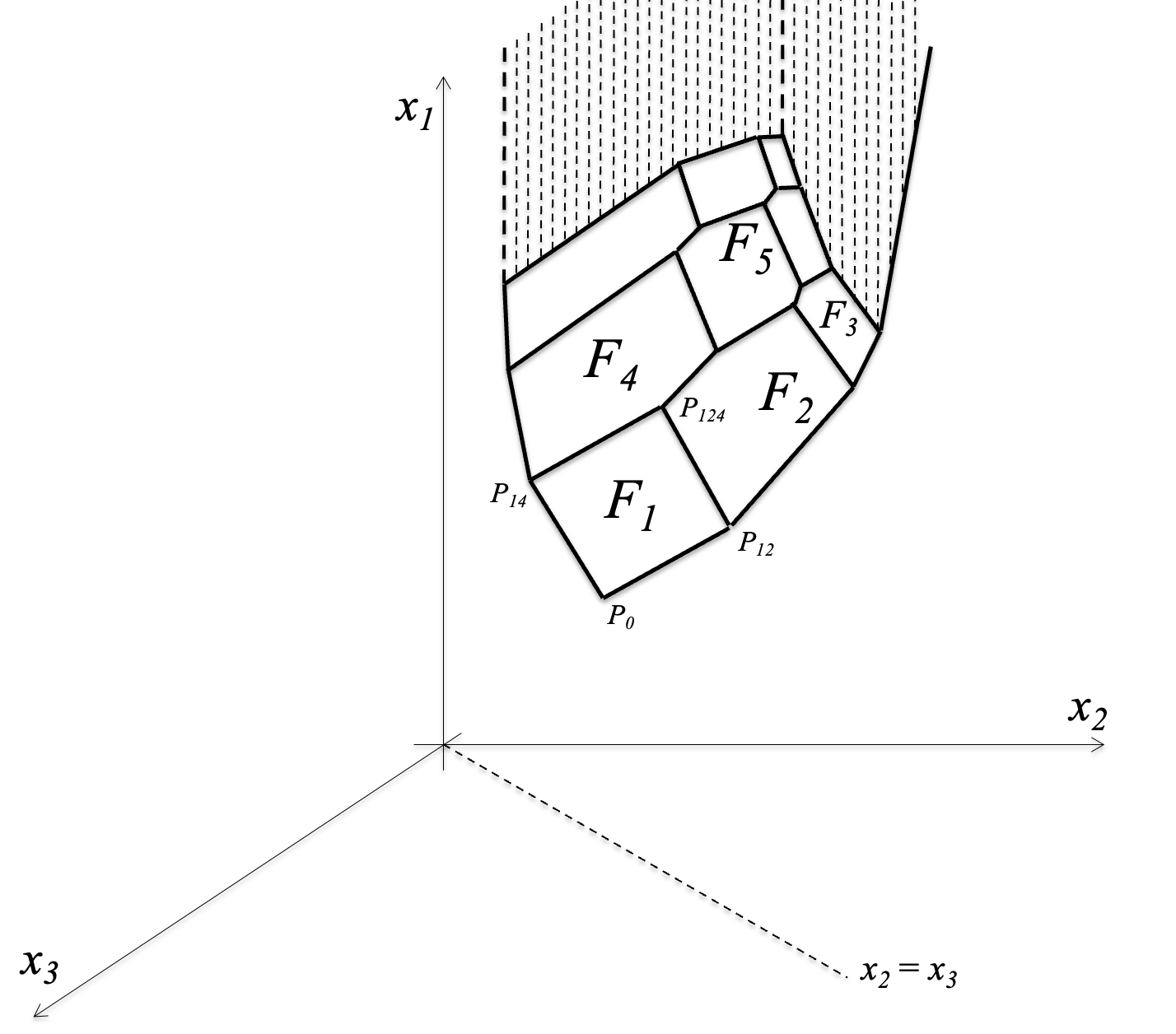}         \\ \\
      \hskip-0.21cm  
	$(e)$\includegraphics[angle=0,width=3.97in]{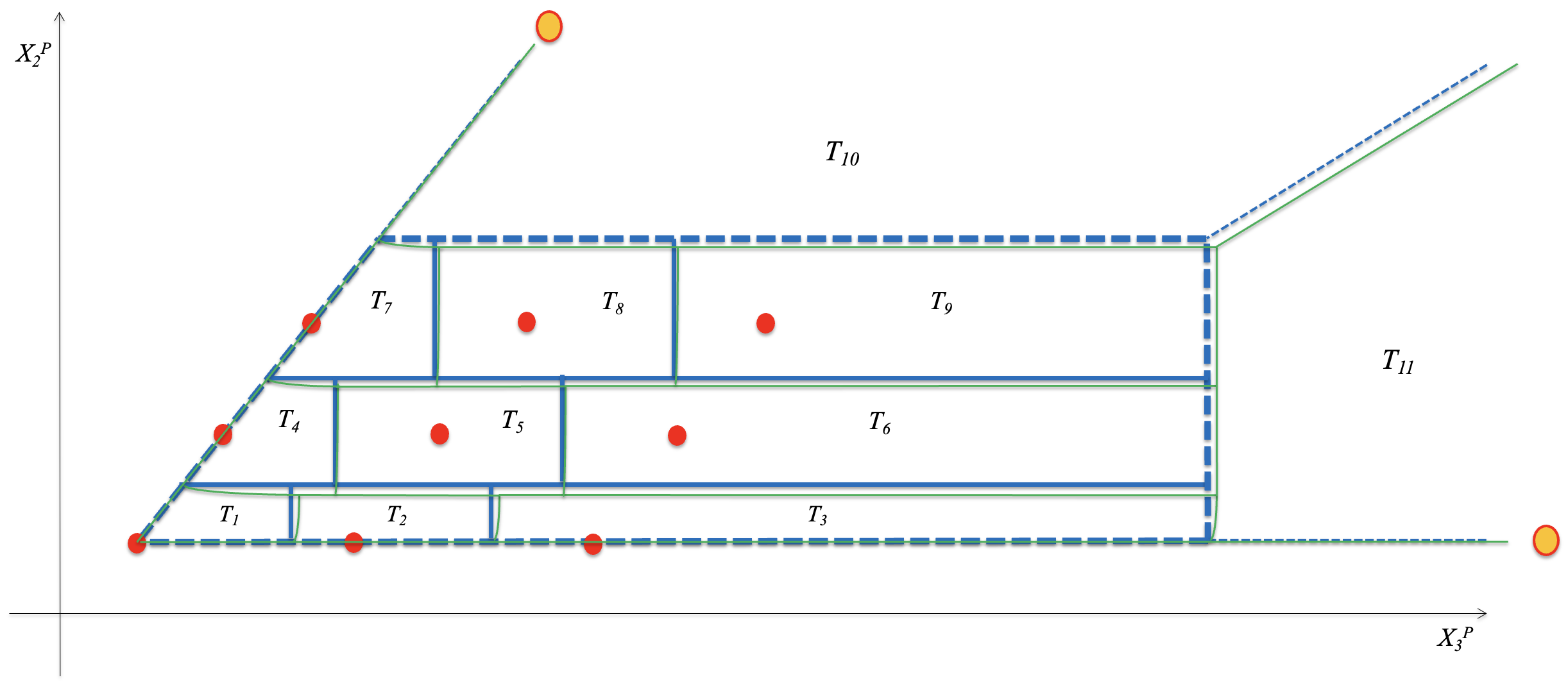}            
    \end{tabular}
  \end{center}
  \caption{\scriptsize \label{fig:3_bcd} Some details of the construction of the polyhedral surface $\P(\B(\mathcal{F}_2^P))$. In $(a)$ we ``zoom in" on some  tiles of the ZSH blueprint $\B(\F_2^P)$. In $(b)$ we show  the corresponding faces of the polyhedral surface $\P(\B(\mathcal{F}_2^P))$. In $(c)$ we zoom in very closely on the tile $T_1$, in order to show where $\Psi_3^P(F_1)$ lies with respect to $T_1$, for two different choices of $\varepsilon \ll 1$. In $(d)$ we show the complete surface $\P(\B(\mathcal{F}_2^P))$, and especially emphasize  the two faces that are parallel to the $x_1$-axis. In $(e)$ we use the green curves to show the images of the faces of the polyhedral surface $\P(\B(\mathcal{F}_2^P))$ via $\Psi_3^P$. In other words, we show  $\Psi_3^P(\P(\B(\mathcal{F}_2^P)))$, together with the ZSH blueprint $\B(\F_2^P)$, so we can see globally how they match. 
}
\end{figure}
%%%%%%%%%%%%%%%%%%%%%%%%%%%%%%%%%%%%%%%%%%%%%%

Similarly, we do the same within the plane $\{x_1=x_2\}$, i.e., we construct a line segment starting at a point $P_0$ and orthogonal to $(1,1, 1)$ (which is the basepoint of $T_1$) inside this plane, and going into the set $\{x_1 = x_2 \ge x_3\}$, until we ``reach" the point $Q_{12}$, i.e., until we reach a point $P_{12}$ such that we have 
$$
\Psi_3^P(P_{12}) = Q_{12}.
$$ 

Next, we construct an edge of the face $F_1$ that corresponds to the blue edge $[Q_{14} Q_{124}]$ shared by $T_1$ and $T_4$. Note that this edge starts at the point $P_{14}$, which we have already constructed, and \underline{\em we have to} impose the condition that this edge is {\em orthogonal to both of the  basepoints of $T_1$ and $T_4$}\footnote{Because this edge belongs to both faces $F_1$ and $F_4$.}, so the only feature of this edge that remains to be specified is its other endpoint, denoted $P_{124}$.

Indeed, using the notation shown in Figure~\ref{fig:3_bcd}$(b)$, we have already constructed the points $P_0, P_{12}, P_{14}$ (and the line segments $[P_{0}P_{12}]$, $[P_{0}P_{14}]$), and we are now constructing the point $P_{124}$ (and the line segments $[P_{12}P_{124}]$, $[P_{14}P_{124}]$). 

Since (obviously) the point $P_{124}$ has to be the intersection point of the edges $[P_{12}P_{124}]$ and $[P_{14}P_{124}]$, we have to proceed as follows:  construct a half-line that starts at $P_{12}$, goes into the set $\{x_1 \ge x_2 \ge x_3\}$, and is orthogonal to the basepoints of $T_1$ and $T_2$; and also construct a half-line that starts at $P_{14}$, goes into the set $\{x_1 \ge x_2 \ge x_3\}$,  and is orthogonal to the basepoints of $T_1$ and $T_4$; and now define $P_{124}$ as the intersection point of these two half-lines. 

Note that, if such an intersection point exists, then we have successfully obtained the complete face $F_1$, and  this face $F_1$ is guaranteed to be orthogonal to the basepoint of $T_1$, because all its edges are orthogonal to the basepoint of $T_1$. 

Let us show now that this intersection point $P_{124}$ does indeed exist. ({\em A priori}, even though we know that these two half-lines lie in the same plane\footnote{When we build a face $F_i$ we will build its edges one by one, and {\em we keep the resulting polygonal line \underline{connected} at every step}. This, together with fact that these edges are all orthogonal to the same basepoint vector, ensure that all the edges of $F_i$ lie in a common (affine) plane.}, this is not sufficient to immediately conclude that they must have a common point.) One way to see this is to use the following two things:

\smallskip

$(i)$ the {\em restriction} of $\Psi_3^P(x_1, x_2, x_3) = \left( \frac{\log x_2}{\log x_1}, \frac{\log x_3}{\log x_1} \right)$ to the affine plane $\pi_1$ that passes through the point $P_0$ and is orthogonal to the basepoint of $T_1$, i.e., 
$$
%\Psi_3^P | 
{\Psi_3^P \Big |}_{\pi_1 \cap \mathbb{R}^3_{>0}} \to \mathbb{R}^2,
$$
is a {\em diffeomorphism} between $\pi_1 \cap \mathbb{R}^3_{>0}$ and its image via $\Psi_3^P$.

\smallskip

$(ii)$ For small enough $\varepsilon$, the images via $\Psi_3^P$ of the two half-lines we discussed above (one starting at $P_{12}$ and the other at $P_{14}$) are arbitrarily close to their corresponding line segments in the ZSH blueprint, as illustrated in Figure~\ref{fig:3_bcd}$(c)$.

\smallskip

The fact that $(i)$ and $(ii)$ are true follows from Lemma~\ref{lemma_restriction_is_diffeo} and Lemma~\ref{lemma_4.5_OLD}, respectively.

\smallskip

Using these two facts we can immediately conclude that the two half-lines must indeed have an intersection point, because their diffeomorphic images via $\Psi_3^P$ have an intersection point. (Their images via $\Psi_3^P$ are shown in  Figure~\ref{fig:3_bcd}$(c)$, with {\em red} curves for some small value of $\varepsilon$, and with {\em green} curves for an even smaller value of $\varepsilon$.)

Let us now discuss the construction of the face $F_2$.\footnote{The face $F_2$ can be constructed analogously with $F_1$, {\em except for one specific step in its construction:} we have to make a specific choice for the point $P_{245}$ on the face $F_2$, which has no exact analogue on the face $F_1$.} 
We first construct the point $P_{23}$  within the set $\{x_1 \ge x_2 = x_3\}$; its construction is similar to the construction of the point $P_{12}$ (see Figure~\ref{fig:3_bcd}$(b)$). Then we construct the point $P_{245}$ as follows: we consider a line segment that starts at $P_{124}$ and goes in a direction that is orthogonal to the basepoints of $T_{2}$ and $T_{3}$. We choose the point $P_{245}$ along this line such that the $X_3^P$-coordinate of $\Psi_3^P(P_{245})$ matches exactly the $X_3^P$-coordinate of the point $Q_{245}$\footnote{As we have remarked before, according to the construction rules used here, it is impossible to require that we have $\Psi_3^P(P_{245}) = Q_{245}$. We will see that it is sufficient for our purposes to impose that their $X_3^P$-coordinates are the same. Note that other similar conventions would also work, for example we could choose to define $P_{245}$ as the point that minimizes the distance between $\Psi_3^P(P_{245})$ and $Q_{245}$, but it is simpler to just match their $X_3^P$-coordinates.}. 

The other faces (corresponding to the other tiles $T_i$ in Figure~(\ref{fig:3_bcd})$(a)$) can be constructed analogously. The construction rules\footnote{Note that once we have decided that each face $F_i$ must be orthogonal to the basepoint of the tile $T_i$, rules $(ii)$ and $(iv)$ \underline{\em must} be enforced. Rule $(i)$ is very natural. Rule $(iii)$ is a somewhat arbitrary way to say that we want $\Psi_3^P(P_{i})$ to be close to $Q_{i}$, but seems the most natural when compared to possible alternatives. In this sense, we can say that this construction is ``canonical", once we have chosen a value of $\varepsilon$. Even if we don't consider this to be ``canonical", the following is true: {\em if we chose a small enough $\varepsilon$ and we impose the rules $(i)$-$(iv)$, then the resulting polyhedral surface $\P(\B(\mathcal{F}_2^P))$  \underline{is uniquely defined}.}} are always the following: 

$(i)$ The construction of points and edges along the boundary (i.e., within the plane $\{x_1=x_2\}$ or $\{x_2=x_3$\}) is done by analogy to the 2D case. 

$(ii)$ If an edge $e_{ij}$ is adjacent to the faces $F_i$ and $F_j$ then $e_{ij}$ must be orthogonal to the basepoints of the tiles $T_i$ and $T_j$.

$(iii)$ If a vertex $P_i$ of a face $F_i$ corresponds to a ``degenerate" vertex $Q_i$ of the tile $T_i$, then the $X_3^P$-coordinate of $\Psi_3^P(P_{i})$ must match exactly the $X_3^P$-coordinate of the point $Q_{i}$. 

$(iv)$ The ``upper-right" vertex $P_i$ of a face $F_i$ is always the last vertex of $F_i$ that we construct, and is obtained by applying the rule $(ii)$ simultaneously to construct the two edges that meet at $P_i$. 

Note that for edges that correspond to the {\em top side} or the {\em left side} of the blue box we have to impose that they are orthogonal to the basepoint of one of the two ``boundary tiles", and this basepoint lies at infinity. Nevertheless, the orthogonality condition males sense in these cases as well. (If a line segment $\ell$ is orthogonal to a ``regular" basepoint $B_i$ and also is orthogonal to the basepoint $B_j$ that lies at infinity, then it means that $\ell$ is orthogonal to $B_i$ and to the line in the $X_2$-$X_3$-plane whose slope is given by the point $B_j$). 

\medskip

Finally, we now  discuss the construction of the two faces that correspond to the two  ``boundary tiles", denoted $T_{10}$ and $T_{11}$.

%%%%%%%%%%%%%%%%%%%%%%%%%%%%%%%%%%%%%%%%%%%%%%
\begin{figure}[h!]
  \begin{center}
    \begin{tabular}{c}
      %\hskip-0.21cm
	\includegraphics[angle=0,width=3.5in]{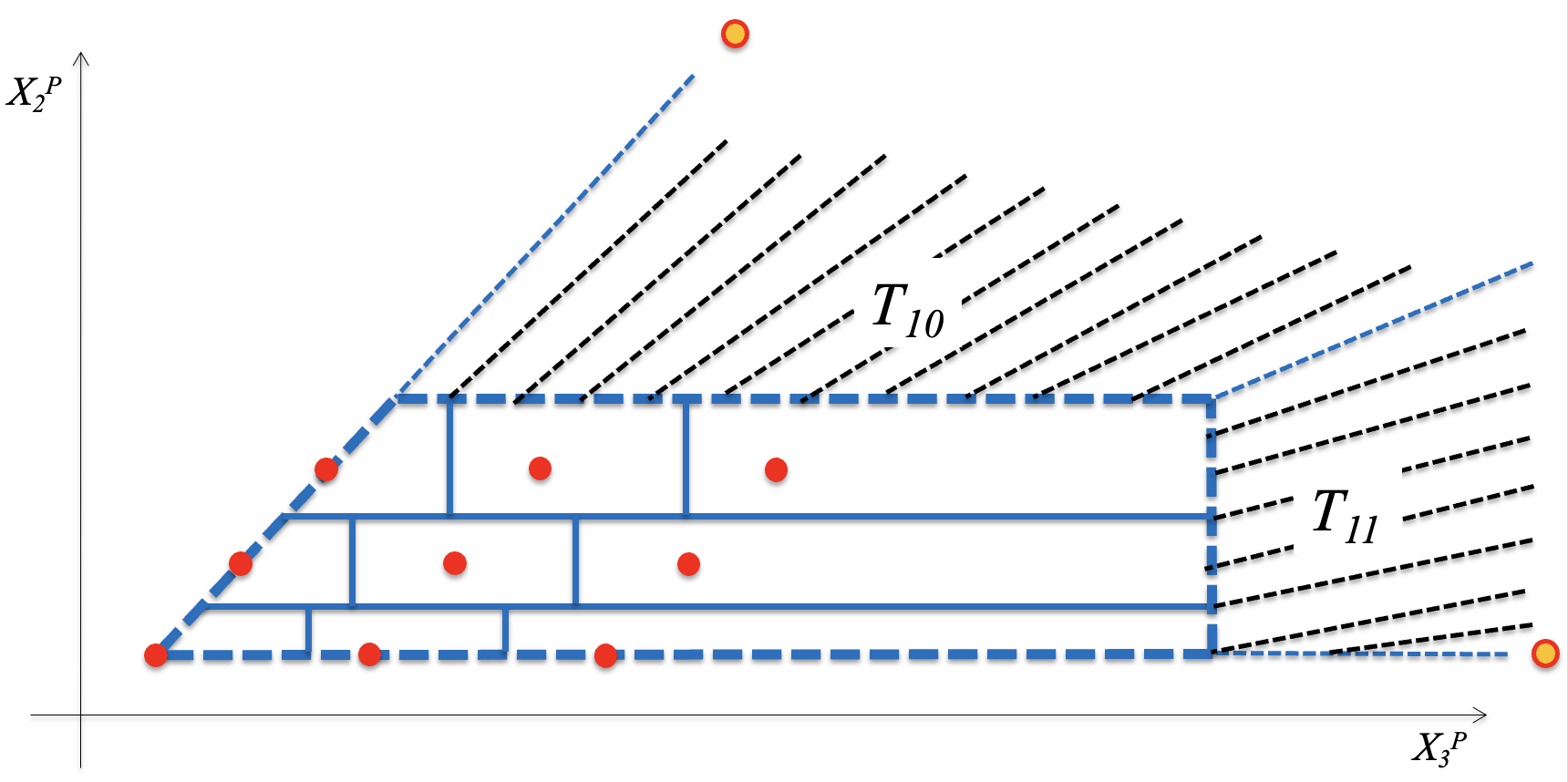}   
    \end{tabular}
  \end{center}
  \caption{\scriptsize \label{fig:5} The dashed black half-lines in this figure show images via $\Psi_3^P$ of  half-lines in $\RR^3_{>0}$ that are parallel to the $x_1$-axis. Note that these dashed half-lines completely fill in the tiles $T_{10}$ and $T_{11}$, and therefore the corresponding  half-lines in $\RR^3_{>0}$ (which are all parallel to the $x_1$-axis) provide a way to construct the faces $F_{10}$ and $F_{11}$.
}
\end{figure}
%%%%%%%%%%%%%%%%%%%%%%%%%%%%%%%%%%%%%%%%%%%%%%

One way to construct the  faces $F_{10}$ and $F_{11}$ is to note that both of them must be parallel to the $x_1$-axis (because each one of them is orthogonal to a basepoint that lies in the $X_2$-$X_3$-plane); therefore, we can construct them as  a ``cylindrical surface patch": we simply draw vertical half-lines (i.e., half-lines that are parallel to the $x_1$-axis), starting along the already constructed\footnote{We assume that we have already constructed the edge of $F_{10}$ that lies within the plane $x_2 = x_3$ (i.e., whose image via $\Psi_3$ goes along the line $X_2^P = X_3^P$), and the edge of $F_{11}$  that lies within the plane $x_2 = x_1$ (i.e., whose image via $\Psi_3$ goes along the line $X_2^P = 1$). Note that this specific edge of $F_{11}$ plays a very important role to help us ensure that our ``cylindrical surface patch" {\em fills up the domain $\D_2^P$ completely}. (If we did not take advantage of  this edge of $F_{11}$ as we described above, then there will be a small  region on the lower right side of $\D_2^P$  ``left uncovered" by our construction, i.e., such that no points of the face $F_{11}$ would map to this small corner via $\Psi_3^P$, see Figure~\ref{fig:5}.)} boundary edges of $F_{10}$ and $F_{11}$. Recall that the images of such half-lines via $\Psi_3^P$ are half-lines in $\mathscr{D}_2^P$ whose support lines pass through the origin, see Figure~\ref{fig:5} for an illustration (the dotted half-lines are these images via $\Psi_3^P$). All we have to do now is to notice in Figure~\ref{fig:5}  that these dotted half-lines {\em completely fill in the tiles $T_{10}$ and $T_{11}$}, and then conclude that the original half-lines (before applying $\Psi_3^P$) completely fill in the faces $F_{10}$ and $F_{11}$.

Note  that the progress of our construction could also be followed within $\mathscr{D}_2^H$ (instead of $\mathscr{D}_2^P$), see Figure~\ref{fig:3_e}. 
In general, {\em it is more convenient to use  $\mathscr{D}_2^P$ for tiles that lie \underline{inside} the blue box} (because we can take advantage of the simplicity of horizontal and vertical lines within $\mathscr{D}_2^P$, while within $\mathscr{D}_2^H$ the same lines are not actually parallel; instead, they meet at some boundary point\footnote{Horizontal lines in $\mathscr{D}_2^P$ are mapped to lines that pass through the orange vertex labeled ``$X_3^H$" in $\mathscr{D}_2^H$, while vertical lines are mapped to lines that pass through the other orange vertex (see Figure~\ref{fig:3_e}$(c)$). Note that these two orange vertices in Figure~\ref{fig:3_e}$(c)$ are the basepoints of their corresponding tiles in $\mathscr{D}_2^H$.}). On the other hand, {\em it can be better to use $\mathscr{D}_2^H$ for tiles that lie \underline{outside} the blue box}; for example, the use of $\mathscr{D}_2^H$ for these ``tiles near infinity" comes with the advantage that their basepoints have a simple ``finite" location in $\mathscr{D}_2^H$ (while their basepoints lie at infinity in $\mathscr{D}_2^P$).

%%%%%%%%%%%%%%%%%%%%%%%%%%%%%%%%%%%%%%%%%%%%%%
\begin{figure}[h!]
  \begin{center}
    \begin{tabular}{c}
      \hskip-2.5cm
$(a)$      \includegraphics[angle=0,width=1.8in]{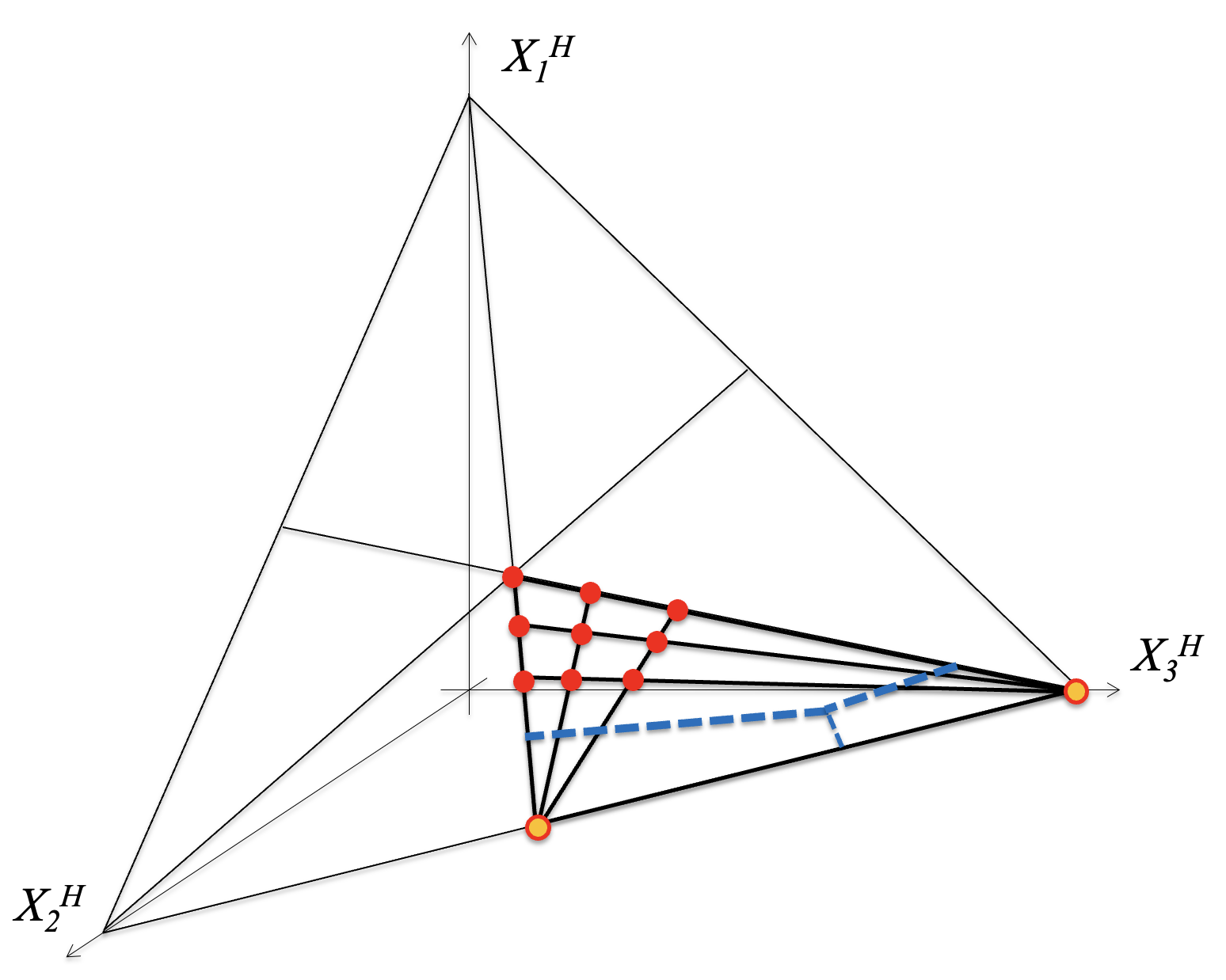} 
$(b)$      \includegraphics[angle=0,width=1.8in]{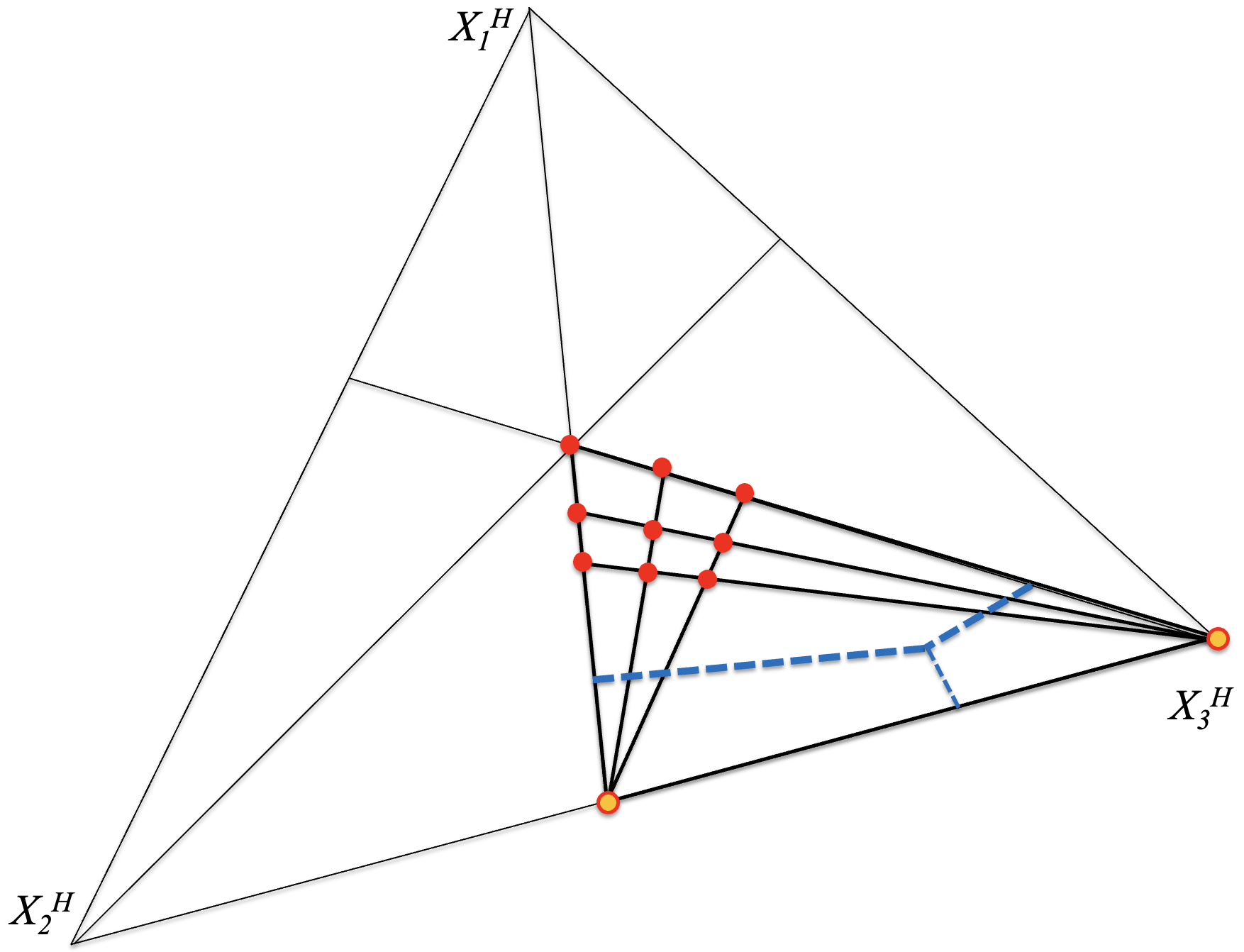}
$(c)$      \includegraphics[angle=0,width=2.4in]{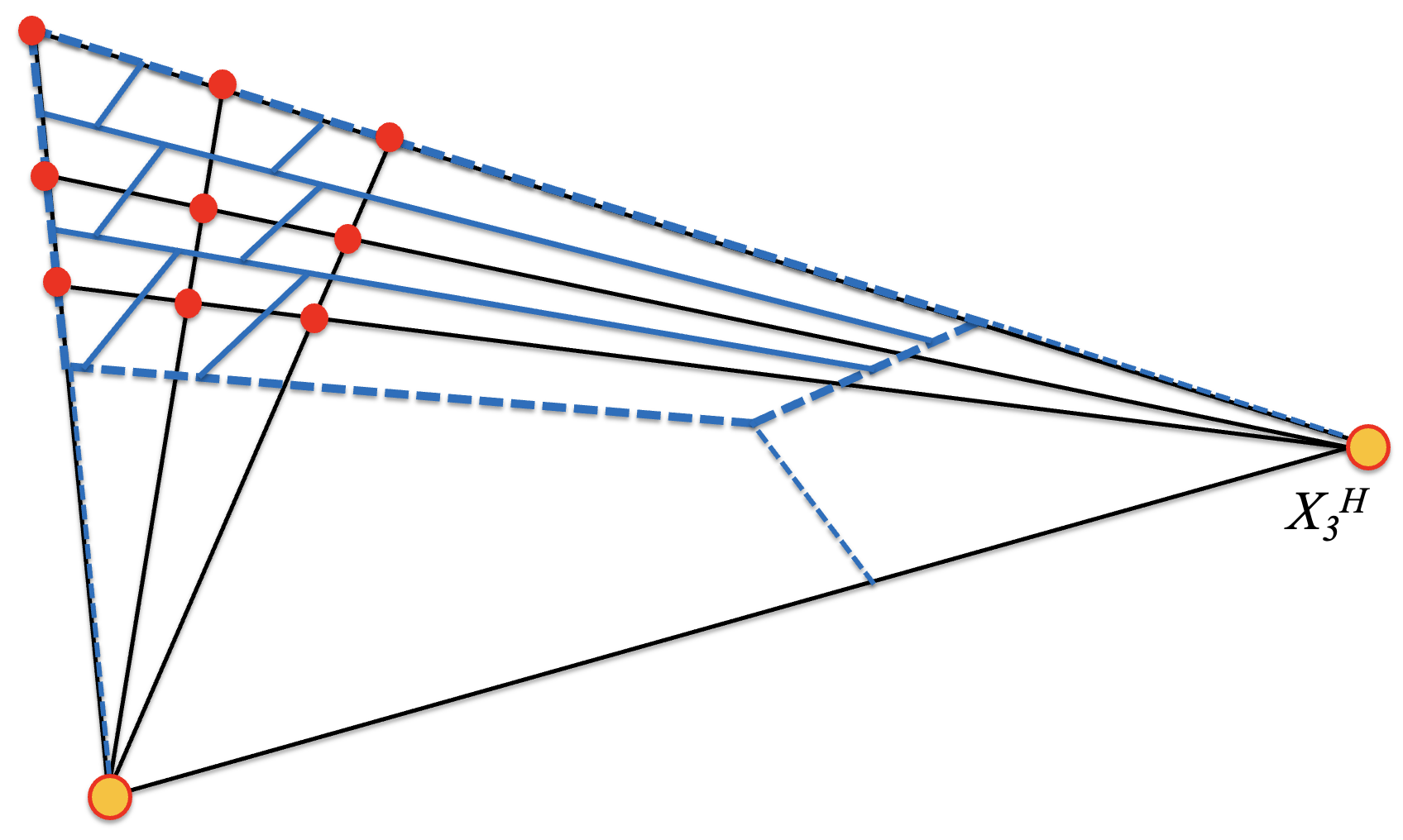}      
    \end{tabular}
  \end{center}
  \caption{\scriptsize \label{fig:3_e} For the same convenient 3D fan $\mathcal{F}_{3}$ that was  shown in Figure~\ref{fig:3a}, we show in $(a)$ its representation as a subdivision $\mathcal{F}_{2}^H$ of $\mathscr{D}_2^H$.  
In $(b)$ we show again the subdivision $\mathcal{F}_{2}^H$, without the coordinate axes. In $(c)$ we zoom in on the set $\mathscr{D}_2^H$, and we show the subdivision $\mathcal{F}_{2}^H$ together with the ZSH blueprint $\B(\mathcal{F}_{2}^H)$. (The blueprints $\B(\mathcal{F}_{2}^P)$ and $\B(\mathcal{F}_{2}^H)$ map to each other via the transformations $\psi_3^P$ and $\psi_3^H$.)
Note also that {\em the two basepoints that lie at infinity} within the ``projective blueprint" $\B(\mathcal{F}_{2}^P)$ in Figure~\ref{fig:3a} (i.e., the points with orange center) look quite similar to other basepoints when regarded within the ``homogenous blueprint" $\B(\mathcal{F}_{2}^H)$. 
}
\end{figure}
%%%%%%%%%%%%%%%%%%%%%%%%%%%%%%%%%%%%%%%%%%%%%%

\subsubsection{The polyhedral surface $\P(\B(\mathcal{F}_2^P))$ gives rise to a ZSH}
\label{polyhedral_surface_is_ ZSH_in_3D}

We now explain why the polyhedral surface $\P(\B(\mathcal{F}_2^P))$ can be used for the construction of a ZSH for any initial condition $x_0$ (for well-chosen $\eps$).

\bigskip

% \bigskip

Let us first explain why, if $\eps_0$ is small enough, then the polyhedral surface $\P(\B(\mathcal{F}_2^P))$ satisfies the ZSH property within the unit cube $(0,1)^3$.

\smallskip

We start by noticing some simple geometric correspondences between the relative locations of {\em vertices, edges,  faces} of the subdivision $\F_2^P$, and the {\em tiles} of $\B(\F_2^P)$, and also the {\em images of faces}\footnote{The observations {\bf 1-3} about {\em images of faces of $\P(\B(\mathcal{F}_2^P))$} take advantage of the assumption that $\eps_0$ is small enough. Indeed, for small enough $\eps_0$ the set $\Psi_3^P(F_i)$ is arbitrarily close to the tile $T_i$, so it is enough to check that ``$T_i$ is a neighborhood of $v$", or ``$T_i \cup T_j$ is a neighborhood of $e$", and so on.} of the polyhedral surface $\P(\B(\mathcal{F}_2^P))$ via the map $\Psi_3^P$, see Figure~\ref{fig:3a}$(b)$ and Figure~\ref{fig:3_bcd}$(e)$:

\smallskip
  
\noindent
{\bf 1.} For any {\em vertex} $v$ of the subdivision $\F_2^P$ there exists a {\em tile} $T_i$ of $\B(\F_2^P)$ which has $v$ as its basepoint. Moreover, the face $F_i$ of $\P(\B(\mathcal{F}_2^P))$ has the property that {\em $\Psi_3^P(F_i)$ \underline{is a neighborhood} of the vertex $v$ within the set $\D_2^P$}. 

\smallskip

\noindent
{\bf 2.}  For any {\em edge} $e$ of the subdivision $\F_2^P$ there exists {\em two tiles} $T_i, T_j$ of $\B(\F_2^P)$ which have the endpoints of $e$ as their basepoints. Moreover, the faces $F_i, F_j$ of $\P(\B(\mathcal{F}_2^P))$ have the property that {\em $\Psi_3^P(F_i \cup F_j)$ is a \underline{neighborhood} of the edge $e$ within the set $\D_2^P$}. 

\smallskip

\noindent
{\bf 3.}  For any {\em face} $f$ of the subdivision $\F_2^P$ consider the {\em tiles} $T_{i_1}, T_{i_2},  ..., T_{i_k}$ of $\B(\F_2^P)$ which have the property\footnote{It is easy to see in Figure~\ref{fig:3a}$(b)$ that we have either $k=3$ or $k=4$. The case $k=3$ corresponds to a face $f$ that has a vertex or an edge at infinity, and the case $k=4$ corresponds to a face $f$ that is contained in the blue box.} that their basepoints belong to $f$. Moreover, the faces $F_{i_1}, F_{i_2},  ..., F_{i_k}$ of $\P(\B(\mathcal{F}_2^P))$ have the property that {\em $\Psi_3^P(F_{i_1} \cup F_{i_2} \cup ... \cup F_{i_k})$ is a \underline{neighborhood} of the face $f$ within the set $\D_2^P$}.

\smallskip

This means that we can now apply  Lemma~\ref{lemma_outer_normal_sufficient} (for cones $K_\C$ that are defined using the neighborhoods we pointed out above\footnote{The fact that $\Psi_3(F_i)$ approximates $T_i$ very closely (see Figure~\ref{fig:3_bcd}$(e)$) implies the following: for any cone $\C \in \F_3^\le$ there exists a  cone $K_\C \subset \{ X_3 \le X_2 \le X1 \le 0 \}$ that contains  $\C$ in its relative interior with respect to the set $\{ X_3 \le X_2 \le X1 \le 0 \}$, and such that $\psi_3^P(K_\C)$ is contained in a union of sets of the form $\Psi_3^P(F_i)$  where the faces $F_i$ have the property that the basepoint of the tile $P_i$ belongs to $\C$ (and therefore the outer normal vector of the face $F_i$ belongs to $\C$). Indeed, the cone $\C$ corresponds to a vertex, an edge, or a face in the subdivision $\F_2^P$, and our claim is easy to check for each one of these cases, by taking advantage of the observations {\bf 1-3} above. }) and conclude that, for $\eps_0$ small enough, the toric differential inclusion $\T_3$ cannot cross the polyhedral surface $\P(\B(\mathcal{F}_2^P))$ within the set $(0,1)^3$.

\bigskip

For a fixed $x_0 \in \RRp^3$ we first fix a specific level surface $\L$ of the Horn-Jackson Lyapunov function such that $x_0$ is below that level surface (see Lemma~\ref{lemma_asymptotics_of_Horn-Jackson_Lyapunov_function} for details on the shape of these  level surfaces). Let us also fix some $M>1$ such that the whole surface $\L$ lies inside the cube $[0,M]^3$.

We claim that, if we choose $\eps_0$ small enough\footnote{Recall that $\eps_0$ was a very small positive number which we chose when we positioned the point $P_0$ at  location $(\eps_0,\eps_0,\eps_0)$.} then we can ensure that the toric differential inclusion $\T_3$ cannot cross the polyhedral surface $\P(\B(\mathcal{F}_2^P))$  anywhere within the cube\footnote{Initially we have focused on the polyhedral surface $\P(\B(\mathcal{F}_2^P))$ within the unit cube $[0,1]^3$, but now we extend our attention to the much larger cube $[0,M]^3$. Note that the polyhedral surface $\P(\B(\mathcal{F}_2^P))$ is easy to ``extrapolate" to any larger cube (and actually all the way to infinity), because there is no problem to extend the cylindrical faces of $\P(\B(\mathcal{F}_2^P))$ outside $[0,1]^3$. We need this larger version of $\P(\B(\mathcal{F}_2^P))$ because later we will intersect it with the level surface $\L$ of the Horn-Jackson Lyapunov function.} $[0,M]^3$.

We already know that   the toric differential inclusion $\T_3$ cannot cross the polyhedral surface $\P(\B(\mathcal{F}_2^P))$  within the cube $[0,1]^3$. Let us also recall that we have constructed $\P(\B(\mathcal{F}_2^P))$ in the set $\{ 0< x_3 \le x_2 \le x_1 \}$ by starting with a polygonal surface $\S_0$ that was contained within a neighborhood of the origin of size $O(\eps_0)$ and then by building cylindrical surfaces along the boundary of $\S_0$. 

Consider now a point $P \in \P(\B(\mathcal{F}_2^P)) \cap [0,M]^3$ within the set $\{ 0<  x_3 \le x_2 \le x_1 \}$, and outside the unit cube $[0,1]^3$. We  will argue that  the toric differential inclusion $\T_3$ cannot cross the polyhedral surface $\P(\B(\mathcal{F}_2^P))$   at $P$. 

The point $P = (\tilde x_1, \tilde x_2, \tilde x_3)$ was obtained from $\S_0$ by the cylindrical construction mentioned above. Therefore, at least one coordinate of $P$ must be very small; more precisely, we  must have $\tilde x_3 = O(\eps_0)$. 
On the other hand, we also must have 
$$
0 < \log(\tilde x_1) < \log(M),
$$ 
which implies 
$$
\frac{\log(\tilde x_1) }{\log(\tilde x_1)  + \log(\tilde x_2)  + \log(\tilde x_3)}  \in (-\eta, 0 ),
$$
for some $\eta = O(1/\log(\eps_0))$. 

Therefore, if $\eps_0$ is small enough, the ray along the vector $\log(P)$ is very close to the cones of $\F_3$ that lie along the boundary of $\D_2^H$ where $X_1^H = 0$, denoted $\partial_1\D_2^H$. 

Note now that {\em every cone $\C$ of $\F_3$ that lies close enough to $\partial_1\D_2^H$ \underline{must intersect}\footnote{Because  any cone that does {\em not} intersect $\partial_1\D_2^H$ must lie at some positive distance from $\partial_1\D_2^H$, and there are finitely many cones in $\F_3$.} $\partial_1\D_2^H$}, and therefore, if $\eps_0$ is small enough, the right-hand side of the toric differential inclusion  $\T_3$ at $P$ will not depend on $\C$, but on the intersection\footnote{Because the right-hand side of a toric differential inclusion $\T$ at a point $P$ depends only on the {\em smallest} cone of the fan of $\T$ that is close enough to the point $\log P$.} between $\C$ and $\partial_1\D_2^H$. But $\C \cap \partial_1\D_2^H$ is one of the points or line segments lying along $\partial_1\D_2^H$, which we have already considered in our construction of the surface $\P(\B(\mathcal{F}_2^P))$. In conclusion,   the toric differential inclusion $\T_3$ cannot cross the polyhedral surface $\P(\B(\mathcal{F}_2^P))$  at $P$.

\bigskip

In the next section 
%we will take full advantage of the representation of the fan in $\mathscr{D}_2^H$. There 
we will consider the case of a {\em fully general fan $\F_3$}, and we will see that the construction of the ZSH blueprint within the blue box can become more complicated, and also that we will need to allow for any finite number of ``tiles near infinity" (as opposed to just two such tiles in the special case of the fan $\F_3$ we used in this section). Nevertheless, a very similar approach will allow us to construct a ZSH in the general case.

%The last step in this (sketch of the) construction of the ZSH (for this specific polyhedral fan) is to expand it via a \underline{cylindrical surface} which starts along the outer boundary of the already constructed surface, and is generated by affine half-lines that are parallel to the $x_1$-axis. 
%
%The image of any such half-line in $\mathscr{D}_2^H$ is a line segment \underline{whose support line} passes through $(1,0,0)$. Note that the union of these line segments fills in $\mathscr{D}_2^H$, with the exception of a small corner near $(0,0,1)$, which corresponds to a face of the ZSH that is parallel to the $x_1x_2$-plane \textit{[diagram arrow points to the missing corner region in Figure 3(e)]}.

%\newpage
%
%%%%%%%%%%%%%%%%%%%%%%%%%%%%%%%%%%%%%%%%%%%%%%%
%\begin{figure}[h!]
%  \begin{center}
%    \begin{tabular}{c}
%%      \hskip-2cm
%      \includegraphics[angle=180,width=4.2in]{IMG_0181.jpeg} \\
%      \includegraphics[angle=180,width=4.2in]{IMG_0185.jpeg}       
%    \end{tabular}
%  \end{center}
%  \caption{\label{fig:0127_0131}  Blackboards 0181 and 0185.}
%\end{figure}
%%%%%%%%%%%%%%%%%%%%%%%%%%%%%%%%%%%%%%%%%%%%%%%

\bigskip

\subsection{The general construction of ZSH in 3D}
\label{section:3d_ZSH_in_general}

In the previous section we have discussed the construction of ZSH in 3D for a convenient choice of polyhedral fan $\mathcal{F}_3$ (and corresponding subdivisions $\mathcal{F}_2^P$ and $\mathcal{F}_2^H$). We now consider the  general case for $\mathcal{F}_3$. 

Recall that $\mathcal{F}_3$ gives rise to two versions of $\mathcal{F}_2$, one (denoted $\mathcal{F}_2^P$) within $\mathscr{D}_2^P$ and another (denoted $\mathcal{F}_2^H$) within $\mathscr{D}_2^H$. (Like before, without loss of generality, we assume that $\mathcal{F}_3$ is $X_1$-$X_2$-$X_3$-symmetric, and this ensures that all relevant information about $\mathcal{F}_3$ is contained in $\mathcal{F}_2^P$, and also in $\mathcal{F}_2^H$.) Note also that the information contained in $\mathcal{F}_2^P$ is equivalent to the information contained in $\mathcal{F}_2^H$. 

Our strategy for the construction of ZSH will rely again on a finite partition of $\mathscr{D}_2^P$ (and, equivalently $\mathscr{D}_2^H$) using well-chosen quadrilateral ``tiles'' (most of these tiles will be rectangular, with the exception of tiles lying along the boundary of $\mathscr{D}_2^P$). Like in the previous section, our plan is to construct a ZSH that has exactly one face for each ``tile''; we will refer to this partition into tiles as a ``ZSH blueprint". Once we have built a ZSH blueprint we will see that the construction of the ZSH can proceed in the same way as in the previous section.

%%%%%%%%%%%%%%%%%%%%%%%%%%%%%%%%%%%%%%%%%%%%%%
\begin{figure}[h!]
  \begin{center}
    \begin{tabular}{c}
      \hskip-3cm
$(a)$      \includegraphics[angle=0,width=2.1in]{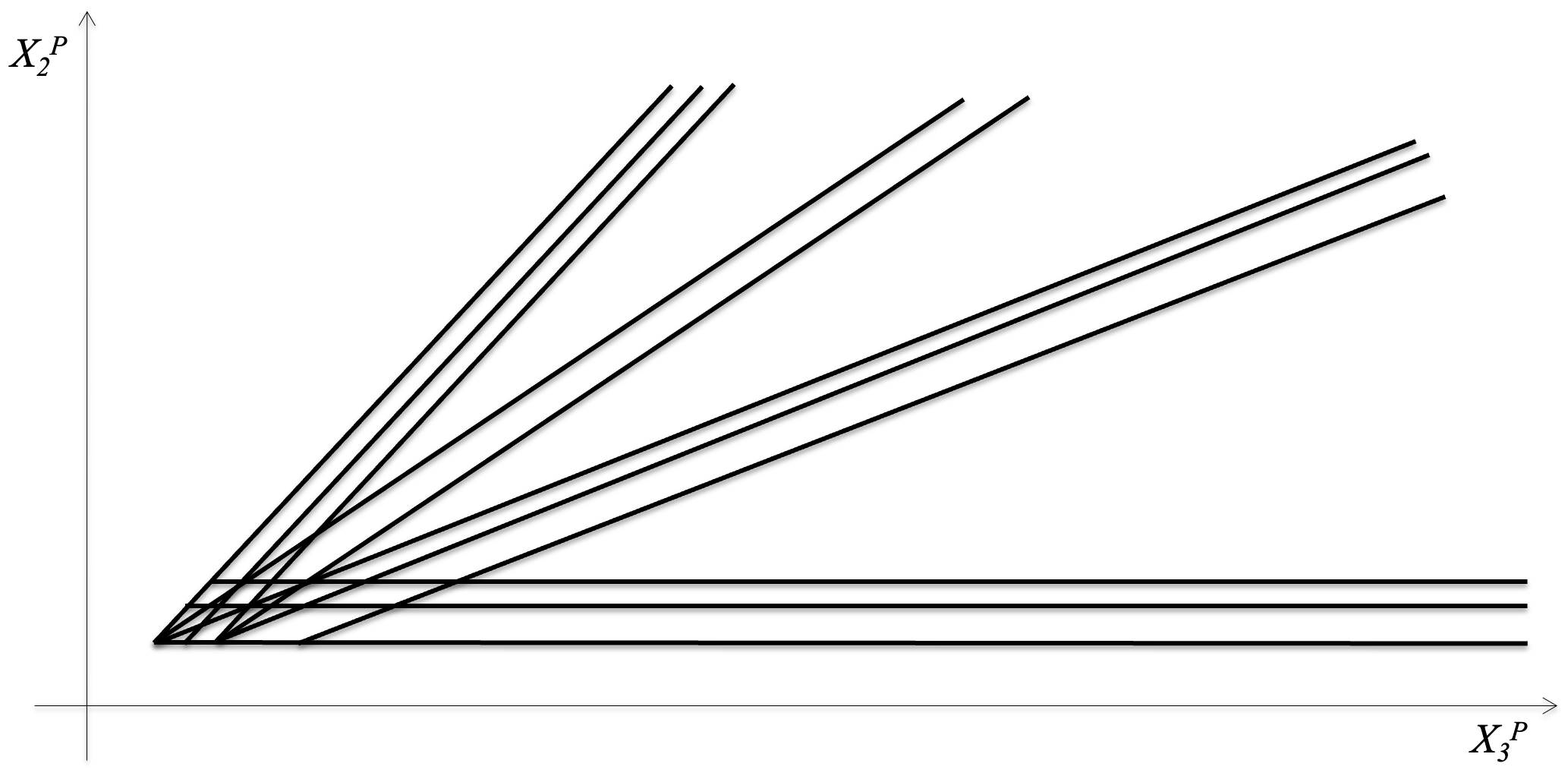} 
$(b)$      \includegraphics[angle=0,width=2.1in]{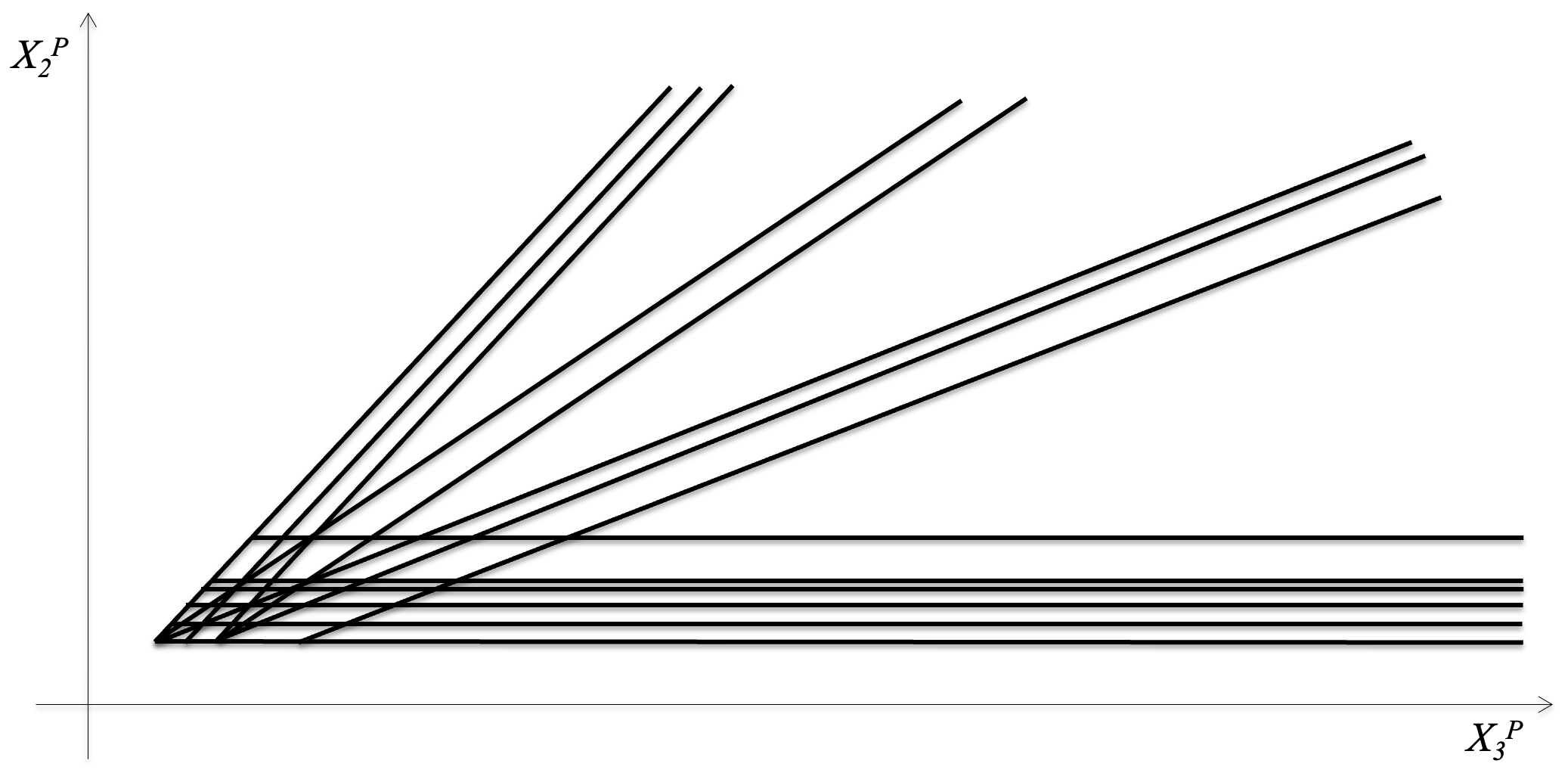}      
$(c)$      \includegraphics[angle=0,width=1.8in]{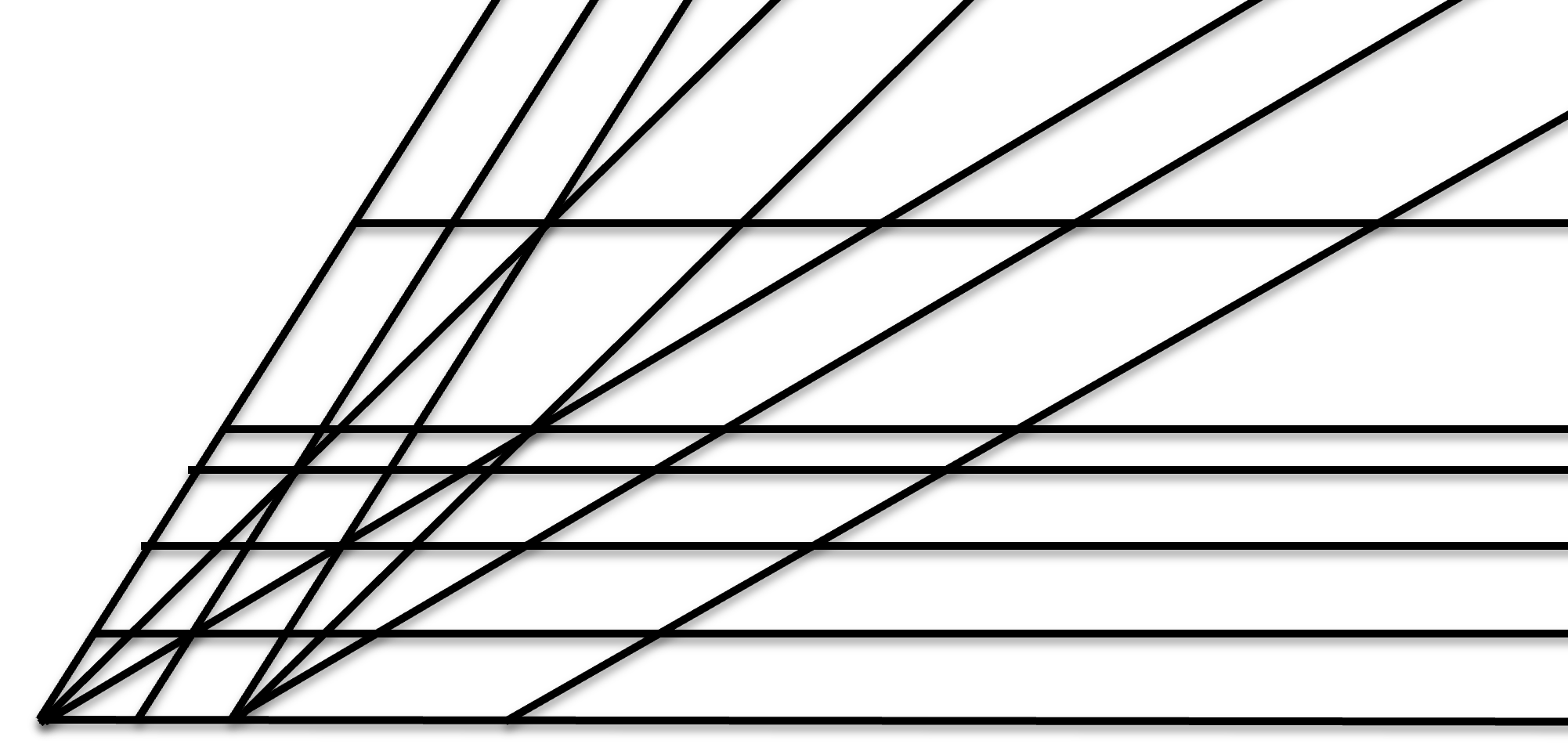}    
    \end{tabular}
  \end{center}
  \caption{\scriptsize \label{fig:6}   A general (symmetric) polyhedral fan $\F_3$ gives rise to a subdivision $\F_2^P$ like the one shown in $(a)$. In $(b)$ we subdivide $\F_2^P$ further, to ensure that each intersection points lies on a horizontal line that is represented in $\F_2^P$. In $(c)$ we zoom in to see the details of $\F_2^P$ near the vertex $(1,1)$.}
\end{figure}
%%%%%%%%%%%%%%%%%%%%%%%%%%%%%%%%%%%%%%%%%%%%%%

%%%%%%%%%%%%%%%%%%%%%%%%%%%%%%%%%%%%%%%%%%%%%%

\begin{figure}[htbp!]
  \begin{center}
  \vspace*{-1.5cm} % THIS "LIFTS UP" THE WHOLE FIGURE !!!!!!
  \hspace*{-2cm} % THIS "MOVES LEFT" THE WHOLE FIGURE !!!!!!
    \begin{tabular}{c}
$(a)$      \includegraphics[angle=0,width=2.8in]{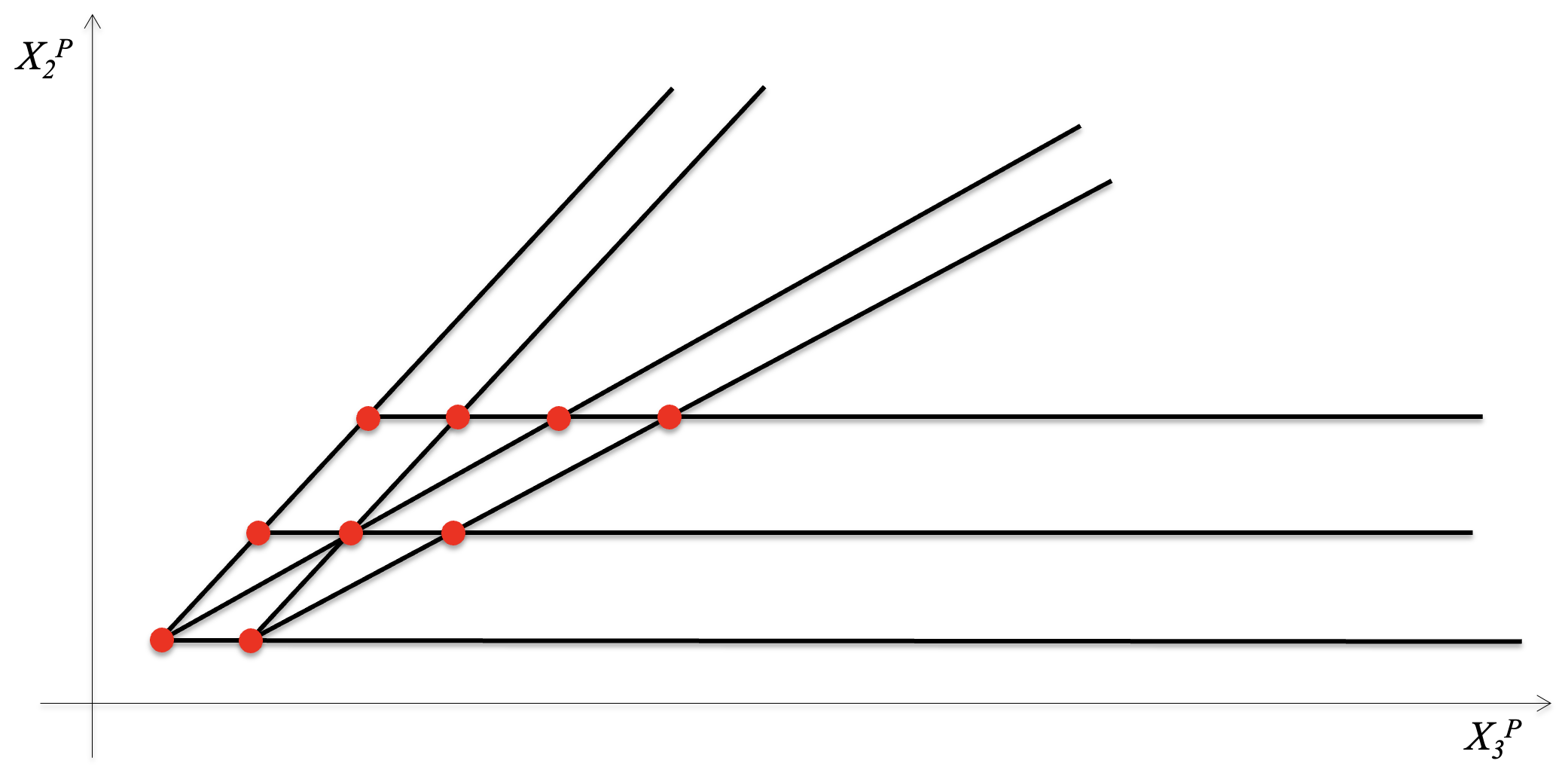} 
$(b)$      \includegraphics[angle=0,width=2.8in]{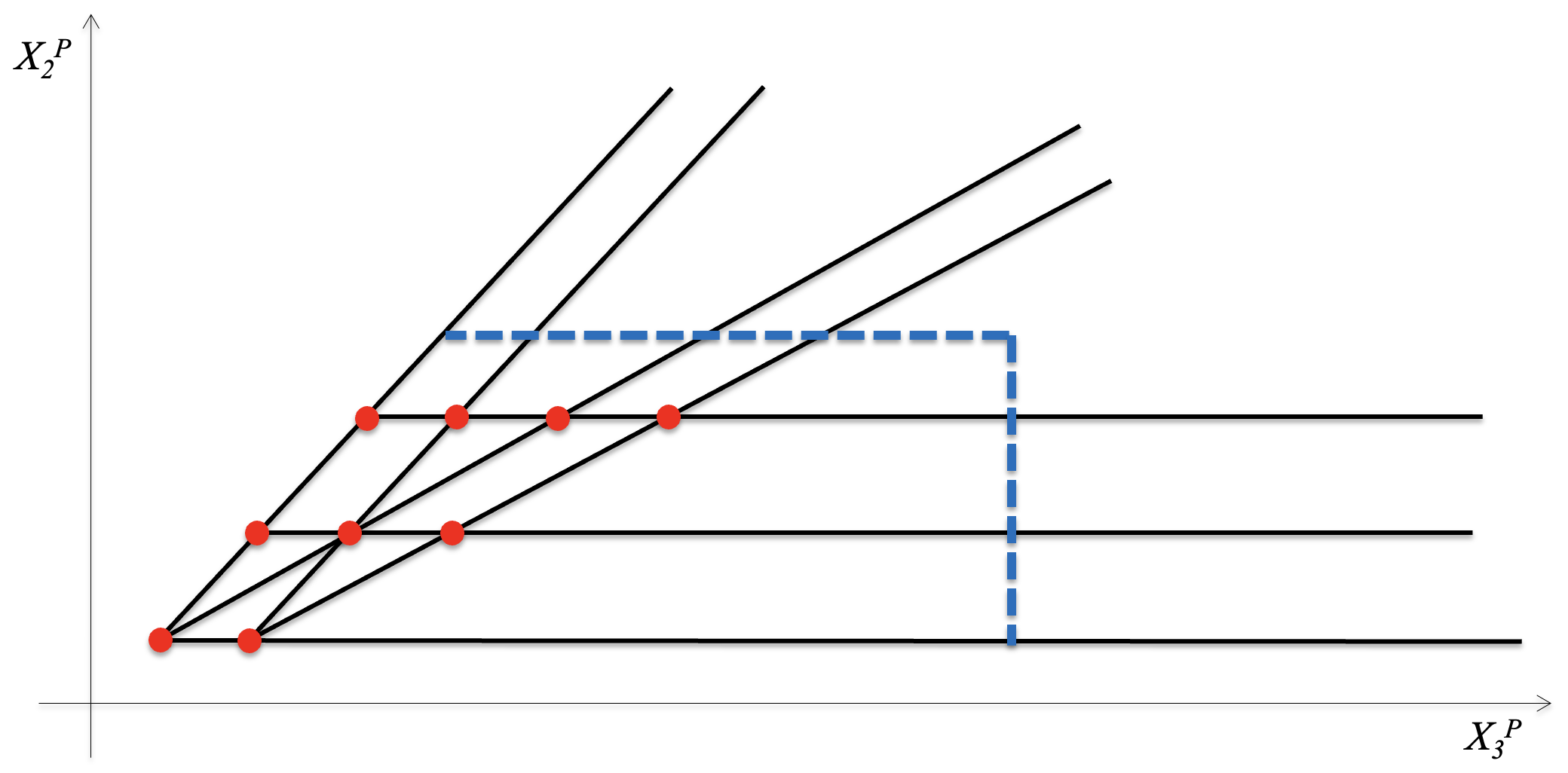}       \\ \\
$(c)$      \includegraphics[angle=0,width=2.8in]{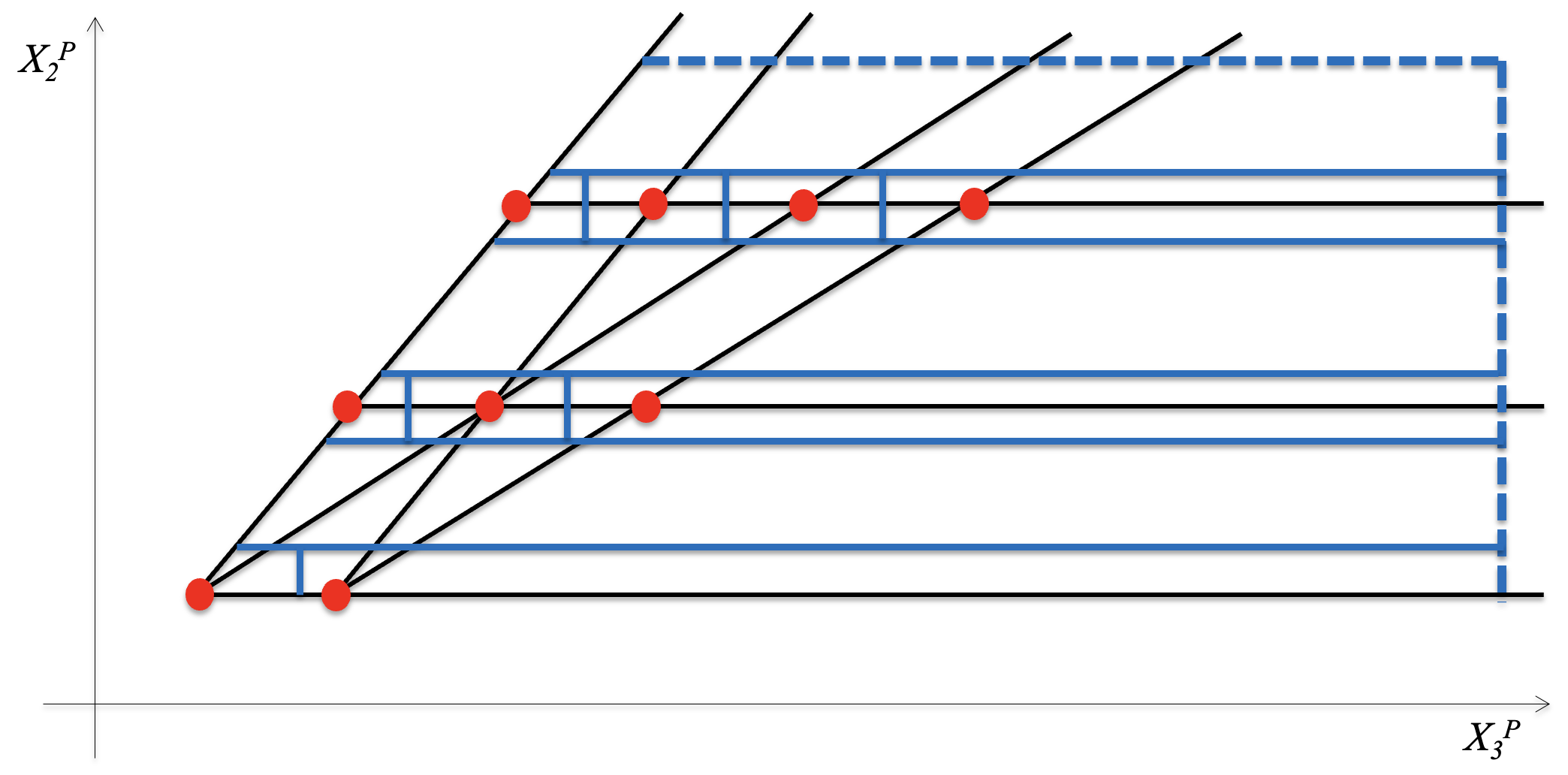} 
$(d)$      \includegraphics[angle=0,width=2.8in]{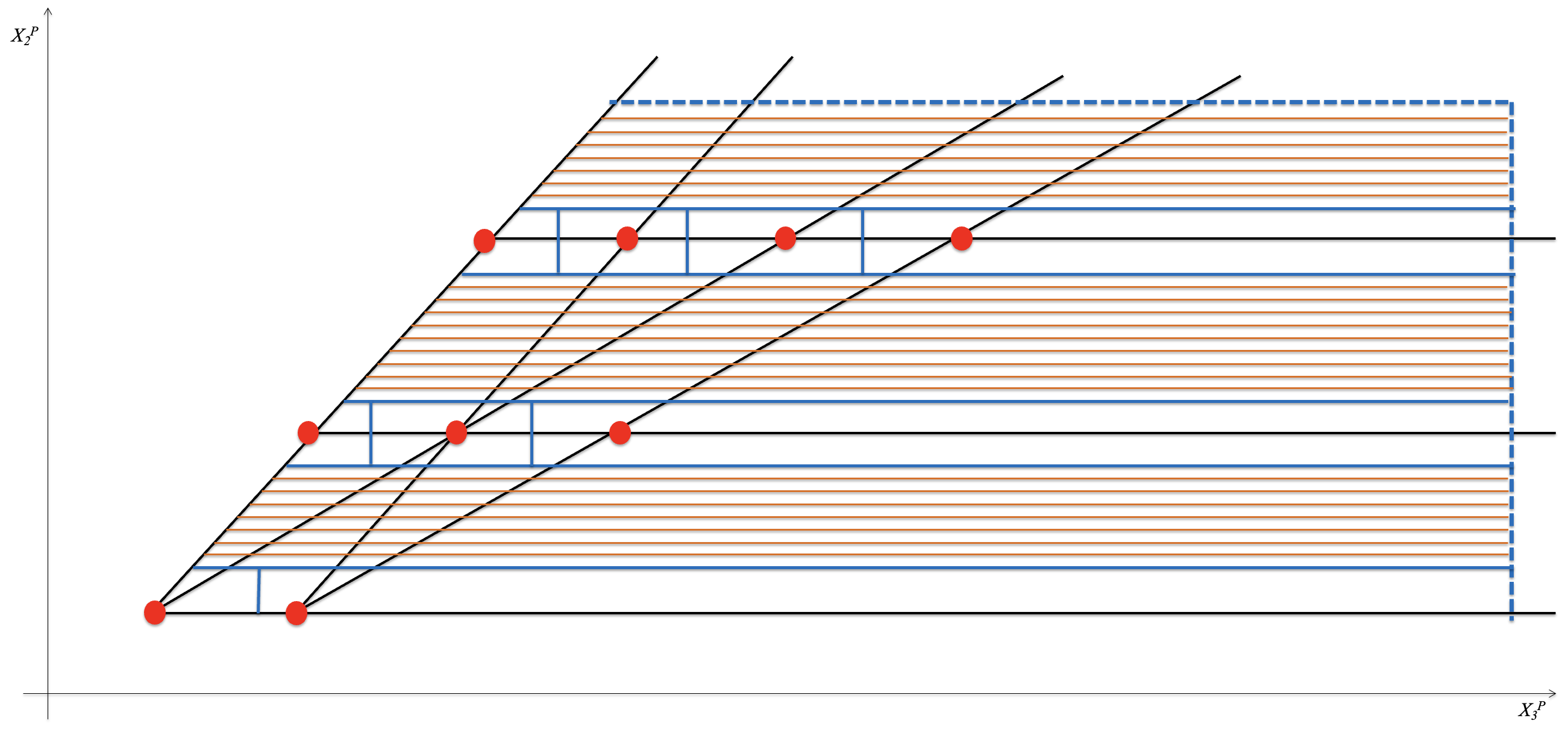}       \\ \\
$(e)$      \includegraphics[angle=0,width=2.8in]{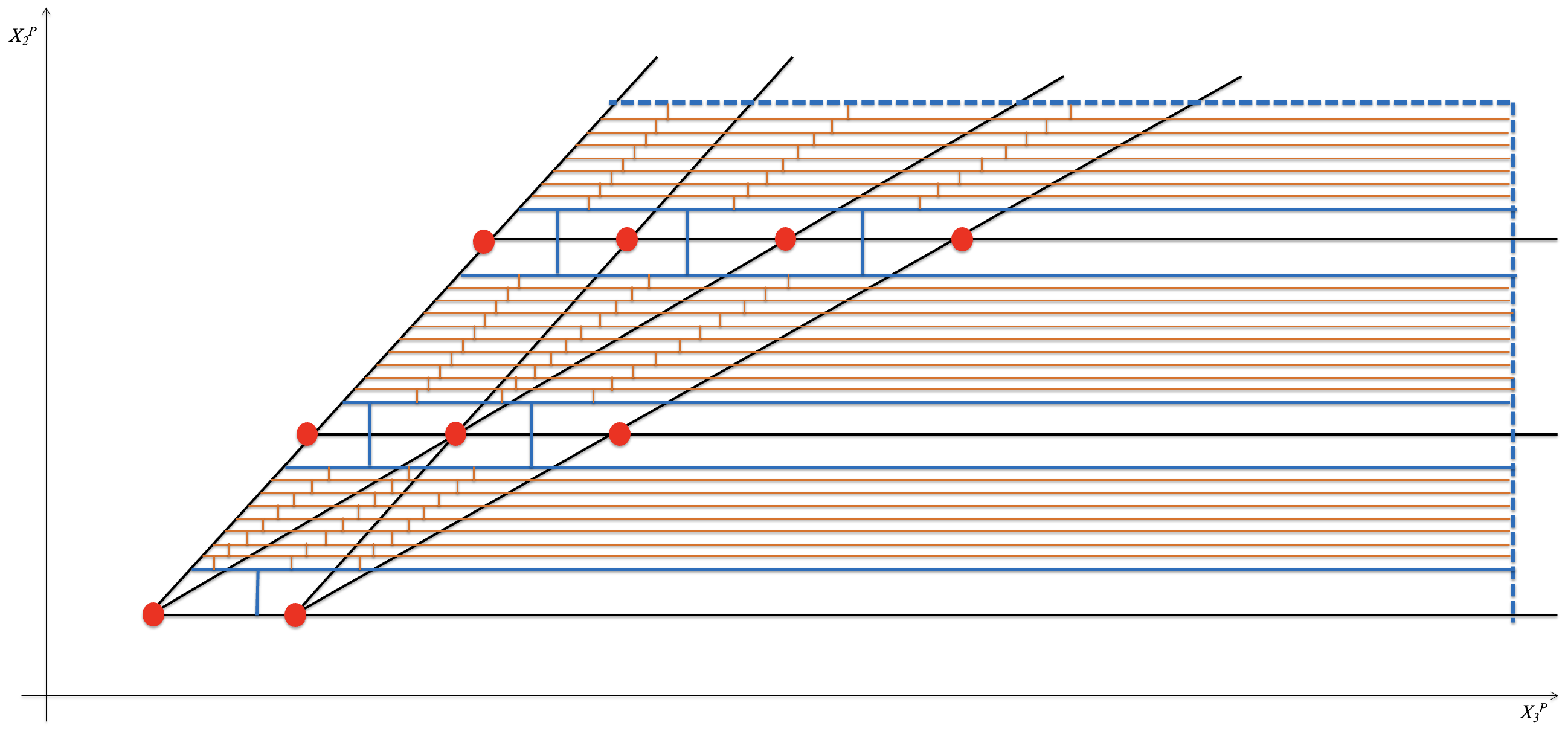} 
$(f)$       \includegraphics[angle=0,width=2.8in]{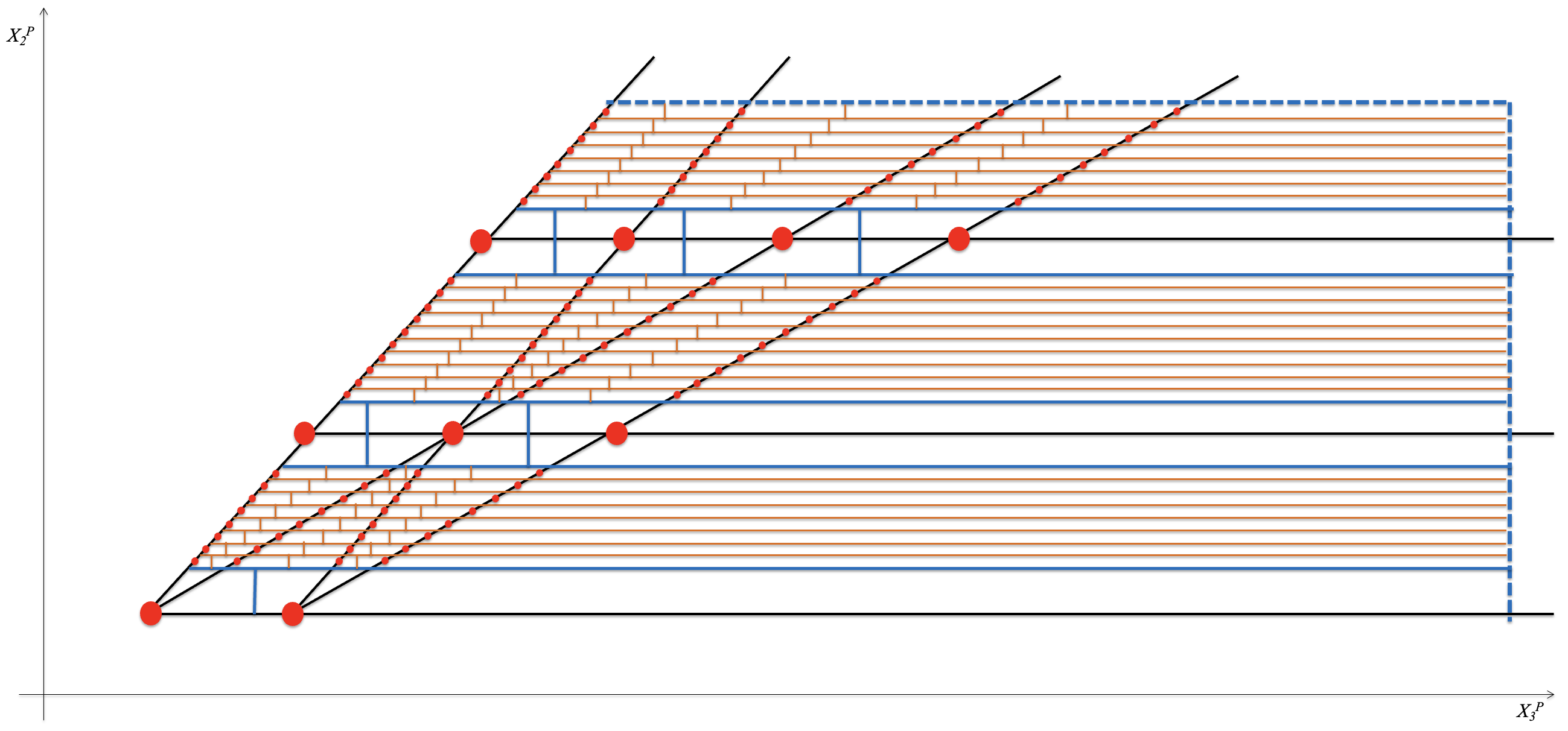}      \\ \\
$(g)$      \includegraphics[angle=0,width=5.9in]{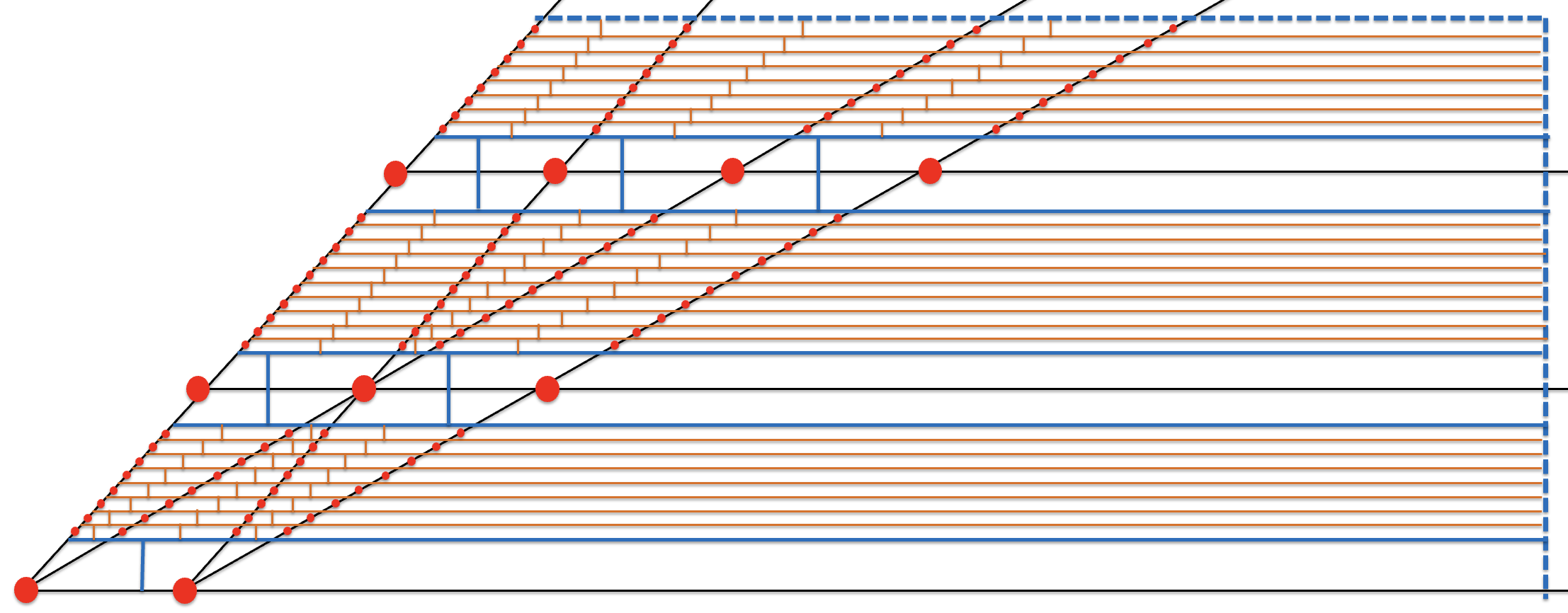} 
    \end{tabular}
  \end{center}
  \caption{\scriptsize \label{fig:7}  In $(a)$ we show a simpler subdivision $\F_2^P$, which is still arbitrary enough to represent the general case. In $(b)$ we show a blue box for $\F_2^P$, in $(c)$ we built tiles along the horizontal lines of  $\F_2^P$, and in $(d)$, $(e)$, $(f)$ we show how to proceed step by step to build a complete ZSH blueprint for $\F_2^P$. In $(g)$ we zoom in, to be able to see this ZSH blueprint in more detail. In particular, note that the smaller tiles also have well defined {\em basepoints} (i.e., the smaller red points).}
\end{figure}
%%%%%%%%%%%%%%%%%%%%%%%%%%%%%%%%%%%%%%%%%%%%%%

\subsubsection{Construction of ZSH blueprint in $\mathscr{D}_2^P$ (for a general $\mathcal{F}_2^P$)}

Consider a given $\mathcal{F}_2^P$ in $\mathscr{D}_2^P$, as shown in Figure~\ref{fig:6}$(a)$. We can see that the $\mathcal{F}_2^P$ in Figure~\ref{fig:6}$(a)$ is much more complicated than the one we have shown in Figure~\ref{fig:3a}$(a)$. 
In particular, we see that the affine half-lines present in $\mathcal{F}_2^P$ have four different slopes, as opposed to just 2 different slopes (zero \& one) in Figure~\ref{fig:3a}$(a)$. 
More importantly, in Figure~\ref{fig:3a} {\em \ul{we have avoided triangular regions} in $\mathcal{F}_2^P$}, while in Figure~\ref{fig:6} there are many triangular regions in $\mathcal{F}_2^P$. 

To illustrate the difficulty introduced by the presence of triangular regions, and also the fact that we can overcome it, let us look at the simpler subdivision $\mathcal{F}_2^P$ shown in Figure~\ref{fig:7}$(a)$. There is no immediate analogue of the type of blueprint we constructed in the previous section for this $\mathcal{F}_2^P$. Nevertheless, we can construct a blueprint for this $\mathcal{F}_2^P$ by proceeding step by step as illustrated in Figure~\ref{fig:7}.

{\em The key idea for the construction of the ZSH blueprint shown in Figure~\ref{fig:7}$(f)$ (see also the ``zoom-in" version in Figure~\ref{fig:7}$(g)$) is to do it \ul{in two stages:}} first choose a small enough\footnote{Here ``small enough" means that we need to be able to accommodate the small blue \underline{\em vertical} line segments, such that each blue vertical line segment {\em intersects only one black line}, i.e.,  intersects only the tile edge that it is  orthogonal on.} height for the \underline{\em blue} horizontal strips that cover the horizontal half-lines of $\mathcal{F}_2^P$, and partition them into quadrilaterals (one for each zero-dimensional element of $\mathcal{F}_2^P$); then, choose an even smaller\footnote{The ``even smaller" height needs to allow us to accommodate the small orange \underline{\em vertical} line segments, such that each orange vertical line segment {\em intersects no black line}, i.e.,  intersects no tile edge.} height for the \underline{\em orange} horizontal strips, in order to be able to partition these orange strips into quadrilaterals as shown, with each quadrilateral intersecting exactly one one-dimensional element of $\mathcal{F}_2^P$. Also, we choose the {\em basepoints} of these orange quadrilaterals to lie on this one-dimensional element of $\mathcal{F}_2^P$, at a height that is exactly along the midline of each horizontal orange strip, as shown in Figure~\ref{fig:7}$(g)$. (A fully general discussion of the existence of such a blueprint for any fan $\mathcal{F}_3$ in $\RRp^3$, and also for any fan $\mathcal{F}_n$ in $\RRp^n$ will be given in Section~\ref{section_construction_of_ZSH_pre-bluprint_and blueprint_in_nD}.)

%%%%%%%%%%%%%%%%%%%%%%%%%%%%%%%%%%%%%%%%%%%%%%
\begin{figure}[h!]
  \begin{center}
    \begin{tabular}{c}
      \hskip-0.5cm
$(a)$      \includegraphics[angle=0,width=2.3in]{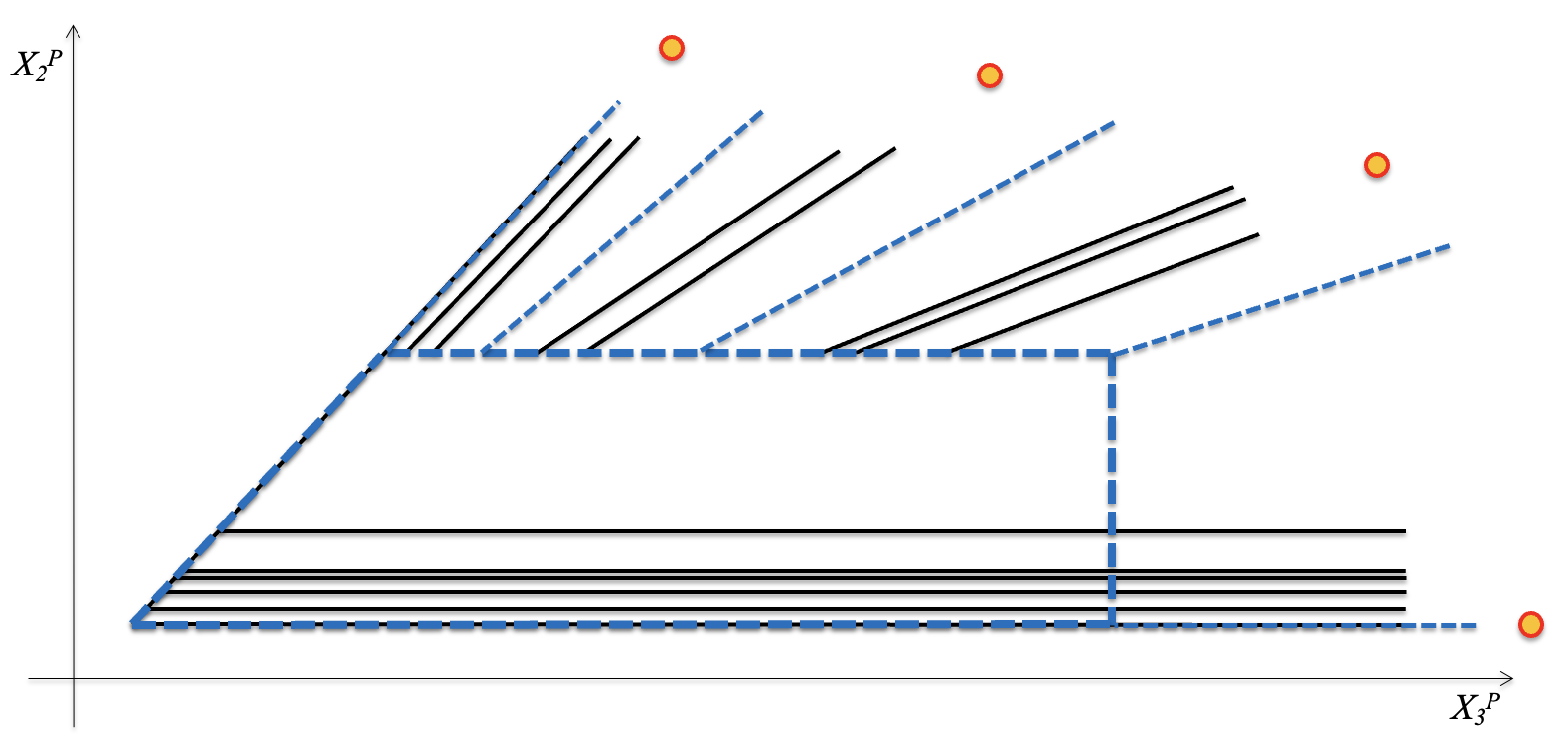} 
$(b)$      \includegraphics[angle=0,width=1.8in]{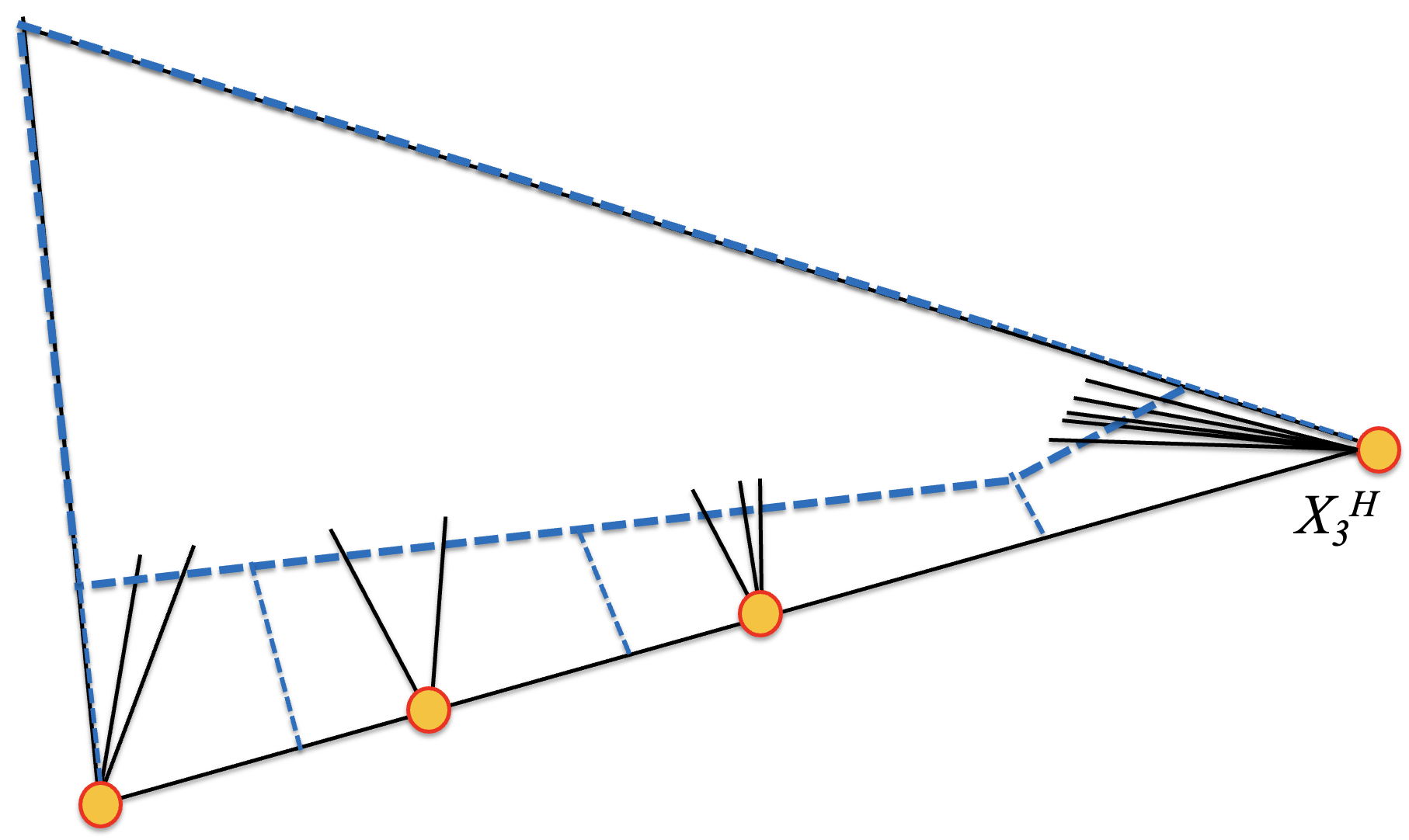}       
    \end{tabular}
  \end{center}
  \caption{\scriptsize \label{fig:8}  Consider again the fan $\F_3$ represented in Figure~\ref{fig:6}$(b)$. The four subdivision tiles that lie outside the blue box of  $\F_2^P$ are shown in $(a)$, together with their basepoints (which lie at infinity). In $(b)$ we show the same setting in $\F_2^H$ (instead of of  $\F_2^P$), including the same four subdivision tiles and their basepoints (which lie in the plane $X_1^H = 0$.)}
\end{figure}
%%%%%%%%%%%%%%%%%%%%%%%%%%%%%%%%%%%%%%%%%%%%%%

\medskip 

Let us now also discuss the tiles that lie \underline{outside} the blue box. {\em We  return to the larger fan $\F_3$} (and its corresponding $\mathcal{F}_2^P$, shown in Figure~\ref{fig:6}$(b)$, since it better illustrates the general case, see Figure~\ref{fig:8}). 

More specifically, in  Figure~\ref{fig:8}$(a)$ we are ignoring what happened inside the blue box, and focus on the elements of $\mathcal{F}_2^P$ {\em outside} the blue box. {\em Then, we  construct one ``boundary tile" for each slope featured by one such affine half-line element of $\mathcal{F}_2^P$}. Similar to the construction in the previous section, we choose the slopes of the three unbounded dotted blue lines (which separate the boundary tiles) such that their support lines pass through the origin $(0,0)$\footnote{Recall from the previous section that images via $\Psi_3^P$ of lines that are parallel to the $x_1$-axis are lines that pass through the origin. Therefore, this particular choice for the  three unbounded dotted blue lines will be compatible later with our construction of the ZSH faces that correspond to boundary tiles as part of a \ul{\em cylindrical surface} whose generating direction is the direction of the $x_1$-axis.}. Also, the basepoint of each boundary tile lies at infinity, and represents the (common) slope of the black half-lines that are contained within that tile\footnote{The reason we have four tiles in Figure~\ref{fig:8}$(a)$ is that there are four different such slopes for half-lines of $\mathcal{F}_2^P$.}.

The same four boundary tiles can be seen in Figure~\ref{fig:8}$(b)$. Note that now  the four basepoints lie within the plane $\{ X_1^H = 0 \}$. 

\subsubsection{The use the ZSH blueprint in the construction of the ZSH}

In general, once we have a ZSH blueprint as described in Figure~\ref{fig:7} and Figure~\ref{fig:8}, {\em  the construction of the polyhedral ZSH is analogous to the construction described in the previous section.} 

For faces of the ZSH that correspond to tiles that lie \underline{inside} the blue box the construction is completely similar. (Note that, like in the previous section, we have ensured that pairs of tiles that are adjacent horizontally {\em satisfy the important property that their basepoints lie on the same horizontal line.} This implies that any nonzero vector $v=(v_1, v_2, v_3)$ that is orthogonal to two such basepoints must have $v_3 = 0$. This allows us to apply Lemma~\ref{lemma_4.5_OLD}$(ii)$ to conclude that for any edge $e$ shared by such a pair of tiles its corresponding edge $\hat e$ within the ZSH satisfies the property that $\Psi_3^P(\hat e)$ will be arbitrarily close\footnote{The most important information that we derive from Lemma~\ref{lemma_4.5_OLD}$(ii)$ is that the tangent direction to any point of $\Psi_3^P(\hat e)$ is very close to vertical.} to $e$ if $\eps_0$ is small enough.)

For faces of the ZSH that correspond to tiles that lie \underline{outside} the blue box the construction is also similar (i.e., we just build the cylindrical surface that starts along the outer edges of the blue box (when seen in $\mathscr{D}_2^P$) and uses the direction of the $x_1$-axis as its generator). On the other hand, the setting in Section 3 was simpler, with only two such ``boundary tiles'', so we will look again at the representation of $\mathcal{F}_3$ in $\mathscr{D}_2^H$ (beside its representation in $\mathscr{D}_2^P$) in order to understand precisely these boundary faces. 

When we look at the four boundary tiles in $\mathscr{D}_2^H$ we can see that each one has a well defined basepoint (which specifies the normal of the corresponding face of the ZSH) and that these basepoints lie on the plane $\{X_1=0\}$, which confirms the fact that each one of these four ``outer faces'' are supported by planes that are parallel to the $x_1$-axis.

The fact that the polyhedral surface $\P(\B(\mathcal{F}_2^P))$ discussed in this section is indeed a ZSH (and that $\T_3$ cannot cross this ZSH within any cube $[0,M]^3$ if we choose $\eps$ small enough) can be argued exactly like in the previous section\footnote{Some of the key observations {\bf 1-3} about {\em images of faces of $\P(\B(\mathcal{F}_2^P))$} need to be modified slightly. For example, in observation ``{\bf 2.}" for some edges $e$ of the subdivision $\F_2^P$ there are more than two tiles whose basepoints lie on $e$ (and these basepoints are not necessarily endpoints of $e$), so we need to use a union of faces that correspond to these tiles, but this  does not  change anything relevant in the proof. Also, in observation ``{\bf 3.}" it is no longer true that any face $f$ of  $\F_2^P$ can be covered with just three or four tiles of $\B(\F_2^P)$, but this also does not  change anything relevant in the proof.}.

\bigskip

\bigskip

\newpage

\section{Subdivisions  in $\RR^n$, their ZSH pre-blueprints, and their faithful ZSH blueprints}
\label{section_construction_of_ZSH_pre-bluprint_and blueprint_in_nD}

We have already seen that key properties of 3D toric differential inclusion $\T_3$ can be represented by a 2D  {\em subdivision} $\F_2^P$ that is obtained from the 3D polyhedral fan $\F_3$ of $\T_3$. Also, we have seen how a 2D  {\em ZSH blueprint} $\B(\F_2^P)$ can be used for constructing a ZSH in 3D.

Here we will extend this approach to $n$D, {\em while  also modifying the 2D construction accordingly}, in order to be able to include it in the framework of an  induction on dimension.  

For example, the general 2D {\em subdivisions} that we consider here are no longer restricted to the subset $\D_2^P \subset \RR^2$, see Figure~\ref{Fig_1}$(a)$; similarly, in this section $n$D subdivisions will be defined \ul{\em on the whole $\RRn$}, instead of being restricted to a subset of it. 

Moreover, instead of just {\em ZSH blueprints}, we will also discuss \ul{\em ZSH pre-blueprints}, and use them as a key tool for constructing \ul{\em faithful ZSH blueprints}, which will play an especially important role here. Also, in this section these pre-blueprints and blueprints \ul{\em will be defined on the whole Euclidean space} (i.e., they will not be restricted to a subset of it, as in Section~\ref{section:proof_in_3D}). 

See Figure~\ref{Fig_1}$(b)$ for an example of a {\em ZSH pre-blueprint}, and see Figure~\ref{Fig_101}$(g)$ for an illustration of a {\em faithful ZSH blueprint}.

\begin{figure}[h!]
\begin{center}
$(a)$ \includegraphics[width=10cm]{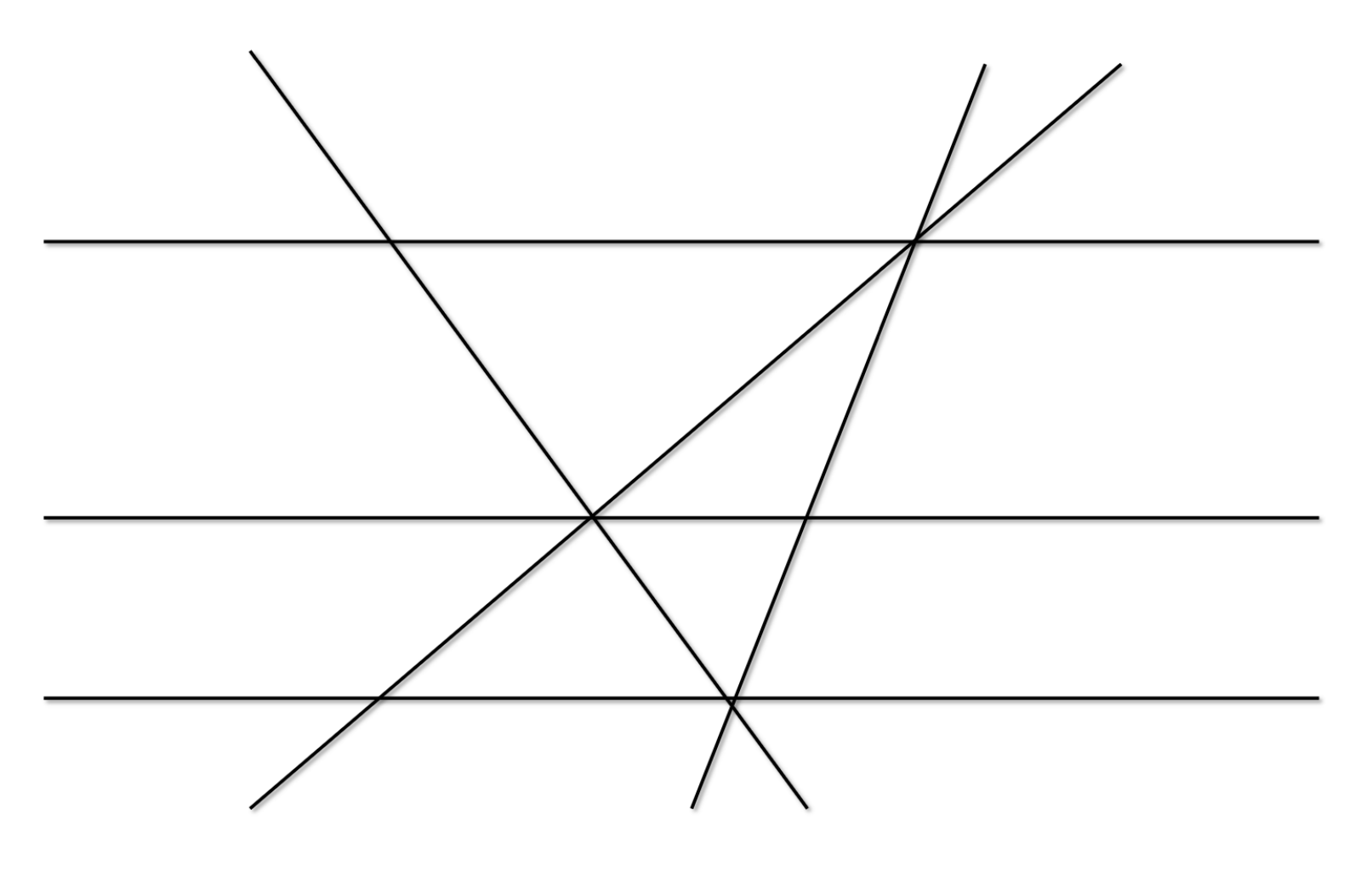} \\
$(b)$ \includegraphics[width=10cm]{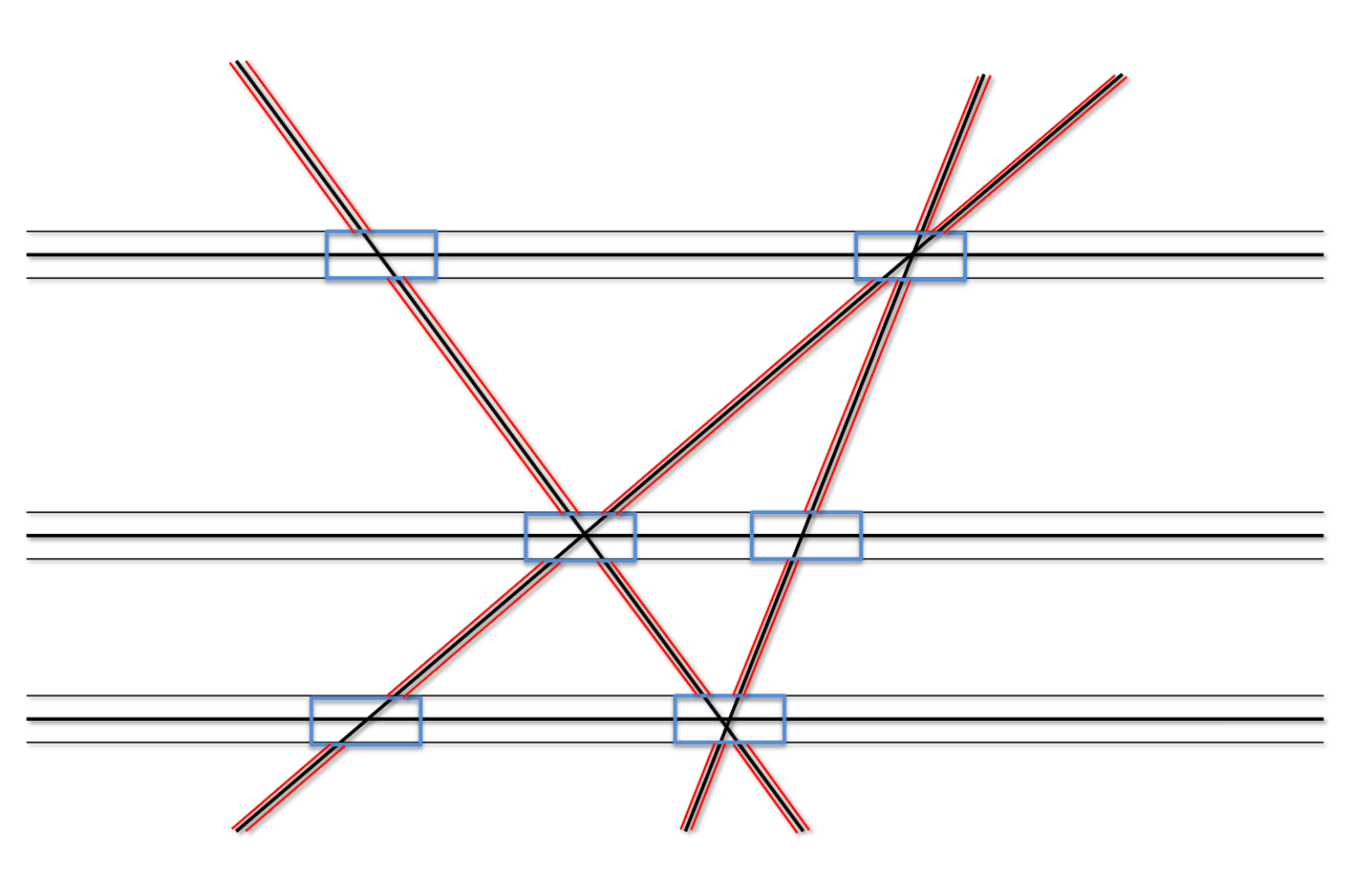} 
\end{center}
\caption{
\label{Fig_1}
\scriptsize  In $(a)$ we see an example of an $x_1$-$x_2$-subdivision $\S_2$, and in $(b)$ we see a ZSH pre-blueprint  for it.
}
\end{figure}

\subsection{Polyhedral fans and subdivisions in $\RR^n$}

Recall that by a {\em ``fan"} we mean a {\em complete and pointed polyhedral fan}. We start by defining a special kind of fan, called {\em $x_1$-$x_n$-fan}. The definition proceeds inductively, for all $n\ge2$, as follows. 

An {\em $x_1$-$x_2$-fan} $\F_2$ is a fan in $\RR^2$ that contains both of the 1D cones (i.e., rays) that lie along the $x_2$-axis. 

An {\em $x_1$-$x_3$-fan} $\F_3$ is a fan in $\RR^3$ such that there exists a $x_1$-$x_2$-fan $\F_2$ in $\RR^2$ with the property that $\pi_3(\F_3)=\F_2$, where $\pi_3$ is the standard projection from $\RR^3$ to $\RR^2$. 
This means that, in particular, the half-planes in $\RR^3$ that project to rays of $\F_2$ must be unions of cones of $\F_3$. Also, the set $\pi_3^{-1}(\{ 0 \})$, i.e., the $x_3$-axis, must be a union of cones of $\F_3$, namely the positive ray and the negative ray of the $x_3$-axis\footnote{Note that we could have also defined an ``$x_1$-$x_1$-fan" to be {\em any} fan in $\RR^1$, and then an equivalent definition of an $x_1$-$x_2$-fan could have been {\em a fan $\F_2$ in $\RR^2$ such that there exists a $x_1$-$x_1$-fan $\F_1$ in $\RR^1$ such that $\pi_2(\F_2)=\F_1$}, i.e., would have been analogous to the definition of an $x_1$-$x_3$-fan. (By the way, note also that {\em there exists actually only one fan in $\RR^1$,} because we only consider {\em complete} fans.)}. 

\noindent
Similarly, in general, a {\em $x_1$-$x_n$-fan} $\F_n$ is a fan in $\RR^n$ such that there exists a $x_1$-$x_{n-1}$-fan $\F_{n-1}$ in $\RR^{n-1}$ with the property that $\pi_n(\F_n)=\F_{n-1}$.

\bigskip

We will now define some polyhedral subdivisions of  $\RR^n$ called {\em $x_1$-$x_n$-subdivisions}, which are strongly related to $x_1$-$x_n$-fans. Consider the embedding 
$$
i_n : \RR^n \to \RR^{n+1}
$$ 
obtained by mapping $\RR^n$ to the hyperplane $\{ x_1 = 1\}$ in $\RR^{n+1}$ via 
$$
i_n(x_1, ..., x_n) = (1, x_1, ..., x_n).
$$
Then an {\em $x_1$-$x_n$-subdivision} $\S_n$ is a polyhedral subdivision of $\RR^n$ that can be obtained by using the embedding $i_n$ to identify $\RR^n$ with  the hyperplane $\{ x_1 = 1\}$ of $ \RR^{n+1}$, and then {\em taking the intersection of this hyperplane with all the cones of a $x_1$-$x_{n+1}$-fan $\F_{n+1}$.} In other words, up to the identification given by $i_n$, we have
$$
\S_n = \{  C \cap \{ x_1 = 1\}  \  | \   C \in \F_{n+1}  \}.
$$
In particular, note that this implies that for any $x_1$-$x_n$-subdivision $\S_n$ there exists an $x_1$-$x_{n-1}$-subdivision $\S_{n-1}$ such that $\pi_n(\S_n)=\S_{n-1}$.

An element of an $x_1$-$x_n$-subdivision $\S_n$ is called a {\em face} of $\S_n$, i.e., up to the identification given by $i_n$, a {\em face} of $\S_n$ is a nonempty intersection between some cone $C \in \F_{n+1}$ and the plane $\{ x_1 = 1\}$.

In subsection~\ref{subsection:ZSH_pre-blueprints} will define ZSH pre-blueprints for any $x_1$-$x_n$-subdivision $\S_n$, and in subsection~\ref{subsection:from_ZSH_pre-blueprints_to_faithful_ZSH_blueprints} we will use ZSH pre-blueprints in order to define faithful ZSH blueprints for any $x_1$-$x_n$-subdivision $\S_n$.

\subsection{Binary words for faces of an $x_1$-$x_n$-subdivision $\S_n$}

Consider an $x_1$-$x_n$-subdivision $\S_n$. 
First, note that to any face
%\footnote{All elements of $\S_n$ are simplexes, possibly with one or more points, faces, etc., lying at infinity.} 
$\Delta_0 \in \S_n$ we can associate {\em a binary word of length $n$} denoted 
$$
\alpha(\Delta_0) = \overline{\a_1 a_2 ... a_n} \in \{ 0,1 \}^n,
$$
as follows. 
Consider the face $\pi_n(\d_0)$ in $\RR^{n-1}$;  if we have
$$
\dim(\pi_n(\d_0)) = \dim(\d_0)
$$
then we define $\a_n = 0$, and if 
$$
\dim(\pi_n(\d_0)) = \dim(\d_0) - 1
$$
then we define $\a_n = 1$.
Similarly, we define
$$
\a_{n-1} = \dim(\pi_n(\d_0)) - \dim(\pi_{n-1}(\pi_n(\d_0))),
$$
and so on. In other words, as we consider the successive projections $\pi_n$, $\pi_{n-1}$, $\pi_{n-2}, ..., \pi_1$ of $\d_0$, the binary letter $\a_j$ tells us if the $j^\mathrm{th}$ projection $\pi_j$ did or did not cause its dimension to decrease by one unit. Equivalently, but in the opposite order, we can think about $\d_0$ as having been ``grown" by starting from a single point $0 \in \RR^0$, and then its dimension went up by one  unit at every stage $j$ that has $\a_j = 1$.

\bigskip

\subsection{ZSH pre-blueprints of an $x_1$-$x_n$-subdivision $\S_n$}
\label{subsection:ZSH_pre-blueprints}

In this subsection we introduce and study the notion of {\em ZSH pre-blueprint} associated to an $x_1$-$x_n$-subdivision $\S_n$. Our main goal in this section is to show that we can construct {\em ZSH blueprints} for $\S_n$, and the construction of {\em ZSH pre-blueprints} in this subsection will be the most important step towards that goal.

\bigskip

We first define {\em one-dimensional ZSH pre-blueprints}; these are sets of intervals with disjoint interiors that cover $\RR^1$ and have a strong relationship with a given $x_1$-$x_1$-subdivision. Consider an $x_1$-$x_1$-subdivision $\S_1$, as shown in Figure~\ref{fig:x1-simplicial-subdivision}. In general, $\S_1$ consists of a finite set of points (i.e., 0-faces) and a  finite set of closed intervals (i.e., 1-faces) whose endpoints are those 0-faces.

\bigskip

\begin{figure}[h]
\begin{center}
\includegraphics[width=10cm]{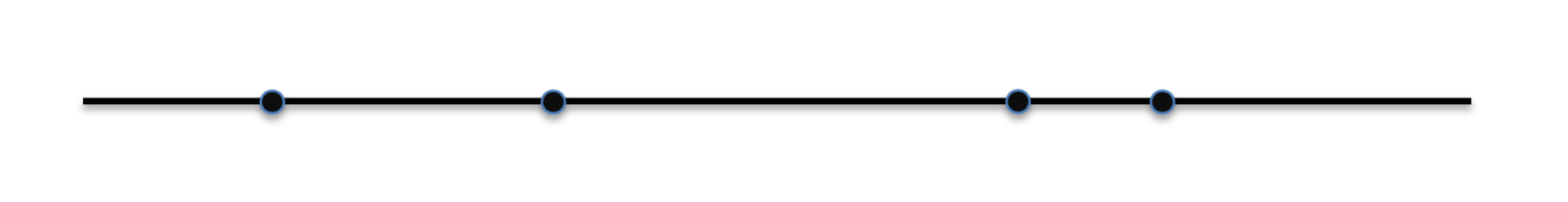}
\end{center}
\caption{\small \it An example of an $x_1$-$x_1$-subdivision $\S_1$ of $\RR^1$.}
\label{fig:x1-simplicial-subdivision}
\end{figure}

\bigskip

A {\em ZSH pre-blueprint} of $\S_1$ is given by an assignment of a closed interval denoted $\nbhd(\sigma)$ for any $\sigma \in \S_1$ such that there exists some positive number denoted $\eps(0)$ for which we have 
$$
\nbhd(\sigma_0) = \sigma_0 + [ - \eps(0), \eps(0) ]
$$
for any 0-face $\sigma_0 \in \S_1$, and 
$$
\nbhd(\sigma_1) = \sigma_1 \setminus \displaystyle \bigcup_{\sigma_0 \in \S_1} \nbhd(\sigma_0)
$$
for any 1-face $\sigma_1 \in \S_1$ (where the union above is over 0-faces $\sigma_0 \in \S_1$). 
Moreover, we require that $\eps(0)$ is small enough so the interiors of the intervals $\nbhd(\sigma)$ are disjoint for all $\sigma \in \S_1$. 

More precisely, in general, {\em a ZSH pre-blueprint} $^p\!\B_1$ of $\S_1$ as described above is defined as the set of intervals $\nbhd(\sigma)$, i.e., we have
$$
^p\!\B_1 = \{  \nbhd(\sigma) \  |  \  \sigma \in \S_1  \}.
$$
Informally, we can think of  ``$\nbhd$" as a function that associates a subset of $\RR$ to any $\sigma \in \S_1$, according to the rules described above. Note then that the function  $\nbhd$ uniquely determines the set $^p\!\B_1$; also, less obviously, the set of intervals $^p\!\B_1$ uniquely determines the function $\nbhd$. Indeed, given the set $^p\!\B_1$ and some zero-dimensional $\sigma_0 \in \S_1$, we can define $\nbhd(\sigma_0)$ to be the (unique) element of $^p\!\B_1$ that contains the point $\sigma_0$; and for any one-dimensional $\sigma_1 \in \S_1$, we can define $\nbhd(\sigma_1)$ to be the (again, unique) element of $^p\!\B_1$ that intersects $\sigma_1$ and was not already assigned to be the $\nbhd(\sigma_0)$ for some zero-dimensional $\sigma_0 \in \S_1$. 

\bigskip

We will now define {\em two-dimensional ZSH pre-blueprints}, i.e.,  ZSH pre-blueprints associated to $x_1$-$x_2$-subdivisions; these are composed of some special polygonal sets with disjoint interiors, which cover $\RR^2$ and have a strong relationship with a given $x_1$-$x_2$-subdivision. 

Consider an $x_1$-$x_2$-subdivision $\S_2$ that projects to an $x_1$-$x_1$-subdivision denoted $\S_1$. We define a  {\em ZSH pre-blueprint} $^p\!\B_2$ of $\S_2$ as follows. First, we assume that we have already constructed a ZSH pre-blueprint $^p\!\B_1$ of $\S_1$, for some $\eps(0) > 0$ (on which additional restrictions will be imposed below).

\noindent
Note now that there are four types of faces in $\S_2$, corresponding to the binary words 00, 01, 10, 11. Specifically, zero-dimensional faces $\sigma_{00} \in \S_2$ have $\a(\sigma_{00}) = 00$ and are just points in $\RR^2$; one-dimensional faces $\sigma_{01} \in \S_2$ with $\a(\sigma_{01}) = 01$ are edges that are parallel to the $x_2$-axis; one-dimensional faces $\sigma_{10} \in \S_2$ with $\a(\sigma_{10}) = 10$ are edges that are {\em not} parallel to the $x_2$-axis; and, finally, two-dimensional faces $\sigma_{11} \in \S_2$ with $\a(\sigma_{11}) = 11$ are polygons\footnote{Some edges and some polygons are unbounded.}. 

\noindent
Consider some positive numbers $\eps(00)$ and $\eps(10)$ and define $\eps(01) = \eps(0)$. Informally, it will be convenient to assume that we have 
$$
0 < \eps(10) \ll \eps(01) \ll \eps(00) \ll1; 
$$
this will be made more precise below. 

\noindent
As in the one-dimensional case, we will now define a set $\nbhd(\sigma)$ for any $\sigma \in \S_2$.
Consider first a zero-dimensional face  $\sigma_{00} \in \S_2$. We define 
$$
\nbhd(\sigma_{00}) = \sigma_{00} + [ - \eps(01), \eps(01) ] \times  [ - \eps(00), \eps(00) ].
$$
If $\sigma_{10} \in \S_2$ has $\a(\sigma_{10}) = 10$ then we define
$$
\nbhd(\sigma_{10}) = \left( \sigma_{10} +  [ - \eps(10), \eps(10) ] \right) \medcap \left( \nbhd(\pi_2(\sigma_{10}))\times\RR \right).
$$

\noindent
For the faces of the form $\sigma_{01}$ and $\sigma_{11}$ (i.e., having corresponding binary words 01 and 11, respectively) we define their $\nbhd$ by taking advantage of the ones already defined above. Specifically, for an edge $\sigma_{01} \in \S_2$ we define $\nbhd(\sigma_{01})$ to be the subset of 
$$
\nbhd(\pi_2(\sigma_{01})) \times \RR
$$
that is delimited by $\nbhd(\sigma'_{00})$ and $\nbhd(\sigma''_{00})$, where $\sigma'_{00}$ and $\sigma''_{00}$ are the endpoints of $\sigma_{01}$. (If $\sigma_{01}$ is unbounded then only one such delimiter is needed.)

\noindent
The definition of $\nbhd(\sigma_{11})$ for some triangle $\sigma_{11} \in \S_2$ is similar, except that the delimiting points $\sigma'_{00}$ and $\sigma''_{00}$ are replaced by delimiting edges $\sigma'_{10}$ and $\sigma''_{10}$ that are of binary type 10 and lie on the boundary of $\sigma_{11}$. (And again, if $\sigma_{11}$ is unbounded, then only one such delimiting edge is needed.)

\noindent
Moreover, we impose the condition that the interiors of $\nbhd(\sigma')$ and $\nbhd(\sigma'')$ are disjoint for all $\sigma' \neq \sigma'' \in \S_2$. It is easy to see in 2D that (for a fixed $x_1$-$x_2$-subdivision $\S_2$) if we assume that if the positive number
$$
\max \Bigl\{  \eps(00), \frac{\eps(01)}{\eps(00)}, \frac{\eps(10)}{\eps(01)}  \Bigr\}
$$
is small enough\footnote{Note that saying that this maximum is ``small enough" is essentially the same as saying that we have 
$
0 < \eps(10) \ll \eps(01) \ll \eps(00) \ll 1 
$.
}, then the construction of 2D ZSH pre-blueprints described above will satisfy this condition\footnote{See Figure~\ref{Fig_1}$(b)$: the number $\eps(00)$ is the (small) width of the blue rectangles;  the number $\eps(01)$ is the (smaller) height of the blue rectangles;  and  the number $\eps(10)$ is the (even smaller) width of the red parallelograms.}. Indeed, if $\eps(00)$ is small enough then the sets $\nbhd(\sigma')$ and $\nbhd(\sigma'')$ are disjoint for all $\sigma' \neq \sigma'' \in \S_2$ of binary type 00, which implies the same for type 01, by construction. Moreover, the inequalities $\eps(10) \ll \eps(01) \ll \eps(00)$ imply that  for any  $\sigma \in \S_2$ of binary type 10 we have that $\nbhd(\sigma)$ can be obtained from the set 
$$
\sigma + [-eps(10), \eps(10)]
$$
by subtracting the sets $\nbhd(\sigma')$ and $\nbhd(\sigma'')$, where $\sigma'$ and $\sigma''$ are the endpoints\footnote{Because, even if the edge $\sigma$ of binary type 10 has {\em very small slope (in absolute value)}, we still can conclude that the top and bottom of $\nbhd(\sigma)$ only intersect the top or bottom of $\nbhd(\sigma')$ and $\nbhd(\sigma'')$ (i.e., they \ul{don't intersect the left and right sides} of $\nbhd(\sigma')$ and $\nbhd(\sigma'')$).} of $\sigma$. 
This also implies that  the images through the projection $\pi_2$ of the sets $\nbhd(\sigma) \in$ $^p\!\B_2$ are either identical or have disjoint interior. Finally, this allows us to conclude that the construction above gives rise to sets $\nbhd(\sigma)$ that always have disjoint interiors.

\bigskip

Assume now that we have defined ZSH pre-blueprints for $x_1$-$x_{n-1}$-subdivisons; we will now  define ZSH pre-blueprints of an $x_1$-$x_n$-subdivison $\S_n$. 

\noindent
First, we consider some real-valued function $\hat\eps(\a)$ whose domain is the set of all binary words of length $k \le n$, such that the  function $\hat\eps$ satisfies 
$$
\hat\eps(\a_1\a_2 ... \a_{k-1}0) > 0
$$
and also 
$$
\hat\eps(\a_1\a_2 ... \a_{k-1}1) = 0.
$$
Then we define the {\em fiber} of a binary word $\a = \a_1\a_2 ... \a_n$ to be 
$$
\fiber(\a) = [-\hat\eps(\hat\a_1), \hat\eps(\hat\a_1)] \times [-\hat\eps(\hat\a_2), \hat\eps(\hat\a_2)] \times ... \times [-\hat\eps(\hat\a_n), \hat\eps(\hat\a_n)], 
$$
where $\hat\a_j$ denotes the prefix of $\a$ of length $j$, i.e., $\hat\a_j = \a_1\a_2...\a_j$.

\noindent
Then, if $\sigma \in \S_n$ has binary type $\a = \a_1\a_2 ... \a_{n-1}0$, we define 
$$
\nbhd(\sigma) = \left( \sigma +  \fiber(\a) \right) \medcap \nbhd(\pi_n(\sigma)) \times \RR.
$$
On the other hand, if $\a = \a_1\a_2 ... \a_{n-1}1$, then we define $\nbhd(\sigma)$ to be the closed subset of 
$$
\nbhd(\pi_n(\sigma)) \times \RR
$$
that is delimited by\footnote{This step in the constructions \ul{\em assumes} that $\nbhd(\sigma')$ and $\nbhd(\sigma'')$ have disjoint interiors, and  therefore there  is a well defined notion of ``subset of $\nbhd(\pi_n(\sigma)) \times \RR$ delimited by $\nbhd(\sigma')$ and $\nbhd(\sigma'')$". We will {\em prove} that this is the case for well chosen $\hat\eps(\a)$ values, see Theorem~\ref{thm:ZSH_preblueprints_exist}.} $\nbhd(\sigma')$ and $\nbhd(\sigma'')$, where $\sigma'$ and $\sigma''$ are the facets\footnote{
The {\em facets} are codimension-one faces.
} of  $\sigma$ that satisfy 
$$
\pi_n(\sigma') = \pi_n(\sigma'') = \pi_n(\sigma).
$$
(Like we saw before, some unbounded $\sigma$ may only have one such facet $\sigma'$. If so, then $\nbhd(\sigma'')$ is defined to be an unbounded set, which is delimited on one side by $\nbhd(\sigma')$.) 
Moreover, we impose that the interiors of $\nbhd(\sigma')$ and $\nbhd(\sigma'')$ are disjoint for all $\sigma' \neq \sigma'' \in \S_n$.

\bigskip

Of course, it is not obvious that the construction of $n$-dimensional ZSH pre-blueprints described above is always possible. The theorem below tells us that for well-chosen values of $\hat\eps(\a)$ {\em we can guarantee that this construction is possible.}

In order to formulate this theorem, recall that we have discussed above the real-valued function $\hat\eps(\a)$; based on such a function we define another real-valued function $\eps(\a)$, which is only defined for binary words $\a$ of length {\em exactly $n$}, and is given by 
$$
\eps(\a_1\a_2...\a_{n-1}0) = \hat\eps(\a_1\a_2...\a_{n-1}0)
$$
and 
$$
\eps(\a_1\a_2...\a_{k-1}011...1) = \hat\eps(\a_1\a_2...\a_{k-1}0).
$$
In other words, the value of $\eps(\a)$ is the same as the value of $\hat\eps(\hat\a)$, where $\hat\a$ is obtained from $\a$ by erasing all the ``1" bits at the end of $\a$. Note that, in particular, we have $\eps(\a) > 0$ for all\footnote{Except for the word 11...11, because $\eps(11...11)$ \ul{is not defined.}} binary words $\a$ of length  $n$.

\begin{thm}
\label{thm:ZSH_preblueprints_exist}
Consider an $x_1$-$x_n$-subdivison $\S_n$ and some functions $\hat\eps(\a)$ and $\eps(\a)$ as described above. If the number 
\begin{equation}
\label{max_eps}
\eps(^p\!\B_n) := \max \Bigl\{  \eps(00...00), \frac{\eps(00...01)}{\eps(00...00)}, \frac{\eps(00...10)}{\eps(00...01)}, ... ,  \frac{\eps(11...10)}{\eps(11...01)} \Bigr\}
\end{equation}
is small enough\footnote{Note that this is essentially equivalent to saying that we have $ \eps(11...10) \ll \eps(11...01) \ll ... \ll \eps(00...10) \ll \eps(00...01) \ll \eps(00...00) \ll 1 $, and therefore the magnitude sizes of the numbers $\eps(\a)$ grow exactly in the \ul{opposite} direction than the standard ordering of the binary numbers given by the binary words $\a$.}, then the construction of a ZSH pre-blueprint $^p\!\B_n$ of $\S_n$ described above is well defined. In particular, if $\sigma' \neq \sigma'' \in \S_n$, then the interiors of $\nbhd(\sigma')$ and $\nbhd(\sigma'')$ are disjoint. 
\end{thm}

%\begin{defn}
%We will refer to the positive number defined in (\ref{max_eps}) as the {\em $\eps$-ratio} of $^p\!\B_n$. 
%\end{defn}

\begin{proof}
We proceed by induction, so we assume that we already know this to be true in dimensions $\le n-1$.

%Note first that, by definition, the projections of the elements of $^p\!\B_n$ via $\pi_n$ are contained in a ZSH pre-blueprint $^p\!\B_{n-1}$ of the fan $\S_{n-1} = \pi_n(\S_n)$, constructed using only {\em some} of the values of the function $\eps$ from~(\ref{max_eps}).

Note that if we can impose some upper bound on the quantity shown in~(\ref{max_eps}), then we can also impose the same upper bound on the corresponding quantity that is restricted only to binary words $\a$ of the form $\a_1\a_2...\a_{n-1}1$, and these can be regarded as {\em nbhd} sizes for ZSH pre-blueprints constructed for the projections of $\S_n$ in dimensions 1, 2, ..., $n-1$.

\noindent
Then, if the quantity in~(\ref{max_eps}) is small enough, we can apply the induction hypothesis to the ZSH pre-blueprint 
%$^p\!\B_n \cap \{x_n = 0\}$ 
of $\pi_n(\S_n)$ that appears in the definition of $^p\!\B_n$, and it follows that the interiors of $\nbhd(\pi_n(\sigma'))$ and $\nbhd(\pi_n(\sigma''))$ are disjoint for all $\sigma', \sigma'' \in \S_n$ such that $\pi_n(\sigma') \neq \pi_n(\sigma'')$, and therefore the interiors of $\nbhd(\sigma')$ and $\nbhd(\sigma'')$ are also disjoint for all $\sigma', \sigma'' \in \S_n$ such that $\pi_n(\sigma') \neq \pi_n(\sigma'')$.

Therefore, it is now sufficient to show that $\nbhd(\sigma')$ and $\nbhd(\sigma'')$ are disjoint for those $\sigma', \sigma'' \in \S_n$ such that $\pi_n(\sigma') = \pi_n(\sigma'')$. Moreover, it is sufficient to consider the case where $\sigma'$ and $\sigma''$ have binary type $\a$ of the form $\a_1\a_2...\a_{n-1}0$\footnote{Because, once we have handled this case, in the other cases (i.e., both of the form $\a_1\a_2...\a_{n-1}1$, or one of the form $\a_1\a_2...\a_{n-1}0$ and the other of the form $\a_1\a_2...\a_{n-1}1$), the interiors of $\nbhd(\sigma')$ and $\nbhd(\sigma'')$ are disjoint by definition.}; in particular, we then have $\dim(\sigma') = \dim(\sigma'')$.

If $\dim(\sigma') \le n-2$, then  $\sigma'$ and $\sigma''$ lie in a common linear subspace of dimension $\le n-1$, and we can use the induction hypothesis to conclude that $\nbhd(\sigma')$ and $\nbhd(\sigma'')$ are disjoint. (Alternatively, we could use the same argument as presented below for the case $\dim(\sigma') = n-1$, but formulate it in a subspace of dimension equal to $\dim(\sigma')+1$.)

If $\dim(\sigma') = n-1$, then  $\sigma'$ and $\sigma''$ have binary type $11...110$. In this case, it is sufficient to consider the case where  $\dim(\sigma' \cap \sigma'')$ is as large as possible\footnote{Because otherwise we can intercalate an additional $\sigma'''$ {\em between} $\sigma'$ and $\sigma''$ such that $\sigma' \neq \sigma''' \neq \sigma''$ and $\dim(\sigma' \cap \sigma''')$ is as large as possible; more precisely, we can choose this $\sigma'''$ to be the convex hull of $(i)$ a facet of $\sigma'$ that contains $\sigma' \cap \sigma''$ and $(ii)$ the vertex of $\sigma''$ that is not on $\sigma'$ but is {\em closest} to $\sigma'$. (Note that this $\sigma'''$ is not necessarily a face of $\S_n$.) Then, if we obtain the desired result for $\sigma'$ and $\sigma'''$ it will also follow for $\sigma'$ and $\sigma''$.}, i.e., $\dim(\sigma' \cap \sigma'') = n-2$.

Assume first that $\sigma' \cap \sigma''$ has binary type $11...100$. Then, if we look at how $\sigma'$ and $\sigma''$ intersect a specific $\fiber$ of $\nbhd(\sigma' \cap \sigma'')$, i.e., a set of the form 
$$
P_0 + \fiber(\sigma' \cap \sigma'')
$$
for some $P_0$ that belongs to the set 
$$
(\sigma' \cap \sigma'') \cap (\nbhd(\pi_n(\sigma' \cap \sigma'')) \times \RR),
$$ 
we notice that it looks like in the figure below:

\bigskip

\bigskip

\begin{figure}[h]
\begin{center}
\includegraphics[width=10cm]{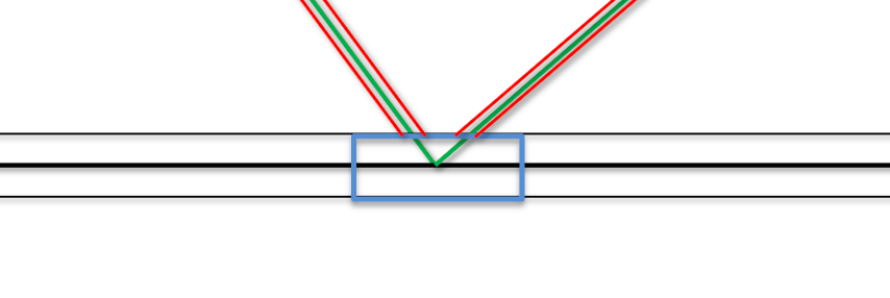}
\end{center}
\caption{\small \it Example of a specific intersection between $\sigma'$ and $\sigma''$ within a fiber of $\nbhd(\sigma' \cap \sigma'')$.}
\end{figure}

\bigskip

\noindent
Here one of the green line segments shows the intersection between  $\sigma'$ and the affine span of $P_0 + \fiber(\sigma' \cap \sigma'')$, and the other green line segment does the same for $\sigma''$. 
Note also that the distance between the two green rays restricted to the outside of this specific fiber is a constant multiple of $\eps(11...101)$, and therefore the distance between  $\sigma' \cap (\pi_n(\sigma') \times \RR)$ and $\sigma'' \cap (\pi_n(\sigma'') \times \RR)$  is a constant multiple of $\eps(11...101)$. 
Note also that 11...101 $<$ 11...110, and therefore we know that  $\eps(11...110)$ is of smaller order of magnitude than $\eps(11...101)$, which implies that $\nbhd(\sigma')$ and $\nbhd(\sigma'')$ are disjoint. 

The case where $\sigma' \cap \sigma''$ has binary type 11...1011...10 is completely similar, because again we have  11...1011...01 $<$  11...1011...10.
\end{proof}

\bigskip 

\noindent
We would like to emphasize a consequence of a specific step in the proof above, since it will be useful later:

\begin{rem}
\label{rem:large_enough_distance_between_adjacent_faces}
In the proof of the Theorem~\ref{thm:ZSH_preblueprints_exist} we have also obtained an 
explicit \ul{positive} lower bound for the distance between $\nbhd(\sigma')$ and $\nbhd(\sigma'')$, for any $\sigma' \neq \sigma'' \in \S_n$ such that $\pi_n(\sigma') = \pi_n(\sigma'')$ and $\sigma'$ and $\sigma''$ have (the same) binary type $\a$, which is of the form $\a_1\a_2...\a_{n-1}0$. Indeed, it follows that, given an $x_1$-$x_n$-subdivison $\S_n$ and any number $C>0$, 
%there exists some $\eps_0 > 0$ with the property that 
if  $^p\!\B_n$ is a ZSH pre-blueprint of $\S_n$ such that the number $\eps(^p\!\B_n)$ in (\ref{max_eps}) is small enough,
% than $\eps_0$ 
then we have 
$$
\dist( \nbhd(\sigma'),   \nbhd(\sigma'') ) > C \eps(\a).
%\max\left(  \eps(\a(\sigma')), \eps(\a(\sigma''))  \right),
$$
\end{rem}

\bigskip

\noindent
Moreover, we also have:

\begin{lem}
\label{lem:face_is_contained_in_nbhds_of_subfaces}

If \, $^p\!\B_n$ is a ZSH pre-blueprint of $\S_n$ such that the number $\eps(^p\!\B_n)$ in (\ref{max_eps}) is small enough,
then the following holds. For any face $\sigma \in \S_n$ there exists some $\rho>0$ such that, for any point $P \in \sigma$ we have 
\begin{equation}
\label{eqn:neighborhood_inside_union_of_nbhds}
B(P, \rho) \subset \bigcup_{\sigma_0 \subset \sigma} \nbhd(\sigma_0),
\end{equation}
where the union above is over the faces $\sigma_0 \in \S_n$ that are contained in $\sigma$; in other words, these are either $\sigma$ itself, or faces that lie on $\partial \sigma$.
\end{lem}

\begin{proof}
A proof of Lemma~\ref{lem:face_is_contained_in_nbhds_of_subfaces}  follows by induction. The claim above can be checked directly in dimensions 1 or 2, and if we assume it to be true in dimension $n-1$, then we notice that in dimension $n$ we only have to check it for faces $\sigma$ that satisfy $\dim{\pi_n(\sigma)} = n-1$, since for all other $\sigma \in \S_n$ the desired result follows from the induction hypothesis (this is similar to our use of induction hypothesis in the proof of Theorem~\ref{thm:ZSH_preblueprints_exist}).  

Now choose $\sigma$ that satisfies $\dim{\pi_n(\sigma)} = n-1$ and consider first the case where $\sigma$ has binary type of the form $11...10$. If we choose some point $P \in \sigma$, then the point  $\pi_n(P)$ belongs to $\pi_n(\sigma)$, and by the induction hypothesis, there exists a ball $B(\pi_n(P), \eta_0)$ such that we have
$$
B(\pi_n(P), \rho_0) \subset \bigcup_{\sigma_0 \subset \pi_n(\sigma)} \nbhd(\sigma_0),
$$
where $B(\pi_n(P), \rho_0)$ is a ball within $\RR^{n-1}$, and the sets $\nbhd(\sigma_0)$ are from the $(n-1)$-dimensional pre-blueprint of the $x_1$-$x_{n-1}$-subdivision  $\pi_n(\S_n)$.  But then, by the construction process of $^p\!\B_n$, it follows that there exists some $\eta \in (0, \eta_0)$ such that~(\ref{eqn:neighborhood_inside_union_of_nbhds}) holds\footnote{We can guarantee that we can find $\eta \in (0, \eta_0)$ that does {\em not} depend on the choice of point $P$, because we know that the smallest length in the $x_n$-direction of a set of the form $\nbhd(\sigma_0)$ for $\sigma_0 \subset \sigma$ is at least of size $\eps(\a(\sigma))$, i.e, $\eps(11...10)$. Note that we are \ul{\em not} saying that we can choose the radius $\eta = \frac{\eps(11...10)}{2}$, because the boundaries in the $x_n$-direction of the sets $\nbhd(\sigma_0)$ are {\em not} necessarily parallel to the $x_1$-$x_{n-1}$-coordinate hyperplane, and some of them can be nearly orthogonal on this coordinate hyperplane; but, since there is only a finite number of such sets, we can choose a single $\eta = \O(\eps(11...10))$ for all $P$.}. 

Finally, let us consider the case  where $\sigma$ has binary type of the form $11...11$. In this case we can reason in a way that is similar to the previous case: we find a ball $B(\pi_n(P), \rho_0)$ that is contained inside $\bigcup_{\sigma_0 \subset \pi_n(\sigma)} \nbhd(\sigma_0)$, for the $(n-1)$-dimensional pre-blueprint of the $x_1$-$x_{n-1}$-subdivision  $\pi_n(\S_n)$, and we look at the steps of the construction process of $^p\!\B_n$, to conclude that there exists some $\eta \in (0, \eta_0)$ such that~(\ref{eqn:neighborhood_inside_union_of_nbhds}) holds\footnote{We can again guarantee that we can find $\eta \in (0, \eta_0)$ that does {\em not} depend on the choice of point $P$, because in this case we still know that the smallest length in the $x_n$-direction of a set of the form $\nbhd(\sigma_0)$ for $\sigma_0 \subset \sigma$ is at least of size $\eps(11...10)$. Note that, if $\sigma_0$ has binary type of the form $\a_1...\a_{n-1}1$, then we have to rely on~Remark \ref{rem:large_enough_distance_between_adjacent_faces} to obtain this conclusion. }.
\end{proof}

\bigskip

\noindent
An immediate consequence of Lemma~\ref{lem:face_is_contained_in_nbhds_of_subfaces} is the following:

\begin{rem}
\label{rem:neighborhood_of_face_inside_union_of_nbhds}
If  \,  $^p\!\B_n$ is a ZSH pre-blueprint of $\S_n$ such that the number $\eps(^p\!\B_n)$ in (\ref{max_eps}) is small enough,
then for any face $\sigma \in \S_n$ there exists some $\rho>0$ such that  we have 
\begin{equation}
\label{eqn:neighborhood_of_face_inside_union_of_nbhds}
B(\sigma, \rho) \subset \bigcup_{\sigma_0 \subset \sigma} \nbhd(\sigma_0),
\end{equation}
where the union above is over the faces $\sigma_0 \in \S_n$ that are contained in $\sigma$, and the set $B(\sigma, \rho)$ denotes the ball of radius $\rho$ around the face $\sigma$.
In particular, note that we have 
\begin{equation}
\label{eqn:face_inside_union_of_nbhds}
\sigma \subset \bigcup_{\sigma_0 \subset \sigma} \nbhd(\sigma_0).
\end{equation}
\end{rem}

\bigskip

\bigskip

\bigskip

\bigskip

\newpage

\subsection{From  ZSH pre-blueprints  to  faithful ZSH blueprints 
%of an $x_1$-$x_n$-subdivision $\S_n$
}
\label{subsection:from_ZSH_pre-blueprints_to_faithful_ZSH_blueprints}

In the previous section we have discussed the construction of {\em ZSH pre-blueprints} of an $x_1$-$x_n$-subdivision $\S_n$ for any $n$. On the other hand, our ultimate goal is actually construction of ZSH,  which requires  a \ul{\em ZSH blueprint} of $\S_n$. Here we discuss how we can use ZSH pre-blueprints in order to construct associated {\em \ul{faithful} ZSH blueprints}, and why these are useful.

A {\em faithful ZSH blueprint} $\B_n^\mathit{ff}$  of $\S_n$ is a special kind of ZSH blueprint, which is strongly related to a given {\em pre-blueprint} $^p\!\B_n$ of $\S_n$. In particular,  for each $\sigma \in \S_n$, if we denote by $f$-$\nbhd(\sigma)$ the union of all the tiles $T_n$ of $\B_n^\mathit{ff}$ such that the basepoint of $T_n$ belongs to $\sigma$, then $f$-$\nbhd(\sigma)$ can approximate the set  $\nbhd(\sigma)$ of $^p\!\B_n$ arbitrarily well, in a sense that will be clarified below (see Figure~\ref{Fig_101} for an illustration).

More importantly, an {\em $n$-dimensional   faithful ZSH blueprint} can be used for the construction of an {\em $(n+1)$-dimensional  ZSH blueprint}. Moreover, given enough additional information\footnote{We will see that this additional information will come from properties of an \ul{\em $(n+2)$-dimensional}  {\em ZSH pre-blueprint}. This is not a problem, because we already know how to construct {\em ZSH pre-blueprints} in any dimension.} (in a sense that will be discussed below),  an {\em $n$-dimensional  faithful ZSH blueprint}  can be used for the construction of an {\em $(n+1)$-dimensional  \ul{faithful}  ZSH blueprint}. Therefore, we will see that the use of {\em faithful ZSH blueprints} will allow us to construct ZSH blueprints in any dimension.

\begin{figure}[htbp]
\begin{center}
  \hspace*{-0.6cm}
    \begin{tabular}{c}
$(a)$ \includegraphics[width=6.2cm]    
{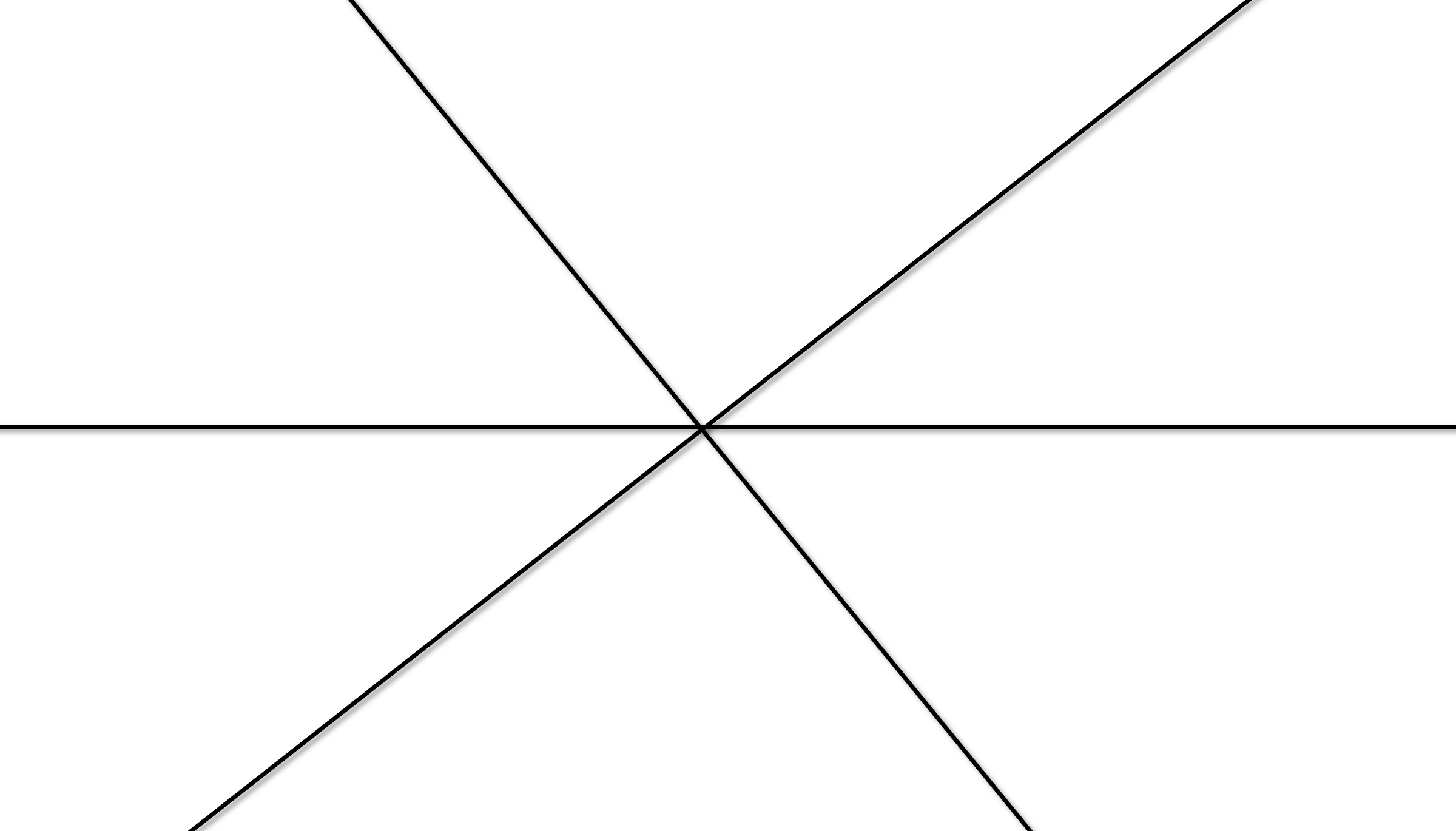} \ \ \ \ 
$(b)$ \includegraphics[width=6.2cm]
{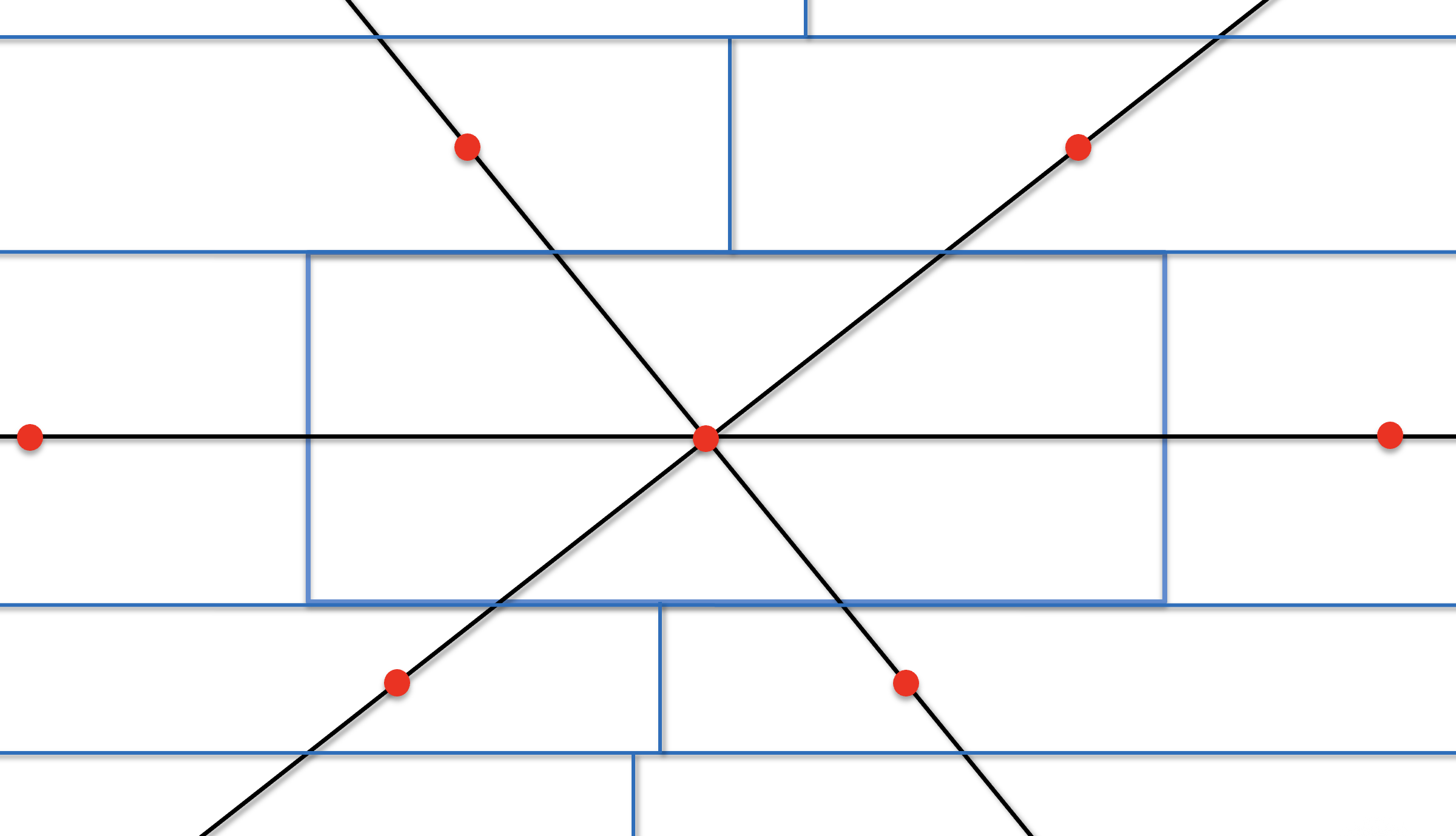} \\ 
   \hspace*{-1.0cm}
$(c)$ \includegraphics[width=0.133cm]{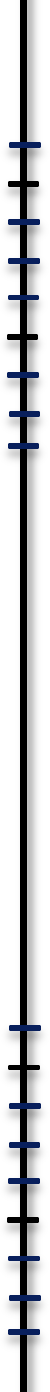} \  
$(d)$ \includegraphics[width=6.2cm]{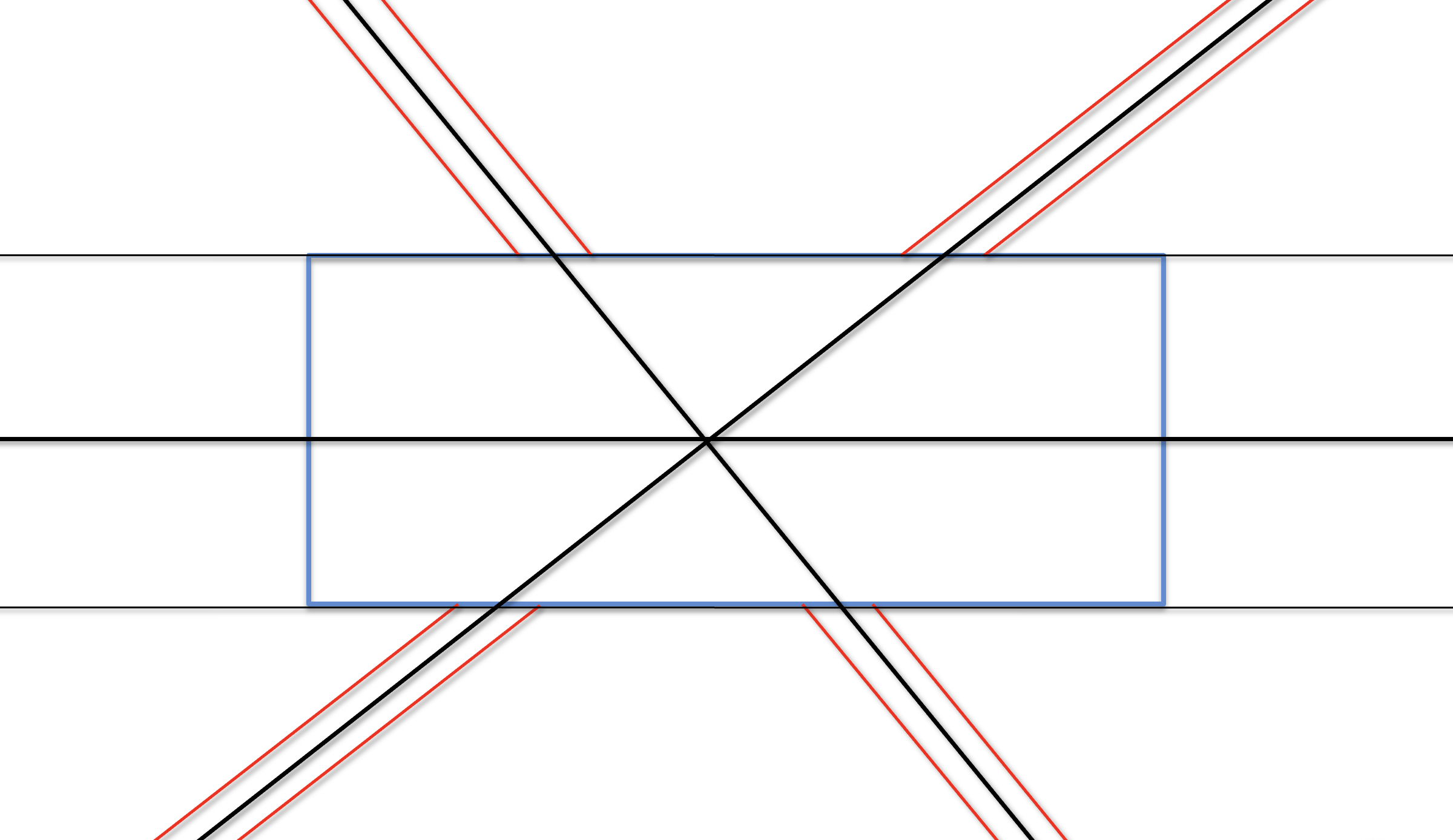} \ \ \ \ 
$(e)$ \includegraphics[width=6.2cm]{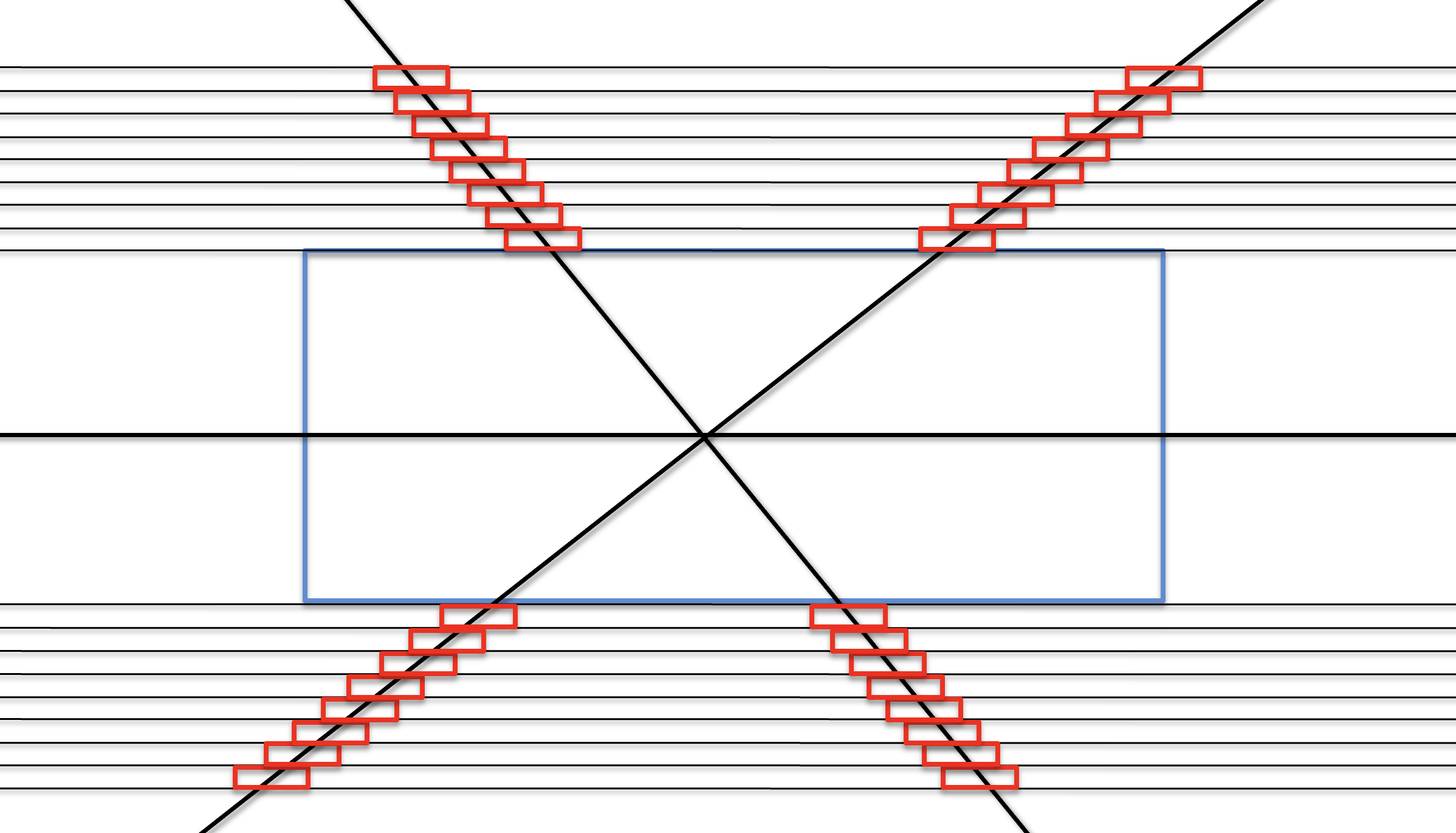} \\
$(f)$ \includegraphics[width=6.2cm]{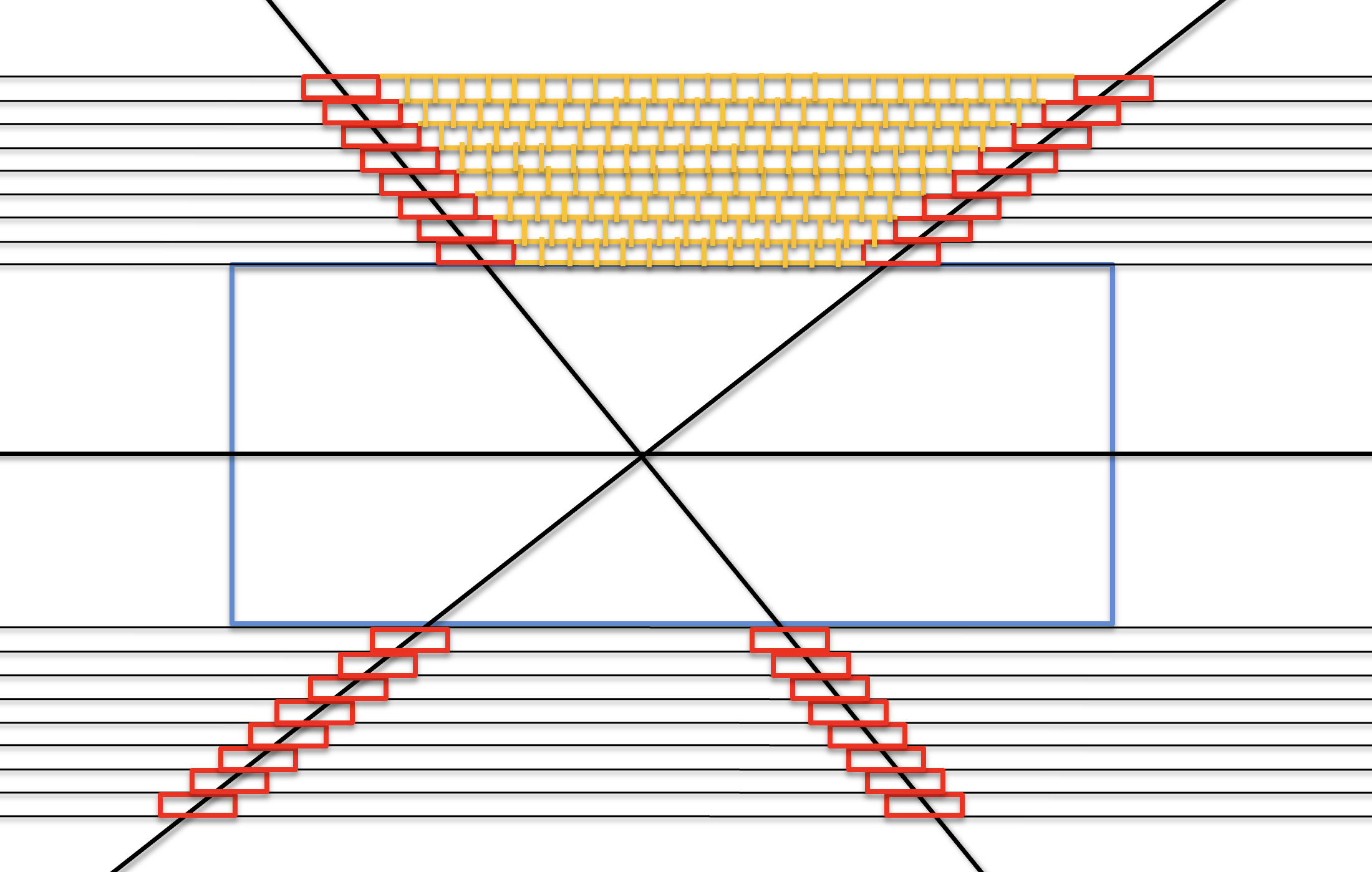} \ \ \ \ 
$(g)$ \includegraphics[width=6.2cm]{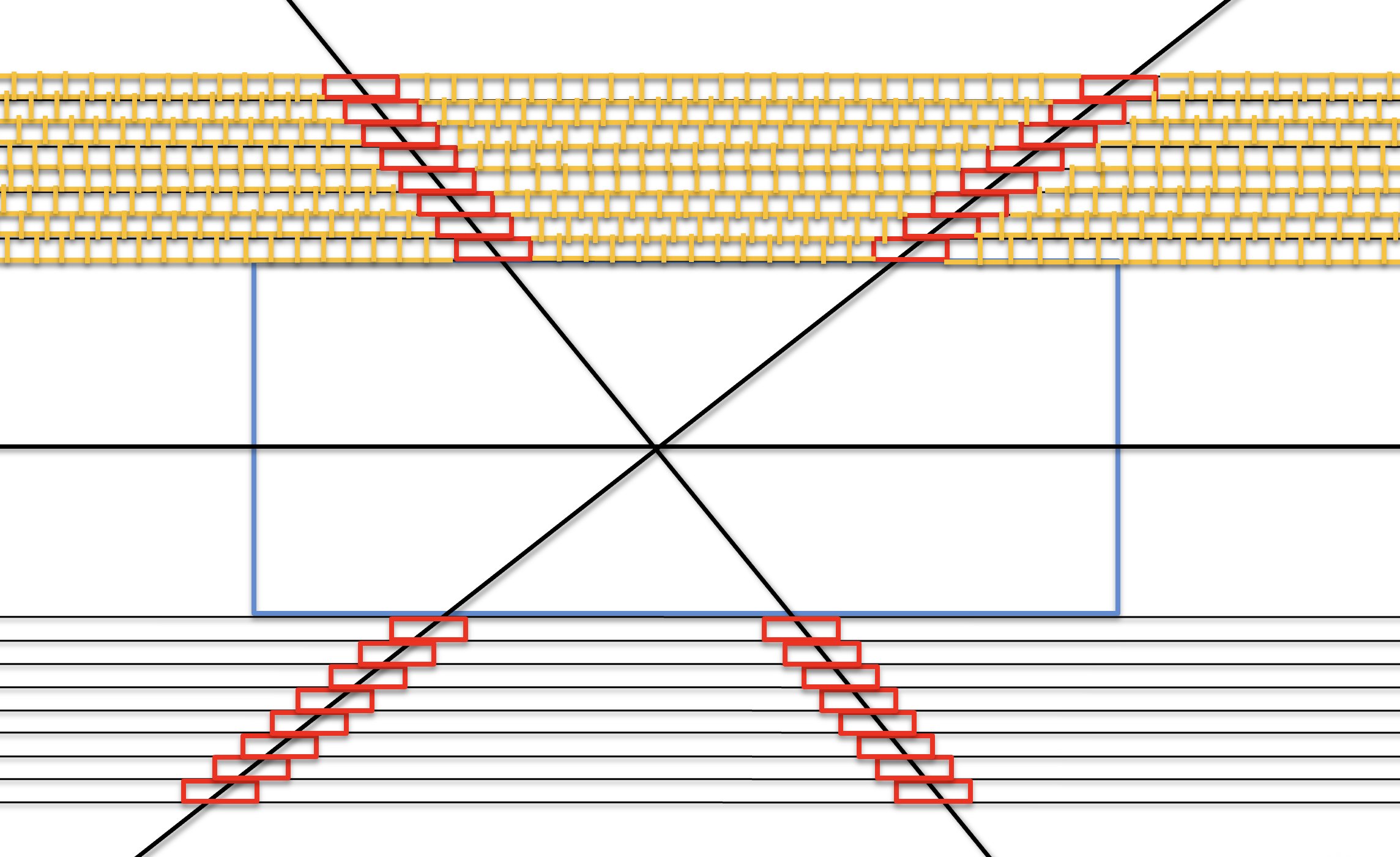}
   \end{tabular}
\end{center}
\caption{
\label{Fig_101}
\scriptsize 
In $(a)$ we {\em zoom in} within a two-dimensional $x_1$-$x_2$-subdivision, like the one shown in Figure~\ref{Fig_1} (note the similarity the the figure in $(a)$ to the region {\em near the middle-left vertex} in Figure~\ref{Fig_1}$(a)$). 
In $(b)$ we see what a ZSH blueprint (which is \ul{\em not a faithful ZSH blueprint}) could look like in this region. 
In $(c)$ we show that, in preparation for the construction of a two-dimensional {\em faithful} ZSH blueprint, we have transformed the one-dimensional ZSH pre-blueprint $^p\!\B_1 = \pi_2(^p\!\B_2)$ into a {\em one-dimensional faithful ZSH blueprint} $\B_1^\mathit{ff}$, by subdividing all its tiles of size $O(1)$ into much smaller tiles. (We will see that a convenient size of these very small 1D tiles is one that is $\ll \eps(10)$.) 
In $(d)$ we zoom in within a two-dimensional ZSH pre-blueprint $^p\!\B_2$; note the relationship between $(d)$ and $(c)$.
In $(e)$ we illustrate how it is possible to ``approximate" $\nbhd(\sigma)$ with a set of very small (red) tiles for any $\sigma$ with $\alpha(\sigma) = 10$, which will later become part of the 2D {\em \ul{faithful} ZSH blueprint} that we are constructing. {\em Note that we choose the very small red tiles such that their projections via $\pi_2$ are exactly the very small 1D tiles of $^p\!\B_1$ that are shown in $(c)$.} The small red tiles and the long black tiles in $(e)$, together with the big blue tiles, form the ZSH blueprint $\B_2$, and we will obtain $\B_2^\mathit{ff}$ by further subdividing the tiles of $\B_2$ (while we keep the same sets ``$f$-$\nbhd(\sigma)$" for both $\B_2$ and $\B_2^\mathit{ff}$).
In $(f)$ we see a specific step in the transition from $\B_2$ to $\B_2^\mathit{ff}$ where we replace some of the long black tiles with the very small orange tiles, within a two-dimensional face $\sigma$ of binary type 11, which sits {\em between} two  edges $\sigma'$ and $\sigma''$ of binary type 10, and satisfy $\pi_2(\sigma) = \pi_2(\sigma') = \pi_2(\sigma'')$. Note that the dimension of each orange tile in the horizontal direction approximates the number $\et(11)$. 
In $(g)$ we see how the same type of  transition from $\B_2$ to $\B_2^\mathit{ff}$ can be done everywhere within the set $\nbhd(\pi_2(\sigma)) \times \RR$.
}
\end{figure}

\subsubsection{The one-dimensional case}

We define a one-dimensional \ul{\em faithful ZSH blueprint}  of a one-dimensional $x_1$-$x_1$-subdivision $\S_1$ as being obtained from a one-dimensional ZSH pre-blueprint\footnote{Note that a one-dimensional {\em ZSH pre-blueprint} is  actually also a one-dimensional {\em ZSH blueprint.}} $^p\!\B_1$
of $\S_1$ as follows. We assume that  we already are given  an $x_1$-$x_2$-subdivision $\S_2$ such that $\S_1 = \pi_2(\S_2)$, and a two-dimensional ZSH pre-blueprint $^p\!\B_2$ of $\S_2$ such that $^p\!\B_1 = \pi_2(^p\!\B_2)$\footnote{It is important to note that, by definition, we {\em cannot} just look at a one-dimensional ZSH blueprint of $\S_1$ {\em in isolation} and say whether it is ``faithful" or not. The property of being ``faithful" {\em depends on relationships between this one-dimensional ZSH blueprint of $\S_1$ and some two-dimensional ZSH pre-blueprint of $\S_2$}.}. 
%choosing some $\eps^*(1) \ll 1$ and 
Let us fix a number denoted $\et(1)$, which satisfies\footnote{Note that $\et(1)$ depends on $\eps(10)$, and the information about $\eps(10)$ comes from $^p\!\B_2$. Also, here is the \ul{\em precise} meaning of the informal-looking condition ``$\et(1) \ll \eps(10)$": we mean that, once we are given the $\S_2$ and $^p\!\B_2$, then there exists some $\et_0 \in (0,\eps(10))$ which depends only on $\S_2$ and $^p\!\B_2$, such that if we have $\et(1) < \et_0$, then we say that we have $\et(1) \ll \eps(10)$.}
$$
\et(1) \ll \eps(10).
$$
For each $\sigma \in \S_1$ such that $\eps(\sigma) = 1$ we divide the interval $\nbhd(\sigma)$ into very small intervals of equal size, and this size is $\le \et(1)$, and as close to $\et(1)$ as possible\footnote{Since we have  $\et(1) \ll 1$, and the length of any interval $\nbhd(\sigma)$ considered above  is $O(1)$, it follows that the size of these small intervals is guaranteed to be very close to $\et(1)$. Note that we {\em cannot} guarantee their size to be {\em exactly} $\et(1)$, because we don't know that this number is an integer factor of the length of $\nbhd(\sigma)$.}. The {\em basepoints} of all these new intervals are defined to be their centerpoints.

\subsubsection{The two-dimensional case}

Some of the main ideas on how  to construct  a {\em faithful ZSH blueprint} in the two-dimensional case are illustrated in Figure~\ref{Fig_101}. 

Note that the sets $\nbhd(\sigma)$ {\em are already rectangular} if $\alpha(\sigma) = 00$ or $\alpha(\sigma) = 10$, see Figure~\ref{Fig_1}$(b)$. 

Our goal is to ``modify" the ZSH pre-blueprint in such a way that $\nbhd(\sigma)$ {\em is approximated by a set of rectangles} for the cases where $\alpha(\sigma) = 10$ or $\alpha(\sigma) = 11$. 

We can see in Figure~\ref{Fig_101}$(b)$ that it is easy to place a ``tower of offset rectangles" such that they  approximate the region of $\nbhd(\sigma)$ in the case where $\alpha(\sigma) = 10$. The most convenient way to do this is {\em to use rectangles that  \ul{have the same type of aspect ratio} as the blue rectangles that we have already used in the ZSH pre-blueprint}, when we built $\nbhd(\sigma)$ for zero-dimensional $\sigma$, i.e., with $\alpha(\sigma) = 00$, see the red rectangles in Figure~\ref{Fig_1}$(b)$.

Here is why using red rectangles with this particular type aspect ratio (i.e., {\em ``thin and long" rectangles}) is useful: recall that we constructed the blue rectangles in Figure~\ref{Fig_1}$(b)$
such that their {\em width} is much larger that their {\em height}\footnote{Because the width is given by $\eps(00)$, the height is given by $\eps(01)$, and  $\eps(00) \gg \eps(01)$.}. Then this type of ``thin and long" rectangle \ul{\em spans across}\footnote{The fact that each such red rectangle {\em spans across} the edge $\sigma$ (by which we mean that the edge $\sigma$  intersects the horizontal edges of the rectangle, but not its vertical edges) is very important, because it will allow us to prove some key properties of the faithful ZSH blueprint $\B_2^\mathit{ff}$ that we are constructing. 
%: that fact the if a tile $T \in \B_2^\mathit{ff}$ intersects a face $\sigma$ of $\S_2$, then the basepoint of $T$ lies on the face $\sigma$.
} the edge $\sigma$ if we place its  centerpoint\footnote{Recall that we will also need to define a \ul{\em basepoint} for each red rectangle, and we need the basepoint to lie on the edge $\sigma$. The simplest choice is to make the centerpoint be the basepoint.} \ul{\em on} the edge $\sigma$.

\bigskip

We would like to think about the construction of a two-dimensional faithful ZSH blueprint $\B_2^\mathit{ff}$ for $\S_2$ as {\em part of an inductive process}, i.e., a preparation for the later construction of an $n$-dimensional faithful ZSH blueprint $\B_n^\mathit{ff}$ for $\S_n$.

\bigskip 

\noindent
{\bf Step 1.} 
Before we construct $\B_2^\mathit{ff}$ let us discuss how we can construct just an ordinary (not ``faithful") ZSH blueprint $\B_2$. Then, in Step 2, we will further refine this construction to obtain $\B_2^\mathit{ff}$. 
%Therefore, we want the construction of $\B_2^\mathit{ff}$ to
 
In order to construct $\B_2$ we will rely on the following three pieces of information: $(i)$ a one-dimensional ZSH pre-blueprint $^p\!\B_1$ of $\S_1 = \pi_2(\S_2)$, $(ii)$ a two-dimensional ZSH pre-blueprint $^p\!\B_2$ of $\S_2$ such that $^p\!\B_1 = \pi_2(^p\!\B_2)$, and $(iii)$ a one-dimensional {\em faithful} ZSH blueprint $\B_2^\mathit{ff}$ of $\S_1$\footnote{As discussed in the previous subsection.}. 

%Our first goal is to use these three pieces of information to construct a {\em ZSH blueprint} for  $\S_2$. After that we will discuss how we can further refine our construction to obtain a {\em \ul{faithful} ZSH blueprint} for $\S_2$. 

\smallskip

%Now that we have defined {\em one-dimensional faithful ZSH blueprint} $\B_1^\mathit{ff}$, we want to construct a {\em two-dimensional faithful ZSH blueprint} $\B_2^\mathit{ff}$ for a two-dimensional $x_1$-$x_2$-subdivision $\S_2$. We are assuming that we have already constructed  a two-dimensional {\em ZSH pre-blueprint} $^p\!\B_2$ of $\S_2$ (like we discussed in the previous section). 
%
%First, we look at a one-dimensional ZSH {\em pre-blueprint} associated to the projected $x_1$-$x_1$-subdivision $\pi_2(\S_2)$, and denoted $^p\!\B(\pi_2(\S_2))$. We use this pre-blueprint to construct a one-dimensional {\em faithful} ZSH blueprint from it, denoted $\B^\mathit{ff}(\pi_2(\S_2))$. 

\smallskip

For each {\em vertex} $\sigma_0 \in \S_2$ (i.e., zero-dimensional face\footnote{Equivalently, $\sigma_0$ is a face of $\S_2$ such that $\a(\sigma_0) = 00$.} of $\S_2$) we define the tile of $\sigma_0$ to be  precisely $\nbhd(\sigma_0)$, and we define its basepoint to be the vertex $\sigma_0$.

\smallskip

For each {\em edge} $\sigma_1 \in \S_2$ of binary type  $\a(\sigma_1) = 01$ (i.e., {\em horizontal} edge in Figure~\ref{Fig_101}$(a)$) similarly we define the tile of $\sigma_1$ to be  precisely $\nbhd(\sigma_1)$, and we define its basepoint to be its centerpoint.
%we subdivide the rectangle $\nbhd(\sigma_1)$ into smaller tiles that have approximately the same length as the tile $\nbhd(\sigma_0)$, where $\sigma_0$ is a  vertex  of $\S_2$. 

\smallskip

For any {\em edge} $\sigma \in \S_2$ of binary type  $\a(\sigma) = 10$, we define rectangular tiles along $\sigma$ as illustrated by the red rectangles in Figure~\ref{Fig_101}$(e)$, i.e., such that each such tile $T_2$ projects via $\pi_2$ to a tile $T_1$ of $\pi_2(\S_2)$ that covers a part of $\nbhd(\pi_2(\sigma))$. 
Also, we choose the {\em basepoint} of each such tile $T_2$ to be its centerpoint, and to lie on $\sigma$. 
Note that for each such tile $T_2$ its height is determined by the size of $T_1$, so we only have to specify its width; we choose its width to be $\eps(10)$\footnote{Note that $\eps(10)$ is exactly the width of $\nbhd(\sigma)$, so in this sense the set of tiles $T_2$ built along $\sigma$ as described above form a very good approximation for $\nbhd(\sigma)$. Note also that, since we ensured that the size of $T_1$ is approximately equal to $\et(1)$, and  $\et(1) \ll \eps(10)$, it follows that we obtain there desired  \ul{\em aspect ratio} of each tile $T_2$, i.e., each tile $T_2$ has ``height $\ll$ width". (That is why we made the \ul{\em specific restriction}  $\et(1) \ll \eps(10)$ for the size of $T_1$ in the previous subsection.)}.

\smallskip

Finally, for any {\em face} $\sigma_2 \in \S_2$ of binary type  $\a(\sigma) = 11$, note that $\sigma_2$ has two\footnote{If $\sigma_2$ is unbounded then there is just one such edge, say $\sigma'$, and $\nbhd(\sigma)$ is delimited by $\nbhd(\sigma')$.} edges $\sigma'$ and $\sigma''$ such that
$$
\pi_2(\sigma_2) = \pi_2(\sigma') = \pi_2(\sigma''),
$$
and the set $\nbhd(\sigma_2)$ is delimited within the set 
$$
\nbhd(\pi_2(\sigma_2)) \times \RR
$$ 
by $\nbhd(\sigma')$ and $\nbhd(\sigma'')$. 

Moreover, the edges $\sigma'$ and $\sigma''$ have binary type $10$,  so we have already constructed the tiles that are associated to them (i.e., the tiles that have basepoints on them). For each tile $T_1$ of $\pi_2(\S_2)$ that has basepoint on $\pi_2(\sigma)$ there exist two tiles $T'_2$ and $T''_2$ that have basepoints on  $\sigma'$ and $\sigma''$ (as suggested in Figure~\ref{Fig_101}$(e)$) and we use them as delimiters within the set 
$$
T_1  \times \RR,
$$ 
in order to  define the tiles associated to $\sigma_2$ to be the rectangles delimited as we discussed above (i.e.,  the thin black rectangles in Figure~\ref{Fig_101}$(e)$).

%fill in the resulting rectangular region with equal tiles that have the same height as $T_1$ and have width $\le \eps(01)$, and as close to  $\eps(01)$ as possible\footnote{Note that, according to Remark~\ref{rem:large_enough_distance_between_adjacent_faces}, the distance between $T'_2$ and $T''_2$ is much larger than $\eps(01)$, and therefore the tiles that fill in this region will have width that is {\em very close} to  $\eps(01)$.}.

\bigskip 

\noindent
{\bf Step 2.} Let us now discuss the construction of the two-dimensional {\em faithful} ZSH blueprint $\B_2^\mathit{ff}$. This construction will be a refinement of the construction of $\B_2$, and will rely on additional information that comes from a three-dimensional $x_1$-$x_3$-subdivision $\S_3$, and a ZSH pre-blueprint $^p\!\B_3$ of $\S_3$, such that we have $\pi_3(\S_3) = \S_2$ and $\pi_3(^p\!\B_3) = \ ^p\!\B_2$\footnote{Recall that, in the earlier construction of  the one-dimensional {\em faithful} ZSH blueprint $\B_1^\mathit{ff}$, we also needed such additional information, and it came from a two-dimensional subdivision $\S_2$ and corresponding pre-blueprint $^p\!\B_2$, such that $\pi_2(\S_2) = \S_1$ and $\pi_2(^p\!\B_2) = \ ^p\!\B_1$.}. 

Recall that when we constructed the {\em one}-dimensional {\em faithful} ZSH blueprint $\B_1^\mathit{ff}$ we made the decision to further subdivide $\nbhd(\sigma_1)$ for any edge $\sigma_1$ of $\S_1$; for this we needed to decide {\em how small} will the new tiles be within $\nbhd(\sigma_1)$, and we decided to impose the condition that the new tiles will have length approximately $\et(1)$, such that 
$$
\et(1) \ll \eps(10).
$$
Recall also  {\em the reason} behind this decision was that we wanted to make sure that if we were to use one-dimensional tiles of (approximate) size $\et(1)$ in a $\pi_2$-compatible two-dimensional construction, then the resulting two-dimensional tiles of (approximate) size $\et(1) \times \eps(10)$ can be guaranteed to ``{\em span across}"\footnote{By ``{\em span across}" we meant that, for any such $\sigma_1$, {\em we did not  want  the two vertical sides of these rectangular tiles to intersect the edge $\sigma_1$}, even though the two horizontal sides {\em do} intersect $\sigma_1$, as discussed in the previous subsection.} the edges $\sigma_1$ of $\S_2$ of binary type 10, see Figure~\ref{Fig_101}$(e)$.

\medskip

In order to construct the faithful ZSH blueprint $\B_2^\mathit{ff}$ we will need to make similar additional subdivisions, and impose similar additional conditions. We are in a two-dimensional case, and therefore the additional subdivisions and conditions will be based on information from the  {\em three-dimensional} $x_1$-$x_3$-subdivision $\S_3$, and  ZSH pre-blueprint $^p\!\B_3$ of $\S_3$.

In this case there will be not just one, but {\em four} binary types\footnote{In the previous case we only had to worry about $\sigma$ of type 10, because this was the only type of face $\sigma$ of  $\S_2$ that projected via $\pi_2$ onto a face $\sigma'$ of $\S_1$ such that $\dim \sigma = \dim \sigma'$. (Equal dimensions means that we have to worry about {\em spanning across}.) In the case of $\S_3$ there are {\em four} types of faces $\sigma$ of  $\S_3$ that project via $\pi_3$ onto a face $\sigma'$ of $\S_2$ such that $\dim \sigma = \dim \sigma'$. These are the faces of type 110, 100, 010, and 000.} of faces $\sigma \in \S_3$ 
%regions of the form $\nbhd(\sigma)$ in the ZSH pre-blueprint $^p\!\B_3$ of $\S_3$ where 
for which we have to make sure that corresponding three-dimensional tiles can {\em span across}\footnote{These three-dimensional tiles need to project via $\pi_3$ on the two-dimensional tiles that we are constructing now. Therefore, the dimensions of two of the three sides of any such three-dimensional tile come from the dimensions of the two sides of the  two-dimensional tiles of the faithful ZSH blueprint $\B_2^\mathit{ff}$, {\em which is now under construction}.} $\sigma$.

We will write down all these conditions {\em systematically}, 
by imposing for each such binary type that we have 
$$
length \gg width \gg height
$$
for each corresponding three-dimensional tile\footnote{The rigorous meaning of ``$\gg$" here is similar to the previous case: we mean that there exists some $\et_0 < 1$ which depends only on $\S_3$ and $^p\!\B_3$, such that we have 
$
\et_0^2 \ length > \et_0 \ width > height.
$
    This makes sense if we also note that we will use these order-of-magnitude conditions \ul{only} in statements of the form ``if $\xi_1 \gg \xi_2$ then...", which is a short and convenient way of expressing the rigorous statement ``there exists some $\et_0 < 1$ such that, if $\et_0 \xi_1 > \xi_2$ then...". 
}.
    
\noindent
Then the conditions are the following:
\begin{eqnarray}
\label{additional_eps_conditions_in_2D}
& & \mathrm{For} \  \sigma \in \S_3 \  \mathrm{of \ type} \ 110:  \ \      
\eps(110)    \gg   \boldsymbol{\et(11)}   \gg   \boldsymbol{\et(1)} 
\\ \nonumber
& & \mathrm{For} \  \sigma \in \S_3 \  \mathrm{of \ type} \ 100:  \ \      
\eps(100)    \gg     \eps(10)                \     \gg   \boldsymbol{\et(1)} 
\\ \nonumber
& & \mathrm{For} \  \sigma \in \S_3 \  \mathrm{of \ type} \ 010:  \ \      
\eps(010)    \gg  \boldsymbol{\et(01)}   \gg             \eps(0)
\\ \nonumber
& & \mathrm{For} \  \sigma \in \S_3 \  \mathrm{of \ type} \ 000:  \ \      
\eps(000)    \gg     \eps(00)               \     \gg           \eps(0) 
\end{eqnarray}

\noindent
Note that we have emphasized in~(\ref{additional_eps_conditions_in_2D}) the \ul{\em new} tile sizes $\et(11)$, $\et(01)$, and $\et(1)$ using {\bf bold} fonts; all the other numbers  in~(\ref{additional_eps_conditions_in_2D}) are values of the form $\eps(\alpha)$ that come from the ZSH pre-blueprint $^p\!\B_3$ of $\S_3$. 

\medskip

\noindent
Let us make two observations about imposing the restrictions~(\ref{additional_eps_conditions_in_2D}). 

Our first observation is that  these restrictions {\em \ul{are already compatible} with the properties of $^p\!\B_3$}. Indeed, recall that for any $n$-dimensional pre-blueprint $^p\!\B_n$ we have 
$$
\eps(\a_1...\a_i) = \eps(\a_1...\a_i 11...1), 
$$ 
where the binary word $\a_1...\a_i 11...1$ has $n$ letters. Also recall  that if $\alpha$ and $\beta$ are two binary words of $n$ letters which, regarded as binary numbers, satisfy $\alpha < \beta$, then we have 
$$
\eps(\alpha) \gg \eps(\beta).
$$ 
Therefore, if the binary word $\a_1...\a_i$ is a {\em strict prefix} of a binary word of the form $\a_1'...\a_j'0$, then we have
$$
\eps(\a_1'...\a_j'0) \gg \eps(\a_1...\a_i),
$$ 
which are exactly the types of inequalities of the form $\eps(\alpha) \gg \eps(\beta)$ that appear in~(\ref{additional_eps_conditions_in_2D}). 

Our second observation is that  {\em we \ul{can} choose the new numbers $\et(\alpha)$} such that they satisfy the conditions~(\ref{additional_eps_conditions_in_2D}). 
%Note that what we need is 
While this two-dimensional case is quite simple, a general way to do it (which works in the $n$-dimensional case) is to 
proceed in two steps: in a {\em first step} define the new numbers  $\et(\a_1...\a_i 1)$ such that the $n$-dimensional analogue of~(\ref{additional_eps_conditions_in_2D}) is {\em almost} satisfied, but some ``$\gg$" are actually ``$=$"; and in a {\em second step} modify all the numbers where we have ``$=$" to obtain ``$\gg$", while  the previously attained ``$\gg$" still holds. This will be discussed in detail in the next subsection.

\bigskip

Let us also make one more remark: once we have chosen the numbers $\et(1)$, $\et(01)$, and $\et(11)$ such that they satisfy~(\ref{additional_eps_conditions_in_2D}), here is how we use them. First, as we discussed in the previous subsection, the number $\et(1)$ will be used in the construction of the faithful ZSH blueprint $\B_1^\mathit{ff}$ of $\S_1$. (Recall that we have used the number $\et(1)$ to specify (to an arbitrarily good approximation) the dimension along the $x_1$-axis of the tiles of $f$-$\nbhd(\sigma_1)$, for any one-dimensional face $\sigma_1 \in \S_1$.

Then the numbers $\et(01)$ and $\et(11)$ will be used in the construction of the faithful ZSH blueprint $\B_2^\mathit{ff}$ of $\S_2$, as follows. The number $\et(01)$ will be used to specify the (approximate) dimension along the $x_2$-axis of the tiles of $f$-$\nbhd(\sigma_1)$, for any one-dimensional face $\sigma_1 \in \S_2$ to binary type 01.
Similarly, the number $\et(11)$ will be used to specify the (approximate) dimension along the $x_2$-axis of the tiles of $f$-$\nbhd(\sigma_2)$, for any two-dimensional face $\sigma_2 \in \S_2$ to binary type 11. 

Of course, as we have mentioned earlier, we could have done much less work in order to obtain just some ZSH blueprint $\B_2$ of $\S_2$. {\em The \ul{key point} here is that once we have a \ul{faithful} ZSH blueprint $\B_2$ of $\S_2$, then, by definition, everything is ``prearranged"  to allow us to construct a  ZSH blueprint $\B_3$ of $\S_3$, and \ul{most importantly}, to construct a \ul{faithful}  ZSH blueprint $\B_3$ of $\S_3$, and therefore to keep this inductive process going and reach any desired dimension~$n$.}

Note also that as the dimension increases {\em we \ul{do} actually  impose additional conditions on the numbers $\et(\a_1...\a_i 1)$} which we have already used in previous construction steps. For example, when we wanted  $\B_1^\mathit{ff}$ to be faithful  we imposed the restriction
$$ 
\eps(10) \gg \et(1),
$$
on $\et(1)$, but when we {\em additionally} wanted  $\B_2^\mathit{ff}$ to be faithful  we  imposed {\em the additional restriction} 
$$ 
\eps(110) \gg \et(1)
$$
on $\et(1)$, which  {\em was \ul{not} needed} in order for $\B_1^\mathit{ff}$ to be faithful, but \ul{\em was  needed} in order for $\B_2^\mathit{ff}$ to be faithful. Therefore, we will see that there are {\em different degrees of faithfulness}, and in this case $\B_1^\mathit{ff}$ is not just {\em faithful}, but it is {\em 2-faithful}\footnote{In general, if a $k$-dimensional faithful ZSH blueprint $\B_k^\mathit{ff}$ satisfies the conditions required for the construction (using the approach that we are discussing here) of an $\ell$-dimensional faithful ZSH blueprint $\B_\ell^\mathit{ff}$ (with $k < \ell$), then we say that $\B_k^\mathit{ff}$ is {\em $\ell$-faithful}. In other words, in order for $\B_k^\mathit{ff}$ to be $\ell$-faithful we have to check conditions similar to~(\ref{additional_eps_conditions_in_2D}) up to dimension $\ell+1$. See the next subsection for more details.}.

\bigskip

\subsubsection{The $n$-dimensional case}

In this subsection we discuss the general $n$-dimensional case of our construction of a  {\em faithful ZSH blueprint} $\B_n^\mathit{ff}$ for an $x_1$-$x_n$-subdivision $\S_n$. Our setting will be the following: we will assume that, for the $x_1$-$x_i$-subdivisions $\S_1, ..., \S_i, ..., \S_n, \S_{n+1}$  we have already  constructed a ``chain" of ZHS \ul{\em pre-blueprints} 
$^p\!\B_1$, $^p\!\B_2$, ..., $^p\!\B_n$,  $^p\!\B_{n+1}$,  
and a ``chain" of {\em \ul{$n$-faithful} ZSH blueprints} 
$\B_1^\mathit{ff}$, $\B_2^\mathit{ff}$, ..., $\B_{n-1}^\mathit{ff}$  
, respectively, where we have 
$$
\S_{i-1} = \pi_i(\S_i), \ ^p\!\B_{i-1} = \pi_i(^p\!\B_i), \ \B_{j-1}^\mathit{ff} = \pi_j(\B_j^\mathit{ff}), 
$$
for all integers $i, j$ such that $2 \le i \le n+1$ and $2 \le j \le n-1$.

\bigskip

%\begin{rem}*
\noindent
{\bf Remark.}  As we have mentioned in previous subsections, the reason we need \ul{\em faithful}  ZSH blueprints is the following. 
First of all, \ul{\em not} every $(n-1)$-dimensional ZSH blueprint $\B_{n-1}$ of the $x_1$-$x_{n-1}$-subdivision $\S_{n-1} = \pi_n(\S_n)$ can serve as an induction hypothesis for the construction of a $n$-dimensional ZSH blueprint $\B_n$ for the $x_1$-$x_n$-subdivision $\S_n$ such that $\B_{n-1} = \pi_n(\B_n)$. 
{\em 
On the other hand, if $\B_{n-1}^\mathit{ff}$ is an  \ul{$n$-faithful} ZSH blueprint for $\S_{n-1}$, then not only can we build a ZSH blueprint $\B_n$ of $\S_n$ such that  $\pi_n(\B_n) = \B_{n-1}^\mathit{ff}$, but also\footnote{As we have seen in the previous subsection, for this second part we will need  to impose some additional requirements, related to the $(n+1)$-dimensional subdivision $\S_{n+1}$. We will discuss  these additional requirements in this subsection, and {\em we will make sure that they can  be satisfied.}} we can build a \ul{\em faithful}  ZSH blueprint  $\B_n^\mathit{ff}$ for $\S_n$ that is also based on $\B_{n-1}^\mathit{ff}$. Therefore, since $n$ was arbitrary, we can  extend this construction of {\em faithful}  ZSH blueprints  by induction. 
}
%\end{rem}*

\bigskip

Similar to the previous subsection, our construction will proceed in two steps. In Step 1 we will just construct a {\em  ZSH blueprint} $\B_n$ for $\S_n$ ({\em not} yet a {\em faithful} ZSH blueprint). 
Then, in Step 2, we will further refine our construction to obtain a  {\em faithful} ZSH blueprint $\B_n^\mathit{ff}$.

\bigskip 

Before we start Step 1, let us first discuss the numbers $\et(\a)$ that will be needed in our construction, and their relationships to the numbers $\eps(\beta)$ that are provided to us by the previously constructed  ``chain" of ZHS {\em pre-blueprints} $^p\!\B_1$, $^p\!\B_2$, ..., $^p\!\B_n$, $^p\!\B_{n+1}$.  

Recall the conditions~(\ref{additional_eps_conditions_in_2D}) we discussed in the previous subsection. It was not difficult to see that all these conditions can be enforced; on the other hand, in the $n$-dimensional case we get $2^n$ such rows of conditions, and it is less clear that they can all be satisfied at the same time. 

\noindent
For example, for $n=3$ here is what these $2^3$ conditions look like\footnote{Recall where these conditions are coming from: we want to construct $\B_3^\mathit{ff}$ which is a blueprint diagram for $\S_3$, and, moreover, can be used as ``induction hypothesis" for the construction of a blueprint diagram $\B_4$ for $\S_4$ such that $\pi_4(\B_4) = \B_3^\mathit{ff}$. 
Moreover, if $\sigma \in \S_4$ has binary type $\a_1\a_2\a_3 0$, then we want the set $f$-$\nbhd(\sigma)$ of $\B_4$ to approximate the set $\nbhd(\sigma)$ of $^p\!\B_4$; in particular, we have required  that all tiles of $\B_4$ that are contained in  $f$-$\nbhd(\sigma)$ must have length along the $x_4$-axis be exactly  $\eps(\a_1\a_2\a_3 0)$.  
Then, in order to be sure that we can enforce all these restrictions, we have decided to impose conditions on $\B_3^\mathit{ff}$ which will guarantee that any three-dimensional tile $\T_3 \in \B_3^\mathit{ff}$  allows us to construct four-dimensional tiles $\T_4 \in \S_4$ such that $\T_4$ {\em  \ul{``spans across"} any adjacent face $\sigma \in \S_4$ of binary type $\a_1\a_2\a_30$}. We are enforcing this guarantee by making very strong requirements on the sizes of  tiles of $\B_3^\mathit{ff}$. More precisely, if a tile  of $\B_3^\mathit{ff}$ is part of $f$-$\nbhd(\sigma)$ of some $\sigma \in \S_3$ of binary type $\a_1\a_2\a_3$, and we denote by ``$\Delta x_i$" the  length of $T_3$ along the $x_i$-axis, then we require that we have $\eps(\a_1\a_2\a_3 0) \gg \Delta x_3 \gg \Delta x_2 \gg \Delta x_1$. When we write down  these  requirements for all possible binary words $\a_1\a_2\a_3$ we obtain  exactly the conditions~(\ref{additional_eps_conditions_in_3D}).}:
%\hskip{-1cm}
%
\begin{align}
\label{additional_eps_conditions_in_3D}
& \mathrm{For} \  \sigma \in \S_4 \  \mathrm{of \ type} \ 1110:  \ \      
\eps(1110)    \gg   \boldsymbol{\et(111)}   \gg   \boldsymbol{\et(11)}      \gg   \boldsymbol{\et(1)} 
\\ \nonumber
& \mathrm{For} \  \sigma \in \S_4 \  \mathrm{of \ type} \ 1100:  \ \      
\eps(1100)    \gg     \eps(110)                \     \gg   \boldsymbol{\et(11)}    \gg   \boldsymbol{\et(1)} 
\\ \nonumber
& \mathrm{For} \  \sigma \in \S_4 \  \mathrm{of \ type} \ 1010:  \ \      
\eps(1010)    \gg  \boldsymbol{\et(101)}   \gg             \eps(10)               \gg   \boldsymbol{\et(1)} 
\\ \nonumber
& \mathrm{For} \  \sigma \in \S_4 \  \mathrm{of \ type} \ 1000:  \ \      
\eps(1000)    \gg     \eps(100)               \     \gg           \eps(10)               \gg   \boldsymbol{\et(1)} 
\\ \nonumber
& \mathrm{For} \  \sigma \in \S_4 \  \mathrm{of \ type} \ 1110:  \ \      
\eps(0110)    \gg   \boldsymbol{\et(011)}   \gg   \boldsymbol{\et(01)}      \gg    \eps(0)  
\\ \nonumber
& \mathrm{For} \  \sigma \in \S_4 \  \mathrm{of \ type} \ 1100:  \ \      
\eps(0100)    \gg     \eps(010)                \     \gg   \boldsymbol{\et(01)}    \gg    \eps(0)  
\\ \nonumber
& \mathrm{For} \  \sigma \in \S_4 \  \mathrm{of \ type} \ 1010:  \ \      
\eps(0010)    \gg  \boldsymbol{\et(001)}   \gg             \eps(00)               \gg     \eps(0)  
\\ \nonumber
& \mathrm{For} \  \sigma \in \S_4 \  \mathrm{of \ type} \ 1000:  \ \      
\eps(0000)    \gg     \eps(000)               \     \gg           \eps(00)               \gg    \eps(0)   
\end{align}

\smallskip

\noindent
On the other hand, all these conditions simply say that, if we {\em remove the last letter} of a binary word $\a$, then the corresponding epsilon value (i.e., $\eps(\a)$ or $\et(\a)$, depending on whether the last letter of $\a$ is 0 or 1, respectively\footnote{If the last letter of $\a$ is 0 then the value $\eps(\a)$ comes from the ZSH pre-blueprints $^p\!\B_1, ^p\!\B_2, ,..., ^p\!\B_{n+1}$. On the other hand, if the last letter of $\a$ is 1 then we have denoted by $\et(\a)$ the \ul{\em new} tile size that we need in our construction of $\B_n^\mathit{ff}$. This notation reminds us \ul{\em how we used the numbers $\et(\a_1...\a_{i-1} 1)$} for all $i < n$: we  looked at the tiles of $\B_i$ that are part of  $f$-$\nbhd(\sigma)$ for any $\sigma \in \S_i$ of binary type  $\a_1...\a_{i-1} 1$, and we subdivided these tiles along the $x_i$-direction into the maximal number of equal  sub-tiles such that their new $x_i$-length is at most $\et(\a_1...\a_{i-1} 1)$, and best approximates  $\et(\a_1...\a_{i-1} 1)$. This simple operation was all we needed to do in order  to transition from $\B_i$ to  $\B_i^\mathit{ff}$. Of course, after we construct  $\B_n$, we will to use exactly the same idea in order to transition from $\B_n$ to  $\B_n^\mathit{ff}$.}) must be much smaller. 

Therefore, in the general $n$-dimensional case, \ul{\em all} \ the inequalities that involve some $\et(\a)$ can be summarized simply as
\begin{eqnarray}
\label{eps_tilde_conditions_in_nD}
\eps(\a_1...\a_j 011...10) \! \gg \! \et(\a_1...\a_j 011...1) \! \gg \! ... \! \gg \! \et(\a_1...\a_j 01) \! \gg \! \eps(\a_1...\a_j 0)
\end{eqnarray}
for any binary ``prefix" $\a_1...\a_i0$ such that the word $\a_1...\a_j 011...10$ above has at most $n+1$ letters\footnote{The reason the conditions~(\ref{eps_tilde_conditions_in_nD})  allow us to address the definition of $\et(\a1)$ for {\em all} the binary words of interest (i.e., all the words that end with an ``1") is the following: $(i)$ we only need to define $\et(\a1)$ for words $\a1$ of length $n$ or less; and  $(ii)$ the word ``$\a_1...\a_j 0$" above {\em \ul{is} allowed to be the \ul{empty} word}, in which case we {\em replace} ``$\eps(\a_1...\a_j 0)$" by ``0" above (this lets us cover the case of $\et(11...1)$ for words ``11...1" of length from 1 to $n$).}.

Moreover, the same inequalities can be summarized \ul{\em in an even simpler way}\footnote{A useful difference between (\ref{eps_tilde_conditions_in_nD_SHORTER_VERSION}) and (\ref{eps_tilde_conditions_in_nD}) is that in the set of inequalities (\ref{eps_tilde_conditions_in_nD_SHORTER_VERSION}) {\em each $\et(\a_1...\a_j 011...1)$ appears only once.}}, as
\begin{eqnarray}
\label{eps_tilde_conditions_in_nD_SHORTER_VERSION}
\eps(\a_1...\a_j 011...10) \! \gg \! \et(\a_1...\a_j 011...1) \! \gg \! ... \! \gg \! \et(\a_1...\a_j 01) \! \gg \! \eps(\a_1...\a_j 0)
\end{eqnarray}
for any binary ``prefix" $\a_1...\a_i0$ such that the word $\a_1...\a_j 011...10$ above has \ul{\em exactly} $n+1$ letters\footnote{This is because we already know that, if a word of the form $\a_1...\a_j 011...110$ {\em has a longer sequence of 1's (denoted ``11...11") than the shorter word} $\a_1...\a_j 011...10$, then we have $\eps(\a_1...\a_j 011...110) \! \gg \eps(\a_1...\a_j 011...10)$.}.

%By the way, note also that  there is {\em no} $\eps(11...11)$ for words of length $n+1$.

We will now show that the inequalities~(\ref{eps_tilde_conditions_in_nD_SHORTER_VERSION}) {\em can} be simultaneously satisfied. Indeed, recall how in the previous section we defined\footnote{Recall that  in the previous section, when discussing $n$-dimensional ZSH {\em pre-blueprints}, at first we only considered $\eps(\a)$ for binary  words $\a$ that ended in a zero; but then it became convenient to define $\eps(\a)$ for {\em all binary  words $\a$ of length exactly $n$} (with the exception of the word ``11...1").}
$$
\eps(\a_1...\a_i 0 11...11) := \eps(\a_1...\a_i 0),
$$
where we obtained the word $\a_1...\a_i 0 11...11$ by adding a sufficient number of 1's at the end of the word $\a_1...\a_i 0$ to make it have \ul{\em maximal length} (which was $n$ in our discussion in the previous section, but it is $n+1$ here).
Our plan now is to try to use a similar idea\footnote{But now we do this ``in opposite direction", i.e., now we try to extend the definition to the {\em shorter} words by using the fact that we already have a definition for the {\em longest} words; in the previous section its was the other way around.} in order to define $\et(\a_1...\a_i 1)$ for words that end with ``1"; we will see that this {\em almost} 
\hskip0.06cm 
solves the problem raised by the inequalities~(\ref{eps_tilde_conditions_in_nD_SHORTER_VERSION}), and then we will make one more adjustment in order  to solve it fully.

\noindent
Indeed, let us first (temporarily) define 
$$
\et(\a_1...\a_i 1) := \eps(\a_1...\a_i 1 11...11), 
$$
where, like above,  we added a sufficient number of 1's at the end of the word $\a_1...\a_i 1$ to make it have length $n+1$.

Then, if we check the inequalities~(\ref{eps_tilde_conditions_in_nD_SHORTER_VERSION}), we will see that they are  ``almost" satisfied, except we {\em actually} have  
$$
\et(\a_1...\a_i 11) = \et(\a_1...\a_i 1),
$$ 
because when we add a sufficient number of 1's at the end of the word $\a_1...\a_i 11$, and also add a sufficient number of 1's at the end of the word $\a_1...\a_i 1$, we actually obtain {\em the same} word of length $n+1$. 
%$$
%\et(\a_1...\a_i 11) \gg \et(\a_1...\a_i 1).
%$$
%\noindent
%Instead, we actually have 
%$$
%\et(\beta 11) = \et(\beta 1),
%$$
%because when we add a sufficient number of 1's at the end of the word $\a_1...\a_i 11$, and also add a sufficient number of 1's at the end of the word $\a_1...\a_i 1$, we actually obtain {\em the same} word of length $n+1$. 
%Therefore, all that remains to be done is to note that we have 
Therefore, instead of~(\ref{eps_tilde_conditions_in_nD_SHORTER_VERSION}) we actually have  
$$
\eps(\a_1...\a_j 011...10) \gg \et(\a_1...\a_j 011...1)  =  ...  =   \et(\a_1...\a_j 01) = \eps(\a_1...\a_j 0),
$$
where the  word ``$\a_1...\a_j 011...10$" has {\em exactly} $n+1$ letters, and the common prefix ``$\a_1...\a_j 0$" may also be the empty word, in which case we replace $\eps(\a_1...\a_j 0)$ by ``0".

Then,  since ``there is enough space" between the order of magnitude of  $\eps(\a_1...\a_j 011...10)$   and the order of magnitude of $\eps(\a_1...\a_j 0)$, we can now \ul{\em modify} all the values  $\et(\a_1...\a_j 01),  ..., \et(\a_1...\a_j 011...1)$ to impose that  we have 
$$
\eps(\a_1...\a_j 011...10) \gg \et(\a_1...\a_j 011...1) \gg ... \gg \et(\a_1...\a_j 01) \gg \eps(\a_1...\a_j 0),
$$
and we obtain a solution for  the inequalities~(\ref{eps_tilde_conditions_in_nD_SHORTER_VERSION}).

\bigskip

\bigskip

\noindent
Let us now proceed with the Steps 1 and 2 in the construction of $\S_n$.

\bigskip

\noindent
{\bf Step 1.} 
Before we construct $\B_n^\mathit{ff}$, we just construct  an ordinary (i.e., not yet ``faithful") ZSH blueprint $\B_n$. Then, in Step 2, we will  refine this construction to obtain $\B_n^\mathit{ff}$. 
 
In order to construct $\B_n$ we will rely on 
an $(n-1)$-dimensional {\em faithful} ZSH blueprint $\B_{n-1}^\mathit{ff}$ of $\S_{n-1}$. We will proceed in a way that is very similar to the construction in the previous subsection.

\bigskip

\noindent
{\em Case 1.1.} Consider some $\sigma_n \in \S_n$ of binary type $\a_1\a_2...\a_{n-1}0$. We construct the tiles that form the set $f$-$\nbhd(\sigma_n)$
as follows. 
For each tile $T_{n-1}$ of $\B_{n-1}^\mathit{ff}$ such that  $T_{n-1} \subset  f\text{-}\nbhd(\pi_n(\sigma_n))$, we first ``lift" the basepoint $P_{n-1}$ of $T_{n-1}$ from $\pi_n(\sigma_n)$ to a point $P_n \in \sigma_n$, such that we have 
$$
\pi_n(P_n) = P_{n-1}.
$$
Note that $P_n$ is well defined, since we have $\dim(\pi_n(\sigma_n)) = \dim(\sigma_n)$\footnote{This is because we assumed that the binary type of $\sigma_n$ is of the form $\a_1\a_2...\a_{n-1}0$.}. 

Once we have constructed the basepoint $P_n \in \sigma_n$, we construct a tile $T_n$ of $f$-$\nbhd(\sigma_n)$ such that: $T_n$ is centered at $P_n$, satisfies the condition $\pi_n(T_n) = T_{n-1}$, and the length of $T_n$ in the $x_n$-direction is $\eps(\alpha(\sigma_n))$.
Then, we do this for all the tiles $T_{n-1} \subset f\text{-}\nbhd(\pi_n(\sigma_n))$. Note that the tiles denoted $T_{n-1}$ cover the whole $f\text{-}\nbhd(\pi_n(\sigma_n))$, and therefore we can define 
$$
f\text{-}\nbhd(\sigma_n) = \bigcup_{T_{n-1} \subset f\text{-}\nbhd(\pi_n(\sigma_n))} T_n,
$$
for all $\sigma_n \in \S_n$ of binary type $\a_1\a_2...\a_{n-1}0$, , see Figure~\ref{Fig_101}$(e)$. 

\bigskip

{\em Case 1.2.} Consider  $\sigma_n \in \S_n$ of binary type $\a_1\a_2...\a_{n-1}1$. Recall that $\nbhd(\sigma_n)$  is delimited by $\nbhd(\sigma'_n)$ and $\nbhd(\sigma''_n)$ for two\footnote{If $\sigma_n$ is unbounded, then we only have $\sigma'_n$ (i.e., there is no $\sigma''_n$), and $\nbhd(\sigma_n)$ is also unbounded, and delimited just by $\nbhd(\sigma'_n)$. } boundary faces $\sigma'_n$ and $\sigma''_n$ of $\sigma_n$, such that we have
$$
\pi_n(\sigma_n) = \pi_n(\sigma'_n) = \pi_n(\sigma''_n). 
$$
Moreover, $\sigma'_n$ and $\sigma''_n$ have  binary type $\a_1\a_2...\a_{n-1}0$, which means that we have already constructed $f\text{-}\nbhd(\sigma'_n)$ and $f\text{-}\nbhd(\sigma''_n)$.

Then we define the region $f$-$\nbhd(\sigma_n)$ to be the subset of the set 
$$
f\text{-}\nbhd(\pi_n(\sigma_n)) \times \RR
$$
that is delimited by $f\text{-}\nbhd(\sigma'_n)$ and $f\text{-}\nbhd(\sigma''_n)$.
Also we define each subset of the form 
$$
\pi_n^{-1} (T_{n-1}) \cap f \text{-} \nbhd(\sigma_n)
$$
to be a (single) tile $T_n$ within   $f$-$\nbhd(\sigma_n)$, see Figure~\ref{Fig_101}$(f)$ and $(g)$.

\bigskip

\bigskip

\noindent
{\bf Step 2.} 
In order to obtain the {\em faithful} ZSH blueprint $\B_n^\mathit{ff}$ we start with the ZSH blueprint $\B_n$ above and we further partition the tiles constructed in Case 1.2, as follows.
Consider  $\sigma_n \in \S_n$ of binary type $\a_1\a_2...\a_{n-1}1$. 
For each tile $T_{n-1} \subset f\text{-}\nbhd(\pi_n(\sigma_n))$ we partition the tile $T_n$ constructed above 
into the minimal number of equal tiles such that their length in the $x_n$-direction  is $\le \et(\a_1\a_2...\a_{n-1}1)$\footnote{The other dimensions of any such tile $\tilde T_n$ are determined by the fact that we still need to have $\pi_n(\tilde T_n) = T_{n-1}$.}.

Note that, according to Remark~\ref{rem:large_enough_distance_between_adjacent_faces}, if the number $\eps(^p\!\B_n)$ in (\ref{max_eps}) is small enough, then for any fixed tile $T_{n-1}$ the number of tiles $T_n$ considered above is very large, so the length in the $x_n$-direction of each such tile $T_n$ can be made arbitrarily close to $\et(\a_1\a_2...\a_{n-1}1)$.
Finally, we choose the basepoint of each such $T_n$ to be its centerpoint.

\bigskip

We denote by $\B_n^\mathit{ff}$ the set of all the tiles constructed above. Then, if the number $\eps(^p\!\B_n)$ in (\ref{max_eps}) is small enough, it follows by construction that the tiles of $\B_n^\mathit{ff}$ have disjoint interiors, we have
$$
\pi_n(\B_n^\mathit{ff}) = \B_{n-1}^\mathit{ff},
$$
and also the set $f$-$\nbhd(\sigma_n)$ aproximates the set $\nbhd(\sigma_n)$ arbitrarily well\footnote{
Here is what we mean when we say that ``the set $f$-$\nbhd(\sigma_n)$ approximates the set $\nbhd(\sigma_n)$ arbitrarily well". Note that for the set $\nbhd(\sigma_n)$ there exist two adjacent sets $\nbhd(\sigma'_n)$ and $\nbhd(\sigma''_n)$ such that $\pi_n(\sigma_n) = \pi_n(\sigma'_n) = \pi_n(\sigma''_n)$. If we choose either one of $\sigma'$ or $\sigma''$ (say $\sigma'$) then we can look at the common boundary set 
$$
\SS(\sigma_n, \sigma'_n) := \nbhd(\sigma_n) \cap \nbhd(\sigma'_n), 
$$
and we can also  look at the common boundary set 
$$
\SS^\mathit{ff}(\sigma_n, \sigma'_n) := f-\nbhd(\sigma_n) \cap f-\nbhd(\sigma'_n),
$$
and we  have 
$$
\eps(\a(\sigma_n \cap \sigma'_n)) \gg \dist(\SS(\sigma_n, \sigma'_n), \SS^\mathit{ff}(\sigma_n, \sigma'_n)),
$$
where the distance between sets above  is the usual Hausdorff distance.
}, for all $\sigma_n \in \S_n$. This follows from the fact the face $\sigma_n \cap \sigma'_n$ (which is just the smaller one of $\sigma_n$ and $\sigma'_n$, since we have either $\sigma_n \subset \partial \sigma'_n$ or $\sigma'_n \subset \partial \sigma_n$) is of binary type $\a 0$, and  the we know that the tiles of $f$-$\nbhd(\sigma'_n \subset \partial \sigma_n)$ ``span across" this face. 

Most importantly, it follows that  $\B_n^\mathit{ff}$ can be used as ``induction hypothesis"\footnote{Because we have already enforced the conditions~(\ref{eps_tilde_conditions_in_nD}) for the numbers $\eps(\a 0)$ and $\et(\a 1)$ used in the construction of $\B_n^\mathit{ff}$.} for the construction of a  ZSH blueprint $\B_{n+1}$ for $\S_n$ such that $\pi_{n+1}(\B_{n+1}) = \B_n$, and therefore $\B_n^\mathit{ff}$ is indeed a {\em faithful ZSH blueprint for $\S_n$}.

\bigskip

\bigskip

\bigskip

\bigskip

\noindent
Not now that the proof of Lemma~\ref{lem:face_is_contained_in_nbhds_of_subfaces} (and the closely related Remark~\ref{rem:neighborhood_of_face_inside_union_of_nbhds}) can be repeated with essentially no change for the case of $f$-$\nbhd(\sigma)$ instead of $\nbhd(\sigma_n)$, and we obtain the following:

\begin{rem}
\label{rem:neighborhood_of_face_inside_union_of_nbhds_FOR_FAITHFUL}
If $\B_n^\mathit{ff}$ is a faithful ZSH blueprint of $\S_n$ generated using a ZSH pre-blueprint $^p\!\B_n$ such that the number $\eps(^p\!\B_n)$ in (\ref{max_eps}) is small enough,
then for any face $\sigma \in \S_n$ there exists some $\rho>0$ such that  we have 
\begin{equation}
\label{eqn:neighborhood_of_face_inside_union_of_nbhds_FOR_FAITHFUL}
B(\sigma, \rho) \subset \bigcup_{\sigma_0 \subset \sigma} f\mathrm{-}\nbhd(\sigma_0),
\end{equation}
where the union above is over the faces $\sigma_0 \in \S_n$ that are contained in $\sigma$, and the set $B(\sigma, \rho)$ denotes the ball of radius $\rho$ around the face $\sigma$.
In particular, note that we have 
\begin{equation}
\label{eqn:face_inside_union_of_nbhds}
\sigma \subset \bigcup_{\sigma_0 \subset \sigma} f\mathrm{-}\nbhd(\sigma_0).
\end{equation}
\end{rem}

Note now that the basepoint of a tile $T_n \in \B_n^\mathit{ff}$ belongs to the {\em relative interior} of a face $\sigma_0 \in \S_n$ if and only if we have $T_n \subset f\mathrm{-}\nbhd(\sigma_0)$. Moreover, a face $\sigma \in \S_n$ intersects a tile $T_n \subset f\mathrm{-}\nbhd(\sigma_0)$ if and only if the basepoint of $T_n$ belongs to $\sigma$.  But then $\sigma$ intersects the relative interior of $\sigma_0$, which implies that $\sigma_0 \subset \sigma$; moreover, if $\sigma_0 \neq \sigma$ then it implies $\sigma_0 \subset \partial \sigma$.

%Note now that, if a tile  $T_n \in \B_n^\mathit{ff}$  intersects a  face $\sigma \in \S_n$  then the basepoint of $T_n$ must belong to $\sigma$

\bigskip

\noindent
Therefore, we obtain the following:

\begin{rem}
\label{rem:basepoints_of_tiles_on_interior_of_faces_FOR_FAITHFUL}
If $\B_n^\mathit{ff}$ is a faithful ZSH blueprint of $\S_n$ generated using a ZSH pre-blueprint $^p\!\B_n$ such that the number $\eps(^p\!\B_n)$ in (\ref{max_eps}) is small enough, then the following is true. 
If the basepoint of a tile $T_n \in \B_n^\mathit{ff}$ belongs to the relative interior of a face $\sigma_0 \in \S_n$, and for some other face $\sigma \in \S_n$ we have 
$$
T_n \cap \sigma \neq \emptyset,
$$
then it follows that $\sigma_0 \subset \partial \sigma$.
In particular,  we have 
\begin{equation}
\label{eqn:face_inside_union_of_nbhds}
\bigcup_{\sigma_0 \subset \sigma} f\mathrm{-}\nbhd(\sigma_0) = \bigcup_{\text{basepoint}(T_n) \in \sigma} T_n,
\end{equation}
where the second union above is over all tiles $T_n \in \B_n^\mathit{ff}$ such that the basepoint of $T_n$ belongs to $\sigma$.
\end{rem}

\bigskip

\bigskip

\noindent
We can combine the two remarks above to obtain:

\begin{rem}
\label{rem:neighborhood_of_face_inside_union_of____TILES____FOR_FAITHFUL}
If $\B_n^\mathit{ff}$ is a faithful ZSH blueprint of $\S_n$ generated using a ZSH pre-blueprint $^p\!\B_n$ such that the number $\eps(^p\!\B_n)$ in (\ref{max_eps}) is small enough,
then for any face $\sigma \in \S_n$ there exists some $\rho>0$ such that  we have 
\begin{equation}
\label{eqn:neighborhood_of_face_inside_union_of____TILES____FOR_FAITHFUL}
B(\sigma, \rho) \subset \bigcup_{\text{basepoint}(T_n) \in \sigma} T_n,
\end{equation}
where the  union above is over all tiles $T_n \in \B_n^\mathit{ff}$ such that the basepoint of $T_n$ belongs to $\sigma$.
\end{rem}

\newpage 

\section{Proof of the global attractor conjecture in the $n$-dimensional case}

Consider some toric dynamical system ${\bf T}_n$ in $\RRpn$. 
As we discussed in previous sections, we embed ${\bf T}_n$ into a toric differential inclusion $\T_n$, and we already know that, if we can construct an \ul{\em exhaustive family of ZSH for $\T_n$}, then we obtain a proof of the global attractor conjecture. 

\medskip

We know that we can construct a ZSH blueprint $\B_n$ for $\T_n$. After that, our plan is the following:  {\bf (Step 1)} explain how to convert $\B_n$ into the kind of ``restricted ZSH blueprint" denoted  $\B_n^P$, whose tiles are contained in the set $\D_n^P$, such that $\B_n^P$  can be used for the  construction of ZSH\footnote{Recall that, when we used a ZSH blueprint $\B_2$ to construct ZSH in 3D, the blueprint $\B_2$ was {\em restricted to the set $\D_2^P$}, or equivalently, was restricted to the set $\D_2^H$.} for $\T_n$, and {\bf (Step 2)} explain how we can use $\B_n^P$  for the construction of an {\em exhaustive family of ZSH for $\T_n$}. We now discuss these two steps,  in this  order.

\bigskip 
\bigskip 

\noindent
{\bf Step 1.}  Recall that, when we constructed a (say, two-dimensional) ZSH blueprint $\B_2$ in section~\ref{section_construction_of_ZSH_pre-bluprint_and blueprint_in_nD}, we did not restrict its tiles to the set $\D_2^P$, which (up to identifying $\RR^2$ with the affine plane $\{ X_1 =1 \}$ of $\RR^3$) was the compactification of the set 
$$
\{ (X_2^P, X_3^P) \in \RR^2 |  1 \le X_2^P \le X_3^P \},
$$
but when we discussed two-dimensional ZSH blueprints in Section~\ref{section:proof_in_3D} (where we actually {\em used} these blueprint to construct two-dimensional polyhedral ZSH), all the tiles of a ZSH blueprint were contained inside $\D_2^P$ (for example, see Figure~\ref{fig:3a} $(b)$).

Note though that {\em the boundary rays of $\D_2^P$ are faces of the subdivision $\F_2^P$.}  Then there is a very simple way to convert an already constructed ZSH blueprint $\B_2$ that is defined on the whole $\RR^2$ (for an $x_1$-$x_2$-subdivision $\S_2$ that contains the faces of a given subdivision subdivision $\F_2^P$ of $\D_2^P$), into a ``restricted ZSH blueprint" $\B_2^P$ whose tiles are contained in $\D_2^P$: for each tile $T_2$ of $\B_2$ whose interior intersects $\D_2^P$ we define a tile $T_2^P$ of $\B_2^P$ given by
$$
T_2^P = T_2 \cap \D_2^P.
$$
Also, we define the basepoint of $T_2^P$ to be the same as the basepoint of $T_2$. 
Note that, since the two boundary rays of $\D_2^P$ are union of faces of the subdivision $\F_2^P$, it follows that if a tile $T_2^P$ intersects the boundary of $\D_2^P$ then the basepoint of $T_2^P$ lies on the boundary of $\D_2^P$ (see Figure~\ref{fig:3a}, Figure~\ref{fig:7}, and Figure~\ref{Fig_101} for examples.). 

Also, in order to obtain {\em exactly} the kind of two-dimensional ZSH blueprint that we have used in Section~\ref{section:proof_in_3D}, we need to also intersect all the tiles $T_2^P$ with a ``blue box"  of the form 
$$
[0,M_3] \times [0,M_2],
$$
for $M_3 \gg M_2 \gg 1$, to obtain all the {\em bounded tiles} of $\B_2^P$\footnote{We also need to make sure that the horizontal side of the blue box (i.e., the side where we have  $\{ X_3 = M_3 \}$) is a union of horizontal sides of tiles of $\B_2$, and the vertical side of the blue box (i.e., the side where we have  $\{ X_2 = M_2 \}$)  is a union of vertical sides of tiles of $\B_2$. In general, for the {\em latter} condition, we need to lengthen or shorten some (horizontal) sides for tiles of $\B_2$.}. 

Finally, in order to finish the construction of $\B_2^P$, we need to construct the ``boundary tiles" of $\B_2^P$ (see Figure~\ref{fig:8}$(a)$). For this, we choose tile vertices $V_i$ along the sides of the blue box where we have $\{ X_3 = M_3 \}$ or  $\{ X_2 = M_2 \}$, in such a 
way that, when we draw rays that connect the origin to each $V_i$, the rays {\em separate the slopes of rays contained within the subdivision $\S_2$\footnote{In other words, if two rays that are contained in $\S_2$ have the same slope, then the parts of those rays that are outside the big blue box are contained in {\em the same} boundary tile of $\B_2^P$.}}.

\medskip

The two-dimensional case discussed above suggests the following process for the transition from our construction of a ``global" ZSH blueprint $\B_n$ to the definition of a ``restricted" ZSH blueprint denoted  $\B_n^P$, in any dimension. Let us start from the beginning, with the $(n+1)$-dimensional fan $\F_{n+1}$\footnote{For ease of notation, let us say  that our toric differential inclusion of interest is $(n+1)$-dimensional,  instead of $n$-dimensional}. Then, as we discussed in previous sections, without loss of generality we can assume that  $\F_{n+1}$ is $x_1$-$x_2$-...-$x_{n+1}$-symmetric, and contains all the hyperplanes of the form $\{ X_i = X_j \}$ and $\{ X_i = 0 \}$.

%\smallskip

In particular, $\F_{n+1}$ gives rise to a ``restricted fan"  $\F_{n+1}^\le$, whose cones are just those cones of $\F_{n+1}$ that are contained within the set\footnote{If we are given $\F_{n+1}^\le$, then, due to symmetry, we can reconstruct all the cones of $\F_{n+1}$ that are contained in the closure of  the negative orthant.} 
$$
\{  X_{n+1} \le X_n \le ... \le X_1 \le 0  \}. 
$$ 
Let us now  identify the set
$$
\{  X_{n+1} \ge X_n \ge ... \ge X_1 =1  \} \subset \RRp^{n+1},
$$
with the set\footnote{This is an abuse of notation; strictly speaking, $\D_n^P$ is the {\em compactification} of this set.} 
$$
\D_n^P := \{ (X_2, X_3, ..., X_{n+1}) |  X_{n+1} \ge X_n \ge ... \ge X_2 \ge 1  \} \subset \RRp^n.
$$

Then, for any cone $\C \in \F_{n+1}^\le$ we consider the set 
$$
\{  X_1 = 1  \}   \cap   (-\C),
$$
and, each such set (via the identification above)  gives rise to a subset of $\D_n^P$; these subsets form a {\em subdivision of $\D_n^P$}, which we embed into an $x_1$-$x_n$-subdivision of $\RRn$ denoted $\S_n$. 

Finally, for the  $x_1$-$x_n$-subdivision  $\S_n$ we construct (as discussed in the previous section) a ZSH blueprint $\B_n$, and then for each each tile $T_n$  of $\B_n$ whose interior intersects $\D_n^P$ we construct the tile 
$$
T_n^P  =  T_n \cap \D_n^P.
$$

Now we only need to make {\em two more changes}, in order  to go  from the set of all the tiles $T_n^P$ described above to the  ``restricted" ZSH blueprint $\B_n^P$: we need to restrict the tiles $T_n^P$ to a ``$n$-dimensional blue box" of the form 
$$
[0, M_{n+1}] \times [0, M_n] \times ... \times [0, M_2], 
$$
and, {\em outside} this blue box, we need to construct the ``boundary tiles" of $\B_n^P$. 

The construction of the boundary tiles of $\B_n^P$ takes advantage of a blueprint diagram denoted $\B_{n-1}^P$, which lies on the boundary of $\D_n^P$. Like in Section~\ref{section:proof_in_3D}, it is easier to think about $\B_{n-1}^H$ instead of $\B_{n-1}^P$, because $\B_{n-1}^H$ is simply a blueprint diagram for the (restricted) fan 
$$
\F_n^\le := F_{n+1}^\le \cap \{ X_1 = 0 \},
$$
see Figure~\ref{fig:3_e}$(a)$. 

Then the boundary tiles of $\B_n^P$ are defined by first constructing a blueprint diagram  $\B_{n-1}^H$ for $\F_n^\le$ ( both of which   live in the hyperplane $\{ X_1 = 0 \}$), and then extending it radially to meet the boundary of the blue box, like in Figure~\ref{fig:3_e}$(a)$ and in the (more representative) Figure~\ref{fig:8}$(b)$.

\bigskip 
\bigskip 
\bigskip 
\bigskip 
\bigskip 
\bigskip 
\bigskip 
\bigskip

\bigskip 
\bigskip 
\bigskip 
\bigskip 
\bigskip 
\bigskip 

\noindent
{\bf Step 2.}  We now want to use  $\B_n^P$  for the construction of an {\em exhaustive family of ZSH for $\T_n$}. 

% Let us go back again to the two-dimensional case, to see what we need to do. 

Our plan is to use the (restricted) ZSH blueprint $\B_n^P$ from  Step 1, together with a choice of ``initial point" of the form $(\eps_0, \eps_0, ..., \eps_0) \in \RRpn$, in order to construct a ZSH denoted $\Z_{\eps_0}$ , and then argue that  $\Z_{\eps_0}$ form an {\em exhaustive family of ZSH} for $\T_n$.

Recall also that we only have to construct the ZSH in one of the regions 
$$
\R_{n+1} = \{  x_1 \ge x_2 \ge ... \ge x_{n+1}  \} \subset \RRp^{n+1},
$$
because after that we can extend it by symmetry to the other $n!$ regions of this type.

Note also that along any boundary piece of $\R_{n+1}$ of the form 
$$
\R_{n+1}^{i,i+1} = \{  x_1 \ge x_2 \ge ... x_i = x_{i+1}  ... \ge x_{n+1}  \} 
$$
there is a $n$-dimensional subproblem of the same type, and we can solve it as an induction hypothesis\footnote{We saw an instance of this when we solved the 3D problem in Section~\ref{section:proof_in_3D}. Indeed, in that case there were two 2D problems that could be solved first (i.e., two one-dimensional ZSH that could be constructed first), one in the plane $\{ x_1 = x_2 \}$, and the other one in the plane $\{ x_2 = x_3 \}$. Note that {\em both} of the constructions of these curves started at the {\em same}  initial point, of coordinates $(\eps_0, \eps_0, \eps_0)$.}.  

Once we have made the constructions along all the boundary pieces of the form $\R_{n+1}^{i,i+1}$, we can continue the construction corresponding to tiles $T_n^P$ within the  {\em blue box}, by filling in the face every polyhedral face $F_i$ corresponding to a tile $T_i^P$ of $\B_n^P$, in {\em lexicographic order}\footnote{The construction has to proceed in  an  order that ensures that  any information that is needed later has already been figured out earlier. Lexicographic order is one such good choice, because it allows us to argue by induction that the construction does not create any ``vicious cycle".}. 

Then, after we are done with all the tiles that lie within the {\em blue box}, we just fill in the rest of the ZSH {\em cylindrically}, using the ``outer boundary"\footnote{This ``outer boundary" consists of the points of the ZSH ({\em before} the cylindrical construction) that correspond via $\Psi_3^P$ to boundary points of the {\em blue box}, on any one of the sides where we have $x_j = M_j$ for some $j$.} of the already constructed ZSH hypersurface; the cylinder uses a generating direction that goes along the  $x_1$-axis (see Figure~\ref{fig:3_bcd}$(d)$ for an example).

Let us now explain why these hypersurfaces $\Z_{\eps_0}$ (constructed as discussed above: each one of them uses a different starting point $(\eps_0, \eps_0, ..., \eps_0) \in \RRpn$, where we have $0 < \eps_0 \ll 1$) form an {\em exhaustive family of ZSH for $\T_n$}. Consider some fixed by arbitrary $\x_0 \in \RRpn$. We need to argue that, for well-chosen $\eps_0$, the hypersurface $\Z_{\eps_0}$ {\em ``lies below"} the point $\x_0$, and also, that the trice differential inclusion  $\T_n$ {\it cannot cross} this $\Z_{\eps_0}$ within a hypercube $[0, M]^{n+1}$ that contains $\x_0$.

For the first condition, we just choose $\eps_0 < \x_0^i$ for any $1 \le i \le n+1$. Then, for this choice of $\eps_0$ it is easy to see that the  set 
$$
(\eps_0, \eps_0, ..., \eps_0) + \RRpn
$$
contains $\x_0$, and {\em lies above} the hypersurface $\Z_{\eps_0}$.

Finally, in order to argue that for any fixed $M$ we can find  $\eps_0$ such that $\T_n$ {\it cannot cross} this $\Z_{\eps_0}$ within the hypercube $[0, M]^{n+1}$, we reason like in subsection~\ref{polyhedral_surface_is_ ZSH_in_3D}, i.e., first we notice that we can apply Lemma~\ref{lemma_outer_normal_sufficient}, and then we reason based on proximity to cones of $\F_{n+1}$ that lie along the boundary of  $\D_n^H$ where we have $X_1^H = 0$.

\bigskip
\bigskip
\bigskip
\bigskip
\bigskip
\bigskip
\bigskip
\bigskip
\bigskip  
\bigskip

{
%\color{olive}
\noindent
{\bf Remark.} Recall that the {\em Persistence Conjecture} says that {\em any $n$-dimensional weakly reversible system is persistent}~\cite{CNP}. A stronger version of this conjecture (called the {\em Extended Persistence Conjecture}~\cite{CNP}) says that any {\em variable-$k$ weakly reversible} system is persistent. 
We obtain the following partial result for the (Extended) Persistence Conjecture: any \ul{\em bounded} trajectory of a variable-$k$ weakly reversible\footnote{Or, more generally,  variable-$k$ {\em endotactic} system, because variable-$k$ endotactic systems can also be embedded into toric differential inclusions, see~\cite{Craciun_Deshpande_Endotactic_networks_and_TDIs}.} system is persistent. 
}

\bigskip

\newpage

\section{Appendix: some general results about fans, toric differential inclusions,  and ZSH}

In this section we discuss some technical results that we used in previous sections. For some of these results, for the sake of a simpler notation, {\em we formulate and prove them in 3D instead of nD}; the proof in nD is analogous. 

\subsection{Finer and coarser fans, embeddings of toric differential inclusions, forward invariant regions}

The following definition clarifies what we mean by ``the fan $\F_1$ is coarser than the fan $\F_2$", or equivalently, ``the fan $\F_2$ is finer than the fan $\F_1$". 

\begin{defn}
Given two  fans $\F_1$ and $\F_2$ in $\RR^n$, we say that \underline{$\F_2$ is finer than $\F_1$} (or \underline{$\F_1$ is coarser than $\F_2$}) if  any cone in $\F_1$ can be written as a union of cones in $\F_2$.
\end{defn}

We now want to show that forward invariant regions of toric differential inclusions generated by the finer fan are also forward invariant regions of toric differential inclusions generated by the  coarser fan; for this we need the notion of ``embedding" for differential inclusions.

\begin{defn}
Given two differential inclusions $\frac{d{\mathbf x}}{dt} \in F_1({\mathbf x})$ and $\frac{d{\mathbf x}}{dt} \in F_2({\mathbf x})$, we say that the first one \underline{is embedded in} the second one if $F_1({\mathbf x}) \subseteq F_2({\mathbf x})$ for all ${\mathbf x}$.
\end{defn}

The following lemma and its corollary allow us  to replace a toric differential inclusion $\T_1$ with a more convenient one (denoted $\T_2$) and argue that  ZSH for the new toric differential inclusion are actually also ZSH for the old toric differential inclusion we started with.

\begin{lem}
\label{lemma_1}
Consider a number $\delta >0$ and two  fans $\F_1$ and $\F_2$ in $\RR^n$ such that $\F_2$ is finer than $\F_1$. If $\T_1$ is the toric differential inclusion generated by $(\F_1, \delta)$, and $\T_2$ is the toric differential inclusion generated by $(\F_2, \delta)$, then $\T_1$ is embedded in $\T_2$.
\end{lem}

\begin{proof}
Consider some ${\bf x} \in \RR^n_{>0}$. For each cone $K_1\in\F_1$ at distance $<\delta$ from the point $\log({\bf x})$, there exists some cone $K_2\in\F_2$ that is also at distance $<\delta$ from the point $\log({\bf x})$, such that $K_2 \subseteq K_1$. But this implies that the polars of these two cones satisfy $K_1^o \subseteq K_2^o$, so we obtain
$$
F_{\F_1,\delta} (\log({\bf x}))  \subseteq F_{\F_2,\delta} (\log({\bf x})),
$$
i.e., $\T_1$ is embedded in $\T_2$.
\end{proof}

Lemma~\ref{lemma_1} immediately implies the following result.

\begin{cor}
\label{cor_1}
Consider a number $\delta >0$ and two  fans $\F_1$ and $\F_2$ in $\RR^n$ such that $\F_2$ is finer than $\F_1$. Denote by $\T_1$  the toric differential inclusion generated by $(\F_1, \delta)$, and by $\T_2$  the toric differential inclusion generated by $(\F_2, \delta)$. Assume that $\T_2$ cannot cross some surface $\Z$ within a set $\Omega \subset \RR^n$. Then $\T_1$ also cannot cross  $\Z$ within $\Omega$.
\end{cor}

\bigskip

\subsection{Necessary and sufficient conditions for ZSH}

Consider a ZSH denoted $\Z$ inside the unit cube $(0,1)^3$, and recall that we can assume that the normal line at any smooth point of $\Z$ is contained in the union of the (closed) positive orthant and the (closed) negative orthant. 
The  \ul{\em outer normal} of $\Z$ at a point $\x \in \Z$ is defined as the normal direction of $\Z$ at  $\x$ pointing into the (closed) {\em negative} orthant. Then we have:

\begin{lem}
\label{lemma_outer_normal_necessary}
Consider a toric differential inclusion $\T_3$ in $\RR^3_{>0}$ with polyhedral fan $\F_3$ and assume that a piecewise smooth surface $\Z$ is a ZSH for $\T_3$ and is contained in the unit cube $(0,1)^3$. 
Then, for any point $\x \in \Z$ where $\Z$ is smooth we have: if $\C \in \F_3$ is a polyhedral cone such that $\log(\x) \in \C$, then the outer normal vector of the surface $\Z$ at the point $\x$ must belong to the cone $\C$. 
\end{lem}

\begin{proof}
Since $\Z$ is a zero-separating surface for $\T_3$ and $\log(\x) \in \C$, it follows that the tangent plane to $\Z$ at $\x$ is a supporting plane of the polar cone $\C^o$ (otherwise the cone $\C^o$ would provide a way for trajectories of $\T_3$ to ``cross through" the zero-separating surface $\Z$ at $\x$). Therefore, the outer normal of $\Z$ at $\x$ must belong to the polar cone of $\C^o$.  In conclusion, the outer normal of $\Z$ at $\x$ belongs to the polar of the polar of the cone $\C$, which is simply $\C$.
\end{proof}

\begin{rem}
\label{remark_outer_normal}
Lemma~\ref{lemma_outer_normal_necessary} provides a very convenient necessary condition that any ZSH $\Z$ should satisfy: its outer normal direction at any (smooth) point $\x$ must belong to a cone of $\F_3$ that is easy to specify, because it is exactly the cone that contains the point $\log(\x)$. If several cones $\C \in \F_3$ contain the point $\log(\x)$, then we obtain more information from this lemma if we select the smallest such cone. Most importantly for us, this lemma provides a geometric condition for $\Z$ that does not involve  polar cones any more; it only involves cones from the polyhedral fan $\F_3$.  
\end{rem}

We will see below that this necessary condition can be modified  to obtain a very useful  \underline{\em sufficient} {\em condition}. 

\begin{lem}
\label{lemma_outer_normal_sufficient}
Consider a toric differential inclusion $\T_3$ in $\RR^3_{>0}$ given by a  fan\footnote{Recall that we are assuming that our fans are $x_1$-$x_2$-$x_3$-symmetric and contain the planes $x_i=x_j$ and $x_i=0$.}  $\F_3$ and a fixed number $\delta>0$. 
Assume that for any cone $\C \in \F_3^\le$ we are able to specify a cone $K_\C \subset \{ X_3 \le X_2 \le X_1 \le 0 \}$ such that $\C$ is contained in the relative interior of $K_\C$ with respect  to the set $\{ X_3 \le X_2 \le X_1 \le 0 \}$, and the property below is satisfied: 

For any $\hat\eps > 0$ there exists a piecewise smooth surface $\Z_{\hat\eps}$ such that

\smallskip

$(i)$ $\Z_{\hat\eps} \subset (0,1)^3 \setminus (\hat\eps + [0,1]^3)$

\noindent
and

$(ii)$ for any smooth point $P \in \Z_{\hat\eps}$ we have: if $\log(P) \in K_\C$, then there is a normal vector $\vec{n}_P$ of $\Z_{\hat\eps}$ at $P$ such that\footnote{We will refer to this $\vec{n}_P$ as ``the outer normal" of $\Z_{\hat\eps}$ at $P$. Note that $\vec{n}_P$ belongs to the (closed) negative orthant of $\RR^3$.} $\vec{n}_P \in \C$.

\smallskip

\noindent
Then,  for any $\hat\eps$ small enough, $\T_3$ cannot cross $\Z_{\hat\eps}$ within $(0,1)^3$. 
\end{lem}

\begin{proof}
Consider some $\C \in \F_3^\le$. 
We need to show that for any $\hat\eps$ small enough the following is true: if a point $P \in \Z_{\hat\eps}$ satisfies 
$$
\log(P) \in B(\C, \delta),
$$ 
where $B(\C, \delta)$ is defined as the set of points at  distance at most $\delta$ from $\C$,  then there exists a normal vector of $\Z_{\hat\eps}$ at $P$ that belongs to $\C$. 

Note first the following \underline{\em remark}: for any cones $\C_1, \C_2 \subset \RR^n$ such that $\C_1$ is contained in the interior of $\C_2$, it follows that there exists some number $M>0$ such that we have 
$$
B(\C_1, \delta) \setminus B(0, M) \subset  \C_2,
$$
i.e., all the points of $B(\C_1, \delta)$  that lie outside the ball of radius $M_\C$ must belong to $K_\C$.\footnote{Geometrically, we could formulate it like this: if we look {\em outside} a big enough sphere centered at the origin, then {\em the ``constant width neighborhood" $B(\C_1, \delta)$ of $\C_1$ will be contained inside the ``growing width neighborhood" $\C_2$ of $\C_1$}. A short proof of this fact follows from the observation that $\dist(\C_1 \cap B(0, 1), \partial \C_2) > 0$, which implies that $\dist(\C_1 \cap B(0, M), \partial \C_2) > \delta$ for large enough $M$.}

Note now that since $\C$ is contained in the {\em relative interior} of $K_\C$ with respect to the set $\{ X_3 \le X_2 \le X_1 \le 0 \}$, it follows that $\C$ is contained in the {\em interior} of $K_\C^{sym}$, where  $K_\C^{sym}$ is the union between $K_\C$ and all its symmetric ``mirror images" within $\RR^n$\footnote{This is because $\C^{sym}$ is contained in the {\em interior} of $K_\C^{sym}$ and, of course, $\C$ is contained in $\C^{sym}$.}. 

Then we can apply the remark above to conclude that there exists some number $M_\C>0$ such that we have 
$$
B(\C, \delta) \setminus B(0, M_\C) \subset  K_\C^{sym}.
$$

Finally, if $\hat\eps$ is small enough, it follows that $\log(P)$ lies outside the ball $B(0, M_\C)$, and therefore $\log(P) \in K_\C^{sym}$; but we also have $\log(P) \in \{ X_3 \le X_2 \le X_1 \le 0 \}$, which implies $\log(P) \in K_\C$. Then  the outer normal $\vec{n}_P$ of $\Z_{\hat\eps}$ at $P$ satisfies $\vec{n}_P \in \C$.
Now the desired conclusion follows from the fact that there are only finitely many cones $\C \in \F_3$, so for $\hat\eps$ small enough we can guarantee that the reasoning above works  for all $\C \in \F_3^\le$.
\end{proof}

\bigskip

\begin{rem}
\label{rem_outer_normal_sufficient}
Lemma~\ref{lemma_outer_normal_sufficient} is very useful for our process for constructing the ZSH, because it allows us to ``trade" the restriction that $\eps$ is very small (which we need to make anyway) to account for the  possibility that $\delta$ might be quite large. More specifically, in practice, it follows that \ul{it is sufficient to construct a single ``good" ZSH blueprint}, and the same blueprint will work for any (fixed) $\delta$; for larger $\delta$ we just need to start the construction of the ZSH (for \ul{the same} ZSH blueprint) with a smaller $\eps$.

\end{rem}

\bigskip

\subsection{Restrictions of $\Psi_3^P$, and images of straight line segments via $\Psi_3^P$}

\noindent
Consider the function $\Psi_3^P : (0,1)^3 \to \RRp^2$ given by
$$
\Psi_3^P(x,y,z) = \left( \frac{log(x)}{log(z)}, \frac{log(y)}{log(z)} \right).
$$

We show that if we restrict  $\Psi_3^P$ to an affine planes near the origin, and if this affine plane  ``cuts off" the corner of the positive orthant, then the restriction is a diffeomorphism: 

\begin{lem}
\label{lemma_restriction_is_diffeo} 
Consider an affine plane $\pi$ that intersects the positive orthant and has a normal vector inside the negative orthant. If the distance between the origin and $\pi$ is small enough, then the restriction of 
$$
\Psi_3^P(x_1, x_2, x_3) = \left( \frac{\log x_2}{\log x_1}, \frac{\log x_3}{\log x_1} \right)
$$ 
to $\pi$, i.e., the map 
$$
{\Psi_3^P \Big |}_{\pi \cap \mathbb{R}^3_{>0}} : \pi \cap \mathbb{R}^3_{>0} \to \mathbb{R}^2,
$$
is a {\em diffeomorphism} between $\pi \cap \mathbb{R}^3_{>0}$ and its image via $\Psi_3^P$.
\end{lem}

\begin{proof}
We just need to check that the restricted map is injective and that it is {\em locally} a diffeomorphism (i.e., its Jacobian is not singular anywhere). These can both be verified using elementary calculations. For example, injectivity follows if we write the plane as 
$$
a_1x_1 + a_2x_2 + a_3x_3 = b,
$$
for some $a_i \ge 0$ (not all zero) and some $b > 0$, and note that for any $u,v>0$ the curve $t \to (t, t^u, t^v)$ intersects this plane for at most (and actually exactly) one value $t_0 \in (0,1)$. This also shows that we can allow the normal vector to lie {\em on the boundary} of the negative orthant, as long as we further restrict the domain to be the set $\pi \cap (0,1)^3$. Also, we don't even need the hypothesis that $\pi$ is very close to the origin; it is actually  sufficient that $\pi \cap (0,1)^3 \neq \emptyset$.
\end{proof}

\bigskip

\bigskip

%\subsection{Images of straight line segments via $\Psi_3^P$}

\noindent
Next, we analyze the slopes of images of straight lines near the origin of $\RRp^3$, via the map $\Psi_3^P$.

The goal of the following lemma is to show (as illustrated in Figure~\ref{fig:Comparison_images_of_lines_via_Psi3}, and discussed before in Figure~\ref{fig:3_bcd}$(e)$) that  images via $\Psi_3^P$ of straight line segments near the origin of $\RRp^3$ are {\em either approximately \ul{horizontal} or approximately \ul{vertical},} and this approximation\footnote{The approximation of a horizontal or vertical line segment in $\RR^2$ by the curve ``$\Psi_3^P(\mathrm{line \ segment})$" is in the sense of $C^0$ distance, not $C^1$ distance. Alternatively, we can regard both as subsets of $\RR^2$, and then the approximation is in the sense of Hausdorff distance.} can be made arbitrarily good, if we focus on a line segment in $\RRp^3$ with fixed (but arbitrary) {\em support vector $v$}, and we position the line segment at a distance $\eps \ll 1$ from the origin. {\em To be more precise, this Lemma focuses on \ul{\em exactly} the kinds of straight line segments  that appear when we build faces of our polyhedral ZSH, and then it shows that their images via $\Psi_3^P$ can approximate arbitrarily well the sides of the \ul{\em tiles} $T_i$ of a ZSH blueprint\footnote{Recall that a side of a tile $T_i$ of $\B_2$ in the {\em interior} of $\D_2^P$ is either horizontal or vertical. See also Figure~\ref{fig:3_bcd}$(e)$.} $\B_2$.} 

\begin{figure}[htbp!]
  \begin{center}
    \begin{tabular}{c}
%      \hskip-2cm
      \includegraphics[angle=0,width=4.2in]{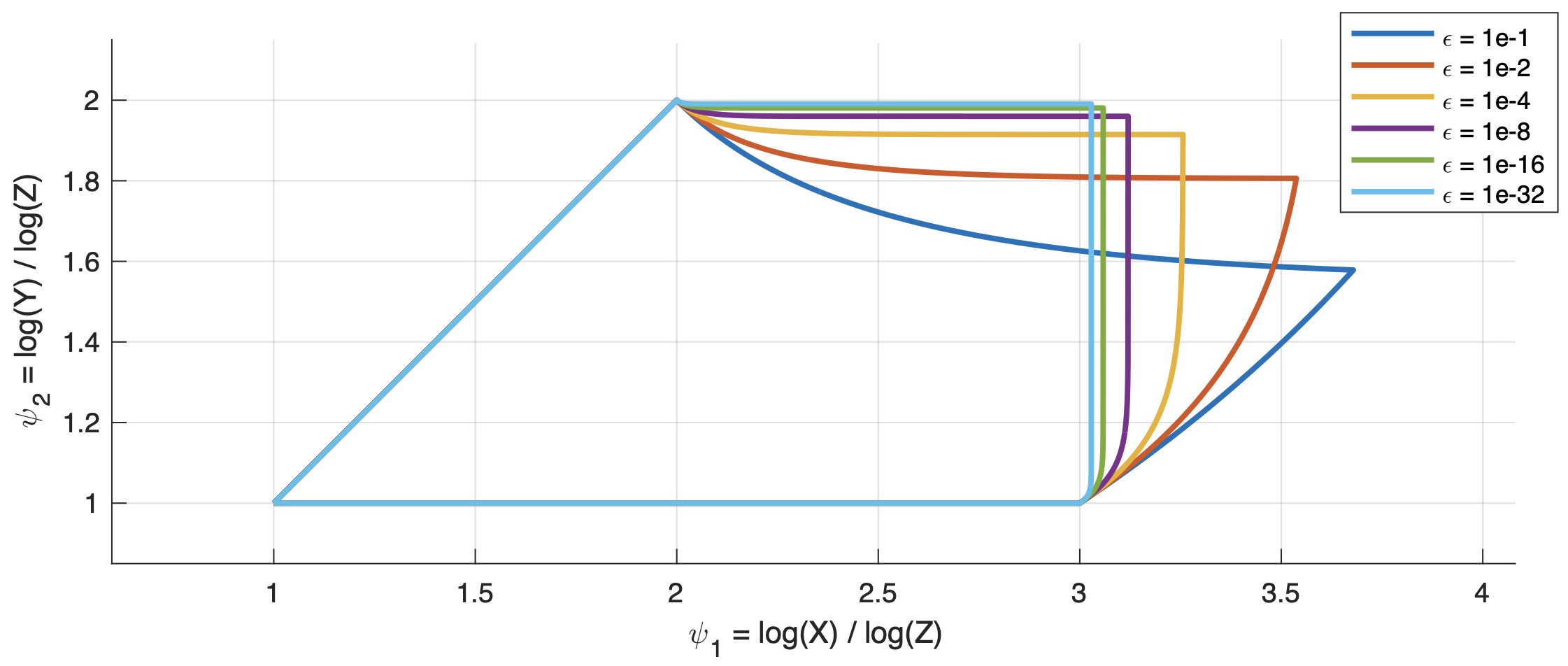}
    \end{tabular}
  \end{center}
\caption{\scriptsize\label{fig:Comparison_images_of_lines_via_Psi3} Comparison of several images of lines via $\Psi_3^P$.}
\end{figure}

\begin{lem}
\label{lemma_4.5_OLD} 
Consider the function $\Psi_3^P : (0,1)^3 \to \RRp^2$ given by
$$
\Psi_3^P(x,y,z) = \left( \frac{log(x)}{log(z)}, \frac{log(y)}{log(z)} \right).
$$
Then we have:

(i) Consider a vector $v = (v_1,v_2,v_3)$ with $v_1<0$ and denote by $w_v(t^X, t^Y, t)$ the vector in $\RR^2$ obtained as the image of $v$ through the Jacobian matrix of $\Psi_3$ evaluated at the point $(t^X, t^Y, t)$, where $(X,Y) \in \D_2$ and $0<t<1$. Then for any $\eps_1, \eps_2>0$ there exists a neighborhood $V_{\eps_1, \eps_2}$ of the origin such that if $(t^X, t^Y, t) \in V_{\eps_1, \eps_2}$  and if $X-Y>\eps_1, Y-1>\eps_1$ then the angle between the vector $w_v(t^X, t^Y, t)$ and the direction of the $X$-axis (i.e., the horizontal axis) is at most $\eps_2$.

(ii) Consider exactly the same setting as above, {\em except} $v_1=0$ and $v_2<0$, and we assume  we also know that $X$ is bounded from above. Then there exists a neighborhood $V_{\eps_1, \eps_2}$ of the origin such that if $(t^X, t^Y, t) \in V_{\eps_1, \eps_2}$  then the angle between the vector $w_v(t^X, t^Y, t)$ and the direction of the $Y$-axis (i.e., the vertical axis)  is at most $\eps_2$.

(iii) If  $v = (-1,1,0)$, then for any $\eps_1, \eps_2>0$ there exists a neighborhood $V_{\eps_1, \eps_2}$ of the origin such that  if $(t^X, t^Y, t) \in V_{\eps_1, \eps_2}$  and if  $1+\eps_1 < X < \frac{1}{\eps_1}$ and $X-Y < \frac{\eps_1}{2}$ then the (signed) angle between the vector $w_v(t^X, t^Y, t)$ and the direction of the $X$-axis belongs to the interval $[-\frac{\pi}{4} - \eps_2, 0]$.

(iv) If $v = (0,-1,1)$, then for any $\eps_1, \eps_2>0$ there exists a neighborhood $V_{\eps_1, \eps_2}$ of the origin such that  if $(t^X, t^Y, t) \in V_{\eps_1, \eps_2}$  and if  $Y<1+\eps_1$ and $X > 1+2\eps_1$ then the (signed) angle between the vector $w_v(t^X, t^Y, t)$ and the direction of the $X$-axis belongs to the interval $[\tan^{-1}\frac{Y+1}{X} - \eps_2, \frac{\pi}{2}]$.
\end{lem}
\begin{proof}
Using the fact that $\log(t)<0$ we can write the Jacobian matrix of $\Psi_3$ in the form
\begin{equation} \label{JacPsi}
\left|
\frac{\log t}{t}
\right|
\left(
\begin{array}{ccc}
-t^{1-X} &    0     &  X \\
   0    & -t^{1-Y}  &  Y 
\end{array}
\right),
\end{equation}
Then we multiply it by the vector $v$ to obtain an explicit formula for the vector $w_v(t^X, t^Y, t)$:
$$
w_v(t^X, t^Y, t) = \left| \frac{\log t}{t} \right|  \left(  - v_1t^{1-X} + v_3 X ,  -v_2t^{1-Y} + v_3 Y \right).
$$
Then, for example, in case $(i)$ we want to show that 
$$ 
- v_1t^{1-X} + v_3 X  \gg   |-v_2t^{1-Y} + v_3 Y|.
$$ 
This follows if we take into account that, for small a enough neighborhood $V_{\eps_1, \eps_2}$ (i.e., if $t$ is small enough), the hypotheses $X-Y>\eps_1, Y-1>\eps_1$ imply   $t^{1-X}  \gg  t^{1-Y}$, $t^{1-X}  \gg  X$, $t^{1-Y}  \gg  Y$, which imply 
$$
- v_1t^{1-X}   \gg   |-v_2t^{1-Y} + v_3 Y| + |v_3 X|.
$$ 

\noindent
The other cases are similar.
\end{proof}

\bigskip

\subsection{Asymptotic level sets of the Horn-Jackson Lyapunov function}

Consider a Horn-Jackson Lyapunov function $V : \RRp^n \to \RR$ given by 
$$
V(x_1, ..., x_n) = \sum_{i=1}^n (x_i \log \frac{x_i}{x_i^0} -x_i + x_i^0),
$$
for some fixed $(x_1^0, ..., x_n^0) \in \RRp^n$. Then we have:

\begin{lem}
\label{lemma_asymptotics_of_Horn-Jackson_Lyapunov_function}
For large values of $M$, the level set of the Horn-Jackson Lyapunov function given by the equation
$$
V(x_1, ..., x_n) = M,
$$
is contained between two simplicial sets of the form 
$$
\sum_{i=1}^n x_i = \alpha \left( 1 \pm \frac{C}{\log M} \right)
$$
for some value $\alpha \sim M/log(M)$ and for some for some constant $C>0$ that depends only on $n$.
\end{lem}

\begin{proof}
The proof of this lemma follows from elementary considerations. 
For example we can calculate the {\em minimum} of $V$ and the {\em maximum} of $V$ along an arbitrary simplex of the form 
$
\sum_{i=1}^n x_i = \alpha.
$
for large values of $\alpha$.
(The minimum is attained at a point $x_{min}$ that is proportional with the vector $x_0$ that shows up in the expression of $V$, and the maximum is attained at a point $x_{max}$ where all but one of the coordinates are zero; the nonzero coordinate of $x_{max}$ is the one where $x_0$ has {\em its} smallest coordinate.)
\end{proof}

\newpage

\section*{Acknowledgements} This work was was supported by NIH grant R01GM086881, NSF grants DMS-1412643, DMS-1816238, DMS-2051498, and by a Simons Foundation fellowship. 

\noindent
This work greatly benefitted from research fellowships while visiting collaborators Alicia Dickenstein (at University of Buenos Aires), Heinz Koeppl (at ETH Zurich), Bernd Sturmfels (at the Max Planck Institute for Mathematics in the Sciences), Radek Erban (at Oxford University) and Costin V\^{i}lcu (at the Institute of Mathematics of the Romanian Academy).  

\noindent
This work also benefited from very helpful discussions with participants at the {\em Workshop on the Global Attractor Conjecture}, organized by Elizabeth Gross, Matthew Johnston, Nikki Meshkat, and Bernd Sturmfels in March 2016. 

\noindent
The author also had very helpful and very detailed discussions about aspects of this work with Casian Pantea, Fedor Nazarov, Polly Yu, and Costin V\^{i}lcu.

\bigskip

\bigskip

\bigskip

\bigskip

\bigskip

\bigskip

% \clearpage
\phantomsection

\end{document}